\documentclass[12pt]{article}
\usepackage[margin=24truemm]{geometry}
\usepackage{float}
\usepackage{amsmath,amssymb,bm}
\usepackage{mathtools}
\usepackage{ascmac}
\usepackage{array}
\usepackage{amsthm}
\usepackage[all]{xy}
\usepackage{graphics}
\usepackage[backend=biber, style=ieee-alphabetic, sorting=nyt, maxbibnames=99]{biblatex}
\theoremstyle{plain}
\newtheorem{thm}{Theorem}[section]
\newtheorem*{thm*}{Theorem}
\newtheorem{cor}[thm]{Corollary}
\newtheorem{lem}[thm]{Lemma} 
\newtheorem{conj}[thm]{Conjecture}
\newtheorem{prop}[thm]{Proposition}

\theoremstyle{definition}
\newtheorem{dfn}[thm]{Definition}
\newtheorem{rem}[thm]{Remark}

\begin{document}

\begin{center}

\textbf{On the potential automorphy and the local-global compatibility for the monodromy operators at $p \neq l$ over CM fields}

\vspace{0.5 \baselineskip}

Kojiro Matsumoto

\vspace{0.5 \baselineskip}

\end{center}
    
Abstract: Let $F$ be a CM field. In this paper, we prove the local-global compatibility for cohomological cuspidal automorphic representations of $\mathrm{GL}_n(\mathbb{A}_F)$ at $p \neq l$ by using certain potential automorphy theorems in some cases including higher dimensional cases. Moreover, we also prove the Ramanujan conjecture for the cohomological cuspidal automorphic representations of $\mathrm{GL}_2(\mathbb{A}_F)$.

\tableofcontents

\section{Introduction}

Let $F$ be a CM field (i.e. imaginary CM field or totally real field), $n$ be a positive integer, $F^+$ be the maximal totally real subfield of $F$ and $G_{F}$ be the absolute Galois group of $F$. We fix a prime $l$, an isomorphism of fields $\iota : \overline{\mathbb{Q}}_l \stackrel{\sim}{\rightarrow} \mathbb{C}$ and a regular algebraic (i.e. cohomological) cuspidal automorphic representation $\pi$ of $\mathrm{GL}_n(\mathbb{A}_F)$. 

In \cite{GH}, Harris-Lan-Taylor-Thorne constructed the $l$-adic Galois representation $r_{\iota}(\pi)$ of $G_F$ corresponding to $\pi$. (In \cite{SG}, Scholze also constructed $r_{\iota}(\pi)$ by another method.) It is important to prove the compatibility of this construction and the local Langlands correspondence of $\mathrm{GL}_n(F_v)$ at all finite places $v \nmid l$ of $F$: $$\iota\mathrm{WD}(r_{\iota}(\pi)|_{G_{F_v}})^{F-ss} \cong \mathrm{rec}_{F_v}(\pi_v|\mathrm{det}|_v^{\frac{1-n}{2}}).$$

In \cite{ss}, Ila Varma proved this up to semisimplification. (See Theorem \ref{Ila Varma}.) However, the compatibility for the monodromy operators hasn't been proved in general. 

If $\pi$ satisfies $\pi^c \cong \pi^{\vee}$ (i.e. $\pi$ is conjugate self-dual), a base change of $r_{\iota}(\pi)$ \footnote{Precisely, in  general, only a base change of $r_{\iota}(\pi)^{\otimes 2}$ appears in the $\acute{\mathrm{e}}$tale cohomology of a unitary Shimura variety over $F$.} appears in the $\acute{\mathrm{e}}$tale cohomology of a unitary Shimura variety over $F$ and the correspondence of the monodromy operators at a finite place $v \nmid l$ of $F$ has been proved by showing the weight monodromy conjecture (i.e. the purity) for $r_{\iota}(\pi)|_{G_{F_v}}$.  (See \cite{GLLC}, \cite{CS}, \cite{CM}, \cite{pWM} and 2 of Proposition \ref{purity local-global}.) 

In general cases, such a geometric interpretation of $r_{\iota}(\pi)$ is not known. Therefore, we need other ideas.

Recently, by using certain potential automorphy theorems, Allen-Newton (resp. Yang) proved the local-global compatibility at all $v \nmid l$ in many two-dimensional weight zero crystalline cases (resp. in many two-dimensional ordinary cases). (See \cite{AN} and \cite{YY}.) However, their methods don't work in higher dimensional cases.

\vspace{0.5 \baselineskip}

In this paper, we modify their methods and prove the local-global compatibility at $p \neq l$ in some cases including higher dimensional cases without assuming essential conjugate self-duality. (We will explain more detailed structure of the proof after stating the main theorems of this paper.)

\subsection{Main theorems}

Let $F$ be a CM field and $n$ be a positive integer. The main theorems of this paper are the following.

\begin{thm} (Theorem \ref{ordinary local-global}) \label{1}

    Let $l$ be a prime such that $l > n^2$, $\iota: \overline{\mathbb{Q}}_l \stackrel{\sim}{\longrightarrow} \mathbb{C}$ be an isomorphism of fields and $\pi$ be an $\iota$-ordinary cohomological cuspidal automorphic representation of $\mathrm{GL}_{n}(\mathbb{A}_F)$. (See Definition \ref{iota ordinary definition} for the definition of $\iota$-ordinarity.)
    
    We suppose the following conditions.

    (1) \ The residual representation $\overline{r_{\iota}(\pi)}$ of $r_{\iota}(\pi)$ is fully decomposed generic. (See Definition \ref{strongly decomposed generic} for the definition of the fully decomposed genericity.)
    
    (2) \ $\overline{r_{\iota}(\pi)}|_{G_{F(\zeta_l)}}$ is absolutely irreducible. 
    
    (3) \ $\zeta_l \notin F$.

    Then for all finite places $v \nmid l$ of $F$, we have $$\iota\mathrm{WD}(r_{\iota}(\pi)|_{G_{F_v}})^{F-ss} \cong \mathrm{rec}_{F_v}(\pi_v|\mathrm{det}|_{v}^{\frac{1-n}{2}}).$$

\end{thm}

\begin{rem}

We will prove that for almost all primes $l$, $\iota: \overline{\mathbb{Q}}_l \stackrel{\sim}{\rightarrow} \mathbb{C}$, the representation $\overline{r_{\iota}(\pi)}$ is fully decomposed generic. (See Lemma \ref{decomposed genericity 1}.)

\end{rem}

\begin{thm} (Theorems \ref{self-dual local-global} and \ref{sufficiently regular}) \label{2}

Let $\pi$ be a cohomological cuspidal automorphic representation of $\mathrm{GL}_{n}(\mathbb{A}_F)$ and $\mathcal{L}$ be a set of primes having Dirichlet density one. 
        
We suppose the following conditions.
    
1 \ For any $l \in \mathcal{L}$ and $\iota: \overline{\mathbb{Q}}_l \stackrel{\sim}{\longrightarrow} \mathbb{C}$, the representation $\overline{r_{\iota}(\pi)}|_{G_{F(\zeta_l)}}$ is absolutely irreducible.

2 \ We have one of the following properties.

(1) \ There exists an algebraic Hecke character $\chi : \mathbb{A}_{F^+}^{\times}/(F^+)^{\times} \rightarrow \mathbb{C}^{\times}$ such that $\pi \cong \pi^{\vee} \otimes \chi \circ N_{F/F^+} \circ \mathrm{det}$ and $\chi_v(-1) = \chi_w(-1)$ for all $v, w \mid \infty$. ($\pi^{\vee}$ denotes the dual representation of $\pi$.)

(2) \ The weight $\lambda \in (\mathbb{Z}^{n}_+)^{\mathrm{Hom}(F,\mathbb{C})}$ of $\pi$ satisfies $\lambda_{\tau, i} - \lambda_{\tau, i+1} \ge n - 1$ for all $\tau \in \mathrm{Hom}(F, \mathbb{C})$ and $i=1, \cdots, n-1$. (See the last sentence of {\S} \ref{Notation} for the definition of weight.)

Then there exists a subset $\mathcal{L}_0$ of $\mathcal{L}$ having positive Dirichlet density such that for all $l \in \mathcal{L}_0$, $\iota: \overline{\mathbb{Q}}_l \stackrel{\sim}{\longrightarrow} \mathbb{C}$ and finite places $v \nmid l$ of $F$, we have $$\iota\mathrm{WD}(r_{\iota}(\pi)|_{G_{F_v}})^{F-ss} \cong \mathrm{rec}_{F_v}(\pi_v|\mathrm{det}|_{v}^{\frac{1-n}{2}}).$$
    
\end{thm}

\begin{rem}

Note that if $\mathcal{R}_{\pi}$ is an irreducible very weakly compatible system of $l$-adic representation of $G_F$\footnote{This assumption conjecturally always holds.}, then we obtain a set $\mathcal{L}$ of primes having Dirichlet density one such that for all $l \in \mathcal{L}$ and $\iota: \overline{\mathbb{Q}}_l \stackrel{\sim}{\longrightarrow} \mathbb{C}$, the representation $\overline{r_{\iota}(\pi)}|_{G_{F(\zeta_l)}}$ is absolutely irreducible by Proposition \ref{irreducibility}. (See Definition \ref{automorphic compatible system} for the definition of $\mathcal{R}_{\pi}$. ) In particular, if there exists a finite place $v$ of $F$ such that $\pi_v$ is supercuspidal, then we have such an $\mathcal{L}$. (See Lemma \ref{supercuspidal residually irreducible}.)
        
\end{rem}

In two-dimensional cases, we have the following unconditional results. (Note that if $\pi$ has parallel weight, then the following results have already been proved in \cite[Theorem 1.1]{AN}, \cite[Corollary 7.1.13]{10} and \cite[Theorem A]{pw}.) 

\begin{thm} (Theorems \ref{two-dimensional local-global}, \ref{Ramanujan} and \ref{potential symmetric power})\label{Ramanujan2}

Let $\pi$ be a cohomological cuspidal automorphic representation of $\mathrm{GL}_2(\mathbb{A}_F)$. Then we obtain the following results.

1 \ There exists a set $\mathcal{L}_0$ of primes having positive Dirichlet density such that for all $l \in \mathcal{L}_0$, $\iota : \overline{\mathbb{Q}}_l \stackrel{\sim}{\longrightarrow} \mathbb{C} $ and finite places $v\nmid l $ of $F$, we have $$\iota \mathrm{WD}(r_{\iota}(\pi)|_{G_{F_v}})^{F-ss} \cong \mathrm{rec}_{F_v}(\pi_v|\mathrm{det}|_v^{-\frac{1}{2}}).$$

2 \ $\pi_v$ is tempered for all finite places $v$ of $F$. ( Ramanujan conjecture )

\end{thm}

\begin{rem}

1 \ Note that the purity of irreducible regular very weakly compatible systems of $2$-dimensional $l$-adic representations of $G_F$ can't be proved by the method in this paper unlike \cite[Corollary 7.1.13]{10} and \cite[Theorem C]{pw}. However, we will prove the Ramanujan conjecture for cohomological cuspidal automorphic representations of $\mathrm{GL}_2(\mathbb{A}_F)$ by using the weak Ramanujan property (see \cite[Definition 3.6]{sym} and 2 of Lemma \ref{decomposed genericity 1}) and Proposition \ref{Hecke field 2}, which are automorphic results.

2 \ In {\S} 6.2, we will also prove the potential automorphy and the purity of some $l$-adic representations of $G_F$. (See Definition \ref{definition of purity} for the definition of the purity.)

3 \ By the results in this paper and a variant of \cite[Theorem 1.0.1]{RG}, we can prove the vanishing of the Bloch-Kato Selmer groups of adjoint representations of many potentially automorphic Galois representations which are potentially crystalline at all $l$-adic places. (See Theorem \ref{selmer group}.) 

\end{rem}

\subsection{New ideas in this paper}

\subsubsection{Local-global compatibility}

\vspace{0.5 \baselineskip}

Let $F$ be an imaginary CM field, $n$ be a positive integer, $l$ be a prime and $\iota : \overline{\mathbb{Q}}_l \stackrel{\sim}{\rightarrow} \mathbb{C}$ be an isomorphism of fields. 

In order to explain the method of proving the local-global compatibility of this paper, we first recall a key observation in the proofs of \cite{AN} and \cite{YY}.

\vspace{0.5 \baselineskip}

Let $\pi$ be a cohomological cuspidal automorphic representation of $\mathrm{GL}_n(\mathbb{A}_F)$ and we fix a finite place $v \nmid l$ such that the monodromy operator of $\mathrm{rec}_{F_v}(\pi_v|\mathrm{det}|_v^{\frac{1-n}{2}})$ is not zero. If the monodromy operator of $\iota \mathrm{WD}(r_{\iota}(\pi)|_{G_{F_v}})$ were zero, then after some base change, we would obtain unramified characters $\psi_1, \cdots, \psi_{n-1} : F_v^{\times} \rightarrow \mathbb{C}^{\times}$ such that $\iota\mathrm{WD}(r_{\iota}(\pi)|_{G_{F_v}})^{F-ss} \cong \iota\mathrm{WD}(r_{\iota}(\pi)|_{G_{F_v}})^{ss} \cong \mathrm{rec}_{F_v}(\pi_v|\mathrm{det}|_v^{\frac{1-n}{2}})^{\mathrm{ss}} = \mathrm{rec}_{F_v}((\psi_1 \boxplus \psi_1| \ |_v \boxplus \psi_2 \boxplus \cdots \boxplus \psi_{n-1})|\mathrm{det}|_v^{\frac{1-n}{2}})$ by the local-global compatibility up to semisimplification \cite{ss} (see Theorem \ref{Ila Varma}) and the assumption that the monodromy operator of $\mathrm{rec}_{F_v}(\pi_v|\mathrm{det}|_v^{\frac{1-n}{2}})$ is not zero. If we can prove that there exist a finite extension $F'/F$ of CM fields and a cohomological cuspidal automorphic representation $\Pi$ of $\mathrm{GL}_n(\mathbb{A}_{F'})$ such that $\Pi_u$ is unramified for some $u|v$ and $r_{\iota}(\Pi) \cong r_{\iota}(\pi)|_{G_{F'}}$, then $\mathrm{rec}_{F'_u}(\Pi_u|\mathrm{det}|_u^{\frac{1-n}{2}}) = \mathrm{rec}_{F'_u}(\Pi_u|\mathrm{det}|_{u}^{\frac{1-n}{2}})^{ss} \cong \iota\mathrm{WD}(r_{\iota}(\pi)|_{G_{F'_u}})^{ss} \cong \mathrm{rec}_{F_u'}((\psi_1 \circ N_{F_u'/F_v} \boxplus \psi_1 \circ N_{F_u'/F_v}| \ |_u \boxplus \psi_2 \circ N_{F_u'/F_v} \boxplus \cdots \boxplus \psi_{n-1} \circ N_{F_u'/F_v})|\mathrm{det}|_u^{\frac{1-n}{2}})$ and consequently $\Pi_u \cong \psi_1 \circ N_{F_u'/F_v} \boxplus \psi_1 \circ N_{F'_u/F_v}| \ |_u \boxplus \psi_2 \circ N_{F_u'/F_v} \boxplus \cdots \boxplus \psi_{n-1} \circ N_{F_u'/F_v}$. 

However, this is absurd since $ \psi_1 \circ N_{F_u'/F_v} \boxplus \psi_1 \circ N_{F'_u/F_v}| \ |_u \boxplus \psi_2 \circ N_{F_u'/F_v} \boxplus \cdots \boxplus \psi_{n-1} \circ N_{F_u'/F_v}$ is not generic, while all local components of all cohomological cuspidal representations of $\mathrm{GL}_n(\mathbb{A}_{F'})$ are generic. (See Proposition \ref{purity}.) This implies that the monodromy operator of $\iota \mathrm{WD}(r_{\iota}(\pi)|_{G_{F_v}})$ is not zero. If $n=2$, this is equivalent to $\iota \mathrm{WD}(r_{\iota}(\pi)|_{G_{F_v}})^{F-ss} \cong \mathrm{rec}_{F_v}(\pi_v|\mathrm{det}|_v^{-\frac{1}{2}})$. 

\vspace{0.5 \baselineskip}

However, if $n > 2$, this is far from the local-global compatibility.

In order to prove the local-global compatibility in higher dimensional cases, it is important to generalize the following method to ramified cases.

\vspace{0.5 \baselineskip}

\textbf{(Previous method)} \ When we prove automorphy lifting theorems\footnote{The automorphy lifting theorem roughly says that for a given Galois representation $r$ of $G_F$, if there exists a cohomological cuspidal automorphic representation $\pi$ of $\mathrm{GL}_n(\mathbb{A}_F)$ such that $\overline{r} \cong \overline{r_{\iota}(\pi)}$ and $r$ and $\pi$ satisfy certain technical conditions, then there exists a cohomological cuspidal automorphic representation $\Pi$ of $\mathrm{GL}_n(\mathbb{A}_F)$ such that $r \cong r_{\iota}(\Pi)$.}, we also prove the cuspidal automorphic representation $\Pi$ of $\mathrm{GL}_n(\mathbb{A}_F)$ corresponding to a given $l$-adic Galois representation $r$ of $G_F$ is unramified at $v \nmid l$ if $r$ and a given cuspidal automorphic representation $\pi$ of $\mathrm{GL}_n(\mathbb{A}_F)$ are unramified at $v$. 

\vspace{0.5 \baselineskip}

This is easily proved by setting the level at $v$ to be $\mathrm{GL}_n(\mathcal{O}_{F_v})$. Note that this method has also been used in the proof of the Ramanujan conjecture of cohomological cuspidal automorphic representations of $\mathrm{GL}_2(\mathbb{A}_{F})$ of parallel weight \cite[Corollary 7.1.13]{10} and \cite[Theorem A]{pw}.

In {\S} 2, we will prove the following result. 

\begin{prop} (Proposition \ref{automorphy lifting and local-global compatibility}) \label{interpretation}

Let $\pi$ be a cohomological cuspidal automorphic representation of $\mathrm{GL}_n(\mathbb{A}_F)$ and $v \nmid l$ be a finite place of $F$. Then the following conditions are equivalent.
    
(1) \ $\iota\mathrm{WD}(r_{\iota}(\pi)|_{G_{F_v}})^{F-ss} \cong \mathrm{rec}_{F_v}(\pi_v|\mathrm{det}|_v^{\frac{1-n}{2}})$.
    
(2) \ The monodromy operators of $\iota\mathrm{WD}(r_{\iota}(\pi)|_{G_{F_v}})$ and $\mathrm{rec}_{F_v}(\pi_v|\mathrm{det}|_v^{\frac{1-n}{2}})$ have the same Jordan canonical form.

If $\pi_v$ has an Iwahori fixed vector, there exists a certain parahoric subgroup $K_v$ defined by the Jordan canonical form of the monodromy operator of $\mathrm{WD}(r_{\iota}(\pi)|_{G_{F_v}})$. (See Proposition \ref{automorphy lifting and local-global compatibility} for the precise definition of $K_v$.) such that (1) and (2) are also equivalent to the following condition.
    
(3) \ $\pi_v^{K_v} \neq 0$.
    
\end{prop}

In order to prove this, we will use the result of Ila Varma \cite{ss} (see Theorem \ref{Ila Varma}) : $\iota \mathrm{WD}(r_{\iota}(\pi)|_{G_{F_v}})^{F-ss} \prec \mathrm{rec}_{F_v}(\pi_v|\mathrm{det}|_v^{\frac{1-n}{2}})$, which implies $\iota \mathrm{WD}(r_{\iota}(\pi)|_{G_{F_v}})^{ss} \cong \mathrm{rec}_{F_v}(\pi_v|\mathrm{det}|_v^{\frac{1-n}{2}})^{ss}$ and the monodromy operator of $\iota \mathrm{WD}(r_{\iota}(\pi)|_{G_{F_v}})$ is smaller than the monodromy operator of $\mathrm{rec}_{F_v}(\pi_v|\mathrm{det}|_v^{\frac{1-n}{2}})$ in some sense (see Definition \ref{Weil-Deligne def} for the precise definition of the notion $\prec$.\footnote{For example, in 3 dimensional cases, we have the three conjugacy classes of nilpotent matrices : zero $\gamma_1 = [\begin{pmatrix}
0 & 0 & 0 \\
0 & 0 & 0 \\
0 & 0 & 0
\end{pmatrix}]$, maximal one $\gamma_2 = [\begin{pmatrix}
0 & 1 & 0 \\
0 & 0 & 1 \\
0 & 0 & 0
\end{pmatrix}]$ and the other one $\gamma_3 = [\begin{pmatrix}
0 & 1 & 0 \\
0 & 0 & 0 \\
0 & 0 & 0
\end{pmatrix}]$, where $[ \ ]$ denotes the conjugacy class. Then we have $\gamma_1 \prec \gamma_3 \prec \gamma_2$. }). 

About the proof of (3) $\Rightarrow$ (1), the key result is that for any partition $\gamma$ of $n$, we can construct a parahoric subgroup $K_{\gamma}$ of $\mathrm{GL}_n(F_v)$ such that if an irreducible smooth representation $\sigma$ of $\mathrm{GL}_n(F_v)$ satisfies $\sigma^{K_{\gamma}} \neq 0$, then the monodromy operator of $\mathrm{rec}_{F_v}(\sigma)$ is smaller than the nilpotent matrix corresponding to $\gamma$ in some sense. (See Lemma \ref{parahoric}. For example, $K_{\gamma}$ is equal to $\mathrm{GL}_n(\mathcal{O}_{F_v})$ when $\gamma$ corresponds to the zero monodromy and $K_{\gamma}$ is equal to the Iwahori subgroup when $\gamma$ corresponds to the maximal monodromy. Roughly speaking, $K_{\gamma}$ becomes bigger when the monodromy becomes smaller.) Therefore, by combining the result of Ila Varma, we can construct a parahoric subgroup $K_v$ defined by the Jordan canonical form of the monodromy operator of $\mathrm{WD}(r_{\iota}(\pi)|_{G_{F_v}})$ such that if $\pi_v^{K_v} \neq 0$, then we obtain $\iota\mathrm{WD}(r_{\iota}(\pi)|_{G_{F_v}})^{F-ss} \cong \mathrm{rec}_{F_v}(\pi_v|\mathrm{det}|_v^{\frac{1-n}{2}})$.

\vspace{0.5 \baselineskip}

From this proposition, we obtain the following generalization of the previous method.

\vspace{0.5 \baselineskip}

\textbf{(New method)} \ When we prove automorphy lifting theorems, we also prove that the cuspidal automorphic representation $\Pi$ of $\mathrm{GL}_n(\mathbb{A}_F)$ corresponding to a given $l$-adic Galois representation $r$ of $G_F$ satisfies the local-global compatibility at $v \nmid l$ if a given cuspidal automorphic representation $\pi$ of $\mathrm{GL}_n(\mathbb{A}_F)$ satisfies the local-global compatibility at $v$ and if the monodromy operators of $\mathrm{WD}(r|_{G_{F_v}})$ and $\mathrm{WD}(r_{\iota}(\pi)|_{G_{F_v}})$ have the same Jordan canonical form. (See Theorems \ref{automorphy lifting theorem in crystalline cases} and \ref{ordinary automorphy lifting}. We used the same notations as in \textbf{Previous method}.)

\vspace{0.5 \baselineskip}

In fact, let $K_v$ be the parahoric subgroup of $\mathrm{GL}_n(F_v)$ defined by the Jordan canonical form of the monodromy operator of $\mathrm{WD}(r_{\iota}(\pi)|_{G_{F_v}})$ (see $(3)$ of Proposition \ref{interpretation}). Note that we may assume $\pi_v^{\mathrm{Iw}_v} \neq 0$ since we can replace $F$ by an appropriate solvable CM extension. Thus we have $\pi_v^{K_v} \neq 0$ by $(1) \Rightarrow (3)$ of Proposition \ref{interpretation}. If we can prove that there exists a cohomological cuspidal automorphic representation $\Pi$ of $\mathrm{GL}_n(\mathbb{A}_F)$ such that $r_{\iota}(\Pi) \cong r$ and $\Pi_v^{K_v} \neq 0$, then $\mathrm{WD}(r_{\iota}(\Pi)|_{G_{F_v}})^{F-ss} \cong \mathrm{rec}_{F_v}(\Pi_v|\mathrm{det}|_v^{\frac{1-n}{2}})$ by $(3) \Rightarrow (1)$ of Proposition \ref{interpretation} and by the assumption that the monodromy operator of $\mathrm{WD}(r_{\iota}(\Pi)|_{G_{F_v}}) \cong \mathrm{WD}(r|_{G_{F_v}})$ and $\mathrm{WD}(r_{\iota}(\pi)|_{G_{F_v}})$ have the same Jordan canonical form. 

\vspace{0.5 \baselineskip}

We will explain an outline of the proof of the local-global compatibility in ordinary cases by using this method. (See Theorem \ref{potential ordinary automorphy}.) Note that the automorphy lifting and the potential automorphy of ordinary Galois representations have already been proved in many cases. (See \cite[Theorem 1.1 and 1.4]{opta}, \cite[Corollary 5.5.2]{10} and their refinements Theorem \ref{ordinary residual potential automorphy}, Proposition \ref{ordinary potential autormorphy} and Theorem \ref{ordinary automorphy lifting}.) 

Let $\pi_0$ be an $\iota$-ordinary cohomological cuspidal automorphic representation such that $\overline{r_{\iota}(\pi_0)}$ is absolutely irreducible. (See Definition \ref{iota ordinary definition} for the definition of the $\iota$-ordinarity.) We fix a finite place $v \nmid l$. By the result \cite[Theorem 1.1]{opta} (see Theorem \ref{ordinary residual potential automorphy}), we obtain a finite CM extension $F'/F$ linearly disjoint from $\overline{F}^{\mathrm{Ker}(\overline{r_{\iota}(\pi_0)})}$ and an $\iota$-ordinary cohomological cuspidal automorphic representation $\pi$ of $\mathrm{GL}_n(\mathbb{A}_{F'})$ such that $\overline{r_{\iota}(\pi)} \cong \overline{r_{\iota}(\pi_0)}|_{G_{F'}}$ and $r_{\iota}(\pi)$ is a direct summand of the $\acute{\mathrm{e}}$tale cohomology of a smooth hypersurface in a projective space over $F'$. Note that $\pi$ satisfies the local-global compatibility for all $v \nmid l$ since $\mathrm{WD}(r_{\iota}(\pi)|_{G_{F'_v}})$ is pure. (See Proposition \ref{purity local-global}.)

If we could take $\pi$ such that the monodromy operators of $\mathrm{WD}(r_{\iota}(\pi)|_{G_{F'_w}})$ and $\mathrm{WD}(r_{\iota}(\pi_0)|_{G_{F'_w}})$ have the same Jordan canonical form for some $w \mid v$, then by using the new method and the ordinary automorphy lifting theorem (Theorem \ref{ordinary automorphy lifting}), we would obtain the $\iota$-ordinary cohomological cuspidal automorphic representation $\Pi$ of $\mathrm{GL}_n(\mathbb{A}_{F'})$ such that $r_{\iota}(\pi_0)|_{G_{F'}} \cong r_{\iota}(\Pi)$ and $\iota \mathrm{WD}(r_{\iota}(\Pi)|_{G_{F'_w}})^{F-ss} \cong \mathrm{rec}_{F'_w}(\Pi_w|\mathrm{det}|_w^{\frac{1-n}{2}})$ and this would imply $\iota \mathrm{WD}(r_{\iota}(\pi_0)|_{G_{F_v}})^{F-ss} \cong \mathrm{rec}_{F_v}(\pi_{0, v}|\mathrm{det}|_v^{\frac{1-n}{2}})$. (See Proposition \ref{potential automorphy and local-global compatibility}.)

However, finding such $\pi$ is difficult in general at least for the author. (Recently, Campo posted the paper \cite{DworkCampo} studying this problem.) In order to overcome this difficulty, we employ the tensor product trick by Harris (see \cite{TPT}) and we use Proposition \ref{conjugate self-dual ordinary}, which is a modification of a strong result \cite[Proposition 3.2.1]{CW} in essentially conjugate self-dual cases. We take $\pi$ such that $\pi_w$ is unramified for some $w|v$. Moreover, we consider an $n$-dimensional ordinary conjugate self-dual automorphic Galois representation $\mathrm{Ind}^{G_{F'}}_{G_M}\phi$. Then Proposition \ref{conjugate self-dual ordinary} provides a conjugate self-dual ordinary automorphic lifting $r_{\iota}(\pi_1)$ of $\overline{\mathrm{Ind}^{G_{F'}}_{G_M}\phi}$ such that the monodromy operators of $\mathrm{WD}(r_{\iota}(\pi_0)|_{G_{F'_w}})$ and $\mathrm{WD}(r_{\iota}(\pi_1)|_{G_{F'_w}})$ have the same Jordan canonical form.\footnote{Precisely, $\mathrm{Ind}^{G_{F'}}_{G_M}\phi$ and $r_{\iota}(\pi_1)$ satisfy some technical conditions.} 

We compare two representations $r_{\iota}(\pi_0)|_{G_{F'}} \otimes \mathrm{Ind}^{G_{F'}}_{G_M}\phi$ and $r_{\iota}(\pi) \otimes r_{\iota}(\pi_1)$. Note that $r_{\iota}(\pi) \otimes r_{\iota}(\pi_1)$ corresponds to the cuspidal automorphic representation satisfying the local-global compatibility at $w$ since $\overline{r_{\iota}(\pi) \otimes r_{\iota}(\pi_1)} \cong \overline{r_{\iota}(\pi) \otimes \mathrm{Ind}^{G_{F'}}_{G_{M}} \phi} \cong \overline{\mathrm{Ind}^{G_{F'}}_{G_{M}} (r_{\iota}(\pi)|_{G_{M}}\phi)}$ is automorphic and $\mathrm{WD}((r_{\iota}(\pi) \otimes r_{\iota}(\pi_1))|_{G_{F'_w}})$ is pure. Moreover, the monodromy operators of $\mathrm{WD}((r_{\iota}(\pi_0)|_{G_{F'}} \otimes \mathrm{Ind}^{G_{F'}}_{G_M}\phi)|_{G_{F'_{w}}})$ and $\mathrm{WD}((r_{\iota}(\pi) \otimes r_{\iota}(\pi_1))|_{G_{F_w'}})$ have the same Jordan canonical form.

Thus by the new method and by the ordinary automorphy lifting theorem, $r_{\iota}(\pi_0)|_{G_{F'}} \otimes \mathrm{Ind}^{G_{F'}}_{G_M}\phi$ corresponds to a cuspidal automorphic representation satisfying the local-global compatibility at $w$. (See Theorem \ref{ordinary automorphy lifting}.) This implies $\iota \mathrm{WD}(r_{\iota}(\pi_0)|_{G_{F_{v}}})^{F-ss} \cong \mathrm{rec}_{F_{v}}(\pi_{0, v}|\mathrm{det}|_v^{\frac{1-n}{2}})$. (See Lemma \ref{purity tensor} for details.)

\subsubsection{Potential automorphy}

As we have seen above in ordinary cases, in order to prove the local-global compatibility for some representations, we need certain potential automorphy theorems for such representations. In this paper, we study the following problem.

\vspace{0.5 \baselineskip}

(*) The potential automorphy of $l$-adic Galois representations of $G_F$ which are Hodge-Tate regular and potentially diagonalizable at all $l$-adic places.

\vspace{0.5 \baselineskip}

For $v \mid l$, we say that a continuous representation $\rho : G_{F_v} \rightarrow \mathrm{GL}_n(\mathcal{O}_{\overline{\mathbb{Q}}_l})$ is potentially diagonalizable if there exists a finite extension $F'_v/F_v$ such that $\rho|_{G_{F_v'}}$ is crystalline and contained in an irreducible component of the crystalline lifting ring of $\overline{\rho}|_{G_{F'_v}}$ containing a direct sum of crystalline characters of $G_{F_v'}$. If $\rho$ is Fontaine-Laffaille, then $\rho$ is potentially diagonalizable. (See Proposition \ref{trace irreducible component}.) Note that the irreducible $n$-dimensional algebraic\footnote{We say that an $l$-adic Galois representation $r$ of $G_F$ is algebraic if $r$ is unramified at almost all finite places of $F$ and de Rham at all $l$-adic places of $F$.} $l$-adic Galois representations of $F$ which are Hodge-Tate regular at all $l$-adic places, are expected to correspond to the cohomological cuspidal automorphic representations of $\mathrm{GL}_n(\mathbb{A}_F)$. (See Theorem \ref{polarizable local-global compatibility} and \ref{Lambert 2}.)

In essentially conjugate self-dual cases, we have already many general results about (*). (See \cite{CW}.) In this paper, we will prove (*) in many essentially self-dual cases and many sufficiently regular weight cases. (See Theorems \ref{potential diagonalizable automorphy} and \ref{potential tensor automorphy}.)

In essentially conjugate self-dual cases, one of the most fundamental works about (*) is a crystalline automorphy lifting theorem \cite[Theorem 7.1]{small}. Now, we can prove it without assuming any conjugate self-duality thanks to the recent works \cite[Theorem 1.3]{MEI} on properties of automorphic Galois representations at $l$-adic places and the generalization \cite{ade} of the work \cite{small} about patching arguments. (See Theorems \ref{Caraiani-Newton}, \ref{Caraiani-Newton 2}, \ref{Lambert 2}, \ref{Patching argument} and \ref{automorphy lifting theorem in crystalline cases}.) By this crystalline automorphy lifting theorem, in order to prove the potential automorphy of a continuous representation $r : G_F \rightarrow \mathrm{GL}_n(\mathcal{O}_{\overline{\mathbb{Q}}_l})$, it suffices to construct a finite CM extension $F'/F$ and a cohomological cuspidal automorphic representations $\pi$ of $\mathrm{GL}_n(\mathbb{A}_{F'})$ satisfying the following conditions. (See Theorem \ref{automorphy lifting theorem in crystalline cases}. In order to prove useful results about the following conditions 2 and 3, we will study irreducible components of local lifting rings in section 3.) \footnote{For simplicity, we ignore some properties, e.g. the irreducibility and decomposed genericity of $\overline{r}|_{G_{F'}}$. See Definition \ref{decomposed generic} for the definition of the decomposed genericity.} 

\vspace{0.5 \baselineskip}

\textbf{Conditions for the crystalline automorphy lifting theorem.}

1 \ $\overline{r}|_{G_{F'}} \cong \overline{r_{\iota}(\pi)}$.

2 \ For each $v \mid l$, $r|_{G_{F'_v}}$ and $r_{\iota}(\pi)|_{G_{F_v'}}$ are contained in the same irreducible component of the crystalline lifting ring of $\overline{r}|_{G_{F_v'}}$.

3 \ For each $v \nmid l$, $r|_{G_{F'_v}}$ and $r_{\iota}(\pi)|_{G_{F_v'}}$ are contained in the same irreducible component of the universal lifting ring of $\overline{r}|_{G_{F_v'}}$. In section 3, we will prove that this is essentially implied by the condition that the monodromy operators of $\mathrm{WD}(r|_{G_{F'_v}})$ and $\mathrm{WD}(r_{\iota}(\pi)|_{G_{F_v'}})$ have the same Jordan canonical form. (See Proposition \ref{monodromy type irreducible component}.)

4 \ For each $v \nmid l$ of $F$, the representation $r_{\iota}(\pi)|_{G_{F_v'}}$ is contained in the unique irreducible component of the universal lifting rings of $\overline{r}|_{G_{F_v'}}$. This is known to be a consequence of the local-global compatibility of $\pi$ at $v$. (See Lemma \ref{regular point} and Proposition \ref{purity}.)

\vspace{0.5 \baselineskip}

Note that the conditions 3 and 4 are also needed for the new method.

The main difficulty is to construct $F'$ and $\pi$ which satisfy the conditions 1 and 2.

About the condition 1, in general, we only have the result \cite[Theorem 1.1]{opta}, which proved that any residual representation is potentially isomorphic to the residual representation of an ordinary automorphic Galois representation with consecutive weights. If the residual representation is essentially conjugate self-dual, then we already have many strong results. (See \cite[Proposition 3.3.1 and Corollary 4.5.3]{CW} and their modifications Proposition \ref{potential conjugate} and Proposition \ref{conjugate self-dual ordinary}. For example, these results work without any restriction on the weights and also work in some non-ordinary cases.)

About the condition 2, in conjugate self-dual cases, it has been a very useful property that for any essentially conjugate self-dual Galois representation $r$ of $G_F$ which is potentially diagonalizable and Hodge-Tate regular at all $l$-adic places, we can construct an essentially conjugate self-dual Galois representation $s$ which is induced from a character (and thus is automorphic) such that $\overline{r}$ and $\overline{s}$ are isomorphic at any $l$-adic places and are contained in the same irreducible component of the crystalline lifting rings at any $l$-adic places after replacing $F$ by an appropriate solvable CM extension.  (See \cite[proof of Proposition 4.1.1]{CW} and Lemma \ref{diagonalizable induction character}.) Without assuming the conjugate self-duality, this property doesn't necessarily hold\footnote{For example, when $r$ is the Tate module of an elliptic curve over $F$ which has good ordinary reduction at $v$ and supersingular reduction at $v^c$ for some $v \mid l$.} and the potential automorphy of non-conjugate self-dual Galois representations has previously been proved only in relatively special cases. (See \cite{pw} and \cite{10} for symmetric powers of parallel regular weight two-dimensional representations and see \cite{opta} for ordinary representations.) However, we will prove the following proposition in {\S}6 and consequently in many cases, we have this property. (Precisely, we will use this property in the proof of Theorem \ref{potential tensor automorphy}. In the proof of Theorem \ref{potential diagonalizable automorphy}, we will use Proposition \ref{potential conjugate}. This proposition is a modification of a stronger result \cite[Corollary 4.5.3]{CW} in conjugate self-dual cases, which is an application of this property.)

\vspace{0.5 \baselineskip}

\begin{prop} (Lemma \ref{Hecke field 2} and Lemma \ref{Hecke field}) \label{conjugate self-duality}

Let $\mathcal{R}:=( M, S, \{Q_v(X)\}_{v \notin S}, \{ r_{\lambda} : G_{F} \rightarrow \mathrm{GL}_n(\overline{M}_{\lambda}) \}_{\lambda}, \{ H_{\tau} \})$ be a rank $n$ very weakly compatible system of $l$-adic Galois representations of $G_F$ which is pure or comes from a cohomological cuspidal representation of $\mathrm{GL}_n(\mathbb{A}_F)$ and $\mathcal{L}$ be a set of primes having Dirichlet density one such that for all $l \in \mathcal{L}$ and $\lambda \mid l$, the representation $\overline{r_{\lambda}}$ is absolutely irreducible. (See {\S} 6.2 for the definition of compatible systems of $l$-adic Galois representations.)

Then there exist an integer $w$\footnote{When $\mathcal{R}$ is pure, this $w$ is equal to the weight.} and a subset $\mathcal{L}_0$ of $\mathcal{L}$ having positive Dirichlet density such that for all $l \in \mathcal{L}_0$, $\lambda|l$ and $v \mid l$, the representation $r_{\lambda}|_{G_{F_v}}$ is potentially diagonalizable and there exists a finite extension $F_v'/F_v$ such that $r_{\lambda}^c|_{G_{F'_v}}$ and $r_{\lambda}^{\vee}\varepsilon_l^{-w}|_{G_{F_{v}'}}$ are contained in the same irreducible components of the crystalline lifting ring of $\overline{r_{\lambda}^c}|_{G_{F'_v}}$. \end{prop}

We explain an outline of the proof of this proposition. We may assume that for all $l \in \mathcal{L}$ and $\lambda \mid l$, $r_{\lambda}$ is Fontaine-Laffaille at all $l$-adic places and consequently $r_{\lambda}$ is potentially diagonalizable at all $l$-adic places. About the second property, we use the property $^c\mathcal{R} \cong \mathcal{R}^{\vee} \otimes \{ \varepsilon_l^{-w} \}$\footnote{$^c\mathcal{R}$ denotes the compatible system obtained by twisting the characteristic polynomial of $\mathcal{R}$ by $c$.} for some integer $w$. This follows from the assumptions that $\mathcal{R}$ is pure or comes from a cohomological cuspidal representation of $\mathrm{GL}_n(\mathbb{A}_F)$. (See Lemma \ref{example 1} and the proof of Lemma \ref{Hecke field}.) We take the subset $\mathcal{L}_0$ of $\mathcal{L}$ consisting of all $l \in \mathcal{L}$ satisfying that the complex conjugation $c$ is $l$-adically continuous on $F$ and $M$. (We may assume that $M$ is a CM field. See Lemma \ref{example 1} and \cite[Lemma 1.2]{irr}.) Fix $l \in \mathcal{L}_0$, $\lambda \mid l$ and $v \mid l$. By using the fact that $r_{\lambda}|_{G_{F_v}}$ is potentially diagonalizable, it suffices to prove that $r_{\lambda}^c|_{G_{F_v}}$ and $^cr_{\lambda}|_{G_{F_v}}( = r_{\lambda}^{\vee}\varepsilon_l^{-w}|_{G_{F_{v}}})$\footnote{Precisely, we need to take an extension $\tilde{c}$ of $c$ to an automorphism of $\overline{M}_{\lambda}$ over $\mathbb{Q}_l$ and we need to replace $^cr_{\lambda}$ by $\overline{M}_{\lambda} \otimes_{\tilde{c}, \overline{M}_{\lambda}} r_{\lambda}$.} are potentially contained in the same irreducible component of the crystalline lifting ring in one-dimensional cases. (See Proposition \ref{trace irreducible component}.)

\vspace{0.5 \baselineskip}

We will see outlines of the proof of the potential automorphy of essentially self-dual Galois representations of $G_F$ which are Hodge-Tate regular and potentially diagonalizable at all $l$-adic places. (See Theorem \ref{potential diagonalizable automorphy}.)

Let $r : G_{F} \rightarrow \mathrm{GL}_n(\mathcal{O}_{\overline{\mathbb{Q}}_l})$ be an $n$-dimensional $l$-adic Galois representation satisfying the following conditions.

(i) For all $v \mid l$, $r|_{G_{F_v}}$ is Hodge-Tate regular and potentially diagonalizable.

(ii) There exists an integer $w$ such that there exists a finite extension $F_v'/F_v$ such that $r^c|_{G_{F_v'}}$ and $r^{\vee} \varepsilon_l^{-w}|_{G_{F_v'}}$ are contained in the crystalline lifting ring of $\overline{r^c}|_{G_{F_v'}}$. (Note that this assumption holds in many cases by Proposition \ref{conjugate self-duality}.)

(iii) $r$ is unramified at almost all finite places of $F$.

(iv) $r \cong r^{\vee} \varepsilon_l^{-w}$.

(v) $\mathrm{Ind}^{G_{F^+}}_{G_{F}}\overline{r}$ is absolutely irreducible.

A basic observation is that $\mathrm{Ind}^{G_{F^+}}_{G_{F}}r$ is an essentially self-dual representation of the totally real field $F^+$ and consequently, conjugate self-dual. However, $\mathrm{Ind}_{G_{F}}^{G_{F+}}r$ is not necessarily Hodge-Tate regular. (In fact, if $r$ is comes from a proper smooth variety over $F$ or a cohomological cuspidal automorphic representation of $\mathrm{GL}_n(\mathbb{A}_F)$, we can prove that $\mathrm{Ind}_{G_{F}}^{G_{F+}}r$ is not Hodge-Tate regular.) In order to overcome this difficulty, we consider $\mathrm{Ind}_{G_{F}}^{G_{F+}}(r \theta)$ for some algebraic $l$-adic character $\theta : G_{F} \rightarrow \overline{\mathbb{Q}}_l^{\times}$ such that $\sigma := \mathrm{Ind}_{G_{F}}^{G_{F+}}(r \theta)$ is Hodge-Tate regular at all $l$-adic places and $\overline{\theta}$ and $\theta \theta^c$ are trivial. (We can take such a character by Proposition \ref{character}.) Then $\overline{\mathrm{Ind}_{G_{F}}^{G_{F+}}(r \theta)} = \overline{\mathrm{Ind}_{G_{F}}^{G_{F+}}r}$ is essentially self-dual. 

Then $\sigma$ satisfies the following properties after replacing $F^+$ by an appropriate solvable CM extension.

$(a)$ \ $\overline{\sigma} \cong \overline{\sigma}^c \cong \overline{\sigma^{\vee}\varepsilon_l^{-w}}$.

$(b)$ \ For all $v \mid l$, $\sigma|_{G_{F^+_v}}$ is potentially diagonalizable, Hodge-Tate regular and $\sigma|_{G_{F_v^+}} = \sigma^c|_{G_{F^+_v}}$ and $\sigma^{\vee}\varepsilon_l^{-w}|_{G_{F^+_v}}$ are contained in the same irreducible component of the crystalline lifting ring of $\overline{\sigma}^c|_{G_{F^+_v}}$. (This follows from the assumption (ii).)

$(c)$ \ For all $v \nmid l$, the monodromy operators of $\mathrm{WD}(\sigma|_{G_{F^+_v}})$ = $\mathrm{WD}(\sigma^c|_{G_{F^+_v}})$ and $\mathrm{WD}(\sigma^{\vee}\varepsilon_l^{-w}|_{G_{F^+_v}})$ have the same Jordan canonical form.

From the above conjugate self-dual properties, the strong result Proposition \ref{potential conjugate} in the conjugate self-dual case gives us a finite CM extension $E/F^+$ and an essentially conjugate self-dual cohomological cuspidal automorphic representation $\pi$ of $\mathrm{GL}_{2n}(\mathbb{A}_{E})$ satisfying the above conditions $1 \sim 4$ of the crystalline automorphy lifting theorem. Thus, we obtain the automorphy of $\sigma|_{G_{E}}$ and the local-global compatibility at non-$l$-adic places for the cuspidal automorphic representation corresponding to $\sigma|_{G_{E}}$ by the new method. This implies the automorphy of $r|_{G_{FE}}$ and the local-global compatibility at non-$l$-adic places for the cuspidal automorphic representation $\Pi$ corresponding to $r|_{G_{FE}}$. (See Lemma \ref{local global induction 2} for details.) 

Note that for a place $v \nmid l$, if $r$ is unramified at $v$, then $\Pi_u$ is unramified for all $u \mid v$ by $\iota\mathrm{WD}(r|_{G_{(FE)_u}})^{F-ss} \cong \mathrm{rec}_{(FE)_u}(\Pi_u |\mathrm{det}|_u^{\frac{1-n}{2}})$. However, the previous method is not enough to prove this property since $\mathrm{Ind}^{G_{F+}}_{G_{F}}r\theta$ is ramified at the place lying below $v$ if $r$ is ramified at $v^c$. For this reason, we also need the new method in the proof of the Ramanujan conjecture for cohomological cuspidal automorphic representations of $\mathrm{GL}_2(\mathbb{A}_F)$.\footnote{Actually, this is the initial motivation to introduce the method of this paper for the author.} (See Theorem \ref{Ramanujan}.)

\vspace{0.5 \baselineskip}

\textbf{Comparison with the paper \cite{pw}}

\vspace{0.5 \baselineskip}

Note that the method of \cite{pw} to prove the Ramanujan conjecture for cohomological cuspidal automorphic representations of $\mathrm{GL}_2(\mathbb{A}_{F})$ in the parallel weight case is based on refining our understanding on local properties of key ingredients used in the potential automorphy argument such as crystalline lifting rings and the Dwork family (see \cite[Theorem D and Proposition 4.2.6]{pw}) while the method of this paper to prove the Ramanujan conjecture for the cohomological cuspidal automorphic representations of $\mathrm{GL}_2(\mathbb{A}_{F})$ is using Proposition \ref{conjugate self-duality} and strong results which were proved in the conjugate self-dual case \cite{CW}.

\subsection{Organization of this paper}

We sketch the content of each section. In section 2, we find the simple interpretation of the local-global compatibility at non-$l$-adic places. (See Proposition \ref{interpretation}.) In section 3, we examine irreducible components of local lifting rings. At non-$l$-adic finite places, we will prove that the irreducible components of the unipotently ramified lifting rings are determined by the monodromy operators in residually trivial cases. (See Proposition \ref{monodromy type irreducible component}. This is the condition 3 of the crystalline automorphy lifting theorem.) At $l$-adic places, we will prove Proposition \ref{trace irreducible component}, which is a refinement of \cite[Lemma 1.4.3]{CW} and is crucial for showing Proposition \ref{conjugate self-duality}. In section 4, we prove automorphy lifting theorems and the local-global compatibility at non-$l$-adic places for cuspidal automorphic representations corresponding to given Galois representations in crystalline and ordinary cases. In section 5, we prove the potential automorphy of some Galois representations and the local-global compatibility at non-$l$-adic places for cuspidal automorphic representations corresponding to them in essentially self-dual potentially diagonalizable cases, ordinary cases and sufficiently regular weight potentially diagonalizable cases.

\subsection{Notation} \label{Notation}

For a field $K$, we write $\overline{K}$ for a separable closure of $K$. We put $G_K:=\mathrm{Gal}(\overline{K}/K)$. For an integer $n \nmid \mathrm{char} \, K$, we write $\zeta_n \in \overline{K}$ for a primitive $n$-th root of unity. For a quadratic Galois extension $L/K$ of fields, we write $\delta_{L/K}$ for the generator of $\mathrm{Hom}(\mathrm{Gal}(L/K), \{\pm 1\})$. 

For a local ring $A$, we write $\mathfrak{m}_A$ for the maximal ideal of $A$ and $\mathbb{F}_A$ for the residue field of $A$. 

For a scheme $X$ and $x \in X$, we write $k(x)$ for the residue field of $X$ at $x$. Let $A$ be a commutative ring and $X$ be a scheme over $A$. For an $A$-algebra $B$, $X(B)$ denotes the set of all $A$-morphisms $\mathrm{Spec} \, B \rightarrow X$.

Let $B_n$ denote the subgroup of $\mathrm{GL}_n$ consisting of all upper triangular matrices and $T_n$ denote the subgroup of $\mathrm{GL}_n$ consisting of all diagonal matrices of $\mathrm{GL}_n$.

For a profinite group $G$, a positive integer $n$, a prime $l$, a continuous representation $r: G \rightarrow \mathrm{GL}_n(\overline{\mathbb{Q}}_l)$ is also called an $l$-adic representation of $G$. For a continuous representation $r:G \rightarrow \mathrm{GL}_n(\mathcal{O}_{\overline{\mathbb{Q}}_l})$, we write $\overline{r}$ for the composition $G \rightarrow \mathrm{GL}_n(\mathcal{O}_{\overline{\mathbb{Q}}_l}) \rightarrow \mathrm{GL}_n(\overline{\mathbb{F}_l})$. For a continuous representation $r: G \rightarrow \mathrm{GL}_n(\overline{\mathbb{Q}}_l)$, we write $\overline{r}$ for the semisimplification of $\overline{grg^{-1}}: G \rightarrow \mathrm{GL}_n(\overline{\mathbb{F}_l})$ for some $g \in \mathrm{GL}_n(\overline{\mathbb{Q}}_l)$ such that $\mathrm{Im}(grg^{-1}) \subset \mathrm{GL}_n(\mathcal{O}_{\overline{\mathbb{Q}}_l})$. The isomorphism class of $\overline{r}$ is independent of the choice of $g$. For an open subgroup $H$ of $G$ and a continuous representation $r : H \rightarrow \mathrm{GL}_n(\overline{\mathbb{Q}}_l)$, we put $\mathrm{Ind}^G_Hr:=\{ f : G \rightarrow \overline{\mathbb{Q}}_l^{\oplus n} \mid f(hg) = r(h)f(g) \ \mathrm{for \ any \ } g \in G,\ h \in H \}$. This become a continuous representation of $G$ by $(g_1f)(g_2):=f(g_2g_1)$ for any $g_1, g_2 \in G$. 

For a group $G$ and a prime $l$, let $G(l)$ denote the pro-$l$-completion of $G$.

For a locally profinite group $G$, a compact open subgroup $K$ of $G$, an open submonoid $\Delta$ of $G$ such that $K \Delta K = \Delta$ and a commutative ring $R$, we write $\mathcal{H}(\Delta, K)_{R}$ for the set of all function $f : \Delta \rightarrow R$ such that $f(k_1gk_2)=f(g)$ for any $g \in \Delta$, $k_1, k_2 \in K$ and $\{ g \in \Delta \mid f(g) \neq 0 \}$ is compact. For $g \in \Delta$, we write $[KgK]$ for the element of $\mathcal{H}(\Delta, K)_R$ satisfying $[KgK](h) = 1$ for $h \in KgK$ and $[KgK](h) = 0$ for $h \notin KgK$. Then $\mathcal{H}(\Delta, K)_{R}$ become an $R$-algebra by the formula $f_1 * f_2(g):=\sum_{x \in G/K} f_1(x)f_2(x^{-1}g)$.

For a (not necessary commutative) ring $S$, we write $D(S)$ for the derived category of $S$-modules. For a positive integer $n$, we write $\mathfrak{S}_n$ for the symmetric group of degree $n$.

\underline{Local theory}

For a prime $p$ and a finite extension $K$ of $\mathbb{Q}_p$, we write $\mathcal{O}_K$ for the ring of integers of $K$, $\mathbb{F}_K$ for the residue field of $\mathcal{O}_K$, $\mathrm{Frob}_K \in W_K$ for a geometric Frobenius lift and $I_K$ for the kernel of the natural map $G_K \rightarrow G_{\mathbb{F}_K}$. We put $q_K:=|\mathbb{F}_K|$ (the order of $\mathbb{F}_K$) and we write $W_K$ for the Weil group of $K$. 

Let $\mathrm{Art}_K: K^{\times} \stackrel{\sim}{\longrightarrow} W_K^{ab}$ denote the local Artin map normalized to send primes to geometric Frobenius lifts. We write $| \ |_K: K^{\times} \rightarrow \mathbb{R}_{>0}$ for the $p$-adic valuation on $K^{\times}$ normalized to send primes to $q_K^{-1}$. We put $|| \ ||_K:=| \mathrm{Art}^{-1}_K |_K: W_{K} \rightarrow \mathbb{R}_{>0}$.

In this paper, we always assume that all smooth representations of $\mathrm{GL}_n(K)$ and Weil-Deligne representations of $W_K$ are defined over $\mathbb{C}$ or $\overline{\mathbb{Q}}_l$ for some prime $l$. For an irreducible smooth representation $\pi$ of $\mathrm{GL}_n(K)$, we write $\omega_{\pi}$ for the central character of $\pi$. For a Weil-Deligne representation $r$, we write $r^{F-ss}$ for the Frobenius semisimplification of $r$. 

For a prime $l \neq p$ and a continuous representation $r: W_K \rightarrow \mathrm{GL}_n(\overline{\mathbb{Q}}_l)$, we write $\mathrm{WD}(r)$ for the Weil-Deligne representation of $W_K$ over $\overline{\mathbb{Q}}_l$ corresponding to $r$.

For a prime $l = p$ and a Hodge-Tate representation $r: G_K \rightarrow \mathrm{GL}_n(\overline{\mathbb{Q}}_l)$ and $\tau \in \mathrm{Hom}_{\mathbb{Q}_l}(K,\overline{\mathbb{Q}}_l)$, we write $\mathrm{HT}_{\tau}(r):=[\underbrace{i, \cdots \cdots,  i}_{\mathrm{dim}_{\overline{\mathbb{Q}}_l}(r \otimes_{\tau, K} \widehat{\overline{K}}(i))^{G_K}}]_i \in (\mathbb{Z}^n/\mathfrak{S}_n)$ for $\tau$-Hodge-Tate weight of $r$. For example, the $l$-adic cyclotomic character $\varepsilon_l$ satisfies $\mathrm{HT}_{\tau}(\varepsilon_l)=-1$ for all $\tau$. We say that $r$ is Hodge-Tate regular if $r$ is Hodge-Tate and $\mathrm{HT}_{\tau}(r)$ consists of $n$-distinct integers for all $\tau \in \mathrm{Hom}_{\mathbb{Q}_l}(K, \overline{\mathbb{Q}}_l)$. For a de Rham representation $r : G_K \rightarrow \mathrm{GL}_n(\overline{\mathbb{Q}}_l)$, we write $\mathrm{WD}(r)$ for the Weil-Deligne representation of $W_K$ over $\overline{\mathbb{Q}}_l$ corresponding to $r$.

Let $\mathrm{rec}_K$ denote the local Langlands correspondence in \cite{LLC}. For a parabolic subgroup $P$ of $\mathrm{GL}_n$ and a smooth representation $\pi$ of a Levi-subgroup $M(K)$ of $P(K)$, we write n-Ind$^{\mathrm{GL}_n(K)}_{P(K)}\pi$ for the normalized induction of $\pi$ to $\mathrm{GL}_n(K)$. For positive integers $n_1, n_2$ such that $n_1 + n_2 = n$ and irreducible smooth representations $\pi_1, \pi_2$ of $\mathrm{GL}_{n_1}(K), \mathrm{GL}_{n_2}(K)$, we define the irreducible smooth representation $\pi_1 \boxplus \pi_2$ of $\mathrm{GL}_n(K)$ by $\mathrm{rec}_K(\pi_1 \boxplus \pi_2) = \mathrm{rec}_K(\pi_1) \oplus \mathrm{rec}_K(\pi_2)$. For positive integers $m|n$ and a supercuspidal representation $\pi$ of $\mathrm{GL}_{\frac{n}{m}}(K)$, we write $\mathrm{St}_m(\pi)$ for the generalized Steinberg representation associated to $\pi$. This is a unique irreducible quotient of n-$\mathrm{Ind}^{\mathrm{GL}_n(K)}_{P(K)}(\pi|\mathrm{det}|_K^{\frac{1-m}{2}} \boxtimes \cdots \boxtimes \pi|\mathrm{det}|_K^{\frac{m-1}{2}})$. For an irreducible Weil-Deligne representation $s$, we define $\mathrm{Sp}_m(s)$ by the condition $\mathrm{Sp}_m(\mathrm{rec}_{K}(\pi)) = \mathrm{rec}_K(\mathrm{St}_m(\pi)$) for any supercuspidal representation $\pi$ of $\mathrm{GL}_{\frac{n}{m}}(K)$. 

For a finite extension $L/K$ of degree $d$ and an irreducible smooth representation $\pi$ of $\mathrm{GL}_n(K)$, we write $\mathrm{BC}_{L/K}(\pi)$ for the irreducible smooth representation of $\mathrm{GL}_n(L)$ defined by $\mathrm{rec}_L(\mathrm{BC}_{L/K}(\pi)) := \mathrm{rec}_{K}(\pi)|_{W_L}$. For an irreducible smooth representation $\pi$ of $\mathrm{GL}_n(L)$, we write $\mathrm{AI}_{L/K}(\pi)$ for the irreducible smooth representation of $\mathrm{GL}_{dn}(K)$ defined by $\mathrm{rec}_{K}(\mathrm{AI}_{L/K}(\pi)) := \mathrm{Ind}^{W_{K}}_{W_L}\mathrm{rec}_{L}(\pi)$.

\underline{Global theory}

For a number field $F$ and a finite place $v$ of $F$, we write $F_v$ for the completion of $F$ by $v$, $\mathbb{F}_v$ for the residue field of $\mathcal{O}_{F_v}$, $\varpi_v$ for a prime of $\mathcal{O}_{F_v}$, $q_v$ for $|\mathbb{F}_v|$ and $\mathrm{Frob}_v \in W_{F_v}$ for a geometric Frobenius lift. Let $\mathbb{A}_F$ denote the adele ring of $F$ and $|| \ ||_F$ denote the adelic norm on $\mathbb{A}_F^{\times}$.

For a number field $F$, a prime $l$, a positive integer $n$ and a continuous representation $r: G_F \rightarrow \mathrm{GL}_n(\overline{\mathbb{Q}}_l)$, we say that $r$ is algebraic if $r$ is unramified at almost all finite places of $F$ and $r$ is de Rham at all $l$-adic places. In this case, for any $v \mid l$ and $\tau \in \mathrm{Hom}_{\mathbb{Q}_l}(F_v, \overline{\mathbb{Q}}_l) \subset \mathrm{Hom}(F, \overline{\mathbb{Q}}_l)$, we put $\mathrm{HT}_{\tau}(r) := \mathrm{HT}_{\tau}(r|_{G_{F_v}})$ and we say that $r$ is Hodge-Tate regular if $r|_{G_{F_v}}$ is Hodge-Tate regular for all $v  \mid l$.

For a CM field $F$, we write $F^+$ for the maximal totally real subfield of $F$. $c$ denotes the generator of $\mathrm{Gal}(F/F^{+})$. For an infinite place $w$ of $F^+$, we write $c_w \in G_{F^+}$ for the complex conjugation induced by $w$.

For a finite solvable extension $L/K$ of number fields of degree $d$ and an isobaric automorphic representation $\Pi$ of $\mathrm{GL}_n(\mathbb{A}_K)$, we write $\mathrm{BC}_{L/K}(\Pi)$ for the isobaric automorphic representation of $\mathrm{GL}_n(\mathbb{A}_L)$ satisfying $\mathrm{BC}_{L/K}(\Pi)_w = \mathrm{BC}_{L_w/K_v}(\Pi_v)$ for all finite places $w|v$. For an isobaric automorphic representation $\pi$ of $\mathrm{GL}_n(\mathbb{A}_L)$, we write $\mathrm{AI}_{L/K}(\pi)$ for the isobaric automorphic representation of $\mathrm{GL}_{dn}(\mathbb{A}_K)$ satisfying $(\mathrm{AI}_{L/K}(\pi))_v = \boxplus_{w|v}\mathrm{AI}_{L_w/K_v}(\pi)$ for all finite places $v$ in $K$. For integers $n_1, n_2$ satisfying $n_1 + n_2 = n$ and isobaric automorphic representations $\pi_1, \pi_2$ of $\mathrm{GL}_{n_1}(\mathbb{A}_K)$, $\mathrm{GL}_{n_2}(\mathbb{A}_K)$, we write $\pi_1 \boxplus \pi_2$ for the isobaric automorphic representation of $\mathrm{GL}_n(\mathbb{A}_K)$ satisfying $(\pi_1 \boxplus \pi_2)_v = (\pi_1)_v \boxplus (\pi_2)_v$ for all finite places $v$ of $K$.

For a number field $K$ and an automorphic representation $\pi$ of $\mathrm{GL}_n(\mathbb{A}_K)$, we say that $\pi$ is cohomological if $\pi_{\infty}$ has the same infinitesimal character as that of an irreducible algebraic representation $V$ of $\prod_{\tau \in \mathrm{Hom}(K, \mathbb{C})}\mathrm{GL}_n(\mathbb{C})$. In this case, we say that $\pi$ has weight $\lambda \in (\mathbb{Z}^n_+)^{\mathrm{Hom}(K, \mathbb{C})}$ if the dual representation $V^{\vee}$ of $V$ has highest weight $\lambda$.

\vspace{0.5 \baselineskip}

\textbf{Acknowledgements}

The author is grateful to his advisor Takeshi Saito for careful reading of an earlier version of this paper and for his constant support and encouragement. He also thanks Yugo Takanashi for answering his questions about the representation theory of $p$-adic reductive groups. This work was supported by the WINGS-FMSP program at the Graduate School of Mathematical Sciences, the University of Tokyo.

\section{Some observations on the local-global compatibility for the monodromy operators at $p \neq l$}

In this section, we find a simple interpretation of the local-global compatibility for the monodromy operators at $p \neq l$. Proposition \ref{automorphy lifting and local-global compatibility} is the main result in this section.

We fix a prime $p$, a finite extension $K$ of $\mathbb{Q}_p$ and a positive integer $n$ in this section. 

\begin{dfn}

1 \ We say that a sequence of positive integers $(n_1, \cdots, n_k)$ is a decomposition of $n$ if they satisfy $n = n_1 + \cdots + n_k$.  

2 \ We say that two decompositions of $n$ are equivalent if they coincide after changing their order.

3 \ A partition of $n$ means the equivalence class of a decomposition of $n$. 

    We write $[n_1, \cdots, n_k]$ for the equivalence class of $(n_1, \cdots, n_k)$. Unless otherwise specified, when we write $[n_1, \cdots ,n_k]$, we assume $n_1 \ge n_2 \ge \cdots \ge n_k$.

    4 \ For two partitions $\sigma := [n_1, \cdots, n_k]$ , $\tau:=[m_1, \cdots ,m_s]$ of $n$, we say that $\sigma \succ \tau$ if $k \le s$ and $n_1 + \cdots + n_i \ge m_1 + \cdots + m_i$ for all $1 \le i \le k$.

    5 \ Let $n_1$ and $n_2$ be positive integers and $\sigma := [a_1, \cdots, a_k]$, $\tau:=[b_1, \cdots ,b_s]$ be partitions of $n_1$, $n_2$ respectively. We write $\sigma \sqcup \tau$ for the partition $$[a_1, \cdots, a_k, b_1, \cdots ,b_s]$$ of $n_1+n_2$.

    6 \ For a partition $[n_1, \cdots, n_k]$ of $n$, we put $m_u:=|\{ t \mid n_t \ge u \}|$ for $1 \le u \le n_1=:j$. Then $[m_1, \cdots, m_j]$ is a partition of $n$. We call $[m_1, \cdots, m_j]$ the dual partition of $[n_1, \cdots, n_k]$.

\end{dfn}

\begin{lem}\label{partition lemma}

    Let $[n_1, \cdots, n_i]$ be a partition of $n$ and $[m_1, \cdots, m_j]$ be the dual partition of $[n_1, \cdots, n_i]$. Then for a partition $[l_1, \cdots, l_k]$ of $n$, the following conditions are equivalent.
    
    (1) \ $[n_1, \cdots, n_i] \succ [l_1, \cdots , l_k]$.
    
    (2) \ There exists a $k \times j$ matrix of integers $(a_{s, u})_{s=1, \cdots, k, u=1, \cdots j}$ such that $a_{s, u} = 0$ or $1$ for all $s$ and $u$, $\sum_{u=1}^{j} a_{s,u} = l_s$ for all $s$ and $\sum_{s=1}^k a_{s, u} = m_u$ for all $u$. 

    \end{lem}

    \begin{proof} We will prove $(2) \Rightarrow (1)$. We take $(a_{s,u})_{s=1, \cdots, k, u=1, \cdots j}$ be a $k \times j$-matrix of integers satisfying the condition (2). Then $k \ge \sum^{k}_{s=1}a_{s, 1} = m_1 = i$ since $a_{s,1} = 0$ or $1$. 
    
 Let $b_{s, u}$ be $1$ if $1 \le s \le m_u$ and $0$ if $m_u + 1 \le s \le k$. Then we have $\sum_{u = 1}^j b_{s, u} = |\{ 1 \le u \le j \mid m_u \ge s \}| = n_s$. Moreover, for any $1 \le u \le j$ and any $1 \le t \le k$, we have $\sum_{s = 1}^t a_{s, u} \le \mathrm{min} \{ m_u, t \} = \sum^{t}_{s=1} b_{s, u}$.
    
Therefore, we obtain $\sum^{t}_{s = 1} l_s = \sum^t_{s=1} \sum_{u=1}^{j} a_{s, u} = \sum_{u=1}^{j} \sum^t_{s=1} a_{s, u} \le \sum_{u=1}^{j} \sum^t_{s=1} b_{s, u} = \sum^t_{s=1} \sum_{u=1}^{j} b_{s, u} = \sum_{s = 1}^{t} n_s$. Consequently, $[n_1, \cdots, n_i] \succ [l_1, \cdots, l_k]$.

\vspace{0.5 \baselineskip}

We prove $(1) \Rightarrow (2)$ by induction on $n$. If $n = 1$, then (2) always holds. We assume $n > 1$ and $[n_1, \cdots, n_i] \succ [l_1, \cdots , l_k]$. Note that $[m_2, \cdots, m_j]$ is the dual partition of $[n_1 - 1, \cdots, n_i-1]$. (We regard $[n_1 - 1, \cdots, n_i-1]$ be the partition of $n-i$ by removing zero components. In the following argument, we use similar conventions.) Let $a_{s,1}$ be $1$ for any $1 \le s \le i$ and $0$ for any $i + 1 \le s \le k$. If we can prove that $[n_1-1, \cdots, n_i - 1] \succ [l_1 - 1, \cdots, l_i - 1, l_{i+1}, \cdots, l_k]$ (note that we need not have $l_i - 1 \ge l_{i+1}$), then by the induction hypothesis, we can take $(a_{s,u})_{s = 1, \cdots, k, u = 2, \cdots, j}$ satisfying $a_{s, u} = 0$ or $1$ for all $s$ and $u$, $\sum_{u=2}^{j} a_{s,u} = l_s - 1$ for all $1 \le s \le i$, $\sum_{u=2}^{j} a_{s,u} = l_s$ for all $i + 1 \le s \le k$ and $\sum_{s=1}^k a_{s, u} = m_u$ for all $2 \le u \le j$. Then, $(a_{s,u})_{s = 1, \cdots, k, u = 1, \cdots, j}$ satisfies the conditions of (2).

Thus, it suffices to prove $[n_1-1, \cdots, n_i - 1] \succ [l_1 - 1, \cdots, l_i - 1, l_{i+1}, \cdots, l_k]$. We put $l_{m-1} > l_m = \cdots = l_i \ge \cdots$ and $[l_1 - 1, \cdots, l_i - 1, l_{i+1}, \cdots, l_k] = [l_1', \cdots, l_k']$ such that $l_1' \ge l_2' \ge \cdots \ge l_k'$. Note that we have $l_s - 1 = l_s'$ for all $1 \le s \le m-1$ and $l_s \ge l_s' \ge l_s - 1$ for all $m \le s \le i$. Then, the condition $[n_1-1, \cdots, n_i - 1] \succ [l_1 - 1, \cdots, l_i - 1, l_{i+1}, \cdots, l_k]$ is equivalent to the condition that for any $m \le t \le i$, we have $\sum^{t}_{s = 1}(n_s - 1) \ge \sum^{t}_{s=1}l_s'$. For $t = i$, this condition always holds by $\sum^{i}_{s = 1}(n_s - 1) = n - i = \sum^{k}_{s=1}l_s' \ge \sum^{i}_{s=1}l_s'$. 

We suppose that there were $m \le t \le i$ such that $\sum^{t}_{s = 1}(n_s - 1) < \sum^{t}_{s=1}l_s'$. We may assume that $t$ is the maximum. Then we would have $t + 1 \le i$, $\sum^{t+1}_{s = 1}(n_s - 1) \ge \sum^{t+1}_{s=1}l_s'$ and consequently $n_{t + 1} \ge l_{t+1}' + 2 \ge l_{t+1} + 1$. On the other hand, $n_{m} \ge \cdots \ge n_{t-1} \ge n_{t} \ge n_{t+1}$ and $l_{m} = \cdots = l_{t} = l_{t+1} \le n_{t+1} - 1$. Therefore, we would obtain $\sum^{t}_{s = 1}(n_s - 1) \ge \sum^{m-1}_{s=1}(l_s - 1) + \sum_{s=m}^{t}l_s \ge \sum^{t}_{s=1}l_s'$ and this is absurd. \end{proof}

\begin{lem}\label{partition sum}

Let $n_1, \cdots n_k$ be positive integers and $\tau_i,  \sigma_i$ be partitions of $n_i$ for any $1 \le i \le k$. We assume that $\tau_i \succ \sigma_i$ for all $i$ and $\tau_1 \sqcup \tau_2 \sqcup \cdots \sqcup \tau_k = \sigma_1 \sqcup \sigma_2 \sqcup \cdots \sqcup \sigma_k$.

Then we obtain $\tau_i = \sigma_i$ for any $1 \le i \le k$.

\end{lem}

\begin{proof} We put $\tau_i:=[a_{i,1}, \cdots, a_{i, m_i}]$ and $\sigma_i:=[b_{i,1}, \cdots, b_{i,l_i}]$. 

By $\sigma_i \succ \tau_i$ for any $i$, we obtain \begin{equation} \sum_{j \le s} a_{i, j} \ge \sum_{j \le s} b_{i,j} \end{equation} for any $i$ and any $1 \le s \le m_i$.

By $\tau_1 \sqcup \tau_2 \sqcup \cdots \sqcup \tau_k = \sigma_1 \sqcup \sigma_2 \sqcup \cdots \sqcup \sigma_k$, we have \begin{equation}\sum_{1 \le i \le k, a_{i, j} \ge u} a_{i, j} = \sum_{1 \le i \le k, b_{i, j} \ge u} b_{i,j} \end{equation} for any positive integer $u$.
    
We put $u_0:= \mathrm{max}\{ a_{i,1} \mid 1 \le i \le k \}$. Then we obtain \begin{equation} \sum_{a_{i, j} \ge u_0} a_{i, j} \ge \sum_{b_{i,j} \ge u_0} b_{i,j} \end{equation} for any $i$ by the inequality (1).

Thus we obtain $a_{i, j} = b_{i, j}$ for any $i, j$ satisfying $a_{i,j} = u_0$ by the equality (2) and the inequality (3). By induction on $n_1 + \cdots + n_k$, we obtain $[a_{i, 1}, \cdots, a_{i, m_i}] = [b_{i, 1}, \cdots, b_{i, l_i}]$ for all $i$. \end{proof}

\begin{dfn}

    Let $(n_1, \cdots, n_k)$ be a decomposition of $n$ and $P$ be the parabolic subgroup of $\mathrm{GL}_n$ corresponding to this decomposition.
    
    The parahoric subgroup of $\mathrm{GL}_n(K)$ corresponding to $(n_1, \cdots, n_k)$ means the inverse image of $P(\mathbb{F}_K)$ by the reduction map $\mathrm{GL}_n(\mathcal{O}_K) \rightarrow \mathrm{GL}_n(\mathbb{F}_K)$.
    
    In particular, if $(n_1, \cdots, n_k) = (1, \cdots, 1)$, we call the corresponding parahoric subgroup the Iwahori subgroup and Iw$_K$ denotes this subgroup. 

\end{dfn}

\begin{lem}\label{Iwahori fixed vectors}

Let $\psi$ be an unramified character of $K^{\times}$.

Then $\mathrm{St}_n(\psi)$ has a nonzero vector fixed by $\mathrm{Iw}_K$ and doesn't have nonzero vectors fixed by the parahoric subgroup corresponding to any other decomposition of $n$.

\end{lem}

\begin{proof} 
    
    Let $(n_1, \cdots, n_k)$ be a decomposition of $n$, $P$ (resp. $\mathfrak{p}$) be a parabolic (resp. parahoric) subgroup corresponding to this decomposition and $P=MN$ be the natural Levi decomposition.

    By \cite[Lemma 5.15]{TW}, we have $\pi^{\mathfrak{p}} \cong r_{N}(\pi)^{M(K) \cap \mathfrak{p}}$ for any smooth representation $\pi$ of $\mathrm{GL}_n(K)$. ($r_N$ denotes the normalized Jacquet module.) By \cite[Proposition 1.1, Proposition 1.5 and Proposition 9.5]{ZI}, there exist unramified characters $ \psi_1, \cdots, \psi_k$ of $K^{\times}$ such that $r_{N}(\mathrm{St}_n(\psi)) \cong \mathrm{St}_{n_1}(\psi_1) \otimes \cdots \otimes \mathrm{St}_{n_k}(\psi_k)$. 

 Note that $\mathrm{St}_{n_i}(\psi_i)$ is unramified if and only if $n_i=1$. Thus, $\mathrm{St}_{n}(\psi)^{\mathfrak{p}} \neq 0$ if and only if $\mathfrak{p}=\mathrm{Iw}_K$.  \end{proof}

\begin{lem}\label{parahoric}

Let $[n_1, \cdots, n_i]$ be a partition of $n$ and $[m_1, \cdots, m_j]$ be the dual partition of $[n_1, \cdots, n_i]$. Then for any partition $[l_1, \cdots, l_k]$ of $n$ and unramified characters $\psi_1, \cdots, \psi_{k}$ of $K^{\times}$, the following conditions are equivalent. ($Q$ denotes the parabolic subgroup corresponding to $(l_1, \cdots, l_k)$ and $\mathfrak{p}$ denotes the parahoric subgroup corresponding to $(m_1, \cdots, m_j)$.)

(1) \ $(\mathrm{n}$-$\mathrm{Ind}^{\mathrm{GL}_n(K)}_{Q(K)}\mathrm{St}_{l_1}(\psi_1) \otimes \cdots \otimes \mathrm{St}_{l_k}(\psi_{k}))^{\mathfrak{p}} \neq 0$.

(2) $[n_1, \cdots, n_i] \succ [l_1, \cdots , l_k]$.

\end{lem}

\begin{proof} By \cite[Lemma 3.3 and 3.4]{ade}, we have $(\mathrm{n}\textrm{-}\mathrm{Ind}^{\mathrm{GL}_n(K)}_{Q(K)}\mathrm{St}_{l_1}(\psi_1) \otimes \cdots \otimes \mathrm{St}_{l_k}(\psi_{k}))^{\mathfrak{p}} = \oplus_{w \in \mathcal{S}} \  \mathrm{St}_{l_1}(\psi_1)^{\mathfrak{p}_1^w} \otimes \cdots \otimes \mathrm{St}_{l_k}(\psi_{k})^{\mathfrak{p}_k^w}.$

Here, $\mathcal{S} := \{ (m_{s, 1}, \cdots, m_{s, j})_{1 \le s \le k} \mid m_{s, u} \in \mathbb{Z}_{\ge 0}, \sum_{u=1}^{j} m_{s,u} = l_s \ \forall s, \ \sum_{s=1}^k m_{s, u} = m_u \ \forall u \}$ and for $w = (m_{s, 1}, \cdots, m_{s, j})_{1 \le s \le k} \in \mathcal{S}$ and $1 \le s \le k$, $\mathfrak{p}^{w}_{s}$ denotes the parahoric subgroup of $\mathrm{GL}_{l_s}(K)$ corresponds to the decomposition $(m_{s, 1}, \cdots, m_{s, j})$ of $l_s$. Note that we regard $(m_{s, 1}, \cdots, m_{s, j})$ as the decomposition of $l_s$ by removing zero components.

By Lemma \ref{Iwahori fixed vectors}, (n-Ind$^{\mathrm{GL}_n(K)}_{Q(K)}\mathrm{St}_{l_1}(\psi_1) \otimes \cdots \otimes \mathrm{St}_{l_k}(\psi_{k}))^{\mathfrak{p}} \neq 0$ is equivalent to the existence of $w = (m_{s, 1}, \cdots, m_{s, j})_{1 \le s \le k} \in \mathcal{S}$ such that $m_{s, u} = 0$ or $1$ for all $s, u$. By Lemma \ref{partition lemma}, this is equivalent to $[n_1, \cdots, n_i] \succ [l_1, \cdots, l_k]$. \end{proof}

\begin{dfn} \label{Weil-Deligne def}

1 \ For a Weil-Deligne representation $r$ of $W_K$, the monodromy type of $r$ means the partition corresponding to the Jordan canonical form of the monodromy operator of $r$.

2 \ We say that two irreducible Weil-Deligne representations $r_1$, $r_2$ of $W_{K}$ are equivalent if there exists an unramified character $\chi$ of $W_{K}$ such that $r_1 \cong r_2 \chi$.

3 \ Let $\mathcal{W}$ be the set of equivalence classes of irreducible Weil-Deligne representations of $W_{K}$.

4 \ For a Frobenius semisimple Weil-Deligne representation $r$ of $W_{K}$ and $w \in \mathcal{W}$, we write $r[w]$ for the maximal subrepresentation of $r$ whose all irreducible subquotients are contained in $w$. Note that if $r[w] \neq 0$, there exist $s_1, \cdots, s_k \in w$ and positive integers $n_1 \ge \cdots \ge n_k$ such that $r[w] \cong \mathrm{Sp}_{n_1}(s_1) \oplus \mathrm{Sp}_{n_2}(s_2) \oplus \cdots \oplus \mathrm{Sp}_{n_k}(s_k)$, $n_1 + \cdots + n_k = \frac{\mathrm{dim} \, r[w]}{\mathrm{dim} \, s_1}$ and the monodromy type of $r[w]$ is equal to $[ \underbrace{n_1, \cdots, n_1}_{\mathrm{dim} \, s_1}, \cdots, \underbrace{n_k, \cdots, n_k}_{\mathrm{dim} \, s_1} ]$. 

5 \ Let $r_1, r_2$ be $n$-dimensional Frobenius semisimple Weil-Deligne representations of $W_{K}$.

We say that $r_1 \succ r_2$ if $r_1^{ss} \cong r_2^{ss}$ and for all $w \in \mathcal{W}$ satisfying $r_1[w] \neq 0$, we have (the monodromy type of $r_1[w]$) $\succ$ (the monodromy type of $r_2[w]$). (Note that if $r_1^{ss} \cong r_2^{ss}$, then $\mathrm{dim} \, r_1[w] = \mathrm{dim} \, r_2[w]$ for any $w \in \mathcal{W}$.)

\end{dfn}

\begin{rem} Let $r_1, r_2$ be $n$-dimensional Frobenius semisimple Weil-Deligne representations of $W_{K}$. If $r_1 \succ r_2$, then $r_1|_{W_L} \succ r_2|_{W_L}$ for any finite extension $L/K$.

\end{rem}

\begin{lem}\label{supercuspidal absolute values}

    Let $\pi$ be an irreducible supercuspidal representation of $\mathrm{GL}_n(K)$ over $\mathbb{C}$. Then there exists $\alpha \in \mathbb{R}_{>0}$ such that all absolute values of the Frobenius eigenvalues of $\mathrm{rec}_{K}(\pi)$ are equal to $\alpha$. Moreover, $\pi$ is unitary if and only if $\alpha = 1$.

    \end{lem}

\begin{proof} We fix a Frobenius lift $\mathrm{Frob}_K \in W_K$ and consider the sum of all eigenspaces corresponding to eigenvalues of $\mathrm{rec}_K(\pi)(\mathrm{Frob}_K)$ whose absolute values are the minimum. This subspace is a nonzero sub-Weil-Deligne representation of $\mathrm{rec}_{K}(\pi)$. Since $\mathrm{rec}_{K}(\pi)$ is irreducible, we obtain the first result. Note that $\pi$ is unitary if and only if the central character $\omega_{\pi}$ of $\pi$ is unitary. Since $\mathrm{det(rec}_{K}(\pi)) = \omega_{\pi} \circ \mathrm{Art}_{K}^{-1}$, we obtain the second result by the first result. \end{proof}
 
We recall the following important results.

\begin{thm}\label{generic unitary}

Let $\pi$ be an irreducible smooth representation of $\mathrm{GL}_n(K)$ over $\mathbb{C}$. Then the following conditions are equivalent.

1 \ $\pi$ is generic unitary.

2 \ There exist irreducible unitary supercuspidal representations $\pi_1, \cdots, \pi_s, \pi_1', \cdots, \pi_r'$ and $0 < a_1, \cdots, a_r < \frac{1}{2}$ such that $\pi \cong \mathrm{St}_{n_1}(\pi_1) \boxplus \cdots \boxplus \mathrm{St}_{n_s}(\pi_s) \boxplus \mathrm{St}_{m_1}(\pi_1')|\mathrm{det}|_K^{a_1} \boxplus \mathrm{St}_{m_1}(\pi_1')|\mathrm{det}|_K^{-a_1} \boxplus \cdots \boxplus \mathrm{St}_{m_r}(\pi_r')|\mathrm{det}|_K^{a_r} \boxplus \mathrm{St}_{m_r}(\pi_r')|\mathrm{det}|_K^{-a_r}$.

\end{thm}

\begin{proof} 

See \cite[Lemma I.3.8]{LLC} or \cite[Theorem A]{GP}. \end{proof}

\begin{thm} \label{tempered}

Let $\pi$ be an irreducible smooth representation of $\mathrm{GL}_n(K)$ over $\mathbb{C}$. Then the following conditions are equivalent.

1 \ $\pi$ is tempered unitary.
    
2 \ There exist irreducible unitary supercuspidal representations $\pi_1, \cdots, \pi_s$ such that $\pi \cong \mathrm{St}_{n_1}(\pi_1) \boxplus \cdots \boxplus \mathrm{St}_{n_s}(\pi_s)$.

\end{thm}

\begin{proof}

    See \cite[Lemma I.3.8]{LLC}. \end{proof}

\begin{dfn}

For a Frobenius semisimple $n$-dimensional Weil-Deligne representation $r = \mathrm{Sp}_{n_1}(s_1) \oplus \cdots \oplus \mathrm{Sp}_{n_k}(s_k)$ over $\mathbb{C}$, we say that $r$ is weakly tempered if there exists a continuous character $\chi : W_K \rightarrow \mathbb{C}^{\times}$ such that for any $i$ and any eigenvalue $\alpha$ of $s_i(\mathrm{Frob}_K)$, we have $q_K^{-\frac{1}{2}} < |\alpha \chi(\mathrm{Frob}_K)| < q_K^{\frac{1}{2}}$.

\end{dfn}

\begin{rem}\label{weakly tempered}

1 \ If $\pi$ is a character twist of a generic unitary irreducible smooth representation of $\mathrm{GL}_n(K)$ over $\mathbb{C}$ (in particular, if $\pi$ is tempered), then $\mathrm{rec}_{K}(\pi)$ is weakly tempered by Theorem \ref{generic unitary} and Lemma \ref{supercuspidal absolute values}.

2 \ For any $n$-dimensional Frobenius semisimple weakly tempered Weil-Deligne representation $r$ and any finite extension $L/K$, the representation $r|_{W_L}$ is weakly tempered.

3 \ Let $r$ be an $n$-dimensional Frobenius semisimple weakly tempered Weil-Deligne representation $r$ and $\pi$ be the irreducible smooth representation $\pi$ of $\mathrm{GL}_n(K)$ satisfying $\mathrm{rec}_K(\pi) \cong r$. Then $\pi$ is generic by Lemma \ref{generic equivalent} later or \cite[p 36]{LLC}. 

\end{rem}

\begin{lem} \label{RG}

Let $r = \mathrm{Sp}_{n_1}(s_1) \oplus \cdots \oplus \mathrm{Sp}_{n_k}(s_k)$ be an $n$-dimensional Frobenius semisimple weakly tempered Weil-Deligne representations of $W_{K}$ over $\mathbb{C}$, $t$ be an irreducible Weil-Deligne representation of $W_K$ over $\mathbb{C}$ and $m$ be a positive integer. We assume that $n_1 \ge \cdots \ge n_k \ge 1$, $t|| \ ||_{K}^{\frac{m-1}{2}}$ and $t|| \ ||^{\frac{1-m}{2}}_{K}$ are contained in $r^{ss}$ and $m \ge n_1$. Then $m=n_1$ and there exists an integer $1 \le l \le k$ such that $t \cong s_l$ and $n_1 = n_l$.

\end{lem}

\begin{proof}

After twisting $r$ and $t$ by a character, we may assume that for any $i$ and any eigenvalue $\alpha$ of $s_i(\mathrm{Frob}_K)$, we have $q_K^{-\frac{1}{2}} < |\alpha| < q_K^{\frac{1}{2}}$. Consequently, for any eigenvalue $\beta$ of $\mathrm{Sp}_{n_i}(s_i)(\mathrm{Frob}_K)$, we have $q_K^{-\frac{n_i}{2}} < |\beta| < q_K^{\frac{n_i}{2}}$. Since $t|| \ ||_{K}^{\frac{m-1}{2}}$ and $t|| \ ||^{\frac{1-m}{2}}_{K}$ are contained in $r^{ss}$ and $m \ge n_1$, we obtain $m = n_1$. We may assume that all eigenvalues $\gamma$ of $t|| \ ||_{K}^{\frac{n_1-1}{2}}(\mathrm{Frob}_K)$ satisfy $|\gamma| \le q_K^{\frac{1-n_1}{2}}$. (Otherwise, all eigenvalues $\gamma$ of $t|| \ ||_{K}^{\frac{1-n_1}{2}}(\mathrm{Frob}_K)$ satisfies $|\gamma| \ge q_K^{\frac{n_1-1}{2}}$.) Thus, we obtain $l$ such that $t|| \ ||_K^{\frac{1-n_1}{2}} \cong s_l|| \ ||_K^{\frac{1-n_l}{2}}$ and $n_1=n_l$.  \end{proof}

\begin{prop}\label{monodromy operator}

Let $r_1$ and $r_2$ be $n$-dimensional Frobenius semisimple Weil-Deligne representations of $W_{K}$ over $\mathbb{C}$. We assume that $r_1^{ss} \cong r_2^{ss}$ and $r_2$ is weakly tempered.

Then the following conditions are equivalent.
        
(1) \ $r_1 \cong r_2$.
        
(2) \ $r_1$ and $r_2$ have the same monodromy type.
        
Let $\pi$ be the irreducible smooth representation of $\mathrm{GL}_n(K)$ satisfying $\mathrm{rec}_{K}(\pi) \cong r_2$. If $\pi^{\mathrm{Iw}_K} \neq 0$ and $r_1 \prec r_2$, the conditions (1) and (2) are equivalent to the following condition.
        
(3) \ $\pi^{\mathfrak{p}} \neq 0$, where $\mathfrak{p}$ is the parahoric subgroup corresponding to the dual partition of the monodromy type of $r_1$.
            
\end{prop}
            
\begin{rem} The equivalence between (1) and (2) is false without assuming that $r_2$ is weakly tempered. For example, the representations $\mathrm{Sp}_m(\chi_1)^{ss} \oplus \mathrm{Sp}_m(\chi_2)$ and $\mathrm{Sp}_m(\chi_1) \oplus \mathrm{Sp}_m(\chi_2)^{ss}$ satisfy the condition (2) but don't satisfy the condition (1) for $m > 1$ and smooth characters $\chi_1 \neq \chi_2$.  \end{rem}

        \begin{proof} $(1) \Rightarrow (2)$ is trivial. 
            
    We assume the condition $(2)$. We put $r_1 = \mathrm{Sp}_{n_1}(s_1) \oplus \cdots \oplus \mathrm{Sp}_{n_i}(s_i)$ and $r_2 = \mathrm{Sp}_{l_1}(t_1) \oplus \cdots \oplus \mathrm{Sp}_{l_k}(t_k)$ ($n_1 \ge \cdots \ge n_i \ge 1$, $l_1 \ge \cdots \ge l_k \ge 1$). By the assumption (2), we have $n_1 = l_1$.
    
    Since $r_1^{ss} \cong r_2^{ss}$, $s_1|| \ ||_K^{\frac{n_1-1}{2}}$ and $s_1|| \ ||_K^{\frac{-n_1+1}{2}}$ are contained in $r_2^{ss}$. By Lemma \ref{RG}, this implies that there exists an integer $j$ such that $s_1 \cong t_j$ and $n_1=l_j$. Thus, we may assume $j = 1$. Note that $(\mathrm{Sp}_{n_2}(s_2) \oplus \cdots \oplus \mathrm{Sp}_{n_i}(s_i))^{ss} \cong (\mathrm{Sp}_{l_2}(t_2) \oplus \cdots \oplus \mathrm{Sp}_{l_k}(t_k))^{ss}$ and $\mathrm{Sp}_{l_2}(t_2) \oplus \cdots \oplus \mathrm{Sp}_{l_k}(t_k)$ is weakly tempered. By induction on $i$, we obtain (1).
            
    We assume $\pi^{\mathrm{Iw}_K} \neq 0$. Then $t_i$ is an unramified character for any $i$. Let $\psi_i:=t_i \circ \mathrm{Art}_{F_v}$ and $Q$ be a parabolic subgroup of $\mathrm{GL}_n$ corresponding to $(l_1, \cdots, l_k)$. Since $\pi$ is generic by 3 of Remark \ref{weakly tempered}, we obtain $\pi = \mathrm{St}_{l_1}(\psi_1) \boxplus \mathrm{St}_{l_2}(\psi_2) \boxplus \cdots \boxplus \mathrm{St}_{l_k}(\psi_k) = $n-$\mathrm{Ind}^{\mathrm{GL}_n(K)}_{Q(K)}\mathrm{St}_{l_1}(\psi_1) \otimes \mathrm{St}_{l_2}(\psi_2) \otimes \cdots \otimes \mathrm{St}_{l_k}(\psi_k)$. By the assumption $r_1 \prec r_2$, we have $[n_1, n_2, \cdots ,n_i] \prec [l_1, l_2, \cdots, l_k]$. Therefore, the condition (3) is equivalent to $[n_1, n_2, \cdots, n_i] = [l_1, l_2, \cdots, l_k]$ by Lemma \ref{parahoric}. This is equivalent to the condition (2). \end{proof}

    \begin{lem}\label{semisimple}

        Let $r_1, r_2$ be Frobenius semisimple weakly tempered Weil-Deligne representations of $W_K$ over $\mathbb{C}$. If $r_1^{ss} \cong r_2^{ss}$, then $r_1 \cong r_2$. 
            
        \end{lem}
            
        \begin{proof} We put $r_1 = \mathrm{Sp}_{n_1}(s_1) \oplus \cdots \oplus \mathrm{Sp}_{n_k}(s_k)$ and $r_2 = \mathrm{Sp}_{m_1}(t_1) \oplus \cdots \oplus \mathrm{Sp}_{m_l}(t_l)$ such that $n_1 \ge \cdots \ge n_k \ge 1$ and $m_1 \ge \cdots \ge m_l \ge 1$. We may assume $n_1 \ge m_1$. Since $s_1|| \ ||_K^{\frac{n_1-1}{2}}$ and $s_1|| \ ||_K^{\frac{1-n_1}{2}}$ are contained in $r_2^{ss}$, we may assume $m_1=n_1$ and $s_1 = t_1$ by Lemma \ref{RG}. We obtain the result by induction. \end{proof}

We recall the following important results.

\begin{thm}\label{Ila Varma}
    
    Let $F$ be a CM field, $\pi$ be a cohomological cuspidal automorphic representation of $\mathrm{GL}_n(\mathbb{A}_F)$, $l$ be a prime and $\iota: \overline{\mathbb{Q}}_l \stackrel{\sim}{\longrightarrow} \mathbb{C}$ be an isomorphism of fields.
    
    Then there exists a continuous semisimple representation $r_{\iota}(\pi) : G_F \rightarrow \mathrm{GL}_n(\overline{\mathbb{Q}}_l)$ satisfying $\iota\mathrm{WD}(r_{\iota}(\pi)|_{G_{F_v}})^{F-ss} \prec \mathrm{rec}_{F_v}(\pi_v|\mathrm{det}|_v^{\frac{1-n}{2}})$ for any $v \nmid l$.
    
    \end{thm}
    
    \begin{proof}
    
    See \cite[Theorem A]{GH}, \cite[Theorem 1.0.4]{SG} and \cite[Theorem 1]{ss}. \end{proof}

\begin{prop} \label{purity}

Let $F$ be a CM field and $\pi$ be a cohomological cuspidal automorphic representation of $\mathrm{GL}_n(\mathbb{A}_F)$ of weight $\lambda$.

Then there exists an integer $w$ such that for all $\tau \in \mathrm{Hom}(F, \mathbb{C})$ and $i=1, \cdots, n$, we have $\lambda_{\tau, i} = w - \lambda_{\tau^c, n+1-i}$. (This implies $|\omega_{\pi}| = || \ ||_F^{-\frac{wn}{2}}$ and $\pi|| \mathrm{det} ||_{F}^{\frac{w}{2}}$ is unitary.)

Moreover, $\pi_v| \mathrm{det} |_v^{\frac{w}{2}}$ is generic unitary for all finite place $v$ of $F$.

\end{prop}

\begin{proof}

The first property follows by \cite[Lemma 4.9]{purity}. The second result follows by \cite[Corollary of Theorem 5.9]{mult}.\end{proof}

The following proposition is the main result in this section.

\begin{prop}\label{automorphy lifting and local-global compatibility}

Let $F$ be a CM field, $\pi$ be a cohomological cuspidal automorphic representation of $\mathrm{GL}_n(\mathbb{A}_F)$, $l$ be a prime, $\iota: \overline{\mathbb{Q}}_l \stackrel{\sim}{\longrightarrow} \mathbb{C}$ be an isomorphism of fields and $v$ be a non-$l$-adic place of $F$.

Then the following conditions are equivalent.

(1) \ $\iota\mathrm{WD}(r_{\iota}(\pi)|_{G_{F_v}})^{F-ss} \cong \mathrm{rec}_{F_v}(\pi_v|\mathrm{det}|_v^{\frac{1-n}{2}})$.

(2) \ $\iota\mathrm{WD}(r_{\iota}(\pi)|_{G_{F_v}})$ and $\mathrm{rec}_{F_v}(\pi_v|\mathrm{det}|_v^{\frac{1-n}{2}})$ have the same monodromy type.

If $\pi_v^{\mathrm{Iw}_v} \neq 0$, these conditions are equivalent to the following condition.

(3) \ $\pi_v^{K_v} \neq 0$, where $K_v$ is the parahoric subgroup corresponding to the dual partition of the monodromy type of $\mathrm{WD}(r_{\iota}(\pi)|_{G_{F_v}})$.

\end{prop}

\begin{proof} 

By Theorem \ref{Ila Varma}, Proposition \ref{purity} and 1 of Remark \ref{weakly tempered}, the representations $r_1:=\iota\mathrm{WD}(r_{\iota}(\pi)|_{G_{F_v}})^{F-ss}$ and $r_2:=\mathrm{rec}_{F_v}(\pi_v|\mathrm{det}|_v^{\frac{1-n}{2}})$ satisfy the conditions of Lemma \ref{monodromy operator}. Thus we obtain the result.  \end{proof}

\begin{prop}\label{potential automorphy and local-global compatibility}

Let $F$, $\pi$, $l$, $\iota$, $v$ as in Proposition \ref{automorphy lifting and local-global compatibility}. 
    
We assume that there exist a finite CM extension $F'$ of $F$ and a cohomological cuspidal automorphic representation $\Pi$ of $\mathrm{GL}_n(\mathbb{A}_{F'})$ such that $r_{\iota}(\pi)|_{G_{F'}} \cong r_{\iota}(\Pi)$ and $\mathrm{WD}(r_{\iota}(\Pi)|_{G_{F'_u}})^{F-ss} \cong \mathrm{rec}_{F'_u}(\Pi_u|\mathrm{det}|_u^{\frac{1-n}{2}})$ for some $u|v$.

Then we have $\iota \mathrm{WD}(r_{\iota}(\pi)|_{G_{F_v}})^{F-ss} \cong \mathrm{rec}_{F_v}(\pi_v|\mathrm{det}|_v^{\frac{1-n}{2}})$.

\end{prop}

\begin{proof}

By Theorem \ref{Ila Varma}, we have $\mathrm{rec}_{F'_u}(\mathrm{BC}_{F'_u/F_v}(\pi_v)|\mathrm{det}|_u^{\frac{1-n}{2}})^{ss} \cong \iota \mathrm{WD}(r_{\iota}(\pi)|_{G_{F'_u}})^{ss} \cong \iota \mathrm{WD}(r_{\iota}(\Pi)|_{G_{F'_u}})^{ss} \cong \mathrm{rec}_{F'_u}(\Pi_u |\mathrm{det}|_u^{\frac{1-n}{2}})^{ss}$. By Lemma \ref{semisimple}, Proposition \ref{purity} and 1 of Remark \ref{weakly tempered}, we obtain $\mathrm{BC}_{F'_u/F_v}(\pi_v) = \Pi_u$ . This implies $\iota \mathrm{WD}(r_{\iota}(\pi)|_{G_{F_v}})^{F-ss}|_{W_{F'_u}} \cong \mathrm{rec}_{F_v}(\pi_v|\mathrm{det}|_v^{\frac{1-n}{2}})|_{W_{F'_u}}$. In particular, $\iota \mathrm{WD}(r_{\iota}(\pi)|_{G_{F_v}})$ and $\mathrm{rec}_{F_v}(\pi_v|\mathrm{det}|_v^{\frac{1-n}{2}})$ have the same monodromy type. Hence we obtain the result by $(2) \Rightarrow (1)$ of Proposition \ref{automorphy lifting and local-global compatibility}. \end{proof}

Next, we recall relations between the purity and the local-global compatibility.

\begin{dfn} \label{definition of purity}

Let $l$ be a prime.

1 \ For $q \in \mathbb{R}$, we say that $x \in \overline{\mathbb{Q}}_l$ is a Weil $q$-number if $|\iota(x)|^2 = q$ for all $\iota : \overline{\mathbb{Q}}_l \stackrel{\sim}{\rightarrow} \mathbb{C}$.

2 \ For a Weil-Deligne representation $r$ of $W_K$ over $\overline{\mathbb{Q}}_l$ and an integer $w$, we say that $r$ is pure of weight $w$ if there exists an increasing filtration $\{ \mathrm{Fil}_i \}_i$ on $r$ by subrepresentations of $r$ satisfying the following conditions.

(1) \ All eigenvalues of $r(\mathrm{Frob}_K)$ on $\mathrm{Fil}_i/\mathrm{Fil}_{i-1}$ are Weil $q_K^{i}$-numbers.

(2) \ For all $i$, we have $N^i : \mathrm{Fil}_{w+i}/\mathrm{Fil}_{w+i-1} \stackrel{\sim}{\rightarrow} \mathrm{Fil}_{w-i}/\mathrm{Fil}_{w-i-1}$, where $N$ denotes the monodromy operator of $r$. (Note that $N(\mathrm{Fil}_i) \subset \mathrm{Fil}_{i-2}$ by the condition (1).)

\end{dfn}

\begin{lem} \label{purity lemma}
    
Let $l$ be a prime. Then we have the following properties.

1 \ For a Weil-Deligne representation $r$ of $W_K$ over $\overline{\mathbb{Q}}_l$, $r$ is pure if and only if $r^{F-ss}$ is pure.

2 \ For a finite extension $L/K$ and a Weil-Deligne representation $r$ of $W_K$ over $\overline{\mathbb{Q}}_l$, $r$ is pure if and only if $r|_{W_L}$ is pure.

3 \ For pure Weil-Deligne representations $r_1$ and $r_2$ of $W_K$ over $\overline{\mathbb{Q}}_l$, $r_1^{ss} \cong r_2^{ss}$ if and only if $r_1^{F-ss} \cong r_2^{F-ss}$.

4 \ For an irreducible smooth representation $\pi$ of $\mathrm{GL}_n(K)$ over $\mathbb{C}$, $\iota: \overline{\mathbb{Q}}_l \stackrel{\sim}{\rightarrow} \mathbb{C}$ and an integer $w$, the representation $\iota^{-1}\mathrm{rec}_{K}(\pi)$ is pure of weight $w$ if and only if \ $^\sigma\pi|\mathrm{det}|_K^{\frac{w}{2}}:=\mathbb{C} \otimes_{\mathbb{\sigma, C}} \pi|\mathrm{det}|_K^{\frac{w}{2}}$ is unitary tempered for all $\sigma \in \mathrm{Aut}(\mathbb{C})$. 

5 \ For a Weil-Deligne representation $r_1$ and a nonzero pure Weil-Deligne representation $r_2$, the purity of $r_1$ is equivalent to the purity of $r_1 \otimes r_2$. 

\end{lem}

\begin{proof}

See \cite[Lemma 1.4]{GLLC} for the proof of 1 $\sim$ 4. (The property 3 also follows from Lemma \ref{semisimple}.)

5 \ By 1 and 2, we may assume that there exist unramified characters $\phi_1, \cdots, \phi_a, \psi_1, \cdots, \psi_b$, positive integers $n_1 \ge \cdots \ge n_a, m_1 \ge \cdots \ge m_b$ and an integer $w_2$ such that $r_1 = \mathrm{Sp}_{n_1}(\phi_1) \oplus \cdots \oplus \mathrm{Sp}_{n_a}(\phi_a)$, $r_2 = \mathrm{Sp}_{m_1}(\psi_1) \oplus \cdots \oplus \mathrm{Sp}_{m_b}(\psi_b)$ and $\psi_i$ are pure of weight $w_2$. Note that for an integer $w_1$, $r_1$ is pure of weight $w_1$ if and only if $\phi_i$ is pure of weight $w_1$ for all $i$. Note also that $r_1 \otimes r_2 = \oplus_{i,j} \oplus_{k = 1}^{\mathrm{min}(n_i,m_j)} \mathrm{Sp}_{n_i + m_j + 1 - 2k}(\phi_i \psi_j)$. (See \cite[(1.6.14.4)]{WC}.) This implies that $r_2$ is pure of weight $w_2$ if and only if $r_1 \otimes r_2$ is pure of weight $w_1 + w_2$. \end{proof}

\begin{lem} \label{purity and Weil number}

Let $r = \mathrm{Sp}_{n_1}(s_1) \oplus \cdots \oplus \mathrm{Sp}_{n_k}(s_k)$ be an $n$-dimensional Frobenius semisimple Weil-Deligne representation of $W_K$ over $\mathbb{C}$, $a$ be an integer and $\iota : \overline{\mathbb{Q}}_l \stackrel{\sim}{\rightarrow} \mathbb{C}$ be an isomorphism of fields. We assume that for any $i$ and any eigenvalue $\alpha$ of $s_i(\mathrm{Frob}_K)$, we have $q_K^{-\frac{1}{2}} < |\alpha q_K^{-\frac{a}{2}}| < q_K^{\frac{1}{2}}$ and any eigenvalue of $\iota^{-1}r(\mathrm{Frob}_K)$ is a Weil $q_K^{w}$-number for some integer $w$. Then $\iota^{-1}r$ is pure of weight $a$.

\end{lem}

\begin{rem} \label{quasi pure}

The first assumption of this lemma is a stronger assumption than that $r$ is weakly tempered. However, this assumption holds if $r$ corresponds to a local component of a cohomological cuspidal automorphic representation of $\mathrm{GL}_n(\mathbb{A}_F)$ by Proposition \ref{purity}, Theorem \ref{generic unitary} and Lemma \ref{supercuspidal absolute values}.

\end{rem}

\begin{proof}

By the assumption, for any $i$, eigenvalue $\alpha$ of $s_i(\mathrm{Frob}_K)$ and $\sigma \in \mathrm{Aut}(\mathbb{C})$, we have $|\sigma(\alpha)|^2 = q_K^{a}$. This implies that $\iota^{-1}r$ is pure of weight $a$. \end{proof}

\begin{prop} \label{purity local-global}

Let $F$, $\pi$, $l$, $\iota$, $v$ as in Proposition \ref{automorphy lifting and local-global compatibility}. Assume that for any $g \in W_{F_v}$ and any eigenvalue $\alpha \in \overline{\mathbb{Q}}_l$ of $r_{\iota}(\pi)(g)$, there exists an integer $w$ such that $\alpha$ is a Weil $q_v^w$-number. Then we have the following results.

1 \ For all $\sigma \in \mathrm{Aut}(\mathbb{C})$, $^\sigma \pi_v:=\mathbb{C} \otimes_{\sigma, \mathbb{C}} \pi_v$ is tempered and $\iota^{-1}\mathrm{rec}_{F_v}(\pi_v|\mathrm{det}|_v^{\frac{1-n}{2}})$ is pure. 

2 \ $\mathrm{WD}(r_{\iota}(\pi)|_{G_{F_v}})$ is pure if and only if $\iota \mathrm{WD}(r_{\iota}(\pi)|_{G_{F_v}})^{F-ss} \cong \mathrm{rec}_{F_v}(\pi_v|\mathrm{det}|_v^{\frac{1-n}{2}})$. 

\end{prop}

\begin{rem} \label{Scholze}

$(a)$ \ As stated in \cite[Lemma I.5.7]{LLC}, by using the weight spectral sequence, the assumption of this proposition is satisfied if $r_{\iota}(\pi)|_{G_{F_v}}$ is a subquotient of $H^*_{\acute{e}t}(X_{\overline{F_v}}, \overline{\mathbb{Q}}_l(i))$, where $X$ is a proper smooth variety over $F_v$ and $i \in \mathbb{Z}$.

$(b)$ \ The first condition of 2 is satisfied if $r_{\iota}(\pi)|_{G_{F_v}}$ is a direct summand of $H^*_{\acute{e}t}(X_{\overline{F_v}}, \overline{\mathbb{Q}}_l(i))$, where $X$ is a smooth hypersurface in a projective space over $F_v$ and $i \in \mathbb{Z}$. In fact, the weight monodromy conjecture for $X$ is already known by \cite[Theorem 1.15]{Per}.

\end{rem}

\begin{proof} 1 \ Note that $\iota\mathrm{WD}(r_{\iota}(\pi)|_{G_{F_v}})^{ss} \cong \mathrm{rec}_{F_v}(\pi_v|\mathrm{det}|_v^{\frac{1-n}{2}})^{ss}$ by Theorem \ref{Ila Varma} and therefore, any eigenvalue of $\iota^{-1}\mathrm{rec}_{F_v}(\pi_v|\mathrm{det}|_v^{\frac{1-n}{2}})(\mathrm{Frob}_v)$ is a Weil $q_v^w$-number for some $w$. By Lemma \ref{purity and Weil number} and Remark \ref{quasi pure}, $\iota^{-1}\mathrm{rec}_{F_v}(\pi_v|\mathrm{det}|_v^{\frac{1-n}{2}})$ is pure. Moreover, by 4 of Lemma \ref{purity lemma}, $^\sigma \pi_v$ is tempered for any $\sigma \in \mathrm{Aut}(\mathbb{C})$.

2 \ This follows from 1 of this proposition, 1 and 3 of Lemma \ref{purity lemma} and Theorem \ref{Ila Varma}. \end{proof}

Finally, we recall the following important result in polarizable cases.

Let $F$ be a CM field. For a cohomological cuspidal automorphic representation $\pi$ of $\mathrm{GL}_n(\mathbb{A}_F)$ and an algebraic Hecke character $\chi : \mathbb{A}_{F^+}^{\times}/(F^+)^{\times} \rightarrow \mathbb{C}^{\times}$, we say that $(\pi, \chi)$ is polarized if we have $\pi^c \cong \pi^{\vee} \otimes \chi \circ N_{F/F^+} \circ \mathrm{det}$ and $\chi_v(-1) = \chi_w(-1)$ for all $v, w|\infty$.

For a cohomological cuspidal automorphic representation $\pi$ of $\mathrm{GL}_n(\mathbb{A}_F)$, we say that $\pi$ is polarizable if there exists an algebraic Hecke character $\chi : \mathbb{A}_{F^+}^{\times}/F^{+ \times} \rightarrow \mathbb{C}^{\times}$ such that $(\pi, \chi)$ is polarized.

\begin{thm}\label{polarizable local-global compatibility}

Let $(\pi, \chi)$ be a polarized cohomological cuspidal automorphic representation of $\mathrm{GL}_n(\mathbb{A}_F)$ of weight $\lambda \in (\mathbb{Z}_+^{n})^{\mathrm{Hom}(F, \mathbb{C})}$, $l$ be a prime and $\iota : \overline{\mathbb{Q}}_l \stackrel{\sim}{\rightarrow} \mathbb{C}$ be an isomorphism of fields.

Then we have the following properties.

(1) \ There exists a perfect symmetric $G_F$-equivariant pairing $r_{\iota}(\pi) \times r_{\iota}(\pi)^c \rightarrow \varepsilon_l^{1-n}r_{\iota}(\chi)|_{G_F}$.

(2) \ For all $v|l$, $r_{\iota}(\pi)|_{G_{F_v}}$ is de Rham and $\mathrm{HT}_{\tau}(r_{\iota}(\pi)|_{G_{F_v}}) = \{ \lambda_{\iota \tau, 1} + n - 1, \lambda_{\iota \tau, 2} + n - 2, \cdots, \lambda_{\iota \tau, n} \}$ for any $\tau \in \mathrm{Hom}_{\mathbb{Q}_l}(F_v, \overline{\mathbb{Q}}_l)$.

(3) \ For all finite places $v$ of $F$, $\mathrm{WD}(r_{\iota}(\pi)|_{G_{F_v}})$ is pure and $\pi_v$ is tempered.

(4) \ For all finite place $v$ of $F$, we have $\iota\mathrm{WD}(r_{\iota}(\pi)|_{G_{F_v}})^{F-ss} \cong \mathrm{rec}_{F_v}(\pi_v|\mathrm{det}|_v^{\frac{1-n}{2}})$.

\end{thm}

\begin{proof} See \cite[Theorem 2.1.1]{CW}, \cite[Theorem 1.2]{CM} and \cite[Theorem 1.1]{pWM}.   \end{proof}

\section{Irreducible components of local lifting rings} 

In this section, we study irreducible components of local lifting rings. Proposition \ref{monodromy type irreducible component} and Proposition \ref{trace irreducible component} are main results in this section.

\subsection{Some commutative algebras}

In this subsection, we recall some elementary properties of lifting rings and commutative algebras. 

We fix a prime $l$ and a finite extension $E$ of $\mathbb{Q}_l$ in $\overline{\mathbb{Q}}_l$. Let $\mathcal{O}$ denote the ring of integers of $E$, $\mathbb{F}$ denote the residue field of $\mathcal{O}$ and $\varpi$ denote a prime of $\mathcal{O}$. Let $\mathrm{CNL}_{\mathcal{O}}$ denote the category of complete Noetherian local $\mathcal{O}$-algebras with residue field $\mathbb{F}$. 

We also fix a positive integer $n$ and a profinite group $\Gamma$ such that any open subgroup $\Gamma'$ of $\Gamma$ satisfies $\mathrm{Hom}_{\mathrm{cont}}(\Gamma', \mathbb{F}_l)$ is finite. This is equivalent to the pro-$l$-completion $\Gamma'(l)$ of $\Gamma'$ is topologically finitely generated for any open subgroup $\Gamma'$ of $\Gamma$ by \cite[Lemma 7.13]{Fer}.

Note that by class field theory, this property holds for $\Gamma = G_K$, where $K$ is a finite extension of $\mathbb{Q}_p$ for some prime $p$ or $\Gamma = \mathrm{Gal}(K_S/K)$, where $K$ is a finite extension of $\mathbb{Q}$, $S$ is a finite set of finite places of $K$ and $K_S$ is the maximal subextension of $\overline{K}/K$ which is unramified outside $S$. See the proof of \cite[Proposition 7.9]{Fer} for global cases. 

For a continuous representation $\overline{r}: \Gamma \rightarrow \mathrm{GL}_n(\mathbb{F})$ and an object $A$ of $\mathrm{CNL}_{\mathcal{O}}$, a lifting of $\overline{r}$ to $A$ means a continuous representation $r : \Gamma \rightarrow \mathrm{GL}_n(A)$ such that $r \mod \mathfrak{m}_A = \overline{r}$. By \cite[proof of Proposition 7.14]{Fer}, there exists an object $R_{\overline{r}, \mathcal{O}}$ of $\mathrm{CNL}_{\mathcal{O}}$ representing the functor $\mathrm{CNL}_{\mathcal{O}} \rightarrow \mathrm{Set},  \ A \mapsto \{ r : \Gamma \rightarrow \mathrm{GL}_n(A) \ \mid \ \mathrm{lifting \ of \ } \overline{r} \}$ and $\mathcal{O}[(r^{\mathrm{univ}}(g))_{i, j} \mid g \in \Gamma, i,j = 1, \cdots, n]$ is a dense subalgebra of $R_{\overline{r}}$, where $r^{\mathrm{univ}} : \Gamma \rightarrow \mathrm{GL}_n(R_{\overline{r}})$ is the universal lifting of $\overline{r}$. If $\mathcal{O}$ is clear, we simply write $R_{\overline{r}}$ for $R_{\overline{r}, \mathcal{O}}$.

\begin{lem} \label{homotopy}

Let $A$ be a Noetherian $\mathcal{O}$-algebra, $I$ be a proper ideal of $A$ containing $\varpi$ such that $A$ is $I$-adically complete and $r: \Gamma \rightarrow \mathrm{GL}_n(A)$ be a continuous representation such that $\mathrm{Im}(r \mod I) \subset \mathrm{GL}_n(\mathbb{F})$. 

Then there exists a unique $\mathcal{O}$-morphism $f: R_{\overline{r}} \rightarrow A$ such that $f(\mathfrak{m}_{R_{\overline{r}}}) \subset I$ and $f \circ r^{\mathrm{univ}} = r$.

\end{lem}

\begin{proof}

We have an exact sequence $1 \rightarrow \mathrm{Ker}(r \mod I )/\mathrm{Ker}(r) \rightarrow \Gamma/\mathrm{Ker}(r) \rightarrow \Gamma/\mathrm{Ker}(r \mod I) \rightarrow 1$. Let $I_n$ denote the unit matrix. Since $\Gamma/\mathrm{Ker}(r \mod I ) \hookrightarrow \mathrm{GL}_n(\mathbb{F})$, $\Gamma/\mathrm{Ker}(r \mod I)$ is finite and $\mathrm{Ker}(r \mod I)$ is an open subgroup of $\Gamma$. Since $\mathrm{Im}(r) \cap (I_n + I \mathrm{M}_n(A))$ is a pro-$l$-group, $\mathrm{Ker}(r \mod I)/\mathrm{Ker}(r)$ is topologically finitely generated by the assumption. This implies that $\Gamma/\mathrm{Ker}(r)$ is topologically finitely generated. We write $\gamma_1, \cdots, \gamma_k$ for topological generators of $\Gamma/\mathrm{Ker}(r)$. By $\mathrm{Im}(r \mod I) \subset \mathrm{GL}_n(\mathbb{F})$, for any $1 \le l \le k$, there exists $a_l \in \mathrm{GL}_n(\mathcal{O})$ and $b_l \in I \mathrm{M}_n(A)$ such that $r(\gamma_l) = a_l + b_l$. Let $A_0$ be the closure of $\mathcal{O}[(b_l)_{i,j} \mid 1 \le l \le k, i,j = 1, \cdots, n]$ in $A$. The $\mathcal{O}$-morphism $\varphi : \mathcal{O}[[X_{l, i, j} \mid 1 \le l \le k, i,j = 1, \cdots n ]] \rightarrow A, X_{l,i,j} \mapsto (b_l)_{i,j}$ induces an isomorphism $\mathcal{O}[[X_{l, i, j}]]/\mathrm{Ker}\varphi \stackrel{\sim}{\rightarrow} A_0$ since $\mathcal{O}[[X_{l,i,j}]]$ is compact. This implies that $A_0$ is a complete Noetherian local $\mathcal{O}$-algebra with residue field $\mathbb{F}$ and $r$ is regarded as a lifting of $\overline{r}$ to $A_0$. Therefore, we obtain a local $\mathcal{O}$-morphism $f: R_{\overline{r}} \rightarrow A_0$ such that $f \circ r^{\mathrm{univ}} = r$. Since $\mathcal{O}[(r^{\mathrm{univ}}(g))_{i, j} \mid g \in G_K, i,j = 1, \cdots,n]$ is a dense subalgebra of $R_{\overline{r}}$, $f$ is unique.  \end{proof}

\begin{cor}\label{coefficient change}

Let $\overline{r} : \Gamma \rightarrow \mathrm{GL}_n(\mathbb{F})$ be a continuous representation, $A$ be a Noetherian $\mathcal{O}$-algebra and $I$ be a proper ideal of $A$ containing $\varpi$ such that $A$ is $I$-adically complete. Then the set $$\{ f : R_{\overline{r}, \mathcal{O}} \rightarrow A \ \mid \ \mathcal{O}\textrm{-}\mathrm{morphism} \ \mathrm{such \ that} \ f(\mathfrak{m}_{R_{\overline{r}}}) \subset I \}$$ is naturally isomorphic to the set $\{ r : \Gamma \rightarrow \mathrm{GL}_n(A) \ \mid \ \mathrm{continuous \ and} \ r \mod I = \overline{r} \}$.

In particular, we have a canonical $\mathcal{O}_{E'}$-isomorphism $R_{\overline{r}, \mathcal{O}} \otimes_{\mathcal{O}} \mathcal{O}_{E'} \cong R_{\overline{r}, \mathcal{O}_{E'}}$ for any finite extension $E'$ of $E$ contained in $\overline{\mathbb{Q}}_l$.

\end{cor}

\begin{proof}

The first result follows from Lemma \ref{homotopy}. For any object $A$ of $\mathrm{CNL}_{\mathcal{O}_{E'}}$, the set $\{ f : R_{\overline{r}, \mathcal{O}} \rightarrow A \ \mid \ \mathrm{local} \ \mathcal{O}\textrm{-}\mathrm{morphism} \}$ is naturally identified with $\mathrm{Hom}_{\mathrm{CNL}_{\mathcal{O}_{E'}}}(R_{\overline{r}, \mathcal{O}} \otimes_{\mathcal{O}} \mathcal{O}_{E'}, A)$. Thus, we obtain the second result by the first result. \end{proof}

\begin{lem} \label{density of closed point}

    Let $A$ be a Noetherian $\mathcal{O}$-algebra and $I$ be a proper ideal containing $\varpi$ such that $A$ is $I$-adically complete. We suppose that for all closed points $x$ of $\mathrm{Spec} \, A/\varpi$, the residue field $k(x)$ is a finite extension of $\mathbb{F}$. 
    
    Then we have the following properties.
    
    1 \ For any maximal ideal $\mathfrak{p}$ of $A[\frac{1}{l}]$, $A[\frac{1}{l}]/\mathfrak{p}$ is a finite extension of $E$ and $A/(\mathfrak{p} \cap A)$ is a finite free $\mathcal{O}$-module and hence an order of $A[\frac{1}{l}]/\mathfrak{p}$. (We simply write $\mathfrak{p} \cap A$ for the inverse image of $\mathfrak{p}$ by $A \rightarrow A[\frac{1}{l}]$. In the following, we use such notations.)
    
    2 \ Assume $A$ is local or equal to $\mathcal{O}\langle T_1, \cdots, T_k \rangle/I$ for some proper ideal $I$. Then for any proper ideal $J$ of $A[\frac{1}{l}]$, the radical of $J$ is equal to the intersection of all maximal ideals of $A[\frac{1}{l}]$ containing $J$. In other words, for any closed subscheme $X$ of $\mathrm{Spec} \, A[\frac{1}{l}]$, $X$ is equal to the closure of the set of all closed points in $X$.
    
    \end{lem}

    \begin{proof} 1 is proved by the same argument as \cite[Lemma 2.6]{IA} and 2 follows from \cite[Lemma 2.6]{IA} and \cite[Appendix]{CW}. \end{proof}

\begin{dfn}

For a field $k$ and $k$-scheme $X$, we say that $X$ is geometrically irreducible if $X \otimes_{k} k'$ is irreducible for any finite extension $k'/k$.

\end{dfn}

\begin{lem} \label{geometrically irreducibility}

For any object $A$ of $\mathrm{CNL}_{\mathcal{O}}$, there exists a finite extension $E'$ of $E$ such that all irreducible components of $\mathrm{Spec} \, A \otimes_{\mathcal{O}} E'$ (resp. $\mathrm{Spec} \, A \otimes_{\mathcal{O}} \mathbb{F}_{E'}$) are geometrically irreducible as $E'$-schemes (resp. $\mathbb{F}_{E'}$-schemes).

\end{lem}

\begin{proof}

Let $Q_1, \cdots, Q_k$ be all minimal primes of $A[\frac{1}{l}]$. We fix $1 \le i \le k$. Let $\widetilde{A[\frac{1}{l}]/Q_i}$ be the normalization of $A[\frac{1}{l}]/Q_i$. Since $A[\frac{1}{l}]/Q_i$ is excellent, $\widetilde{A[\frac{1}{l}]/Q_i}$ is finite over $A[\frac{1}{l}]/Q_i$ and consequently, for any maximal ideal $\mathfrak{p}$ of $\widetilde{A[\frac{1}{l}]/Q_i}$, $(\widetilde{A[\frac{1}{l}]/Q_i})/\mathfrak{p}$ is a finite extension of $E$ by 1 of Lemma \ref{density of closed point}. This implies that the algebraic closure $E_i$ of $E$ in $\mathrm{Frac}(A[\frac{1}{l}]/Q_i)$ is a finite extension of $E$. We fix an embedding $E_i \hookrightarrow \overline{\mathbb{Q}}_l$ over $E$. Let $E'$ be the Galois closure of $E_i$'s over $E$ in $\overline{\mathbb{Q}}_l$. Let $\eta_1, \cdots, \eta_k$ be the generic points of $\mathrm{Spec} \, A[\frac{1}{l}]$ corresponding to $Q_1, \cdots, Q_k$ respectively.

Since $\mathrm{Spec} \, A[\frac{1}{l}] \otimes_E E' \rightarrow \mathrm{Spec} \, A[\frac{1}{l}]$ is flat, any generic point of $\mathrm{Spec} \, A[\frac{1}{l}] \otimes_E E'$ is contained $\eta_i \otimes_E E'$ for some $i$. Since $\eta_i \otimes_{E} E' = \prod_{\tau \in \mathrm{Hom}_E(E_i, E')}\eta_i \otimes_{E_i, \tau} E'$ and each $\eta_i \otimes_{E_i, \tau} E'$ is geometrically irreducible as an $E'$-scheme, any irreducible component of $\mathrm{Spec} \, A[\frac{1}{l}] \otimes_{E} E'$ is geometrically irreducible as an $E'$-scheme.

By a similar argument for $A/\varpi$, after extending $E'$, we obtain the result. \end{proof}

\begin{lem}\label{completed tensor product irreducibility}

1 \ For any objects $A, B$ of $\mathrm{CNL}_{\mathcal{O}}$, if $A$ and $B$ are $\mathcal{O}$-flat and $\mathrm{Spec} \, A[\frac{1}{l}]$ and $\mathrm{Spec} \, B[\frac{1}{l}]$ are geometrically irreducible, then $A \widehat{\otimes}_{\mathcal{O}} B$ is $\mathcal{O}$-flat and $A \widehat{\otimes}_{\mathcal{O}} B[\frac{1}{l}]$ is geometrically irreducible of dimension $\mathrm{dim} \, A[\frac{1}{l}] + \mathrm{dim} \, B[\frac{1}{l}]$. ($A \widehat{\otimes}_{\mathcal{O}} B$ denotes the completion $A \otimes_{\mathcal{O}} B$ by the maximal ideal $(\mathfrak{m}_A, \mathfrak{m}_B)$.)

2 \ Let $A \ (resp. B)$ be an object of $\mathrm{CNL}_{\mathcal{O}}$, $\mathfrak{p}_1, \cdots, \mathfrak{p}_{s}$ (resp. $\mathfrak{q}_1, \cdots, \mathfrak{q}_{u}$) be the different minimal primes of $A$ (resp. $B$) such that $A/\mathfrak{p}_i$ (resp. $B/\mathfrak{q}_j$) is flat over $\mathcal{O}$ for any $i$ (resp. $j$). We assume that $\mathrm{Spec} \, A/\mathfrak{p}_i[\frac{1}{l}]$ (resp. $\mathrm{Spec} \, B/\mathfrak{q}_j[\frac{1}{l}]$) is geometrically irreducible for any $i$ (resp. $j$). Then the different minimal primes of $A \widehat{\otimes}_{\mathcal{O}} B$ whose residue fields have characteristic zero are given by $(\mathfrak{p}_i, \mathfrak{q}_j)$ for $i = 1, \cdots, s$ and $j = 1, \cdots, u$.

3 \ For any objects $A, B$ of $\mathrm{CNL}_{\mathbb{F}}$, if $A$ and $B$ are geometrically irreducible, then $A \widehat{\otimes}_{\mathbb{F}} B$ is geometrically irreducible. (Here, $\mathrm{CNL}_{\mathbb{F}}$ denotes the full subcategory of $\mathrm{CNL}_{\mathcal{O}}$ consisting of all complete Noetherian local $\mathbb{F}$-algebras with residue field $\mathbb{F}$.)

4 \ Let $A \ (resp. B)$ be an object of $\mathrm{CNL}_{\mathbb{F}}$, $\mathfrak{p}_1, \cdots, \mathfrak{p}_{s}$ (resp. $\mathfrak{q}_1, \cdots, \mathfrak{q}_{u}$) be the different minimal primes of $A$ (resp. $B$). We assume that $A/\mathfrak{p}_i$ (resp. $B/\mathfrak{q}_j$) is geometrically irreducible for any $i$ (resp. $j$). Then the different minimal prime of $A \widehat{\otimes}_{\mathbb{F}} B$ are given by $(\mathfrak{p}_i, \mathfrak{q}_j)$ for $i = 1, \cdots, s$ and $j = 1, \cdots, u$.

\end{lem}

\begin{proof}

See \cite[Lemma 3.3]{Calabi}. \end{proof}

\begin{lem}\label{completed tensor product regularity}

1 \ Let $A, B$ be Noetherian local rings and $f : A \rightarrow B$ be a flat local morphism. 

(a) \ If $B$ is regular, then $A$ is regular.

(b) \ If $A$ and $B/\mathfrak{m}_AB$ are regular, then $B$ is regular.

2 \ Let $k$ be a field and $A$ be a complete Noetherian local $k$-algebra with residue field $k$. Then the following conditions are equivalent.

(c) \ $A$ is regular.

(d) \ $A \cong k[[X_1, \cdots, X_m]]$ for some nonnegative integer $m$.

3 \ For objects $A, B$ of $\mathrm{CNL}_{\mathcal{O}}$, $\mathfrak{p} \in \mathrm{Spec} \, A(E)$ and $\mathfrak{q} \in \mathrm{Spec} \, B(E)$, we have $\widehat{A_{\mathfrak{p}}} \widehat{\otimes}_{E} \widehat{B_{\mathfrak{q}}} \cong \widehat{(A \widehat{\otimes}_{\mathcal{O}} B)}_{(\mathfrak{p}, \mathfrak{q})}$. In particular, if $\mathfrak{p}$, $\mathfrak{q}$ are regular points of $\mathrm{Spec} \, A$, $\mathrm{Spec} \, B$ respectively, then $(\mathfrak{p}, \mathfrak{q})$ is a regular point of $\mathrm{Spec} \, A \widehat{\otimes}_{\mathcal{O}} B$.

\end{lem}

\begin{proof}

See \cite[Theorem 23.7]{Mat} for $(a)$ and $(b)$ of 1.

2 \ $(d) \Rightarrow (c)$ is trivial. If $(c)$ holds, we have a surjective $k$-morphism $k[[X_1, \cdots, X_{\mathrm{dim} \, A}]] \twoheadrightarrow A$ and the kernel of this morphism is zero since their dimensions are equal and $A$ is an integral domain.

3 \ We have canonical $E$-morphisms $\widehat{A_{\mathfrak{p}}} \rightarrow \widehat{(A \widehat{\otimes}_{\mathcal{O}} B)}_{(\mathfrak{p}, \mathfrak{q})}$, $\widehat{B_{\mathfrak{q}}} \rightarrow \widehat{(A \widehat{\otimes}_{\mathcal{O}} B)}_{(\mathfrak{p}, \mathfrak{q})}$. This induces $\widehat{A_{\mathfrak{p}}} \widehat{\otimes}_{E} \widehat{B_{\mathfrak{q}}} \rightarrow \widehat{(A \widehat{\otimes}_{\mathcal{O}} B)}_{(\mathfrak{p}, \mathfrak{q})}$. Note that for any positive integer $s$, we have $\widehat{A_{\mathfrak{p}}} \widehat{\otimes}_E \widehat{B_{\mathfrak{q}}}/(\mathfrak{p}^s, \mathfrak{q}^s) = A_{\mathfrak{p}}/\mathfrak{p}^s \otimes_{E} B_{\mathfrak{q}}/\mathfrak{q}^s$ and $A \widehat{\otimes}_{\mathcal{O}} B/(\mathfrak{p}^s, \mathfrak{q}^s) = A/\mathfrak{p}^s \otimes_{\mathcal{O}} B/\mathfrak{q}^s$ since $A/\mathfrak{p}^s \otimes_{\mathcal{O}} B/\mathfrak{q}^s$ is finite over $\mathcal{O}$. Thus, we obtain the inverse morphism $\widehat{(A \widehat{\otimes}_{\mathcal{O}} B)}_{(\mathfrak{p}, \mathfrak{q})} \rightarrow \widehat{A_{\mathfrak{p}}} \widehat{\otimes}_{E} \widehat{B_{\mathfrak{q}}}$ of the above morphism. \end{proof}

In the following subsections, we fix a prime $p$ and a finite extension $K$ of $\mathbb{Q}_p$, and we put $\Gamma := G_K$.

\subsection{$p \neq l$}

In this subsection, we assume $p \neq l$.

\begin{prop} \label{dimension of unristricted}

    For a continuous representation $\overline{r}: G_K \rightarrow \mathrm{GL}_n(\mathbb{F})$, all irreducible components of $\mathrm{Spec} \, R_{\overline{r}}[\frac{1}{l}]$ have dimension $n^2$.
                
    \end{prop}
     
\begin{proof} See \cite[Proposition 1.2.2]{Allen}. \end{proof}

\begin{dfn}

For a Weil-Deligne representation $r$ of $W_K$, we say that $r$ is generic if $\mathrm{Hom}_{\mathrm{WD}}(r, r(1))= 0$.

\end{dfn}

\begin{rem}

For a Weil-Deligne representation $r$ of $W_K$, if $r^{F-ss}$ is generic, then $r$ is generic.

\end{rem}

\begin{lem} \label{generic equivalent}

For an irreducible smooth representation $\pi$ of $\mathrm{GL}_n(K)$ over $\mathbb{C}$, $\pi$ is generic if and only if $\mathrm{rec}_K(\pi)$ is generic.

\end{lem}

\begin{proof} See \cite[Lemma 1.1.3]{Allen}. \end{proof}

\begin{rem}

Note that if $r$ is a pure Weil-Deligne representation of $W_K$, then $r$ is generic. (This is easily proved by definition or by 1 and 4 of Lemma \ref{purity lemma}.)

\end{rem}

For a continuous representation $\overline{r} : G_K \rightarrow \mathrm{GL}_n(\mathbb{F})$ and a closed point $\mathfrak{p} \in \mathrm{Spec} \, R_{\overline{r}}[\frac{1}{l}]$, we write $r_{\mathfrak{p}}$ for the composition $G_K \rightarrow \mathrm{GL}_n(R_{\overline{r}}) \rightarrow \mathrm{GL}_n(R_{\overline{r}}/\mathfrak{p} \cap R_{\overline{r}}) \hookrightarrow \mathrm{GL}_n(\mathcal{O}_{k(\mathfrak{p})})$. Note that $k(\mathfrak{p})$ is finite extension of $E$ and $R_{\overline{r}}/\mathfrak{p} \cap R_{\overline{r}}$ is an order of $k(\mathfrak{p})$ by 1 of Lemma \ref{density of closed point}.

\begin{lem} \label{regular point}

Let $\mathfrak{p}$ be a closed point of $\mathrm{Spec} \, R_{\overline{r}}[\frac{1}{l}]$.

Then $\mathfrak{p}$ is a regular point of $\mathrm{Spec} \, R_{\overline{r}}[\frac{1}{l}]$ if and only if $\mathrm{WD}(r_{\mathfrak{p}} \otimes \overline{\mathbb{Q}}_l)$ is generic.

\end{lem}

\begin{proof}

See \cite[Proposition 1.2.2]{Allen}.  \end{proof}

\begin{dfn} \label{robustly smooth}

Let $\overline{r} : G_K \rightarrow \mathrm{GL}_n(\mathbb{F})$ be a continuous representation and $\mathfrak{p}$ be a closed point of $\mathrm{Spec} \, R_{\overline{r}}[\frac{1}{l}]$.

We say that $\mathfrak{p}$ is robustly smooth if $\mathrm{WD}(r_{\mathfrak{p}}|_{G_L} \otimes \overline{\mathbb{Q}}_l)$ is generic for any finite extension $L/K$.

\end{dfn}

\begin{lem}

Let $\overline{r} : G_K \rightarrow \mathrm{GL}_n(\mathbb{F})$ be a continuous representation.

Then the set of robustly smooth points is Zariski dense in $\mathrm{Spec} \, R_{\overline{r}}[\frac{1}{l}]$.

\end{lem}

\begin{proof}

See \cite[Lemma 1.3.2, (2)]{CW}. \end{proof}

Next, we study irreducible components of $\mathrm{Spec} \, R_{\overline{r}}[\frac{1}{l}]$.

\begin{lem} \label{equivalence irreducible}

Let $\overline{r}: G_K \rightarrow \mathrm{GL}_n(\mathbb{F})$ be a continuous representation and $g \in \mathrm{GL}_n(\mathcal{O})$ such that $\overline{grg^{-1}} = \overline{r}$. 

Then the action $ \mathrm{Spec} \, R_{\overline{r}}[\frac{1}{l}] \rightarrow \mathrm{Spec} \, R_{\overline{r}}[\frac{1}{l}], \rho \rightarrow g \rho g^{-1}$ fixes all irreducible components. 

\end{lem}

\begin{proof} This is \cite[Lemma 1.2.2]{CW}. \end{proof}

We recall the definition of the notion $\sim$.

For continuous representations $r_1, r_2 : G_{K} \rightarrow \mathrm{GL}_n(\mathcal{O})$, we say $r_1 \sim r_2$ if $r_1$ and $r_2$ satisfy the following conditions after replacing $E$ by a finite extension if necessary. 

1 \ $\overline{r}_1 \cong \overline{r}_2$.

2 \ $r_1$ and $r_2$ are contained in the same geometrically irreducible component of $\mathrm{Spec} \, R_{\overline{r}_1}[\frac{1}{l}]$ after fixing an identification $\overline{r}_1 = \overline{r}_2$. 

Note that the condition 2 is independent of the choice of an identification $\overline{r}_1 = \overline{r}_2$ by Lemma \ref{equivalence irreducible}.

For continuous representations $r_1, r_2 : G_{K} \rightarrow \mathrm{GL}_n(\mathcal{O}_{\overline{\mathbb{Q}}_l})$, we can take a sufficiently large finite extension $E'$ of $E$ contained in $\overline{\mathbb{Q}}_l$ such that $\mathrm{Im}(r_1), \mathrm{Im}(r_2) \subset \mathrm{GL}_n(\mathcal{O}_{E'})$ and consequently $r_1$, $r_2$ are regarded as representations $r_1^{E'}, r_2^{E'} : G_K \rightarrow \mathrm{GL}_n(\mathcal{O}_{E'})$ respectively. We say $r_1 \sim r_2$ if $r_1^{E'} \sim r_2^{E'}$. 

We recall the following fundamental properties of the notion $\sim$.

\begin{lem}\label{coincide monodromy type}

Let $r, r_1, r_2 : G_K \rightarrow \mathrm{GL}_n(\mathcal{O})$ and $r_3, r_4 : G_K \rightarrow \mathrm{GL}_m(\mathcal{O})$ be continuous representations.

1 \ If $r_1$ and $r_2$ are unramified and $\overline{r}_1 \cong \overline{r}_2$, then $r_1 \sim r_2$.

2 \ For an unramified character $\mu: G_K \rightarrow \mathcal{O}^{\times}$ such that $\overline{\mu}$ is trivial, we have $r \sim r \otimes \mu$.

3 \ Assume $\overline{r}$ is semisimple. Then for any $G_K$-stable increasing filtration $\{ \mathrm{Fil}_i \}$ on $r$ by $\mathcal{O}$-direct summands, we have $r \sim \oplus_i \mathrm{Fil}_i/\mathrm{Fil}_{i-1}$.

4 If $r_1 \sim r_2$, then $(r_1 \otimes \overline{\mathbb{Q}}_l)|^{ss}_{I_K} \cong (r_2 \otimes \overline{\mathbb{Q}}_l)|^{ss}_{I_K}$.

5 \ If $r_1 \sim r_2$ and for any finite extension $E'/E$, there exists a unique irreducible component $\mathcal{C}_1$ (resp. $\mathcal{C}_2$) of $\mathrm{Spec} \, R_{\overline{r}, \mathcal{O}_{E'}}[\frac{1}{l}]$ such that $r_1$ (resp. $r_2$) is contained in $\mathcal{C}_1$ (resp. $\mathcal{C}_2$), then $(r_1 \otimes \overline{\mathbb{Q}}_l)|_{I_K} \cong (r_2 \otimes \overline{\mathbb{Q}}_l)|_{I_K}$. In particular, $\mathrm{WD}(r_1 \otimes \overline{\mathbb{Q}}_l)$ and $\mathrm{WD}(r_2 \otimes \overline{\mathbb{Q}}_l)$ have the same monodromy type.

6 \ If $r_1 \sim r_2$, then $r_1|_{G_L} \sim r_2|_{G_L}$ for any finite extension $L/K$.

7 \ If $r_1 \sim r_2$ and $r_3 \sim r_4$, then $r_1 \oplus r_3 \sim r_2 \oplus r_4$ and $r_1 \otimes r_3 \sim r_2 \otimes r_4$.

\end{lem}

\begin{proof} See \cite[p 524]{CW}. \end{proof}

For an object $A$ of $\mathrm{CNL}_{\mathcal{O}}$ and a continuous representation $r : G_K \rightarrow \mathrm{GL}_n(A)$, we say that $r$ is unipotently ramified if $\mathrm{det}(TI_n - r(\sigma)) = (T-1)^n$ for all $\sigma \in I_K$. 

Next, we study relations between irreducible components of the unipotently ramified lifting ring and monodromy types.

We recall the construction and some basic properties of the unipotently ramified lifting ring studied in \cite{small} and \cite{IA}. 

\vspace{0.5 \baselineskip}

(*) \ In the following argument in this subsection, we fix a unipotently ramified representation $\overline{r}: G_K \rightarrow \mathrm{GL}_n(\mathbb{F})$.

\vspace{0.5 \baselineskip}

Let $f : R_{\overline{r}} \twoheadrightarrow R^{\mathrm{unip}}_{\overline{r}}$ denote the quotient of $R_{\overline{r}}$ defined by the local deformation problem consisting of all unipotently ramified liftings of $\overline{r}$. (See Definition \ref{local deformation problem} for the definition of the local deformation problem.)

Let $t_l: I_K \rightarrow \mathbb{Z}_l(1)$ be the $l$-adic tamely ramified character. We fix an identification $\mathbb{Z}_l(1) = \mathbb{Z}_l$, an element $\sigma \in t_l^{-1}(1)$ and a geometric Frobenius lift $\phi \in W_K$. 

For any unipotently ramified lifting $r$ of $\overline{r}$, we have $\mathrm{Ker}(t_l) \subset \mathrm{Ker}(r)$. If $r$ is a unipotently ramified lifting of $\overline{r}$ to $\mathcal{O}$, $\mathrm{log}(r(\sigma)) \in \mathrm{M}_n(E)$ is the monodromy operator of $\mathrm{WD}(r \otimes \overline{\mathbb{Q}}_l)$.

Let $\mathcal{M}/\mathcal{O}$ be the moduli space of pairs of $n \times n$-matrices $(\Phi, \Sigma)$ such that $\Phi$ is invertible, $\Sigma$ has characteristic polynomial $(X-1)^n$ and $\Phi \Sigma \Phi^{-1} = \Sigma^{q_K}$. Let $\mathcal{N}/\mathcal{O}$ denote the moduli space of $n \times n$-matrices $\Sigma$ such that $\Sigma$ has characteristic polynomial $X^n$. 

Then we have the canonical map $\mathcal{M} \rightarrow \mathcal{N}$, $(\Phi, \Sigma) \mapsto \Sigma-I_n$. ($I_n$ denotes the identity matrix.) Moreover, we have a morphism of affine $\mathcal{O}$-schemes $F: \mathrm{Spec} \, R_{\overline{r}}^{\mathrm{unip}} \rightarrow \mathcal{M}, \ \rho \mapsto (\rho(\phi), \rho(\sigma))$. We put $x:=F(\overline{r})$. Then $F$ induces $\widehat{\mathcal{O}_{\mathcal{M}, x}} \stackrel{\sim}{\rightarrow} R_{\overline{r}}^{\mathrm{unip}}$.

Let $\tau=[n_1, \cdots, n_s]$ be a partition of $n$ and $\rho=[m_1, \cdots, m_u]$ be the dual partition of $\tau$.

We write $\mathcal{N}(\tau)$ for the reduced closed subscheme of $\mathcal{N}$ defined by the condition that for all $i$, all $n + 1 - m_1 - \cdots - m_i$ minors of $N_0^i$ vanish, where $N_0$ is the universal nilpotent matrix over $\mathcal{N}$. We put $\mathcal{N}(\tau)^0 := \mathcal{N}(\tau) \setminus (\cup_{\tau' \prec \tau, \tau' \neq \tau} \mathcal{N}(\tau'))$. For a field $L$ which is an $\mathcal{O}$-algebra and $N \in \mathcal{N}(L)$, the condition $N \in \mathcal{N}(\tau)^0(L)$ is equivalent to that the Jordan canonical form of $N$ corresponds to $\tau$. This implies $\mathcal{N} = \cup_{\tau} \mathcal{N}(\tau) = \sqcup_{\tau} \mathcal{N}(\tau)^0$.

Let $\mathcal{M}(\tau)^0 := \mathcal{M} \times_{\mathcal{N}} \mathcal{N}(\tau)^0$ and $\mathcal{M}(\tau)$ be the reduced closed subscheme of $\mathcal{M}$ defined by the closure of $\mathcal{M}(\tau)^0$ in $\mathcal{M}$. Note that $\mathcal{M} = \cup_{\tau} \mathcal{M}(\tau) = \sqcup_{\tau} \mathcal{M}(\tau)^0$, $\mathcal{M}(\tau) \subset \mathcal{M} \times_{\mathcal{N}} \mathcal{N}(\tau)$ and consequently $\mathcal{M}(\tau)^0$ is an open subset of $\mathcal{M}(\tau)$.

\begin{prop} \label{moduli of Weil-Deligne} Let $\tau$ be a partition of $n$.

1 \ The scheme $\mathcal{M}(\tau)^0$ is smooth over $\mathrm{Spec} \, \mathcal{O}$ and all geometric fibers of $\mathcal{M}(\tau)^0 \rightarrow \mathrm{Spec} \, \mathcal{O}$ are irreducible schemes of dimension $n^2$.

2 \ The scheme $\mathcal{M}(\tau)$ is an irreducible component of $\mathcal{M}$ of dimension $n^2+1$ and flat over $\mathrm{Spec} \, \mathcal{O}$, and all geometric fibers of $\mathcal{M}(\tau) \rightarrow \mathrm{Spec} \, \mathcal{O}$ are irreducible schemes of dimension $n^2$.

3 \ By sending $\rho$ to $\mathcal{M}(\rho)$ (resp. $\rho$ to $\mathcal{M}(\rho) \otimes_{\mathcal{O}} \mathbb{F}$), we have a bijection between the set of partitions of $n$ and the set of irreducible components of $\mathcal{M}$ (resp. $\mathcal{M} \otimes_{\mathcal{O}} \mathbb{F}$).

\end{prop}

\begin{proof} 
    
This follows from \cite[proof of Lemma 3.15]{small}. \end{proof}

\begin{cor} \label{moduli of Weil-Deligne 2}

1 \ All irreducible components of $\mathrm{Spec} \, R_{\overline{r}}^{\mathrm{unip}}$ are flat over $\mathcal{O}$ and have dimension $n^2+1$.

2 \ $\mathrm{Spec} \, R_{\overline{r}}^{\mathrm{unip}}[\frac{1}{l}]$ is a union of irreducible components of $\mathrm{Spec} \, R_{\overline{r}}[\frac{1}{l}]$.

\end{cor}

\begin{proof}

1 follows from Proposition \ref{moduli of Weil-Deligne}. This implies 2 by Proposition \ref{dimension of unristricted} (or we obtain 2 by 4 of Lemma \ref{coincide monodromy type}). \end{proof}

\begin{cor}  \label{regular morphism}

Let $\tau$ be a partition of $n$ and $\mathcal{X}(\tau):=\mathrm{Spec} \, R_{\overline{r}}^{\mathrm{unip}} \times_{F, \mathcal{M}} \mathcal{M}(\tau)^0$.

1 \ $\mathcal{X}(\tau)$ is a regular scheme.

2 \ For any $y \in \mathrm{Spec} \, R_{\overline{r}}^{\mathrm{unip}}[\frac{1}{l}]$, the condition $y \in \mathcal{X}(\tau)[\frac{1}{l}]$ is equivalent to that the Jordan canonical form of $\mathrm{log}(r_y(\sigma)) \in \mathrm{M}_n(k(y))$ corresponds to $\tau$, where $r_y:=r^{\mathrm{univ}} \mod y$.

In particular, for $r \in \mathrm{Spec} \, R_{\overline{r}}^{\mathrm{unip}}(E)$, the condition $r \in \mathcal{X}(\tau)(E)$ is equivalent to that the monodromy type of $\mathrm{WD}(r \otimes \overline{\mathbb{Q}}_l)$ is equal to $\tau$.

\end{cor}

\begin{proof}

1 \ Since $\mathcal{M}$ is excellent, $\mathcal{X}(\tau)=\mathrm{Spec} \, \widehat{\mathcal{O}_{\mathcal{M}, x}} \times_{\mathcal{M}} \mathcal{M}(\tau)^0  \rightarrow \mathcal{M}(\tau)^0$ is a regular morphism. In particular, this is flat and all fibers of this morphism are regular schemes. Since $\mathcal{M}(\tau)^0$ is smooth over $\mathrm{Spec} \, \mathcal{O}$ by 1 of Proposition \ref{moduli of Weil-Deligne}, $\mathcal{X}(\tau)$ is a regular scheme by (2) of 1 of Lemma \ref{completed tensor product regularity}.

2 \ The condition $y \in \mathcal{X}(\tau)[\frac{1}{l}] = \mathrm{Spec} \, \widehat{\mathcal{O}_{\mathcal{M}, x}} \times_{\mathcal{N}} \mathcal{N}(\tau)^0[\frac{1}{l}]$ is equivalent to that the Jordan canonical form of $r_y(\sigma) - I_n$ corresponds to the partition $\tau$. This implies the result because the nilpotent matrices $\mathrm{log}(r_y(\sigma))$ and $r_y(\sigma) - I_n$ have the same Jordan canonical form.

Note that for $r \in \mathrm{Spec} \, R_{\overline{r}}^{\mathrm{unip}}(E)$, $\mathrm{log}(r(\sigma))$ is the monodromy operator of $\mathrm{WD}(r \otimes \overline{\mathbb{Q}}_l)$. \end{proof}

\begin{lem}\label{robustly smooth monodromy type}

Let $r \in \mathrm{Spec} \, R_{\overline{r}}^{\mathrm{unip}}(E)$. Then after replacing $E$ by a finite extension, there exists a robustly smooth unipotently ramified lifting $r' : G_K \rightarrow \mathrm{GL}_n(\mathcal{O})$ of $\overline{r}$ such that $\mathrm{WD}(r \otimes \overline{\mathbb{Q}}_l)$ and $\mathrm{WD}(r' \otimes \overline{\mathbb{Q}}_l)$ have the same monodromy type and $r \sim r'$.
        
\end{lem}
        
\begin{proof} 
After extending $E$, we may assume that all irreducible component of $\mathrm{Spec} \, R_{\overline{r}}^{\mathrm{unip}}[\frac{1}{l}]$ are geometrically irreducible. Let $\tau$ be the monodromy type of $\mathrm{WD}(r \otimes \overline{\mathbb{Q}}_l)$. By 2 of Proposition \ref{moduli of Weil-Deligne}, $\mathcal{M}(\tau)$ is an irreducible component of $\mathcal{M}$. We take an irreducible component $\mathcal{C}$ of $\mathrm{Spec} \, \widehat{\mathcal{O}_{\mathcal{M}, x}}[\frac{1}{l}] = \mathrm{Spec} \, R_{\overline{r}}^{\mathrm{unip}}[\frac{1}{l}]$ containing $r$ whose generic point is sent to the generic point of $\mathcal{M}(\tau)[\frac{1}{l}]$ by $F : \mathrm{Spec} \, R_{\overline{r}}^{\mathrm{unip}}[\frac{1}{l}] \rightarrow \mathcal{M}[\frac{1}{l}]$. This is possible because $F$ is flat. Since $\mathcal{M}(\tau)^0 = \mathcal{M} \times_{\mathcal{N}}\mathcal{N}(\tau)^0$ is a nonempty open subscheme of $\mathcal{M}(\tau) \subset \mathcal{M} \times_{\mathcal{N}} \mathcal{N}(\tau)$, we can take a nonempty open subset $U$ of $\mathrm{Spec} \, R_{\overline{r}}[\frac{1}{l}]$ contained in $\mathcal{C}$ such that $F(U) \subset \mathcal{M}(\tau)^0[\frac{1}{l}]$. The set of all robustly smooth points is Zariski dense in $\mathrm{Spec} \, R_{\overline{r}}[\frac{1}{l}]$ by Lemma \ref{robustly smooth} and therefore in $U$. This implies the lemma. \end{proof}

\begin{prop}\label{monodromy type irreducible component}

We assume that $\overline{r}$ is trivial. Let $r_1, r_2 \in \mathrm{Spec} \, R_{\overline{r}}^{\mathrm{unip}}(E)$.
    
If $\mathrm{WD}(r_1 \otimes \overline{\mathbb{Q}}_l)$ and $\mathrm{WD}(r_2 \otimes \overline{\mathbb{Q}}_l)$ have the same monodromy type, then $r_1 \sim r_2$.
    
\end{prop}

\begin{proof} We may assume that all irreducible components of $\mathrm{Spec} \, R_{\overline{r}}[\frac{1}{l}]$ are geometrically irreducible. 
    
Let $\tau:=[n_1, \cdots, n_k]$ be the monodromy type of $\mathrm{WD}(r_1 \otimes_{\mathcal{O}} \overline{\mathbb{Q}}_l)$. Then $r_1$ and $r_2$ define $E$-valued points of $\mathcal{X}(\tau)[\frac{1}{l}]$ by 2 of Corollary \ref{regular morphism}. Let $\psi_1, \cdots, \psi_k : W_K \rightarrow \overline{\mathbb{Q}}_l^{\times}$ be unramified characters such that $\mathrm{WD}(r_1 \otimes \overline{\mathbb{Q}}_l)^{F-ss} \cong \mathrm{Sp}_{n_1}(\psi_1) \oplus \cdots \oplus \mathrm{Sp}_{n_k}(\psi_k)$. After extending $E$, we may assume $\psi_i(\mathrm{Frob_K}) \subset 1 + \varpi \mathcal{O}$ for any $i$ since $\psi_i(\mathrm{Frob}_K)$ is an eigenvalue of $r_1(\mathrm{Frob}_K)$.

By \cite[Lemma 1.3.1]{CW}, after extending $E$ and changing the order of $(n_1, \cdots, n_k)$ if necessary, there exists an increasing $G_K$-stable filtration $\{ \mathrm{Fil}_i \}$ on $r_1$ by $\mathcal{O}$-direct summands such that $\mathrm{WD}(\mathrm{Fil}_i/\mathrm{Fil}_{i-1} \otimes \overline{\mathbb{Q}}_l) \cong \mathrm{Sp}_{n_i}(\psi_i)$ for all $i$.  Then $r_1':=\oplus_i \mathrm{Fil}_i/\mathrm{Fil}_{i-1}$ is contained in $\mathcal{X}(\tau)[\frac{1}{l}]$ by Corollary \ref{regular morphism}. 

We can take a basis $\{ e_{i, j} \}_{1 \le i \le k, 1 \le j \le t_i}$ of $\mathcal{O}^{\oplus n}$ such that $\{ e_{i',j} \mid 1 \le i' \le i, 1 \le j \le t_{i'} \}$ is a basis of $\mathrm{Fil}_i$ for any $i$. Let $h$ denotes the element of $\mathrm{GL}_n(\mathcal{O} \langle T \rangle [$$\frac{1}{T}$$])$ such that $he_{i,j} = T^{-i}e_{i,j}$ for all $i, j$. Then $h r_1 h^{-1} : G_K \rightarrow \mathrm{GL}_n(\mathcal{O}\langle T \rangle[\frac{1}{T}])$ satisfies $\mathrm{Im}(h r h^{-1}) \subset \mathrm{GL}_n(\mathcal{O}\langle T \rangle)$ and $h r_1 h^{-1} \mod \varpi$ is trivial. By Lemma \ref{homotopy}, we obtain $\varphi : R_{\overline{r}}[\frac{1}{l}] \rightarrow \mathcal{O}\langle T \rangle[\frac{1}{l}]$ such that $\varphi \circ r^{\mathrm{univ}} = h r_1 h^{-1}$. Let $\varphi^* : \mathrm{Spec} \, \mathcal{O}\langle T \rangle[\frac{1}{l}] \rightarrow \mathrm{Spec} \, R_{\overline{r}}[\frac{1}{l}]$ denote the $\mathcal{O}$-morphism corresponding to $\varphi$. 

Note that for any point $\mathfrak{p} \in \mathrm{Spec} \, \mathcal{O}\langle T \rangle[\frac{1}{l}]$, we have $\varphi^*(\mathfrak{p}) \in \mathcal{X}(\tau)[\frac{1}{l}]$. In fact, if $T$ is not zero in the residue field $k(\mathfrak{p})$, then $h \mod \mathfrak{p} \in \mathrm{GL}_n(k(\mathfrak{p}))$ and we obtain $\varphi^*(\mathfrak{p}) \in \mathcal{X}(\tau)[\frac{1}{l}]$ by 2 of Corollary \ref{regular morphism}. If $\mathfrak{p}=(T)$, then we have $\varphi^*((T)) = r_1'$. Since $\mathcal{X}(\tau)[\frac{1}{l}]$ is a subscheme of $\mathrm{Spec} \, R_{\overline{r}}[\frac{1}{l}]$ and $\mathcal{O}\langle T \rangle[\frac{1}{l}]$ is reduced, $\varphi^*$ factors through $\mathcal{X}(\tau)[\frac{1}{l}] \rightarrow \mathrm{Spec} \, R_{\overline{r}}[\frac{1}{l}]$. This implies that $r_1$ and $r_1'$ are contained in the same irreducible component of $\mathcal{X}(\tau)[\frac{1}{l}]$.

Let $\tilde{\psi_i}: G_K \rightarrow \mathcal{O}\langle T \rangle^{\times}$ be an unramified character such that $\tilde{\psi_i}(\mathrm{Frob}_K) = T + (1-T)\psi_i(\mathrm{Frob}_K)$. Then there exists $\varphi': R_{\overline{r}}[\frac{1}{l}] \rightarrow \mathcal{O}\langle T \rangle[\frac{1}{l}]$ such that $\varphi' \circ r^{\mathrm{univ}} = \oplus_i (\mathrm{Fil}_i/\mathrm{Fil}_{i-1} \otimes \tilde{\psi_i}^{-1})$ by Lemma \ref{homotopy}. From 2 of Corollary \ref{regular morphism}, $\varphi'^*$ factors through $\mathcal{X}(\tau)[\frac{1}{l}] \rightarrow \mathrm{Spec} \, R_{\overline{r}}[\frac{1}{l}]$. Therefore, $r_1' = \oplus_i \mathrm{Fil}_i/\mathrm{Fil}_{i-1}$ and $\oplus_i (\mathrm{Fil}_i/\mathrm{Fil}_{i-1} \otimes \psi_i^{-1})$ are contained in the same irreducible component of $\mathcal{X}(\tau)[\frac{1}{l}]$. Since $\mathcal{X}(\tau)[\frac{1}{l}]$ is regular by 1 of Corollary \ref{regular morphism}, $r_1$ and $\oplus_i (\mathrm{Fil}_i/\mathrm{Fil}_{i-1} \otimes \psi_i^{-1})$ are contained in the same irreducible component of $\mathcal{X}(\tau)[\frac{1}{l}]$. Since $\mathrm{WD}(\oplus_i (\mathrm{Fil}_i/\mathrm{Fil}_{i-1} \otimes \psi_i^{-1}) \otimes \overline{\mathbb{Q}}_l) \cong \mathrm{Sp}_{n_1}(1) \oplus \cdots \oplus \mathrm{Sp}_{n_k}(1)$ is pure of weight $0$, $\oplus_i (\mathrm{Fil}_i/\mathrm{Fil}_{i-1} \otimes \psi_i^{-1})$ is a regular point of $\mathrm{Spec} \, R_{\overline{r}}[\frac{1}{l}]$ by Lemma \ref{regular point}.

Thus, we may assume that there exist subrepresentations of $s_1, \cdots, s_k$ (resp. $t_1, \cdots, t_k$) of $r_1$ (resp. $r_2$) such that $r_1 = \oplus_i s_i$ (resp. $r_2 = \oplus_i t_i$) and $\mathrm{WD}(s_i \otimes \overline{\mathbb{Q}}_l) \cong \mathrm{Sp}_{n_i}(1)$ (resp. $\mathrm{WD}(t_i \otimes \overline{\mathbb{Q}}_l) \cong \mathrm{Sp}_{n_i}(1)$) for all $i$.

We have $s_i \sim t_i$ for all $i$ by the irreducibility of the Steinberg lifting ring (see \cite[4 of Proposition 3.1]{IA}). Thus, we obtain $r_1 \sim r_2$ by 7 of Lemma \ref{coincide monodromy type}. \end{proof}

\subsection{$p = l$}

In this subsection, we assume $p=l$ and $\tau(K) \subset E$ for all $\tau \in \mathrm{Hom}_{\mathbb{Q}_l}(K, \overline{\mathbb{Q}}_l)$.

\begin{dfn} Let $\lambda$ be an element of $(\mathbb{Z}^n_+)^{\mathrm{Hom}_{\mathbb{Q}_l}(K, E)}$.

(a) \ Let $B$ be a finite $E$-algebra and $r : G_K \rightarrow \mathrm{GL}_n(B)$ be a continuous representation.

We say that $r$ is de Rham of $l$-adic Hodge type $\mathbf{v}_{\lambda}$ if $r$ satisfies the following conditions.

1 \ $r$ is de Rham as a continuous representation over $\mathbb{Q}_l$.

2 \ For all $\tau \in \mathrm{Hom}_{\mathbb{Q}_l}(K, E)$, $\mathrm{gr}^{\lambda_{\tau, j}+n-j}(D_{\mathrm{dR}}(r)) \otimes_{B \otimes_{\mathbb{Q}_l} K, 1 \otimes \tau} B$ is a finite free $B$-module of rank $1$ for $j = 1, \cdots, n$ and $\mathrm{gr}^{i}(D_{\mathrm{dR}}(r)) \otimes_{B \otimes_{\mathbb{Q}_l} K, 1 \otimes \tau} B = 0$ for other $i$. 

(b) \ We say that $r$ is semistable (resp. crystalline) of $l$-adic Hodge type $\mathbf{v}_{\lambda}$ if $r$ is de Rham of $l$-adic Hodge type $\mathbf{v}_{\lambda}$ and $r$ is semistable (resp. crystalline) as a continuous representation over $\mathbb{Q}_l$.

(c) \ For a continuous representation $r : G_K \rightarrow \mathrm{GL}_n(\overline{\mathbb{Q}}_l)$, we can take a finite extension $E'/E$ contained in $\overline{\mathbb{Q}}_l$ such that $\mathrm{Im}(r) \subset \mathrm{GL}_n(E')$ and consequently $r$ is regarded as a representation $r^{E'} : G_{K} \rightarrow \mathrm{GL}_n(E')$. We say that $r$ is de Rham (resp. semistable, crystalline) of $l$-adic Hodge type $\mathbf{v}_{\lambda}$ if $r^{E'}$ is de Rham (resp. semistable, crystalline) of $l$-adic Hodge type $\mathbf{v}_{\lambda}$. This is independent of the choice of $E'$.

\end{dfn}

\begin{prop} \label{psd}

Let $\overline{r}: G_K \rightarrow \mathrm{GL}_n(\mathbb{F})$ be a continuous representation and $\lambda \in (\mathbb{Z}_+^n)^{\mathrm{Hom}_{\mathbb{Q}_l}(K, E)}$.

1 \ There exists a unique $\mathcal{O}$-flat quotient $R_{\overline{r}}^{\mathrm{cris}, \lambda}$ of $R_{\overline{r}}$ satisfying the following property.

$\cdot$ For any finite $E$-algebra $B$ and $\mathcal{O}$-morphism $f : R_{\overline{r}} \rightarrow B$, $f$ factors through $R_{\overline{r}} \twoheadrightarrow R_{\overline{r}}^{\mathrm{cris}, \lambda}$ if and only if $f \circ r^{\mathrm{univ}}$ is crystalline of $l$-adic Hodge type $\mathbf{v}_{\lambda}$.

2 \ $R_{\overline{r}}^{\mathrm{cris}, \lambda}[\frac{1}{l}]$ is regular (in particular, $R_{\overline{r}}^{\mathrm{cris}, \lambda}$ is reduced) and $\mathrm{Spec} \, R_{\overline{r}}^{\mathrm{cris}, \lambda}$ is equidimensional of dimension $1 + n^2 + [K:\mathbb{Q}_l]\frac{n(n-1)}{2}$.

\end{prop}

\begin{proof} See \cite[Theorem 3.3.3]{MEI} and \cite[Theorem 3.3.8]{psd}. \end{proof}

\begin{lem} \label{p-adic equivalence irreducible}

Let $\lambda \in (\mathbb{Z}_+^n)^{\mathrm{Hom}_{\mathbb{Q}_l}(K, E)}$, $\overline{r} : G_K \rightarrow \mathrm{GL}_n(\mathbb{F})$ be a continuous representation and $g \in \mathrm{GL}_n(\mathcal{O})$ such that $\overline{grg^{-1}} = \overline{r}$. 

Then the action $\mathrm{Spec} \, R_{\overline{r}}[\frac{1}{l}] \rightarrow \mathrm{Spec} \, R_{\overline{r}}[\frac{1}{l}], \rho \mapsto g \rho g^{-1}$ induces the action $\mathrm{Spec} \, R_{\overline{r}}^{\mathrm{cris}, \lambda}[\frac{1}{l}] \rightarrow \mathrm{Spec} \, R_{\overline{r}}^{\mathrm{cris}, \lambda} [\frac{1}{l}]$ and fixes all irreducible components of $\mathrm{Spec} \, R_{\overline{r}}^{\mathrm{cris}, \lambda} [\frac{1}{l}]$.

\end{lem}

\begin{proof} See \cite[(1) of p.530]{CW}. \end{proof}

\begin{lem} \label{p-adic coefficient}

For any continuous representation $\overline{r} : G_K \rightarrow \mathrm{GL}_n(\mathbb{F})$, any finite extension $E'/E$ and $\lambda \in (\mathbb{Z}_+^n)^{\mathrm{Hom}_{\mathbb{Q}_l}(K, E)}$, we have $R^{\mathrm{cris}, \lambda}_{\overline{r}, \mathcal{O}} \otimes_{\mathcal{O}} \mathcal{O}_{E'} \cong R^{\mathrm{cris}, \lambda}_{\overline{r}, \mathcal{O}_{E'}}$. 
        
\end{lem}

\begin{proof}

By Corollary \ref{coefficient change}, $R_{\overline{r}, \mathcal{O}}^{\mathrm{cris}, \lambda}\otimes_{\mathcal{O}} \mathcal{O}_{E'}$ has the same universal property as $R_{\overline{r}, \mathcal{O}_{E'}}^{\mathrm{cris}, \lambda}$. This implies the result. \end{proof}

We recall the definition of the notion $\sim$.

For continuous representations $r_1, r_2 : G_{K} \rightarrow \mathrm{GL}_n(\mathcal{O})$, we say $r_1 \sim r_2$ if $r_1$ and $r_2$ satisfy the following conditions after extending $E$. 

1 \ $\overline{r}_1 \cong \overline{r}_2$.

2 \ There exists $\lambda \in (\mathbb{Z}^n_+)^{\mathrm{Hom}_{\mathbb{Q}_l}(K,\overline{\mathbb{Q}}_l)}$ such that $r_1 \otimes \overline{\mathbb{Q}}_l$ and $r_2 \otimes \overline{\mathbb{Q}}_l$ are crystalline of $l$-adic Hodge-type $\mathbf{v}_{\lambda}$.

3 \ $r_1$ and $r_2$ are contained in the same geometrically irreducible component of $\mathrm{Spec} \, R^{\mathrm{cris}, \lambda}_{\overline{r}_1}[\frac{1}{l}]$ after fixing an identification $\overline{r}_1 = \overline{r}_2$. 

Note that the condition 2 is independent of an identification $\overline{r}_1 = \overline{r}_2$ by Lemma \ref{p-adic equivalence irreducible}.

\vspace{0.5 \baselineskip}

For continuous representations $r_1, r_2 : G_{K} \rightarrow \mathrm{GL}_n(\mathcal{O}_{\overline{\mathbb{Q}}_l})$, we can take a sufficiently large finite extension $E'$ of $E$ contained in $\overline{\mathbb{Q}}_l$ such that $\mathrm{Im}(r_1), \mathrm{Im}(r_2) \subset \mathrm{GL}_n(\mathcal{O}_{E'})$ and consequently $r_1$, $r_2$ are regarded as representations $r_1^{E'}$, $r_2^{E'} : G_K \rightarrow \mathrm{GL}_n(\mathcal{O}_{E'})$ respectively. We say $r_1 \sim r_2$ if $r_1^{E'} \sim r_2^{E'}$.

Note the following elementary properties of the notion $\sim$.

\begin{lem} \label{p=l}

Let $r, r_1, r_2, r_3 : G_K \rightarrow \mathrm{GL}_n(\mathcal{O}_{\overline{\mathbb{Q}}_l})$ and $r_4, r_5 : G_K \rightarrow \mathrm{GL}_m(\mathcal{O}_{\overline{\mathbb{Q}}_l})$ be continuous representations such that $r \otimes \overline{\mathbb{Q}}_l$, $r_1 \otimes \overline{\mathbb{Q}}_l$, $r_2 \otimes \overline{\mathbb{Q}}_l$ and $r_3 \otimes \overline{\mathbb{Q}}_l$ are crystalline of $l$-adic Hodge type $\mathbf{v}_{\lambda}$ for some $\lambda \in (\mathbb{Z}^n_+)^{\mathrm{Hom}_{\mathbb{Q}_l}(K, \overline{\mathbb{Q}}_l)}$.

1 \ If $r_1 \sim r_2$ and $r_2 \sim r_3$, then $r_1 \sim r_3$.

2 \ For any unramified character $\mu: G_K \rightarrow \mathcal{O}^{\times}_{\overline{\mathbb{Q}}_l}$ such that $\overline{\mu}$ is trivial, we have $r \sim r \otimes \mu$.

3 \ Assume $\overline{r}$ is semisimple. Then any $G_K$-stable increasing filtration $\{ \mathrm{Fil}_i \}$ on $r$ by $\mathcal{O}_{\overline{\mathbb{Q}}_l}$-direct summands, we have $r \sim \oplus_i \mathrm{Fil}_i/\mathrm{Fil}_{i-1}$.

4 \ Assume that $K/\mathbb{Q}_l$ is unramified, $\lambda_{\tau, 1} - \lambda_{\tau, n} \le l - 2 - n + 1$ for any $\tau \in \mathrm{Hom}_{\mathbb{Q}_l}(K, E)$ and $\overline{r}_1 \cong \overline{r}_2$. Then $r_1 \sim r_2$.

5 \ If $r_1 \sim r_2$, then $r_1|_{G_L} \sim r_2|_{G_L}$ for any finite extension $L/K$.

6 \ If $r_1 \sim r_2$ and $r_4 \sim r_5$, then $r_1 \otimes r_4 \sim r_2 \otimes r_5$ and $r_1 \oplus r_4 \sim r_2 \oplus r_5$.

\end{lem}

\begin{proof} See \cite[p530]{CW} except the property 4. 4 is \cite[Lemma 2.4.1]{CD}. \end{proof}

Let $\tau \in \mathrm{Aut}_{\mathbb{Q}_l}(K)$. We take $\tilde{\tau} \in \mathrm{Aut}_{\mathbb{Q}_l}(\overline{K})$ such that $\tilde{\tau}|_{K} = \tau$.

For a continuous representation $r: G_K \rightarrow \mathrm{GL}_n(\overline{\mathbb{Q}}_l)$, we define $r^{\tau} : G_{K} \rightarrow \mathrm{GL}_n(\overline{\mathbb{Q}}_l)$ by $r^{\tau}(g):=r(\tilde{\tau} g \tilde{\tau}^{-1})$ for any $g \in G_{K}$. The isomorphism class of $r^{\tau}$ depends only on $\tau$.

For $\sigma \in G_{\mathbb{Q}_l}$ and a continuous representation $r: G_K \rightarrow \mathrm{GL}_n(\overline{\mathbb{Q}}_l)$, we define $^{\sigma}r : G_{K} \rightarrow \mathrm{GL}_n(\overline{\mathbb{Q}}_l)$ by $^{\sigma}r(g):=\sigma(r(g))$ for any $g \in G_K$. 

Note that the following properties.

\begin{lem}\label{coefficient} Let $r : G_K \rightarrow \mathrm{GL}_n(\overline{\mathbb{Q}}_l)$ be a continuous representation, $\sigma \in G_{\mathbb{Q}_l}$, $\tau \in \mathrm{Aut}_{\mathbb{Q}_l}(K)$. Then, we obtain the following results.

1 \ If $r$ is crystalline (resp. semistable, de Rham, Hodge-Tate), then $r^{\tau}$ and $^{\sigma}r$ are crystalline (resp. semistable, de Rham, Hodge-Tate).

2 \ If $r$ is Hodge-Tate, then $\mathrm{HT}_{\delta \tau}(r^{\tau}) = \mathrm{HT}_{\delta}(r)$ and $\mathrm{HT}_{\sigma \delta}(^{\sigma}r) = \mathrm{HT}_{\delta}( r )$ for any $\delta \in \mathrm{Hom}_{\mathbb{Q}_l}(K, \overline{\mathbb{Q}}_l)$. 

\end{lem}

\begin{proof}  Let $B$ be one of the rings $B_{\mathrm{cris}}, B_{\mathrm{st}}$, $B_{\mathrm{dR}}$, $B_{\mathrm{HT}}$. We fix $\tilde{\tau} \in \mathrm{Aut}_{\mathbb{Q}_l}(\overline{K})$ such that $\tilde{\tau}|_{K} = \tau$. Note that the action of $G_K$ on $B$ is extended to the action of $G_{\mathbb{Q}_l}$. Then we obtain $(r^{\tau} \otimes_{\mathbb{Q}_l} B)^{G_K} = (r \otimes_{\mathbb{Q}_l} B^{\tau^{-1}})^{G_K} \xrightarrow[1 \otimes \tilde{\tau}]{\sim} (r \otimes_{\mathbb{Q}_l} B)^{G_K}$ and $(r \otimes_{\mathbb{Q}_l} B)^{G_{K}} \xrightarrow[\sigma \otimes 1]{\sim} (\ ^{\sigma}r \otimes_{\mathbb{Q}_l} B)^{G_{K}}$. This implies the result 1. Moreover, we have $(r^{\tau}(i) \otimes_{\delta \tau, K} \widehat{\overline{K}})^{G_K} = (r(i) \otimes_{\delta \tau, K} \widehat{\overline{K}}^{\tau^{-1}})^{G_K} \xrightarrow[1 \otimes \tilde{\tau}]{\sim} (r(i) \otimes_{\delta, K} \widehat{\overline{K}})^{G_K}$ and  $(r(i) \otimes_{\delta, K} \widehat{\overline{K}})^{G_{K}} \xrightarrow[\sigma \otimes 1]{\sim} (\ ^{\sigma}r \otimes_{\sigma \delta, K} \widehat{\overline{K}})^{G_{K}}$. Here, $\widehat{\overline{K}}$ denotes the completion of $\overline{K}$. This implies the result 2.  \end{proof}

\begin{dfn} \label{potentially diagonalizable definition}
    
For a continuous representation $r: G_K \rightarrow \mathrm{GL}_n(\mathcal{O}_{\overline{\mathbb{Q}}_l})$, we say that $r$ is diagonalizable if $r \otimes \overline{\mathbb{Q}}_l$ is crystalline and there exist crystalline characters $\chi_1, \cdots, \chi_n : G_K \rightarrow \mathcal{O}_{\overline{\mathbb{Q}}_l}^{\times}$ such that $r \sim \chi_1 \oplus \cdots \oplus \chi_n$. 

We say that $r$ is potentially diagonalizable if there exists a finite extension $L/K$ such that $r|_{G_L}$ is diagonalizable. This implies that $r$ is potentially crystalline.

\end{dfn}

\begin{prop} \label{trace irreducible component}
    
    Let $r: G_K \rightarrow \mathrm{GL}_n(\mathcal{O}_{\overline{\mathbb{Q}}_l})$ be a continuous representation.

    We assume that $K/\mathbb{Q}_l$ is unramified, $r \otimes \overline{\mathbb{Q}}_l$ is crystalline of $l$-adic Hodge type $\mathbf{v}_{\lambda}$ satisfying $\lambda_{\tau, 1} - \lambda_{\tau, n} \le l - 2 - n + 1$ for any $\tau \in \mathrm{Hom}_{\mathbb{Q}_l}(K, \overline{\mathbb{Q}}_l)$.

    Then $r$ is potentially diagonalizable and there exists a finite extension $L$ of $K$ such that \ $^{\sigma}r|_{G_L} \sim r^{\tau^{-1}\sigma\tau}|_{G_L}$ for all $\sigma \in G_{\mathbb{Q}_l}$ and $\tau \in \mathrm{Hom}_{\mathbb{Q}_l}(K, \overline{\mathbb{Q}}_l)$.

\end{prop}

\begin{proof} Note that almost all part of the following proof is contained in the proof of \cite[Lemma 1.4.3]{CW}. 

There exists a finite unramified extension $K'/K$ such that $\overline{r}(G_{K'}) = \overline{r}(I_K)$. Then there exists a $G_{K'}$-stable increasing filtration $\{ \mathrm{\overline{Fil}}_i \}_{i=1}^n$ on $\overline{r}|_{G_{K'}}$ such that $\mathrm{dim}_{\overline{\mathbb{F}_l}} \overline{\mathrm{Fil}_i}/\overline{\mathrm{Fil}_{i-1}} = 1$. By \cite[Lemma 1.4.2]{CW}, there exists a lifting $r': G_{K'} \rightarrow \mathrm{GL}_n(\mathcal{O}_{\overline{\mathbb{Q}}_l})$ of $\overline{r}|_{G_{K'}}$ such that $r' \otimes \overline{\mathbb{Q}}_l$ is crystalline of $l$-adic Hodge type $\mathbf{v}_{\lambda}$ and $r'$ has a $G_{K'}$-stable increasing filtration $\{ \mathrm{Fil}_i \}_{i=1}^n$ by $\mathcal{O}_{\overline{\mathbb{Q}}_l}$-direct summands which is a lifting of $\{ \overline{\mathrm{Fil}}_i \}_{i=1}^n$. By 4 of Lemma \ref{p=l}, we have $r|_{G_{K'}} \sim r'$. We take $\tau \in \mathrm{Hom}_{\mathbb{Q}_l}(K', \overline{\mathbb{Q}}_l)$ and $\sigma \in G_{\mathbb{Q}_l}$. By 4 of Lemma \ref{p=l} and Lemma \ref{coefficient}, we have \ $^{\sigma}r|_{G_{K'}} \sim \ ^{\sigma}r'$ and $r^{\tau^{-1} \sigma \tau}|_{G_{K'}} \sim (r')^{\tau^{-1} \sigma \tau}$.

For $1 \le i \le n$, we put $\phi_i := \mathrm{Fil}_i/\mathrm{Fil}_{i-1}$. Since $\phi_i \otimes \overline{\mathbb{Q}}_l$ is crystalline, there exists $\lambda_i \in \mathbb{Z}^{\mathrm{Gal}(\tau(K')/\mathbb{Q}_l)}$ such that $\phi_i \circ \mathrm{Art}_{K'}(u) = \prod_{\rho \in \mathrm{Gal}(\tau(K')/\mathbb{Q}_l)} \rho\tau (u)^{\lambda_{i, \rho}}$ for all $u \in \mathcal{O}_{K'}^{\times}$. By the commutativity of $\mathrm{Gal}(\tau(K')/\mathbb{Q}_l)$, we have $\sigma (\phi_i \circ \mathrm{Art}_{K'}(u)) = \prod_{\rho \in \mathrm{Gal}(\tau(K')/\mathbb{Q}_l)} \sigma\rho\tau(u)^{\lambda_{i, \rho}} = \prod_{\rho \in \mathrm{Gal}(\tau(K')/\mathbb{Q}_l)} \rho\sigma\tau(u)^{\lambda_{i, \rho}}$. Thus we obtain $^{\sigma}\phi_i|_{I_{K'}} = \phi_i^{\tau^{-1}\sigma \tau}|_{I_{K'}}$ and $\overline{\phi_i}(I_K) = \overline{\phi_i}(G_{K'})$. This implies \ $^{\sigma}\phi_i \sim \phi_i^{\tau^{-1} \sigma \tau}$ by 2 of Lemma \ref{p=l}.  By 6 of Lemma \ref{p=l}, we obtain $^{\sigma}(\phi_1 \oplus \cdots \oplus \phi_n) \sim (\phi_1 \oplus \cdots \oplus \phi_n)^{\tau^{-1} \sigma \tau}$.

There exists a finite extension $L$ of $K'$ such that $\overline{r}|_{G_L}$ is trivial. Then $r'|_{G_L} \sim (\phi_1 \oplus \cdots \oplus \phi_n)|_{G_L}$, \ $^{\sigma}r'|_{G_L} \sim \ ^{\sigma}(\phi_1 \oplus \cdots \oplus \ \phi_n)|_{G_L}$ and $r'^{\tau^{-1}\sigma \tau}|_{G_L} \sim (\phi_1 \oplus \cdots \oplus \ \phi_n)^{\tau^{-1}\sigma \tau}|_{G_L}$ by 3 of Lemma \ref{p=l}. This implies that $r$ is potentially diagonalizable by 1 and 5 of Lemma \ref{p=l}. Moreover, we obtain $^{\sigma}r|_{G_L} \sim \ ^{\sigma}r'|_{G_L} \sim \ ^{\sigma}(\phi_1 \oplus \cdots \oplus \phi_n)|_{G_{L}} \sim (\phi_1 \oplus \cdots \oplus \phi_n)^{\tau^{-1}\sigma \tau}|_{G_{L}} \sim (r')^{\tau^{-1} \sigma \tau}|_{G_{L}} \sim r^{\tau^{-1} \sigma \tau}|_{G_L}$ by 5 of Lemma \ref{p=l}. Thus, we obtain the proposition by 1 of Lemma \ref{p=l}. \end{proof}

\vspace{0.5 \baselineskip}

Finally, we recall some basic properties of ordinary representations.

\begin{lem} \label{ordinary Galois representation}

Let $r : G_K \rightarrow \mathrm{GL}_n(\overline{\mathbb{Q}}_l)$ be a continuous representation and $\lambda \in (\mathbb{Z}_+^{n})^{\mathrm{Hom}_{\mathbb{Q}_l}(K, \overline{\mathbb{Q}}_l)}$.

The following conditions are equivalent.

(1) \ There exists an increasing filtration $\{ \mathrm{Fil}_i \}$ on $r$ and an open subgroup $U$ of $\mathcal{O}_K^{\times}$ such that for all $i$, $\chi_i := \mathrm{Fil}_i/\mathrm{Fil}_{i-1}$ is one-dimensional and $\chi_i \circ \mathrm{Art}_K(u) = \prod_{\tau \in \mathrm{Hom}_{\mathbb{Q}_l}(K, \overline{\mathbb{Q}}_l)} \tau(u)^{-\lambda_{\tau, n - i + 1} - i + 1}$ for all $u \in U$.

(2) \ $r$ is de Rham of $l$-adic Hodge type $\mathbf{v}_{\lambda}$ and the eigenvalues $\alpha_1, \cdots, \alpha_n$ of $\mathrm{WD}(r)(\mathrm{Frob}_K)$ satisfies $\mathrm{val}_l(\alpha_i) = \frac{1}{e_K} \sum_{\tau \in \mathrm{Hom}_{\mathbb{Q}_l}(K, \overline{\mathbb{Q}}_l)} ( \lambda_{\tau, n + 1 - i} + i - 1)$ for all $i = 1, \cdots, n$ after arranging $\alpha_i$'s so that $\mathrm{val}_l(\alpha_1) \le \mathrm{val}_l(\alpha_2) \le \cdots \le \mathrm{val}_l(\alpha_n)$. (Here, $\mathrm{val}_l$ denotes the $l$-adic valuation on $\overline{\mathbb{Q}}_l$ normalized to $\mathrm{val}_l(l) = 1$ and $e_K$ denotes the ramification index of $K/\mathbb{Q}_l$.)

\end{lem}

\begin{proof}
    
($(1) \Rightarrow (2)$) By \cite[Lemma 3.1.4]{GG}, $r$ is de Rham. Since $$\mathrm{WD}(\chi_i)(\mathrm{Art}_K(x)) = \chi_i(\mathrm{Art}_K(x)) (\prod_{\tau \in \mathrm{Hom}_{\mathbb{Q}_l}(K, \overline{\mathbb{Q}}_l)}\tau(x)^{\lambda_{\tau, n-i+1}+i-1})$$ for all $x \in K^{\times}$ and $\chi_i (\mathrm{Frob}_K)$ is an $l$-adic unit for all $i$, we obtain the result.
    
$(2) \Rightarrow (1)$ follows by \cite[proof of Theorem 2.4]{red}. \end{proof}
    
\begin{dfn}

Let $r : G_K \rightarrow \mathrm{GL}_n(\overline{\mathbb{Q}}_l)$ be a continuous representation and $\lambda \in (\mathbb{Z}_+^{n})^{\mathrm{Hom}_{\mathbb{Q}_l}(K, \overline{\mathbb{Q}}_l)}$. We say that $r$ is ordinary of weight $\lambda$ if $r$ and $\lambda$ satisfy the equivalent conditions in Lemma \ref{ordinary Galois representation}. 

\end{dfn}

\begin{rem}

If $r : G_K \rightarrow \mathrm{GL}_n(\overline{\mathbb{Q}}_l)$ is ordinary, then the $l$-adic valuations of the eigenvalues $\alpha_1, \cdots, \alpha_n$ of $\mathrm{WD}(r)(\mathrm{Frob}_K)$ are distinct by (2) of Lemma \ref{ordinary Galois representation}. Thus, $\mathrm{WD}(r)$ is Frobenius semisimple.

\end{rem}
        
\begin{prop} \label{semistable ordinary deformation ring}
        
Let $r_1, r_2 : G_K \rightarrow \mathrm{GL}_n(\mathcal{O})$ be residually trivial continuous representations and $\lambda \in (\mathbb{Z}^n_+)^{\mathrm{Hom}_{\mathbb{Q}_l}(K, \overline{\mathbb{Q}})}$. We assume that $r_1 \otimes \overline{\mathbb{Q}}_l$ and $r_2 \otimes \overline{\mathbb{Q}}_l$ are crystalline of $l$-adic Hodge type $\lambda$ and $r_1 \otimes \overline{\mathbb{Q}}_l$ is ordinary. Then $r_1 \sim r_2$ if and only if $r_2 \otimes \overline{\mathbb{Q}}_l$ is ordinary.

\end{prop}

\begin{proof} See \cite[Lemma 3.14]{OG}. \end{proof}

\section{Automorphy lifting and local-global compatibility}

The goal of this subsection is to prove the automorphy lifting theorem in some crystalline cases and ordinary cases and also prove the local-global compatibility for the cohomological cuspidal automorphic representations corresponding to given Galois representations (See Theorems \ref{automorphy lifting theorem in crystalline cases} and \ref{ordinary automorphy lifting}.) The proofs of Theorems \ref{automorphy lifting theorem in crystalline cases} and \ref{ordinary automorphy lifting} are very similar to the proof of the automorphy lifting theorems \cite[Theorems 6.1.1 and 6.1.2]{10}, \cite[Theorem 3.2]{pw}, \cite[Theorem 5.2]{MEI} and \cite[{\S} 2 and 5]{RG}. The main differences and features are the following points.

\begin{itemize}
\item In order to prove the local-global compatibility of the automorphic representation at a non-$l$-adic places $v \nmid l$ corresponding to the given Galois representation $r$, we use the parahoric subgroup corresponding to the dual partition of the monodromy type of $\mathrm{WD}(r|_{G_{F_v}})$ as the level structure at $v$ when we define our Hecke algebra.
\item In this paper, we use two patching arguments : \cite[\S 2]{RG} in crystalline cases and \cite[\S 5.4]{MEI} in ordinary cases. If we are only interested in automorphy lifting theorems, then the argument \cite[\S 5.4]{MEI} suffices and this argument works under weaker assumptions. On the other hand, we can also prove the vanishing of Bloch-Kato Selmer groups of adjoint representations by using the argument \cite[\S 2]{RG}. (See Theorem \ref{vanishing of selmer group}.)
\item In the proof of Theorem \ref{automorphy lifting theorem in crystalline cases}, we don't use the (derived) Ihara avoidance, which was initiated by \cite{IA} (and \cite{10}), was used in \cite[Theorems 6.1.1 and 6.1.2]{10}, \cite[Theorem 3.2]{pw} and \cite[Theorem 5.2]{MEI} and gives us ``non-minimal'' automorphy lifting theorems.\footnote{``non-minimal'' roughly means that we don't assume that a given Galois representation $r$ and a given automorphic Galois representation $r_{\iota}(\pi)$ are contained in the same irreducible component of the universal lifting ring at all non-$l$-adic places.} The reasons for this are the following two points. The first point is that in order to prove the local-global compatibility at $v \nmid l$, clearly we need a certain condition at $v$. In fact, if we want to remove the condition at $v$ by the Ihara avoidance argument, we need to use the Iwahori subgroup as the level structure at $v$ while as stated above, we need to use a certain parahoric subgroup to prove the local-global compatibility. (We will prove Theorem \ref{ordinary automorphy lifting} by combining the Ihara avoidance argument and the local-global compatibility argument.) The second point is that in order to use the Ihara avoidance argument, we require that our local lifting rings at $p$-adic places has the unique generalization property : any generic point of the special fiber has a unique generalization to the generic fiber. (For example, this property holds when the ring is irreducible.) This property was classically known in the Fontaine-Laffaille case and the ordinary case and one of the main innovations of \cite{pw} is that this property also holds for the consecutive weights crystalline lifting rings even when $p$ is ramified in our local field. (See \cite[Theorem D]{pw}.) However, this property is not known beyond these cases \footnote{Actually, \cite[Remark 2.5.6]{pw} says that we can't expect a variant of \cite[Theorem D]{pw} for the non-consecutive weights crystalline lifting ring if we don't assume any condition on the absolute ramification index of our local field.} and for our potential automorphy argument, we need the automorphy lifting theorem in the crystalline case when $p$ is ramified in our CM field because we only have a nice property as in Proposition \ref{conjugate self-duality} after taking a CM extension of $F$ in which $p$ is ramified. 
\item In the proof of Theorem \ref{ordinary automorphy lifting}, we use a similar argument as \cite[Proposition 5.4.2]{MEI} and \cite[{\S} 8]{small} to remove the assumption $\zeta_p \notin \overline{F}^{\mathrm{Ker}{\mathrm{ad} \, \overline{r}}}$\footnote{$\overline{r}$ is a residual representation.} for the neatness of the level structure, which was assumed in \cite[Theorem 6.1.2]{10}.
\end{itemize}

\subsection{Preliminaries on automorphy lifting theorems}

\subsubsection{Global deformation theory}

We recall some notations and basic facts about the global deformation theory.

Let $F$ be a CM field, $n$ be a positive integer, $l$ be a prime such that $l \nmid 2n$, $E$ be a finite extension of $\mathbb{Q}_l$ contained in $\overline{\mathbb{Q}}_l$, $\mathcal{O}$ be the ring of integers of $E$ and $\mathbb{F}$ be the residue field of $\mathcal{O}$. We assume that $\tau(F) \subset E$ for any $\tau \in \mathrm{Hom}(F, \overline{\mathbb{Q}}_l)$. For an object $\Lambda$ of $\mathrm{CNL}_{\mathcal{O}}$, let $\mathrm{CNL}_{\Lambda}$ denote the category of complete Noetherian local $\Lambda$-algebras with residue field $\mathbb{F}$. 

Let $\overline{r} : G_F \rightarrow \mathrm{GL}_n(\mathbb{F})$ be an absolutely irreducible continuous representation and $S$ be a finite set of finite places of $F$ containing all $l$-adic places and ramified places of $\overline{r}$.

Let $F_S$ denote the maximal subextension of $\overline{F}/F$ such that $F_S/F$ is unramified outside $S$. We put $G_{F,S}:=\mathrm{Gal}(F_S/F)$. Then $\overline{r}$ is regarded as a representation of $G_{F, S}$.

For $v \in S$, we fix an object $\Lambda_v$ of $\mathrm{CNL}_{\mathcal{O}}$. Then $R_{\overline{r}|_{G_{F_v}}, \Lambda_v} := R_{\overline{r}|_{G_{F_v}}, \mathcal{O}} \widehat{\otimes}_{\mathcal{O}} \Lambda_v$ represents the functor $\mathcal{D}_v^{\Box} : \mathrm{CNL}_{\Lambda_v} \rightarrow \mathrm{Set}, A \mapsto \{ \rho_v : G_{F_v} \rightarrow \mathrm{GL}_n(A) \ \mid \ \mathrm{lifting \ of \ } \overline{r}|_{G_{F_v}} \}$.  We put $\Lambda:=\widehat{\otimes}_{v \in S} \Lambda_v$.

\begin{dfn}\label{local deformation problem}

    Let $\mathrm{CNL}_{\Lambda_v}^f$ denotes the full subcategory of $\mathrm{CNL}_{\Lambda_v}$ consisting of all Artin local $\Lambda_v$-algebras.

A local deformation problem $\mathcal{D}_v$ of $\overline{r}|_{G_{F_v}}$ over $\Lambda_v$ means a subfunctor $\mathcal{D}_v$ of $\mathcal{D}_v^{\Box}$ satisfying the following conditions. 

$1$ \ $\mathcal{D}_v$ is represented by a nonzero quotient $R_{\overline{r}|_{G_{F_v}}, \mathcal{D}_{v}}$ of $R_{\overline{r}|_{G_{F_v}}, \Lambda_v}$.

$2$ \ For all $A \in \mathrm{CNL}_{\Lambda_v}$, $r \in \mathcal{D}_v(A)$ and $a \in I_n + \mathfrak{m}_A\mathrm{M}_n(A)$, we have $a r a^{-1} \in \mathcal{D}_v(A)$.

Note that the condition 1 is equivalent to the following conditions.
    
$(a)$ \ $\overline{r}|_{G_{F_v}} \in \mathcal{D}_v(\mathbb{F})$.

$(b)$ \ For an object $A$ of $\mathrm{CNL}_{\Lambda}^f$, ideals $I_1, I_2$ of $A$ satisfying $I_1 \cap I_2 = 0$ and $r \in \mathcal{D}^{\Box}_v(A)$, if we have $r \mod I_{1} \in \mathcal{D}_v(A/I_1)$ and $r \mod I_{2} \in \mathcal{D}_v(A/I_2)$, then $r \in \mathcal{D}_v(A)$.
    
$(c)$ \ For an injective morphism $f : A \hookrightarrow B$ in $\mathrm{CNL}^f_{\mathcal{O}}$ and $r \in \mathcal{D}_{v}^{\Box}(A)$, if we have $f \circ r \in \mathcal{D}_{v}(B)$, then $r \in \mathcal{D}_{v}(A)$. 
    
$(d)$ \ For an object $A \in \mathrm{CNL}_{\Lambda_v}$, $\mathcal{D}_v(A) \cong \varprojlim_n \mathcal{D}_v(A/\mathfrak{m}_A^n)$.

In fact, if $\mathcal{D}_v$ satisfies the conditions $(a) \sim (d)$, $R_{\overline{r}, \Lambda_v}/(\cap_{I \in \mathcal{I}}I) \cong \varprojlim_{I \in \mathcal{I}} R_{\overline{r}, \Lambda_{v}}/I$ represents $\mathcal{D}_v$, where $\mathcal{I}$ is the set of all proper open ideals $I$ of $R_{\overline{r}, \Lambda_v}$ satisfying $r^{\mathrm{univ}} \mod I \in \mathcal{D}_{v}(R_{\overline{r}, \Lambda_v}/I)$.
    
\end{dfn}   

\begin{dfn} \label{global deformation}

1 \ A global deformation problem means a tuple $( \overline{r}, S, \{ \Lambda_v \}_{v \in S}, \{ \mathcal{D}_v \}_{v \in S})$, where

$\cdot$ \ $\overline{r}$, $S$ and $\Lambda_v$'s as above.

$\cdot$ \ For any $v \in S$, $\mathcal{D}_v$ is a local deformation problem of $\overline{r}|_{G_{F_v}}$ over $\Lambda_v$.

We fix a global deformation problem $\mathcal{S}=( \overline{r}, S, \{ \Lambda_v \}_{v \in S}, \{ \mathcal{D}_v \}_{v \in S})$.

2 \ For an object $A$ of $\mathrm{CNL}_{\Lambda}$, we say that a lifting $r : G_{F,S} \rightarrow \mathrm{GL}_n(A)$ of $\overline{r}$ is of type $\mathcal{S}$ if $r|_{G_{F_v}} \in \mathcal{D}_v(A)$ for all $v \in S$.

3 \ For an object $A \in \mathrm{CNL}_{\Lambda}$ and liftings $r_1, r_2 : G_{F, S} \rightarrow \mathrm{GL}_n(A)$ of $\overline{r}$, we say that $r_1$ and $r_2$ are equivalent if there exists $a \in I_n + \mathfrak{m}_A \mathrm{M}_n(A)$ such that $a r_1 a^{-1} = r_2$. Since $\overline{r}$ is absolutely irreducible, this is equivalent to the condition that there exists $a \in \mathrm{GL}_n(A)$ such that $a r_1 a^{-1} = r_2$. A deformation of type $\mathcal{S}$ to $A$ means an equivalence class of a lifting of type $\mathcal{S}$ to $A$.

4 \ For an object $A$ of $\mathrm{CNL}_{\Lambda}$, let $\mathcal{D}_{\mathcal{S}}(A)$ be the set of all deformations of type $\mathcal{S}$ to $A$.

5 \ We fix a subset $T$ of $S$. A $T$-framed lifting of type $\mathcal{S}$ to $A$ means a tuple $(r, \{ \alpha_v \}_{v \in T})$, where $r$ is a lifting of type $\mathcal{S}$ to $A$ and $\alpha_v \in I_n + \mathfrak{m}_A \mathrm{M}_n(A)$ for any $v \in T$.

6 \ For an object $A \in \mathrm{CNL}_{\mathcal{O}}$ and $T$-framed liftings $(r_1, \{ \alpha_v \}_{v \in T}), (r_2, \{ \beta_v \}_{v \in T})$ of $\mathcal{S}$ to $A$, we say that $(r_1, \{ \alpha_v \}_{v \in T})$ and $(r_2, \{ \beta_v \}_{v \in T})$ are equivalent if there exists $a \in I_n + \mathfrak{m}_A \mathrm{M}_n(A)$ such that $a r_1 a^{-1} = r_2$ and $a\alpha_v = \beta_v$ for all $v \in T$. A $T$-framed deformation of type $\mathcal{S}$ to $A$ means an equivalence class of a $T$-framed lifting of type $\mathcal{S}$ to $A$.

7 \ For an object $A$ of $\mathrm{CNL}_{\Lambda}$, let $\mathcal{D}^T_{\mathcal{S}}(A)$ be the set of all $T$-framed deformations of type $\mathcal{S}$ to $A$.

\end{dfn}

\begin{lem} \label{framed deformation}

Let $\mathcal{S}=( \overline{r}, S, \{ \Lambda_v \}_{v \in S}, \{ \mathcal{D}_v \}_{v \in S})$ be a global deformation problem. Then we have the following results.

1 \ The functors $\mathcal{D}_{\mathcal{S}}$ and $\mathcal{D}_{\mathcal{S}}^T$ are representable by objects $R_{\mathcal{S}}$ and $R^T_{\mathcal{S}}$ of $\mathrm{CNL}_{\Lambda}$ respectively.

2 \ For any place $v_0 \in T$ and any representative $r_{\mathcal{S}}$ of the universal deformation of type $\mathcal{S}$, we have a canonical isomorphism $R_{\mathcal{S}}^T \cong R_{\mathcal{S}}[[X_{v,i,j} \mid v \in T, i,j = 1, \cdots, n]]/(X_{v_0, 1, 1})$.

\end{lem}

\begin{proof} See \cite[Theorem 7.7]{Fer} for the existence of $R_{\mathcal{S}}$. 

We assume $T$ is not empty and fix $v_0 \in T$ and a representative $r_{\mathcal{S}}$ of the universal deformation of type $\mathcal{S}$.

For an object $A$ of $\mathrm{CNL}_{\Lambda}$, we have a natural surjective map $$\mathrm{Hom}_{\mathrm{CNL}_{\Lambda}}(R_{\mathcal{S}}[[X_{v,i,j}]]/(X_{v_0, 1, 1}), A) \rightarrow \mathcal{D}^T_{\mathcal{S}}(A), f \mapsto [(f(r_{\mathcal{S}}), (I_n + (f(X_{v, i, j}))))].$$

For any representative $\rho$ of an element of $\mathcal{D}_{\mathcal{S}}(A)$, the centralizer $Z_{\mathrm{M}_n(A)}(\mathrm{Im}(\rho))$ of $\rho$ is $AI_n$ since $\overline{r}$ is absolutely irreducible. This implies that the above map is injective. \end{proof}

Note that by sending $[((r, \{ \alpha_v \}_{v \in T}))]$ to $(\alpha_v^{-1} r|_{G_{F_v}} \alpha_v)_{v \in T}$, we get a local $\widehat{\otimes}_{v \in T}\Lambda_v$-morphism $R^{T, \mathrm{loc}}_{\mathcal{S}}:= \widehat{\otimes}_{v \in T} R_{\overline{r}|_{G_{F_v}}, \mathcal{D}_v} \rightarrow R_{\mathcal{S}}^T$.

\begin{lem}\label{trace density}

Let $\mathcal{S}=( \overline{r}, S, \{ \Lambda_v \}_{v \in S}, \{ \mathcal{D}_v \}_{v \in S})$ be a global deformation problem.

Then $\Lambda[\mathrm{tr}(r_{\mathcal{S}}(\mathrm{Frob}_v)) \mid v \notin S ]$ is a dense subalgebra of $R_{\mathcal{S}}$, where $r_{\mathcal{S}}$ is a representative of the universal deformation of type $\mathcal{S}$.

\end{lem}

\begin{proof} This follows from \cite[Proposition 7.7]{Fer} and Chebotarev's density theorem. \end{proof}

We recall some local deformation problems. We fix a place $v \in S$.

\begin{prop}\label{Level raising} We assume $v \nmid l$ and $\overline{r}|_{G_{F_v}}$ is unipotently ramified.

Let $\chi_1, \cdots, \chi_n : \mathbb{F}_{v}^{\times} \rightarrow 1 + \varpi \mathcal{O}$ be characters.

1 \ For an object $A$ of $\mathrm{CNL}_{\mathcal{O}}$, let $\mathcal{D}_v^{\chi}(A)$ be the set of all liftings $r : G_{F_v} \rightarrow \mathrm{GL}_n(A)$ of $\overline{r}|_{G_{F_v}}$ satisfying $\mathrm{det}(TI_n - r(\sigma)) = \prod_{i=1}^n(T - \chi_i (\mathrm{Art}_{F_v}^{-1}(\sigma)))$ for all $\sigma \in I_{F_v}$. Then the functor $\mathcal{D}_{v}^{\chi} : \mathrm{CNL}_{\mathcal{O}} \rightarrow \mathrm{Set}$ defines a local deformation problem for $\overline{r}|_{G_{F_v}}$ over $\mathcal{O}$.

2 \ If $\chi_i$ is trivial for all $i$, then any irreducible component $\mathcal{C}$ of $\mathrm{Spec} \, R_{\overline{r}|_{G_{F_v}}, \mathcal{D}_v^1}$ has dimension $n^2+1$ and is flat over $\mathcal{O}$. Moreover, any generic point $\overline{\eta}$ of $\mathcal{C}/\varpi$ is a generic point of $\mathrm{Spec} \, R_{\overline{r}|_{G_{F_v}}, \mathcal{D}_v^1}/\varpi$ and the generic point of $\mathcal{C}$ is the unique generic point of $\mathrm{Spec} \, R_{\overline{r}|_{G_{F_v}}, \mathcal{D}_v^1}$ specializing to $\overline{\eta}$.

3 \ If $q_v \equiv 1 \mod l$, $\chi_i \neq \chi_j$ for any $i \neq j$ and $\overline{r}|_{G_{F_v}}$ is trivial, then $\mathrm{Spec} \, R_{\overline{r}|_{G_{F_v}}, \mathcal{D}_v^{\chi}}$ is an irreducible scheme of dimension $n^2+1$. Moreover, the generic point of $\mathrm{Spec} \, R_{\overline{r}|_{G_{F_v}}, \mathcal{D}_v^{\chi}}$ has characteristic zero.

\end{prop}

\begin{proof} See \cite[Proposition 3.16 and 3.17]{small}. \end{proof}

\begin{prop}We assume $v | l$. Let $\lambda_v \in (\mathbb{Z}_+^n)^{\mathrm{Hom}_{\mathbb{Q}_l}(F_v, E)}$.

If $R_{\overline{r}|_{G_{F_v}}}^{\mathrm{cris}, \lambda_v}$ is not zero, then the subfunctor $\mathcal{D}_v^{\mathrm{cris}, \lambda_v} := \mathrm{Hom}_{\mathrm{CNL}_{\mathcal{O}}}( R_{\overline{r}|_{G_{F_v}}}^{\mathrm{cris}, \lambda_v}, \ \ )$ of $\mathcal{D}_v^{\Box} = \mathrm{Hom}_{\mathrm{CNL}_{\mathcal{O}}}(R_{\overline{r}|_{G_{F_v}}}, \ \ )$ is a local deformation problem of $\overline{r}|_{G_{F_v}}$ over $\mathcal{O}$. 

\end{prop}

\begin{proof}  Let $r$ be the push-forward of $r^{\mathrm{univ}} : G_{F_v} \rightarrow \mathrm{GL}_n(R_{\overline{r}|_{G_{F_v}}})$ to $\mathrm{GL}_n(R^{\mathrm{cris}, \lambda}_{\overline{r}|_{G_{F_v}}})$. 

It suffices to prove that the morphism $\phi : R_{\overline{r}|_{G_{F_v}}} \rightarrow R_{\overline{r}|_{G_{F_v}}}^{\mathrm{cris}, \lambda_v}[[X_{i,j} \mid i ,j = 1, \cdots, n]]$ corresponding to $(I_n + (X_{i,j})) r (I_n + (X_{i, j}))^{-1}$ factors through $R_{\overline{r}|_{G_{F_v}}} \twoheadrightarrow R_{\overline{r}|_{G_{F_v}}}^{\mathrm{cris}, \lambda_v}$. 

By Proposition \ref{psd}, $R_{\overline{r}|_{G_{F_v}}}^{\mathrm{cris}, \lambda_v}[[X_{i,j} \mid i ,j = 1, \cdots, n]]$ is flat over $\mathcal{O}$ and reduced.

Thus, the result follows from the fact that any closed point of $\mathrm{Spec} \, R_{\overline{r}|_{G_{F_v}}}^{\mathrm{cris}, \lambda_v}[[X_{i,j} \mid i ,j = 1, \cdots, n]][\frac{1}{l}]$ is sent in $\mathrm{Spec} \, R_{\overline{r}|_{G_{F_v}}}^{\mathrm{cris}, \lambda_v}[\frac{1}{l}]$ by the morphism $\mathrm{Spec} \, R_{\overline{r}|_{G_{F_v}}}^{\mathrm{cris}, \lambda_v}[[X_{i,j} \mid i ,j = 1, \cdots, n]][\frac{1}{l}] \rightarrow \mathrm{Spec} \, R_{\overline{r}|_{G_{F_v}}}[\frac{1}{l}]$ corresponding to $\phi$. \end{proof}

\begin{prop}\label{ordinary deformation ring}

We assume $v|l$, $\overline{r}|_{G_{F_v}}$ is trivial and $[F_v:\mathbb{Q}_l] > 1 + \frac{n(n+1)}{2}$. Let $P$ be a minimal prime of $\mathcal{O}[[(\mathcal{O}_{F_v}^{\times}(l))^{\oplus n}]]$, $\Lambda_v:=\mathcal{O}[[(\mathcal{O}_{F_v}^{\times}(l))^{\oplus n}]]/P$ and $\widetilde{\Lambda_v} := \Lambda_v \widehat{\otimes}_{\mathcal{O}[[(\mathcal{O}_{F_v}^{\times}(l))^{\oplus n}]]} \mathcal{O}[[(F_v^{\times}(l))^{\oplus n}]]$. (Here, $(l)$ denotes the pro-$l$-completion.) 

1 \ $\mathcal{O}[[(F_v^{\times}(l))^{\oplus n}]]$ represents the functor of $\mathcal{F} : \mathrm{CNL}_{\mathcal{O}} \rightarrow \mathrm{Set} \ , A \mapsto \{ (\chi_1, \cdots, \chi_n) \mid \ For \ any \  i, \chi_i : G_{F_v} \rightarrow 1+\mathfrak{m}_A \ \mathrm{is \ a \ continuous \ character} \}$.

2 \ For an object $A$ of $\mathrm{CNL}_{\widetilde{\Lambda_v}}$, let $\widetilde{\mathcal{D}}^{\mathrm{det, ord}}_v(A)$ be the set of all lifts $r : G_{F_v} \rightarrow \mathrm{GL}_n(A)$ of $\overline{r}|_{G_{F_v}}$ to $A$ satisfying the following conditions.

$\cdot$ \ $\mathrm{det}(TI_n - r(g)) = \prod_{i=1}^n(T - f \circ \chi_i^{\mathrm{univ}} (g))$ for all $g \in G_{F_v}$.

$\cdot$ \ $(r(g_1)-f \circ \chi_1^{\mathrm{\mathrm{univ}}}(g_1)) \cdots (r(g_n) - f \circ \chi_n^{\mathrm{univ}}(g_n)) = 0$ for all $g_1, \cdots, g_n \in G_{F_v}$. 

(Here, $f : \widetilde{\Lambda_v} \rightarrow A$ denotes the structure morphism and $(\chi_1^{\mathrm{univ}}, \cdots, \chi_n^{\mathrm{univ}}) \in \mathcal{F}(\widetilde{\Lambda_v})$ denotes the image of the universal element by the quotient map $\mathcal{O}[[(F_v^{\times}(l))^{\oplus n}]] \twoheadrightarrow \widetilde{\Lambda_v}$.)

Then the functor $\widetilde{\mathcal{D}}_{v}^{det,ord} : \mathrm{CNL}_{\widetilde{\Lambda_v}} \rightarrow \mathrm{Set}$ defines a local deformation problem for $\overline{r}|_{G_{F_v}}$ over $\widetilde{\Lambda_v}$.

3 \ We put $R_{\overline{r}|_{G_{F_v}}, \mathcal{D}_v^{\mathrm{det, ord}}} := \mathrm{Im}(R_{\overline{r}|_{G_{F_v}}, \Lambda_v} \rightarrow R_{\overline{r}|_{G_{F_v}}, \widetilde{\mathcal{D}}_v^{\mathrm{det, ord}}})$. Then the functor $\mathcal{D}_v^{\mathrm{det, ord}} := \mathrm{Hom}_{\mathrm{CNL}_{\Lambda_v}}(R_{\overline{r}|_{G_{F_v}}, \mathcal{D}_v^{\mathrm{det, ord}}}, \ \ )$ defines a local deformation problem of $\overline{r}|_{G_{F_v}}$ over $\Lambda_v$.

4 \ $\mathrm{Spec} \, R_{\overline{r}|_{G_{F_v}}, \mathcal{D}_v^{\mathrm{det, ord}}}$ has a unique irreducible component $\mathcal{C}$ having the maximum dimension $1 + n^2 + [F_v:\mathbb{Q}_l]\frac{n(n+1)}{2}$.

5 \ $\mathcal{C}$ contains $\mathrm{Spec} \, R_{\overline{r}|_{G_{F_v}}, \mathcal{D}_v^{\mathrm{det, ord}}} \times_{\mathrm{Spec} \, \Lambda_v} U$, where $U$ is the open subscheme of $\mathrm{Spec} \, \Lambda_v$ defined by the condition $\chi_i^{\mathrm{univ}} \neq \chi_j^{\mathrm{univ}}$ for any $i \neq j$.

6 \ Any other irreducible component of $\mathrm{Spec} \, R_{\overline{r}|_{G_{F_v}}, \mathcal{D}_v^{\mathrm{det, ord}}}$ than $\mathcal{C}$ has dimension $ \le - 1 + n^2 + [F_v:\mathbb{Q}_l]\frac{n(n+1)}{2}$.

\end{prop}

\begin{proof}

1 and 2 are trivial.

3 \ It suffices to prove that the morphism $\phi : R_{\overline{r}|_{G_{F_v}}, \Lambda_v} \rightarrow R_{\overline{r}|_{G_{F_v}, \mathcal{D}_v^{\mathrm{det, ord}}}}[[X_{i,j} \mid i ,j = 1, \cdots, n]]$ corresponding to $(I_n + (X_{i,j})) r_{\mathcal{D}^{\mathrm{det, ord}}_v} (I_n + (X_{i, j}))^{-1}$ factors through $R_{\overline{r}|_{G_{F_v}}, \Lambda_v} \rightarrow R_{\overline{r}|_{G_{F_v}}, \mathcal{D}_v^{\mathrm{det, ord}}}$. ($r_{\mathcal{D}_v^{\mathrm{det, ord}}}$ denotes the push-forward of the universal lifting $r^{\mathrm{univ}} : G_{F_v} \rightarrow \mathrm{GL}_n(R_{\overline{r}|_{G_{F_v}}, \Lambda_v})$ of $\overline{r}|_{G_{F_v}}$ to $R_{\overline{r}|_{G_{F_v}, \mathcal{D}_v^{\mathrm{det, ord}}}}$.)

By the property 2 of this Proposition, the morphism $\psi : R_{\overline{r}|_{G_{F_v}}, \widetilde{\Lambda_v}} \rightarrow R_{\overline{r}|_{G_{F_v}, \widetilde{\mathcal{D}}_v^{\mathrm{det, ord}}}}[[X_{i,j}]]$ corresponding to $(I_n + (X_{i,j})) r_{\widetilde{\mathcal{D}}^{\mathrm{det, ord}}_v} (I_n + (X_{i, j}))^{-1}$ factors through $R_{\overline{r}|_{G_{F_v}}, \widetilde{\Lambda_v}} \rightarrow R_{\overline{r}|_{G_{F_v}}, \widetilde{\mathcal{D}}_v^{\mathrm{det, ord}}}$. By the following commutative diagram, we obtain the property 3.

\[\xymatrix{
R_{\overline{r}|_{G_{F_v}}, \Lambda_v} \ar[d] \ar[r]^-{\phi} & R_{\overline{r}|_{G_{F_v}, \mathcal{D}_v^{\mathrm{det, ord}}}}[[X_{i,j}]] \ar@{^{(}-_>}[d] \\
R_{\overline{r}|_{G_{F_v}}, \widetilde{\Lambda_v}} \ar[r]^-{\psi} & R_{\overline{r}|_{G_{F_v}, \widetilde{\mathcal{D}}_v^{\mathrm{det, ord}}}}[[X_{i,j}]]. \\
}\]

See \cite[Proposition 6.2.12]{10} for the proofs of 4, 5 and 6. \end{proof}

\begin{prop}\label{Taylor-Wiles deformation}

We assume $v \nmid l$, $q_v \equiv 1 \mod l$ and $\overline{r}|_{G_{F_v}}$ is unramified. We also assume that $\overline{r}|_{G_{F_v}} = \overline{\psi_v} \oplus \overline{s_v}$, where $\overline{r}(\mathrm{Frob}_v)$ acts on $\overline{\psi_v}$ as a scalar $\alpha_v \in \mathbb{F}^{\times}$ and $\mathrm{det}(\overline{s_v}(\mathrm{Frob}_v) - \alpha_v) \neq 0$. 

For an object $A$ of $\mathrm{CNL}_{\mathcal{O}}$, let $\mathcal{D}_v^{TW, \alpha_v}(A)$ be the set of all liftings of $\overline{r}|_{G_{F_v}}$ to $A$ which are equivalent to $\psi_v \oplus s_v$, where $s_v$ is an unramified lifting of $\overline{s}_v$ and $\psi_v$ is a lifting of $\overline{\psi_v}$ such that $I_{F_v}$ acts on $\psi_v$ as a scalar. 

Then the following properties hold.

1 \ $\mathcal{D}^{TW, \alpha_v}_v$ is a local deformation problem of $\overline{r}|_{G_{F_v}}$ over $\mathcal{O}$.

2 \ For any object $A$ of $\mathrm{CNL}_{\mathcal{O}}$ and any unramified lifting $\rho : G_{F_v} \rightarrow \mathrm{GL}_n(A)$, we have $\rho \in \mathcal{D}^{TW, \alpha_v}(A)$.

\end{prop}

\begin{proof} This easily follows from Hensel's Lemma. See \cite[Lemma 4.2]{small}. \end{proof}

\begin{dfn}

Let $\mathcal{S} = (\overline{r}, S, \{ \Lambda_v \}_{v \in S}, \{ \mathcal{D}_v \})$ be a global deformation problem. 

A Taylor-Wiles datum $(Q, \{ \alpha_v \}_{v \in Q})$ for $(\overline{r}, S)$ means a pair of a finite set $Q$ of finite places of $F$ and a family $\{ \alpha_v \}_{v \in Q}$ of elements of $\mathbb{F}^{\times}$ satisfying the following conditions.

1 \ $S \cap (Q \cup Q^c)$ and $Q \cap Q^c$ are empty.

2 \ For any $v \in Q$, the prime lying below $v$ splits in an imaginary quadratic field contained in $F$.

3 \ For any $v \in Q$, we have $q_v \equiv 1$ \ mod $l$.
   
4 \ For all $v \in Q$, $\alpha_v$ is an eigenvalue of $\overline{r}(\mathrm{Frob}_v)$ and $\overline{r}(\mathrm{Frob}_v)$ acts as a scalar $\alpha_v$ on the $\alpha_v$-generalized eigenspace of $\overline{r}(\mathrm{Frob}_v)$.
    
We put $\mathcal{S}(Q):=( \overline{r}, S \cup Q, \{ \Lambda_v \}_{v \in S} \cup \{ \mathcal{O} \}_{v \in Q}, \{ \mathcal{D}_v \}_{v \in S} \cup \{ \mathcal{D}_v^{TW, \alpha_v} \}_{v \in Q})$.

\end{dfn}

We fix a representative $r_{\mathcal{S}(Q)}$ of the universal deformation of type $\mathcal{S}(Q)$. Then for any $v \in Q$, $r_{\mathcal{S}(Q)}|_{G_{F_v}}$ is equivalent to $\psi_v \oplus s_v$ and the action of $I_{F_v}$ on $\psi_v$ induces a character $\mathbb{F}_v^{\times} \rightarrow 1 + \mathfrak{m}_{R_{\mathcal{S}(Q)}}$. (Note that the action of $G_{F_v}$ on $\psi_v$ is abelian.) Let $\Delta:=\prod_{v \in Q}\mathbb{F}_v^{\times}(l)$ and $\mathfrak{a}$ be the augmentation ideal of $\mathcal{O}[\Delta]$. Then, we obtain the local $\mathcal{O}$-morphism $\mathcal{O}[\Delta] \rightarrow R_{\mathcal{S}(Q)}$ and $R_{\mathcal{S}(Q)}/\mathfrak{a} \cong R_{\mathcal{S}}$ by 2 of Proposition \ref{Taylor-Wiles deformation}.

\begin{dfn}

    For a finite subgroup $H$ of $\mathrm{GL}_n(\overline{\mathbb{F}_l})$, we say that $H$ is adequate if $H$ satisfies the following conditions.
    
    1 \ $H^0(H, \mathfrak{sl}_n(\overline{\mathbb{F}_l})) = 0$.

    2 \ $H^1(H, \overline{\mathbb{F}_l}) = 0$.

    3 \ $H^1(H, \mathfrak{sl}_n(\overline{\mathbb{F}_l}))=0$.
     
4 \ For any nonzero $\mathbb{F}_l[H]$-submodule $W$ of $\mathfrak{sl}_n(\overline{\mathbb{F}_l})$, there exist a semisimple element $g \in H$ and an eigenvalue $\alpha \in \overline{\mathbb{F}_l}$ of $g$ such that $\mathrm{tr} \ e_{g, \alpha} W \neq 0$, where $e_{g, \alpha}$ is the projector of $\overline{\mathbb{F}_l}^{\oplus n}$ to the $\alpha$-eigenspace.

    \end{dfn}

\begin{prop}\label{adequate}

Let $H$ be a finite subgroup of $\mathrm{GL}_n(\overline{\mathbb{F}_l})$ which acts on $\overline{\mathbb{F}_l}^{\oplus n}$ absolutely irreducibly. Let $H^0$ denote the subgroup of $H$ generated by all $l$-power order elements of $H$ and $d$ denote the maximal dimension of irreducible $H^0$-subrepresentations.

If $l \ge 2(d + 1)$, then $H$ is adequate.

\end{prop}

\begin{proof}

See \cite[Theorems A.1 and A.9]{small}. \end{proof}

\begin{prop} \label{chebotarev} We assume that $F$ is an imaginary CM field.

Let $\mathcal{S} = (\overline{r}, S, \{ \Lambda_v \}_{v \in S}, \{ \mathcal{D}_v \}_{v \in S})$ be a global deformation problem. We assume that $\zeta_l \notin F$, $\overline{r}(G_{F(\zeta_l)})$ is adequate and all eigenvalues of $\overline{r}(G_{F(\zeta_l)})$ are contained in $\mathbb{F}$. We fix a positive integer $q \ge \mathrm{dim}_{\mathbb{F}}H^1(G_{F,S}, \mathrm{ad} \, \overline{r}(1))$. 

Then for any positive integer $N$, there exists a Taylor-Wiles datum $(Q_N, \{ \alpha_v \}_{v \in Q_N})$ for $(\overline{r}, S)$ satisfying the following conditions.

1 \ $|Q_N| = q$.

2 \ For any $v \in Q_N$, $q_v \equiv 1 \mod l^N$.

3 \ $\mathrm{dim}_{\mathbb{F}} \mathfrak{m}_{R_{\mathcal{S}(Q_N)}^S}/(\mathfrak{m}_{R_{\mathcal{S}(Q_N)}^S}^2, \mathfrak{m}_{R_{\overline{r}|_{G_{F_v}}, \mathcal{D}_v}})_{v \in S} = q - n^2[F^+:\mathbb{Q}]$.

\end{prop}

\begin{proof} See \cite[Proposition 6.7]{ade}. \end{proof}

\begin{rem}(Convention) \label{convension}

We use the following same convention as \cite{CW}.

For a residually absolutely irreducible continuous representation $r : G_{F} \rightarrow \mathrm{GL}_n(\overline{\mathbb{Q}}_l)$, there exists a continuous representation $r^{\circ} : G_{F} \rightarrow \mathrm{GL}_n(\mathcal{O}_{\overline{\mathbb{Q}}_l})$ such that $r^{\circ} \otimes \overline{\mathbb{Q}}_l \cong r$. This is unique up to conjugations of $\mathrm{GL}_n(\mathcal{O}_{\overline{\mathbb{Q}}_l})$. For a finite place $v$ of $F$ and a continuous representation $\rho_v : G_{F_v} \rightarrow \mathrm{GL}_n(\mathcal{O}_{\overline{\mathbb{Q}}_l})$, we say $r|_{G_{F_v}} \sim \rho_v$ if $r^{\circ}|_{G_{F_v}} \sim \rho_v$. This is independent of the choice of $r^{\circ}$ by Lemma \ref{equivalence irreducible}.

Similarly, for residually absolutely irreducible continuous representations $r_1, r_2 : G_{F} \rightarrow \mathrm{GL}_n(\overline{\mathbb{Q}}_l)$ and a finite place $v$ of $F$, we say $r_1|_{G_{F_v}} \sim r_2|_{G_{F_v}}$ if $r_1^{\circ}|_{G_{F_v}} \sim r_2^{\circ}|_{G_{F_v}}$.

For a residually absolutely irreducible continuous representation $r : G_{F} \rightarrow \mathrm{GL}_n(\overline{\mathbb{Q}}_l)$ and $v \mid l$, we say that $r|_{G_{F_v}}$ is potentially diagonalizable if $r^{\circ}|_{G_{F_v}}$ is potentially diagonalizable.

\end{rem}

Finally, we recall some properties about Bloch-Kato Selmer groups of adjoint representations. 

Let $r : G_F \rightarrow \mathrm{GL}_n(E)$ be a continuous representation which is unramified at almost all finite places. We take a finite set $S$ of finite places of $F$ containing all $l$-adic places and ramified places of $r$. 

We put $$H^1_f(F, \mathrm{ad} \, r) := \mathrm{Ker}( H^1(G_{F,S}, \mathrm{ad} \, r) \rightarrow \prod_{v \in S} (H^1(G_{F_v}, \mathrm{ad} \, r|_{G_{F_v}})/H^1_f(G_{F_v}, \mathrm{ad} \, r|_{G_{F_v}}) )),$$ where $H^1_f(G_{F_v}, \mathrm{ad} \, r|_{G_{F_v}}) := \mathrm{Ker} (H^1(G_{F_v}, \mathrm{ad} \, r|_{G_{F_v}}) \rightarrow H^1(I_{F_v}, \mathrm{ad} \, r|_{G_{F_v}}))$ if $v \nmid l$ and $H_f^1(G_{F_v}, \mathrm{ad} \, r|_{G_{F_v}}) := \mathrm{Ker}( H^1(G_{F_v}, \mathrm{ad} \, r|_{G_{F_v}}) \rightarrow H^1(G_{F_v}, B_{\mathrm{cris}} \otimes_{\mathbb{Q}_l} \mathrm{ad} \, r|_{G_{F_v}}) )$ if $v | l$. Then $H^1_f(F, \mathrm{ad} \, r)$ is independent of the choice of $S$.

\begin{lem} \label{vanishing of selmer group}
    
In the above situation, for any finite Galois extension $F'/F$, there exists a canonical injection $H^1_f(F, \mathrm{ad} \, r) \hookrightarrow H^1_f(F', \mathrm{ad} \, r|_{G_{F'}})$.

\end{lem}

\begin{proof} We take a set $S$ of finite places of $F$ containing all $l$-adic places, ramified places of $r$ and ramified places of $F'/F$. Let $S_{F'}$ denote the set of all finite places of $F'$ lying above places in $S$. Then we have the following commutative diagram.

\small

\[\xymatrix{
 H^1(G_{F, S}, \mathrm{ad} \, r) \ar[d] \ar[r] & H^1(G_{F', S_{F'}}, \mathrm{ad} \, r|_{G_{F'}} ) \ar[d] \\
 \prod_{v \in S}H^1(G_{F_v}, \mathrm{ad} \, r|_{G_{F_v}})/H^1_f(G_{F_v}, \mathrm{ad} \, r|_{G_{F_v}}) \ar[r] & \prod_{w \in S_{F'}}H^1(G_{F'_w}, \mathrm{ad} \, r|_{G_{F'_w}})/H^1_f(G_{F'_w}, \mathrm{ad} \, r|_{G_{F'_w}}). \\
}\]

\normalsize

The kernel of the first row is $H^1(\mathrm{Gal}(F'/F), (\mathrm{ad} \, r)^{G_{F'}})$ and this is zero since $(\mathrm{ad} \, r)^{G_{F'}}$ is divisible. Thus, we obtain an injection $H^1_f(F, \mathrm{ad} \, r) \hookrightarrow H^1_f(F', \mathrm{ad} \, r|_{G_{F'}})$. \end{proof}

\begin{lem} \label{definition of selmer group}

Let $\lambda \in (\mathbb{Z}^n_+)^{\mathrm{Hom}(F,E)}$ and $r : G_{F} \rightarrow \mathrm{GL}_n(E)$ be a residually absolutely irreducible algebraic representation such that $\mathrm{WD}(r|_{G_{F_v}} \otimes \overline{\mathbb{Q}}_l)$ is generic for all $v \nmid l$ and $r|_{G_{F_v}}$ is crystalline of $l$-adic Hodge type $\mathbf{v}_{\lambda_v}$ for all $v \mid l$. (We put $(\lambda_v)_{\tau, i} := \lambda_{\tau, i}$ for any $\tau \in \mathrm{Hom}_{\mathbb{Q}_l}(F_v, E) \subset \mathrm{Hom}(F, E)$ and $i = 1, \cdots, n$.)

We take a continuous representation $r^{\circ} : G_{F} \rightarrow \mathrm{GL}_n(\mathcal{O})$ such that $r^o \otimes_{\mathcal{O}} E \cong r$ and a finite set $S$ of finite places of $F$ containing all $l$-adic places and ramified places of $r$. Let $\mathcal{S}:=(\overline{r^{\circ}}, S, \{ \mathcal{O} \}_{v \in S}, \{ \mathcal{D}_v^{\Box} \}_{v \in S \setminus \{ u|l \}} \cup \{ \mathcal{D}_v^{\mathrm{cris}, \lambda_v} \}_{v | l})$ and $\mathfrak{p}$ denote the point of $\mathrm{Spec} \, R_{\mathcal{S}}$ corresponding to $r^{\circ}$.

Then we have $H^1_f(F, \mathrm{ad} \, r) \cong \mathrm{Hom}_E(\mathfrak{p}R_{\mathcal{S}, \mathfrak{p}}/\mathfrak{p}^2R_{\mathcal{S}, \mathfrak{p}}, E)$ and they are regarded as the set $\mathcal{D}$ of all deformations [$\tilde{r}$] of $r$ to $E[T]/(T^2)$ \footnote{A deformation [$\tilde{r}$] of $r$ to $E[T]/(T^2)$ means a equivalence class by the action of $I_n + T\mathrm{M}_n(E[T]/(T^2))$ of a continuous representation $\tilde{r} : G_{F} \rightarrow \mathrm{GL}_n(E[T]/(T^2))$ satisfying $\widetilde{r} \mod T = r$.} such that $\tilde{r}$ is unramified outside $S$ and $\tilde{r}|_{G_{F_v}}$ is crystalline of $l$-adic Hodge type $\mathbf{v}_{\lambda_v}$ for all $v | l$. 

\end{lem}

\begin{proof} By the same argument as \cite[Lemma 1.2.5]{Allen}, for any $v \mid l$, $H^1_f(G_{F_v}, \mathrm{ad} \, (r|_{G_{F_v}}))$ is regarded as the set of all deformations [$\tilde{\rho}_v$] of $r|_{G_{F_v}}$ to $E[T]/(T^2)$ such that $\tilde{\rho}_v$ is crystalline of $l$-adic Hodge type $\mathbf{v}_{\lambda_v}$ for all $v | l$. 
    
On the other hand, for $v \nmid l$, we have $H^0(G_{F_v}, \mathrm{ad} \, \rho_v(1)) = 0$ since $\mathrm{WD}(\rho_v \otimes \overline{\mathbb{Q}}_l)$ is generic. By local duality, we have $H^2(G_{F_v}, \mathrm{ad} \, \rho_v) = 0$. This implies $\mathrm{dim}_EH^1(G_{F_v}, \mathrm{ad} \, \rho_v) = \mathrm{dim}_EH^0(G_{F_v}, \mathrm{ad} \, \rho_v)$ by local Euler-Poincare characteristic formula. Since this is equal to $\mathrm{dim}_EH^1_f(G_{F_v}, \mathrm{ad} \, \rho_v)$, we obtain $H^1(G_{F_v}, \mathrm{ad} \, \rho_v) = H^1_{f}(G_{F_v}, \mathrm{ad} \, \rho_v)$.

Consequently, we obtain $H^1_f(F, \mathrm{ad} \, r) = \mathrm{Ker}(H^1(G_{F, S}, \mathrm{ad} \, r) \rightarrow \prod_{v|l} H^1(G_{F_v}, B_{\mathrm{cris}}\otimes_{\mathbb{Q}_l} \mathrm{ad} \, r|_{G_{F_v}}))$ and $H^1_f(F, \mathrm{ad} \, r ) \cong \mathcal{D}$.

By the same argument as \cite[Lemma 3.5]{OG}, we obtain the surjection $$\mathrm{Hom}_E(\mathfrak{p}R_{\mathcal{S}, \mathfrak{p}}/\mathfrak{p}^2R_{\mathcal{S}, \mathfrak{p}}, E) = \mathrm{Hom}_E^{\mathrm{loc}}(R_{\mathcal{S}, \mathfrak{p}}, E[T]/(T^2)) \rightarrow \mathcal{D}, f \mapsto f \circ r^{\circ}_{\mathcal{S}}.$$ ($r^{\circ}_{\mathcal{S}}$ denotes a representative of the universal deformation of $\overline{r^{\circ}}$ of type $\mathcal{S}$. ) This map is injective by Lemma \ref{trace density}. Therefore, we obtain $H^1_f(F, \mathrm{ad} \, r) \cong \mathrm{Hom}_E(\mathfrak{p}R_{\mathcal{S}, \mathfrak{p}}/\mathfrak{p}^2R_{\mathcal{S}, \mathfrak{p}}, E)$. \end{proof}

\subsubsection{Locally symmetric spaces and automorphic Galois representations}

We recall some notations and important results about locally symmetric spaces and automorphic Galois representations. Let $F$ be a CM field and $n$ be a positive integer.

For simplicity, we assume that $F$ contains an imaginary quadratic field. Note that Corollary \ref{Caraiani-Newton 2}, Theorem \ref{Lambert 2}, Lemma \ref{iota-ordinary} and Theorem \ref{ordinary semisimple} hold without this assumption.

For $(k_v) \in \mathrm{GL}_n(\mathbb{A}_F^{\infty})$, we write $\Gamma_{k_v}$ for the subgroup of $\overline{F_v}^{\times}$ generated by the eigenvalues of $k_v$. Note that the torsion part $\Gamma_{k_v}^{\mathrm{tor}}$ of $\Gamma_v$ is generated by a root of unity and consequently $\Gamma_{k_v}^{\mathrm{tor}}$ is naturally regarded as a subgroup of $\overline{\mathbb{Q}}^{\times}$.

For an open compact subgroup $K = \prod_v K_v$ of $\mathrm{GL}_n(\mathbb{A}_F^{\infty})$, we say that $K$ is good if $\cap_v \Gamma_{k_v}^{tor} = 1$ for any $(k_v) \in K$. 

For example, an open compact subgroup $K = \prod_vK_v$ of $\mathrm{GL}_n(\mathbb{A}_F^{\infty})$ is good if there exist finite places $v, w$ of $F$ such that $\mathrm{char} \, \mathbb{F}_v \neq \mathrm{char} \, \mathbb{F}_w$, $K_v \subset \mathrm{Iw}_{v,1}:=\mathrm{Ker}(\mathrm{Iw}_{v} \rightarrow (\mathbb{F}_v^{\times})^{ \oplus n})$ and $K_w \subset \mathrm{Iw}_{w,1}$.

We put $X := \mathrm{GL}_n(F \otimes_{\mathbb{Q}} \mathbb{R})/\mathbb{R}_{>0}\prod_{v|\infty}U(n)$. This is diffeomorphic to $\mathbb{R}^{[F^+:\mathbb{Q}]n^2-1}$. We also put $\mathfrak{X}:=\mathrm{GL}_n(F) \setminus X \times \mathrm{GL}_n(\mathbb{A}_{F}^{\infty})$. Here, we consider the discrete topology on $\mathrm{GL}_n(\mathbb{A}_F^{\infty})$. Any good subgroup $K$ acts on $\mathfrak{X}$ freely and properly discontinuously. This implies that $X_K := \mathfrak{X}/K$ is a smooth manifold of dimension $[F^+:\mathbb{Q}]n^2-1$. (See \cite{BS} for more detailed properties.)

Let $l$ be a prime, $E$ be a finite extension of $\mathbb{Q}_l$ contained in $\overline{\mathbb{Q}}_l$ such that $\tau(F) \subset E$ for any $\tau \in \mathrm{Hom}(F, \overline{\mathbb{Q}}_l)$, $\mathcal{O}$ be the ring of integers of $E$ and $\mathbb{F}$ be the residue field of $\mathcal{O}$.

We fix a finite set $S$ of finite places of $F$ satisfying the following conditions.
 
1 \ $S$ contains all $l$-adic places of $F$.

2 \ Any prime $p$ satisfies one of the following conditions.

(1) \ All $v|p$ are contained in $S$.

(2) \ $p$ is unramified in $F$ and any $v|p$ isn't contained in $S$ 

(3) \ $p$ splits in an imaginary quadratic field contained in $F$.

3 \ $S = S^c$.

We fix good subgroups $K, K'$ of $\mathrm{GL}_n(\mathbb{A}_F^{\infty})$ such that $K'$ is a normal subgroup of $K$ and $K_v = K_v' =\mathrm{GL}_n(\mathcal{O}_{F_v})$ for any $v \notin S$.

Let $R$ be a Noetherian $\mathcal{O}$-algebra and $\mathcal{V}$ be a finite free $R$-module with a $K_S$-action. We regard $\mathcal{V}$ as a $K$-module by projection $K \twoheadrightarrow K_S$.  Then $\mathcal{V}$ is regarded as a constant $K$-equivariant sheaf of $R$-modules on $\mathfrak{X}$ and $(\pi_*\mathcal{V})^{K'}$ is a $K/K'$-equivariant locally constant sheaf of $R$-modules on $X_{K'}$. (Here, we write $\pi : \mathfrak{X} \twoheadrightarrow X_K$ for the natural quotient map.) We simply write $\mathcal{V}$ for $(\pi_*\mathcal{V})^{K'}$.

We write $R\Gamma_{K/K'}(X_{K'}, \ \ )$ for the right derived functor of $\Gamma : \mathrm{Sh}_{K/K'}(X_{K'}, R) \rightarrow D(R[K/K'])$, where $\mathrm{Sh}_{K/K'}(X_{K'}, R)$ is the category of $K/K'$-equivariant sheaves of $R$-modules on $X_{K'}$.

If $K = K'$, we have $R\Gamma(X_{K'}, \mathcal{V}) = R\Gamma_{K/K'}(X_{K'}, \mathcal{V})$.

Moreover, since $\mathcal{V}$ is also regarded as a $\mathrm{GL}_n(\mathbb{A}_F^{\infty, S}) \times K_S$-module by projection to $K_S$ and a constant $\mathrm{GL}_n(\mathbb{A}_F^{\infty, S}) \times K_S$-equivariant sheaf of $R$-modules on $\mathfrak{X}$, we obtain an $\mathcal{O}$-morphism $\mathbb{T}^S:=\mathcal{H}(\mathrm{GL}_n(\mathbb{A}_F^{\infty, S}), \mathrm{GL}_n(\widehat{\mathcal{O}}_F^{S}))_{\mathcal{O}} \rightarrow \mathrm{End}_{D(R[K/K'])}(R\Gamma_{K/K'}(X_{K'}, \mathcal{V}))$ by the result 2 of the following proposition.

\begin{prop} \label{perfect complex} In the above situation, we have the following results.
    
1 \ The complex $R\Gamma_{K/K'}(X_{K'}, \mathcal{V})$ is a perfect complex of $R[K/K']$-modules.

2 \ We have the following commutative diagram.

\[\xymatrix{
D^+(\mathrm{Sh}_{\mathrm{GL}_n(\mathbb{A}_F^{\infty, S}) \times K_S}(\mathfrak{X}, R)) \ar[d]^{R\Gamma(\mathfrak{X}, \ )} \ar[r]^-{\mathrm{forget}} &  D^+(\mathrm{Sh}_{K}(\mathfrak{X})) \ar[d]^{\pi_*^{K'}} \\
D^+(R[\mathrm{GL}_n(\mathbb{A}_F^{\infty, S}) \times K_S]) \ar[d]^{R\Gamma(K', \ \ )} & D^+(\mathrm{Sh}_{K/K'}(X_{K'}, R)) \ar[d]^{R\Gamma_{K/K'}(X_{K'}, \ \ )} \\
D^+(\mathcal{H}(\mathrm{GL}_n(\mathbb{A}_F^{\infty, S}) \times K_S, K')_{R}) \ar[r]^-{\mathrm{forget}} & D^+(R[K/K']) \\
}\]

3 \ $R\Gamma(\mathfrak{X}, \mathcal{V}) \cong \Gamma(\mathfrak{X}, R) \otimes_{R} \mathcal{V}$ in $D(R[\mathrm{GL}_n(\mathbb{A}_F^{\infty, S}) \times K])$.

\end{prop}

\begin{proof} 1 \ See \cite[Lemma 2.1.7]{10}. 

2 \ See \cite[Proposition 2.18]{NTT}.

3 \ We have a canonical morphism $\Gamma(\mathfrak{X}, R) \otimes_{R} \mathcal{V} \cong \Gamma(\mathfrak{X}, \mathcal{V}) \rightarrow R\Gamma(\mathfrak{X}, \mathcal{V})$ in $D(R[\mathrm{GL}_n(\mathbb{A}_F^{\infty, S}) \times K])$. Thus it suffices to show that $H^i(\mathfrak{X}, \mathcal{V}) = 0$ for $i > 0$ in $D(R)$. This follows from $\mathfrak{X} \cong \sqcup_{g \in \mathrm{GL}_n(F) \setminus \mathrm{GL}_n(\mathbb{A}_F^{\infty})} \mathbb{R}^{[F^+:\mathbb{Q}]n^2-1}$. (Note that we now consider the discrete topology on $\mathrm{GL}_n(\mathbb{A}_F^{\infty})$.) \end{proof}

For a (not necessary commutative) ring $R$, a commutative ring $T$ and a bounded complex $C$ of $R$-modules with a ring morphism $T \rightarrow \mathrm{End}_{D(R)}(C)$, we write $T(C):=\mathrm{Im}(T \rightarrow \mathrm{End}_{D(R)}(C))$. (We often put $T =\mathbb{T}^S$ in the following.) Note that we have a canonical ring morphism $T(C) \rightarrow T(H^*(C))$ whose kernel is nilpotent by the following lemma. 

\begin{lem} \label{nilpotent} In the above situation, let $d$ be a positive integer and $q$ be an integer such that $H^i(C)=0$ for any $i \notin [q, d+q]$ .

Then the kernel $I$ of $S(C) \rightarrow \mathrm{End}_{R}(H^*(C))$ satisfies $I^{d+1}=0$.

\end{lem}

\begin{proof} See \cite[Lemma 2.2.4]{10}. \end{proof}

We fix $\lambda \in (\mathbb{Z}_+^{n})^{\mathrm{Hom}(F,E)}$.

For $\tau \in \mathrm{Hom}(F, E)$, we write $\mathcal{V}_{\lambda_{\tau}}$ for the set of all $\mathcal{O}$-valued points of the algebraic induction $(\mathrm{Ind}^{\mathrm{GL}_n}_{B_n}w_0\lambda_{\tau})/_{\mathcal{O}}$. (Here, $w_0\lambda_{\tau}:=(\lambda_{\tau, n}, \cdots, \lambda_{\tau, 1})$.) This is a finite free $\mathcal{O}$-module with a continuous action of $\mathrm{GL}_n(\mathcal{O})$ because $\mathcal{V}_{\lambda_{\tau}}$ is regarded as the set of all global sections of a line bundle on a projective smooth $\mathcal{O}$-scheme by \cite[Proposition 5.12 of I]{JZ}. 

In the following, we assume that $K_v \subset \mathrm{GL}_n(\mathcal{O}_{F_v})$ for any $v|l$. 

Thus, $\mathcal{V}_{\lambda} := \otimes_{\tau \in \mathrm{Hom}(F,E)} \mathcal{V}_{\lambda_{\tau}}$ is a finite free $\mathcal{O}$-module with a continuous $K_S$-action by projection to $K_l$.

\vspace{0.5 \baselineskip}

For any finite place $v$ of $F$, we put $$T_{v,i}:=[\mathrm{GL}_n(\mathcal{O}_{F_v})\mathrm{diag}(\underbrace{\varpi_v, \cdots, \varpi_v}_{i}, 1, \cdots, 1)\mathrm{GL}_n(\mathcal{O}_{F_v})] \in \mathcal{H}(\mathrm{GL}_n(F_v), \mathrm{GL}_n(\mathcal{O}_{F_v}))_{\mathbb{Z}}$$ and $$P_v(X) := \sum_{i=0}^n (-1)^iq_v^{\frac{i(i-1)}{2}}T_{v,i}X^{n-i} \in \mathcal{H}(\mathrm{GL}_n(F_v), \mathrm{GL}_n(\mathcal{O}_{F_v}))_{\mathbb{Z}}[X].$$ For any unramified irreducible smooth representation $\pi$ of $\mathrm{GL}_n(F_v)$ over $\mathbb{C}$, let $\varphi_{\pi} : \mathcal{H}(\mathrm{GL}_n(F_v), \mathrm{GL}_n(\mathcal{O}_{F_v}))_{\mathbb{Z}} \rightarrow \mathbb{C}$ denote the eigensystem corresponding to $\pi^{\mathrm{GL}_n(\mathcal{O}_{F_v})}$. Then we obtain $\mathrm{det}(XI_n - \mathrm{rec}_{F_v}(\pi|\mathrm{det}|_v^{\frac{1-n}{2}})(\mathrm{Frob}_v)) = \varphi_{\pi}(P_v)(X)$ by \cite[Corollary 3.1.2]{CD}.

\begin{thm}\label{torsion Galois} Let $\mathfrak{m}$ be a maximal ideal of $\mathbb{T}^S(R\Gamma(X_K, \mathcal{V}_{\lambda}))$.
    
Then there exists a continuous semisimple representation $\overline{\rho_{\mathfrak{m}}}: G_{F,S} \rightarrow \mathrm{GL}_n(\overline{\mathbb{T}^S(R\Gamma(X_K, \mathcal{V}_{\lambda}))/\mathfrak{m}})$ such that for any $v \notin S$, $\mathrm{det}(TI_n - \overline{\rho_{\mathfrak{m}}}(\mathrm{Frob}_v)) = P_v(T)$. (We simply write $P_v(T)$ for $P_v(T) \mod \mathfrak{m}$. We use similar notations in the following.) 

\end{thm}

\begin{proof} See \cite[Theorem 2.3.5]{10}. \end{proof}

In the following, we often regard a maximal ideal of $\mathbb{T}^S(R\Gamma(X_K, \mathcal{V}_{\lambda}))$ as a maximal ideal of $\mathbb{T}^S$. For a maximal ideal $\mathfrak{m}$ of $\mathbb{T}^S(R\Gamma(X_K, \mathcal{V}_{\lambda}))$, we say that $\mathfrak{m}$ is non-Eisenstein if $\overline{\rho_{\mathfrak{m}}}$ is absolutely irreducible. For a maximal ideal $\mathfrak{m}$ of $\mathbb{T}^S$, let $R\Gamma(X_K, \mathcal{V}_{\lambda})_{\mathfrak{m}}$ denote the image in $D(\mathcal{O})$ of the localization of $R\Gamma(K, R\Gamma(\mathfrak{X}, \mathcal{V}_{\lambda})) \in \mathrm{Ob} D(\mathbb{T}^S)$ at $\mathfrak{m}$. This is a $\mathbb{T}^S$-equivariant direct summand of $R\Gamma(X_K, \mathcal{V}_{\lambda})$ by Proposition \ref{perfect complex}.

\begin{thm}\label{Hecke algebra valued Galois} Let $\mathfrak{m}$ be a non-Eisenstein ideal of $\mathbb{T}^S(R\Gamma(X_K, \mathcal{V}_{\lambda}))$.

Then there exist a positive integer $\delta$ depending only on $n$ and $[F:\mathbb{Q}]$, an ideal $I$ of $\mathbb{T}^{S}(R\Gamma(X_K, \mathcal{V}_{\lambda})_{\mathfrak{m}})$ satisfying $I^{\delta} = 0$ and a continuous representation $\rho_{\mathfrak{m}}: G_{F,S} \rightarrow \mathrm{GL}_n(\mathbb{T}^S(R\Gamma(X_K, \mathcal{V}_{\lambda})_{\mathfrak{m}})/I)$ such that for any $v \notin S$, $\mathrm{det}(TI_n - \rho_{\mathfrak{m}}(\mathrm{Frob}_v)) = P_v(T)$.

\end{thm}

\begin{proof} See \cite[Theorem 2.3.7]{10}. \end{proof}

\begin{thm} \label{cohomology of locally symmetric space}

Let $\iota: \overline{\mathbb{Q}}_l \stackrel{\sim}{\rightarrow} \mathbb{C}$ be an isomorphism of fields and $\iota\lambda$ be the element of $(\mathbb{Z}^{n}_{+})^{\mathrm{Hom}(F,\mathbb{C})}$ satisfying $(\iota\lambda)_{\tau, i} := \lambda_{\iota^{-1} \tau, i}$ for $\tau \in \mathrm{Hom}(F, \mathbb{C})$.

1 \ Let $\pi$ be a cohomological cuspidal automorphic representation of weight $\iota \lambda$ satisfying $(\pi^{\infty})^K \neq 0$. Then we obtain the following results.

$\cdot$ \ $\iota^{-1}(\pi^{\infty})^K:=\overline{\mathbb{Q}}_l \otimes_{\iota^{-1}, \mathbb{C}} (\pi^{\infty})^K$ is a $\mathcal{H}(\mathrm{GL}_n(\mathbb{A}_F^{\infty}), K)_{\overline{\mathbb{Q}}_l}$-equivariant direct summand of $H^*(X_K, \mathcal{V}_{\lambda}) \otimes_{\mathcal{O}} \overline{\mathbb{Q}}_l$.

$\cdot$ \ The Hecke eigensystem $\varphi_{\pi, \iota}: \mathbb{T}^S \rightarrow \overline{\mathbb{Q}}_l$ corresponding to $\iota^{-1}(\pi^{\infty})^K$ factors through $\mathbb{T}^S \twoheadrightarrow \mathbb{T}^S(R\Gamma(X_K, \mathcal{V}_{\lambda}))$.

$\cdot$ \ The maximal ideal $\mathfrak{m}:=\mathrm{Ker}(\mathbb{T}^S(R\Gamma(X_K, \mathcal{V}_{\lambda})) \xrightarrow{\varphi_{\iota, \pi}} \mathcal{O}_{\overline{\mathbb{Q}}_l} \twoheadrightarrow \overline{\mathbb{F}_l})$ induces $\overline{\rho_{\mathfrak{m}}} \cong \overline{r_{\iota}(\pi)}$.

$\cdot$ \ If $\overline{r_{\iota}(\pi)}$ is absolutely irreducible, then $\mathfrak{m}$ is a non-Eisenstein ideal of $\mathbb{T}^S(R\Gamma(X_K, \mathcal{V}_{\lambda}))$.

2 \ Let $\mathfrak{m}$ be a non-Eisenstein ideal of $\mathbb{T}^S(R\Gamma(X_K, \mathcal{V}_{\lambda}))$ and $\mathcal{T}_{\mathfrak{m}}$ be the set of all cohomological cuspidal automorphic representations of $\mathrm{GL}_n(\mathbb{A}_F)$ of weight $\iota\lambda$ such that $(\pi^{\infty})^K \neq 0$ and $\overline{r_{\iota}(\pi)} \cong \sigma \circ \overline{\rho_{\mathfrak{m}}}$ for some isomorphism $\sigma : \overline{\mathbb{T}^S(R\Gamma(X_K, \mathcal{V}_{\lambda}))/\mathfrak{m}} \stackrel{\sim}{\rightarrow} \overline{\mathbb{F}_l}$ over $\mathbb{F}$. 

(1) \ There exists an isomorphism of $\mathcal{H}(\mathrm{GL}_n(\mathbb{A}_F^{\infty}), K)_{\overline{\mathbb{Q}}_l}$-modules $$H^*(X_K, \mathcal{V}_{\lambda})_{\mathfrak{m}} \otimes_{\mathcal{O}} \overline{\mathbb{Q}}_l \cong (\oplus_{\pi \in \mathcal{T}_{\mathfrak{m}}} (\pi^{\infty})^K \otimes_{\mathbb{C}} H^*(\mathfrak{g}, \prod_{v|\infty}U(n);\pi_{\infty} \otimes_{\mathbb{C}} V_{\iota \lambda})) \otimes_{\mathbb{C}, \iota^{-1}} \overline{\mathbb{Q}}_l.$$ (Here, $\mathfrak{g}$ denotes the Lie algebra of the Lie group $\mathrm{Ker}(N_{F/\mathbb{Q}} \circ \mathrm{det} : \mathrm{Res}_{F/\mathbb{Q}}\mathrm{GL}_{n, F} \rightarrow \mathrm{GL}_{1, \mathbb{Q}})(\mathbb{R})$ and $V_{\iota \lambda}$ denote an irreducible algebraic representation of $\prod_{\mathrm{Hom}(F, \mathbb{C})} \mathrm{GL}_n(\mathbb{C})$ having highest weight $\iota \lambda$.)

(2) \ Let $l_0:=[F^+:\mathbb{Q}]n-1$ and $q_0:=[F^+:\mathbb{Q}]\frac{n(n-1)}{2}$.

If $\mathcal{T}_{\mathfrak{m}}$ is not empty, then $H^i(X_K, \mathcal{V}_{\lambda})_{\mathfrak{m}}[\frac{1}{l}] \neq 0$ for any $i \in [q_0, q_0 + l_0]$ and $H^i(X_K, \mathcal{V}_{\lambda})_{\mathfrak{m}}[\frac{1}{l}] = 0$ for any $i \notin [q_0, q_0 + l_0]$.
    
(3) \ The $E$-algebra $\mathbb{T}^S(H^*(X_K, \mathcal{V}_{\lambda})_{\mathfrak{m}}[\frac{1}{l}]) = \mathbb{T}^S(R\Gamma(X_K, \mathcal{V}_{\lambda})_{\mathfrak{m}})[\frac{1}{l}]$ is finite $\acute{e}$tale over $E$. Moreover, by sending $\pi$ to $\varphi_{\pi, \iota}$, we obtain a bijection between $\mathcal{T}_{\mathfrak{m}}$ and $\{ f: \mathbb{T}^S(R\Gamma(X_K, \mathcal{V}_{\lambda})_{\mathfrak{m}}) \rightarrow \overline{\mathbb{Q}}_l \mid \mathcal{O} \textrm{-} \mathrm{morphism} \}$. (Note that $\varphi_{\pi, \iota} \circ \rho_{\mathfrak{m}} \cong r_{\iota}(\pi)$, where $\rho_{\mathfrak{m}}$ denote the Galois representation of Theorem \ref{Hecke algebra valued Galois}.)

\end{thm}

\begin{proof} See \cite[proof of Theorem 2.4.10]{10}. \end{proof}

We recall some important properties of the Galois representation of Theorem \ref{Hecke algebra valued Galois}.

\begin{prop} \label{unipotently ramified}

Let $\mathfrak{m}$ be a non-Eisenstein ideal of $\mathbb{T}^S(R\Gamma(X_K, \mathcal{V}_{\lambda}))$, $R$ be a subset of $S$ not containing any $l$-adic place such that for any $v \in R$, the prime lying below $v$ splits in an imaginary quadratic field contained in $F$.

We assume $K_v=\mathrm{Iw}_v$ for any $v \in R$ and $K_v = \mathrm{GL}_n(\mathcal{O}_{F_v})$ for any $v \in R^c \setminus R$.

Let $\chi_{v,1}, \chi_{v,2}, \cdots, \chi_{v,n} : \mathbb{F}_v^{\times} \rightarrow 1 + \varpi\mathcal{O}$ be characters for $v \in R$ and $\mathcal{V}_{\lambda}(\chi^{-1}) := \mathcal{V}_{\lambda} \otimes_{\mathcal{O}} \mathcal{O}(\chi^{-1})$ be the $\mathcal{O}[K_S]$-module, where $K_S$ acts on $\mathcal{V}_{\lambda}$ by projection to $K_l$ and on $\mathcal{O}(\chi^{-1})$ by projection to $K_{R}$. We put $T := S \setminus (R^c \setminus R)$.

Then there exist an positive integer $\delta$ depending only on $n$ and $[F:\mathbb{Q}]$, an ideal $I$ of $\mathbb{T}^{S}(R\Gamma(X_K, \mathcal{V}_{\lambda}(\chi^{-1}))_{\mathfrak{m}})$ satisfying $I^{\delta} = 0$ and a continuous representation $$\rho_{\mathfrak{m}}: G_{F,T} \rightarrow \mathrm{GL}_n(\mathbb{T}^T(R\Gamma(X_K, \mathcal{V}_{\lambda}(\chi^{-1}))_{\mathfrak{m}})/I)$$ such that for any $v \notin T$, $\mathrm{det}(TI_n - \rho_{\mathfrak{m}}(\mathrm{Frob}_v)) = P_v(T)$ and $\mathrm{det}(T - \rho_{\mathfrak{m}}(\sigma)) = \prod_{i=1}^n(T-\chi_{v,i}(\mathrm{Art}_{F_v}^{-1}(\sigma)))$ for any $v \in R$ and $\sigma \in I_{F_v}$.

\end{prop}

\begin{proof} Let $K_1$ be the good subgroup of $\mathrm{GL}_n(\mathbb{A}_F^{\infty})$ defined by $(K_1)_v = K_v$ for any $v \notin R$ and $(K_1)_v = \mathrm{Iw}_{v,1}:=\mathrm{Ker}(\mathrm{Iw}_{v} \rightarrow (\mathbb{F}_v^{\times})^{\oplus n})$ for any $v \in R$. 

By \cite[Theorem 3.1.1]{10}, we obtain an integer $\delta$ depending only on $n$ and $[F:\mathbb{Q}]$, an ideal $I_1$ of $\mathbb{T}^{T}_R(R\Gamma_{K/K_1}(X_{K_1}, \mathcal{V}_{\lambda}(\chi^{-1}))_{\mathfrak{m}}) = \mathbb{T}^{T}_R(R\Gamma_{K/K_1}(X_{K_1}, \mathcal{V}_{\lambda})_{\mathfrak{m}}(\chi^{-1}))$ satisfying $I_1^{\delta} = 0$ and a continuous representation $$\rho^1_{\mathfrak{m}}: G_{F,T} \rightarrow \mathrm{GL}_n(\mathbb{T}^T_R(R\Gamma_{K/K_1}(X_{K_1}, \mathcal{V}_{\lambda}(\chi^{-1}))_{\mathfrak{m}})/I_1)$$ such that for any $v \notin T$, $\mathrm{det}(TI_n - \rho_{\mathfrak{m}}(\mathrm{Frob}_v)) = P_v(T)$ and $\mathrm{det}(T - \rho_{\mathfrak{m}}(\sigma)) = \prod_{i=1}^n(T-\chi_{v,i}(\mathrm{Art}_{F_v}^{-1}(\sigma))t_{v,i}(\sigma))$ for any $v \in T$ and $\sigma \in I_{F_v}$. 

Here, $t_{v,i}$ is a certain morphism $W_{F_v}^{ab} \rightarrow \mathcal{H}(\mathrm{GL}_n(F_v), \mathrm{Iw}_{v,1})_{\mathcal{O}}^{\times} \subset \mathcal{H}(\mathrm{GL}_n(\mathbb{A}_F^{\infty}), K_1)^{\times}_{\mathcal{O}}$ such that $t_{v,i} \circ \mathrm{Art}_{F_v}(x) = \mathrm{diag}(1, \cdots, 1, \underset{i}{x}, 1, \cdots, 1)\mathrm{Iw}_{v,1} \in \mathrm{Iw}_v/\mathrm{Iw}_{v,1}$ for any $x \in \mathcal{O}_{F_v}^{\times}$ (see \cite[p925]{10} for the precise definition of $t_{v,i}$ ) and $\mathbb{T}^T_R := \mathbb{T}^T[t_{v,i}(\sigma) \mid v \in R, i = 1, \cdots, n, \sigma \in I_{F_v}] \subset \mathcal{H}(\mathrm{GL}_n(\mathbb{A}_F^{\infty}), K_1)_{\mathcal{O}}$. (Note that $\mathbb{T}^T_R$ is commutative.)

This implies the result since the canonical morphism $$\mathbb{T}^T_R(R\Gamma_{K/K_1}(X_{K_1}, \mathcal{V}_{\lambda}(\chi^{-1}))_{\mathfrak{m}}) \rightarrow \mathbb{T}^T_R(R\Gamma(X_K, \mathcal{V}_{\lambda}(\chi^{-1}))_{\mathfrak{m}}) = \mathbb{T}^T(R\Gamma(X_K, \mathcal{V}_{\lambda}(\chi^{-1}))_{\mathfrak{m}})$$ sends $t_{v,i}(\sigma)$ ($\sigma \in I_{F_v}$) to $1$. \end{proof}

We recall the following technical property, which is required for studying properties at $l$-adic places of automorphic Galois representations.

\begin{dfn} \label{decomposed generic}

Let $L$ be a number field, $\overline{r}: G_L \rightarrow \mathrm{GL}_n(\overline{\mathbb{F}_l})$ be a continuous representation and $p \neq l$ be a prime.
        
We say that $p$ is decomposed generic for $\overline{r}$ if $p$ splits completely in $L$ and for all $v|p$, $\overline{r}|_{G_{L_v}}$ is unramified and the eigenvalues $\alpha_{v,1}, \cdots, \alpha_{v,n}$ of $\overline{r}(\mathrm{Frob}_v)$ satisfy $\frac{\alpha_{v,i}}{\alpha_{v,j}} \neq p$ for all $i \neq j$. 
   
We say that $\overline{r}$ is decomposed generic if there exists a prime $p$ which is decomposed generic for $\overline{r}$.

Note that if $\overline{r}$ is decomposed generic, then there exist positive Dirichlet density primes which are decomposed generic for $\overline{r}$ by Lemma \ref{decomposed genericity density}.

\end{dfn}

We recall important results of \cite{MEI}.

\begin{thm} \label{Caraiani-Newton}

Let $\mathfrak{m}$ be a non-Eisenstein ideal of $\mathbb{T}^S(R\Gamma(X_K, \mathcal{V}_{\lambda}))$ and $\overline{v}$ be an $l$-adic place of $F^+$. We suppose the following conditions.

1 \ $\overline{\rho_{\mathfrak{m}}}$ is decomposed generic.

2 \ $l$ splits in an imaginary quadratic field contained in $F$.

3 \ For any $v|\overline{v}$, $K_v=\mathrm{GL}_n(\mathcal{O}_{F_v})$.

4 \ There exists an $l$-adic place $\overline{v}' \neq \overline{v}$ of $F^+$ such that $$\sum_{\overline{v}, \overline{v}' \neq \overline{v}''|l} [F^+_{\overline{v}''}:\mathbb{Q}_l] \ge \frac{1}{2}[F^+:\mathbb{Q}].$$

Then there exist an integer $\delta$ depending only on $n$ and $[F:\mathbb{Q}]$, an ideal $I$ of $\mathbb{T}^S(R\Gamma(X_K, \mathcal{V}_{\lambda})_{\mathfrak{m}})$ and a continuous representation $\rho_{\mathfrak{m}}: G_{K,S} \rightarrow \mathrm{GL}_n(\mathbb{T}^S(R\Gamma(X_K,\mathcal{V}_{\lambda})_{\mathfrak{m}})/I)$ satisfying the following conditions.

(1) \ For any $v \notin S$, we have $\mathrm{det}(TI_n - \rho_{\mathfrak{m}}(\mathrm{Frob}_v)) = P_v(T)$.

(2) \ For any $v|\overline{v}$, the $\mathcal{O}$-morphism $R_{\overline{\rho_{\mathfrak{m}}}|_{G_{F_v}}} \rightarrow \mathbb{T}^S(R\Gamma(X_K,\mathcal{V}_{\lambda})_{\mathfrak{m}})/I$ corresponding to $\rho_{\mathfrak{m}}|_{G_{F_v}}$ factors through $R_{\overline{\rho_{\mathfrak{m}}}|_{G_{F_v}}} \twoheadrightarrow R^{\mathrm{cris}, \lambda_v}_{\overline{\rho_{\mathfrak{m}}}|_{G_{F_v}}}$.

\end{thm}

\begin{proof} See \cite[Theorem 4.2.15]{MEI}. \end{proof}

\begin{cor} \label{Caraiani-Newton 2}

Let $\iota : \overline{\mathbb{Q}}_l \stackrel{\sim}{\rightarrow} \mathbb{C}$ be an isomorphism of fields, $\pi$ be a cohomological cuspidal automorphic representation of $\mathrm{GL}_n(\mathbb{A}_F)$ of weight $\iota \lambda$ and $\overline{v}$ be an $l$-adic place of $F^+$.

We assume the following conditions.

1 \ $\overline{r_{\iota}(\pi)}$ is absolutely irreducible and decomposed generic.

2 \ For all $v|\overline{v}$, $\pi_v$ is unramified.

Then $r_{\iota}(\pi)|_{G_{F_v}}$ is crystalline of $l$-adic Hodge type $\mathbf{v}_{\lambda_v}$ for all $v|\overline{v}$.
 
\end{cor}

\begin{rem} Note that we need not assume that $F$ contains an imaginary quadratic field. \end{rem}

\begin{proof}
    
Let $p \neq l$ be a prime which is decomposed generic for $\overline{r_{\iota}(\pi)}$. By Corollary \ref{extension of Q} and Corollary \ref{imaginary quadratic 2}, there exists a finite solvable CM extension $F'/F$ satisfying the following conditions. (Let $\overline{w}$ be an $l$-adic place of $F'^+$ lying above $\overline{v}$.)

1 \ $F'$ is linearly disjoint from $\overline{F}^{\mathrm{Ker}\overline{r_{\iota}(\pi)}}$ over $F$.

2 \ $l$ splits in an imaginary quadratic field contained in $F'$.

3 \ $p$ splits completely in $F$.

4 \ All $l$-adic places of $F$ split completely in $F'$.

5 \ There exists $\overline{w}' \neq \overline{w} | l$ such that $$\sum_{\overline{w}, \overline{w}' \neq \overline{w}''|l} [F'^+_{\overline{w}''}:\mathbb{Q}_l] \ge \frac{1}{2}[F'^+:\mathbb{Q}].$$

Then $\mathrm{BC}_{F'/F}(\pi)$ is a cohomological cuspidal, $r_{\iota}(\mathrm{BC}_{F'/F}(\pi)) \cong r_{\iota}(\pi)|_{G_{F'}}$ and $r_{\iota}(\mathrm{BC}_{F'/F}(\pi))|_{G_{F'_w}} \cong r_{\iota}(\pi)|_{G_{F_v}}$ if we identify $G_{F'_w} = G_{F_v}$ for all $w|\overline{w}$ and $w|v$ by Proposition \ref{base change}. Therefore, we may assume that $F$ satisfies the conditions of Theorem \ref{Caraiani-Newton}. Moreover, after replacing $K$ and $S$, we may assume that $(\pi^{\infty})^K \neq 0$ and $K_v = \mathrm{GL}_n(\mathcal{O}_{F_v})$ for $v \mid \overline{v}$. 

Let $\mathfrak{m}$ be the non-Eisenstein ideal of $\mathbb{T}^S(R\Gamma(X_{K}, \mathcal{V}_{\lambda}))$ corresponding to $\overline{r_{\iota}(\pi)}$. (See Theorem \ref{cohomology of locally symmetric space}.) Since the Galois representation $\rho_{\mathfrak{m}}$ of Theorem \ref{Caraiani-Newton} satisfies $\varphi_{\iota, \pi} \circ \rho_{\mathfrak{m}} \cong r_{\iota}(\pi)$, we obtain the result. \end{proof}

\begin{thm} \label{Lambert 2}

    Let $\iota : \overline{\mathbb{Q}}_l \stackrel{\sim}{\rightarrow} \mathbb{C}$ be an isomorphism of fields and $\pi$ be a cohomological cuspidal automorphic representation of $\mathrm{GL}_n(\mathbb{A}_F)$ of weight $\iota\lambda$. We assume that $\overline{r_{\iota}(\pi)}$ is absolutely irreducible and decomposed generic.
    
    Then $r_{\iota}(\pi)|_{G_{F_v}}$ is de Rham of $l$-adic Hodge type $\mathbf{v}_{\lambda_v}$ for all $v|l$.
    
    \end{thm}
    
    \begin{rem} Note that we need not assume that $F$ contains an imaginary quadratic field. \end{rem}
    
    \begin{proof} See \cite[Theorem 4.3.3]{RG}. \end{proof}

We also recall results about ordinary parts. 

We assume that $K_v=\mathrm{Iw}_v$ for all $v|l$.

For $0 \le b \le c$ with $1 \le c$, we write $\mathrm{Iw}_{v}(b,c)$ for the set of all elements of $\mathrm{GL}_n(\mathcal{O}_{F_v})$ which are upper-triangular matrices modulo $\varpi_v^c$ and are unipotent upper-triangular matrices modulo $\varpi_v^b$.

Let $K(b,c)$ denotes the good subgroup of $\mathrm{GL}_n(\mathbb{A}_F^{\infty})$ defined by $K(b,c)_v:=K_v$ for all $v \nmid l$ and $K(b,c)_v = \mathrm{Iw}_v(b,c)$ for all $v | l$.

We put $T_n(\mathcal{O}_{F_l})_b := \prod_{v|l}T_n(\mathcal{O}_{F_v}/\varpi_v^b)$ and $T_n(\mathcal{O}_{F_l})(b):=\prod_{v | l} T_n(\mathcal{O}_{F_v})(b) := \prod_{v|l}\mathrm{Ker}(T_n(\mathcal{O}_{F_v}) \rightarrow T_n(\mathcal{O}_{F_v}/\varpi_v^b))$. Then we have $T_n(\mathcal{O}_{F_l})_b = K(0,c)/K(b,c)$.

We have open submonoids $\Delta_v^{T_n}:=\sqcup_{\mu_1 \ge \cdots \ge \mu_n} T_n(\mathcal{O}_{F_v})\mathrm{diag}(\varpi_v^{\mu_1}, \cdots, \varpi_v^{\mu_n})T_n(\mathcal{O}_{F_v})$, $\Delta_v := \mathrm{Iw}_v\Delta_v^{T_n}\mathrm{Iw}_v$ of $T_n(F_v)$, $\mathrm{GL}_n(F_v)$ respectively.

For $\lambda \in (\mathbb{Z}_+^n)^{\mathrm{Hom}(F,E)}$, we have the continuous morphism $\alpha_{\lambda} : \Delta := \prod_v \Delta_v \rightarrow E^{\times}$ defined by $$(k_{v,1}\mathrm{diag}(\varpi_v^{\mu_{v,1}}, \cdots, \varpi_v^{\mu_{v,n}})k_{v,2})_{v|l} \mapsto \displaystyle {\prod_{v|l}\prod_{\tau \in \mathrm{Hom}_{\mathbb{Q}_l}(F_v, E)} \prod_{i=1}^n \tau(\varpi_v)^{\mu_{v,i} (w_0\lambda)_{\tau, i}}}.$$ (We put $w_0\lambda := (\lambda_{\tau, n}, \lambda_{\tau, n-1}, \cdots, \lambda_{\tau, 1})_{\tau} \in (\mathbb{Z}^n)^{\mathrm{Hom}(F, E)}$.)

Then we have an action of $\Delta$ on $\mathcal{V}_{\lambda}$ defined by the formula $\delta \cdot x := \alpha_{\lambda}(\delta)^{-1} \delta x$. ($\delta x$ denotes the element obtained by the natural action of $\delta$ on $x \in \mathcal{V}_{\lambda} \otimes_{\mathcal{O}} E$. See \cite[Lemma 2.2 and Definition 2.8]{OG} for the proof.)

Therefore, $\mathbb{T}^S \otimes_{\mathcal{O}} \mathcal{H}(\Delta, K(b,c)_l)_{\mathcal{O}}$ acts on $R\Gamma_{K(0,c)/K(b,c)}(X_{K(b,c)}, \mathcal{V}_{\lambda}) = R\Gamma(K(b,c), \Gamma(\mathfrak{X}, \mathcal{O}) \otimes_{\mathcal{O}} \mathcal{V}_{\lambda})$. (See Proposition \ref{perfect complex}.)

For $v|l$ and $i=1, \cdots, n$, we put $$U_{v,i}:=[\mathrm{Iw}_v(b,c) \mathrm{diag}(\underbrace{\varpi_v, \cdots, \varpi_v}_{i}, 1, \cdots, 1)\mathrm{Iw}_v(b,c)] \in \mathcal{H}(\mathrm{GL}_n(F_v), \mathrm{Iw}_v(b,c))_{\mathcal{O}}.$$ Note that the map $\mathcal{H}(\Delta_v^{T_n}, T_n(\mathcal{O}_{F_v})(b))_{\mathcal{O}} \rightarrow \mathcal{H}(\Delta_v, \mathrm{Iw}_v(b,c))_{\mathcal{O}}$ defined by $$[T_n(\mathcal{O}_{F_v})(b)gT_n(\mathcal{O}_{F_v})(b)] \mapsto [\mathrm{Iw}_v(b,c)g\mathrm{Iw}_v(b,c)]$$ is a morphism of $\mathcal{O}$-algebras by \cite[Lemma 2.1.12]{10}.

Thus, by $X_{v, i} \mapsto U_{v,i}, \prod_{v|l} (\mathcal{O}_{F_v}^{\times})^{\oplus n} \twoheadrightarrow K(0,c)/K(b,c)$, we obtain an $\mathcal{O}$-morphism $\mathbb{T}^{S, \mathrm{ord}}:=\mathbb{T}^{S} \otimes_{\mathcal{O}} \mathcal{O}[[\prod_{v|l}(\mathcal{O}_{F_v}^{\times})^{\oplus n}]][\{ X_{v, 1}, \cdots, X_{v,n}, X_{v,n}^{-1} \}_{v|l}] \rightarrow \mathbb{T}^S \otimes_{\mathcal{O}} \mathcal{H}(\Delta, K(b,c)_l)_{\mathcal{O}}$. 

By this morphism, we obtain an action of $\mathbb{T}^{S, \mathrm{ord}}$ on $R\Gamma_{K(0,c)/K(b,c)}(X_{K(b,c)}, \mathcal{V}_{\lambda})$.

Note that or any $0 \le b' \le c'$ satisfying $b \le b'$ and $c \le c'$, the canonical morphism $R\Gamma(K(b,c), \Gamma(\mathfrak{X}, \mathcal{O}) \otimes_{\mathcal{O}} \mathcal{V}_{\lambda}) \rightarrow  R\Gamma( K(b',c'), \Gamma(\mathfrak{X}, \mathcal{O}) \otimes_{\mathcal{O}} \mathcal{V}_{\lambda})$ in $D(\mathcal{O})$ is $\mathbb{T}^{S,\mathrm{ord}}$-equivariant. (See \cite[Lemma 2.10]{OG}.)

Let $R\Gamma_{K(0,c)/K(b,c)}(X_{K(b,c)}, \mathcal{V}_{\lambda})^{\mathrm{ord}}$ denote the image in $D(\mathcal{O}[K(0,c)/K(b,c)])$ of the localization of $R\Gamma(K(b,c), R\Gamma(\mathfrak{X}, \mathcal{O}) \otimes_{\mathcal{O}} \mathcal{V}_{\lambda}) \in D(\mathbb{T}^{S, \mathrm{ord}})$ by $\displaystyle \prod_{v|l, i}U_{v,i}$. This is a perfect complex of $\mathcal{O}[K(0,c)/K(b,c)]$-modules by Proposition \ref{perfect complex}.

Then we obtain the following fundamental results.

\begin{prop} (Level independence) \label{level independence}
    
We put $d:=\mathrm{max}\{ 1, b \}$ and identify $K(0,c)/K(b,c) = K(0, d)/K(b,d) = T_n(\mathcal{O}_{F_l})_b$. Then we have a canonical $\mathbb{T}^{S, \mathrm{ord}}$-equivalent isomorphism

$R\Gamma_{K(0,d)/K(b,d)}(X_{K(b,d)}, \mathcal{V}_{\lambda})^{\mathrm{ord}} \cong R\Gamma_{K(0,c)/K(b,c)}(X_{K(b,c)}, \mathcal{V}_{\lambda})^{\mathrm{ord}}$ in $D(\mathcal{O}[T_n(\mathcal{O}_{F_l})_b])$.

\end{prop}

\begin{proof} See \cite[Corollary 5.2.16]{10}. \end{proof}

\begin{prop} (Weight independence) \label{weight independence}

Let $\lambda' \in (\mathbb{Z}^n_{+})^{\mathrm{Hom}(F,E)}$ and $m$ be a positive integer such that $\mathcal{O}(w_0\lambda)/ \varpi^m = \mathcal{O}(w_0\lambda')/ \varpi^m$ as $\mathcal{O}/\varpi^m[T_n(\mathcal{O}_{F_l})(b)]$-modules.

Then we have a canonical $\mathbb{T}^{S, \mathrm{ord}}$-equivalent isomorphism \begin{align*}R\Gamma_{K(0,c)/K(b,c)}(X_{K(b,c)}, \mathcal{V}_{\lambda}/\varpi^m)^{\mathrm{ord}} \otimes_{\mathcal{O}} \mathcal{O}(-w_0\lambda) \\
\cong R\Gamma_{K(0,c)/K(b,c)}(X_{K(b,c)}, \mathcal{V}_{\lambda'}/\varpi^m)^{\mathrm{ord}} \otimes_{\mathcal{O}} \mathcal{O}(-w_0\lambda')\end{align*} in $D(\mathcal{O}/\varpi^m[K(0,c)/K(b,c)])$.

\end{prop}

\begin{proof} See \cite[Corollary 5.2.18]{10}. \end{proof}

For a maximal ideal $\mathfrak{m}$ of $\mathbb{T}^{S, \mathrm{ord}}(R\Gamma_{K(0, c)/K(b,c)}(X_{K(b,c)}, \mathcal{V}_{\lambda})^{\mathrm{ord}})$, we get a maximal ideal $\mathfrak{n}$ of $\mathbb{T}^S(R\Gamma(X_{K(b,c)}, \mathcal{V}_{\lambda}))$ by the pullback $$\mathbb{T}^S(R\Gamma_{K(0,c)/K(b,c)}(X_{K(b,c)}, \mathcal{V}_{\lambda})^{\mathrm{ord}}) \hookrightarrow \mathbb{T}^{S, \mathrm{ord}}(R\Gamma_{K(0,c)/K(b,c)}(X_{K(b,c)}, \mathcal{V}_{\lambda})^{\mathrm{ord}}),$$ the quotient map $\mathbb{T}^S(R\Gamma_{K(0,c)/K(b,c)}(X_{K(b,c)}, \mathcal{V}_{\lambda})^{\mathrm{ord}}) \twoheadrightarrow \mathbb{T}^{S}(R\Gamma(X_{K(b,c)}, \mathcal{V}_{\lambda})^{\mathrm{ord}})$ by a nilpotent ideal (see Lemma \ref{nilpotent}) and the pullback $\mathbb{T}^S(R\Gamma(X_{K(b,c)}, \mathcal{V}_{\lambda})) \twoheadrightarrow \mathbb{T}^{S}(R\Gamma(X_{K(b,c)}, \mathcal{V}_{\lambda})^{\mathrm{ord}})$. Therefore, we obtain the continuous semisimple representation $$\overline{\rho_{\mathfrak{m}}} : G_{F,S} \rightarrow  \mathrm{GL}_n(\overline{\mathbb{T}^{S, \mathrm{ord}}(R\Gamma_{K(0,c)/K(b,c)}(X_{K(b,c)}, \mathcal{V}_{\lambda})^{\mathrm{ord}})/\mathfrak{m}})$$ such that $\mathrm{det}(TI_n - \rho_{\mathfrak{m}}(\mathrm{Frob}_v)) = P_v(T)$ by Theorem \ref{torsion Galois}. We say that $\mathfrak{m}$ is non-Eisenstein if $\overline{\rho_{\mathfrak{m}}}$ is absolutely irreducible.

For $v|l$ and $i = 1, \cdots, n$, let $\chi_{v, \lambda, i} : G_{F_v} \rightarrow \mathbb{T}^{S, \mathrm{ord}}(R\Gamma_{K(0,c)/K(b,c)}(X_{K(b,c)}, \mathcal{V}_{\lambda})^{\mathrm{ord}})^{\times}$ denote the continuous character defined by $$\chi_{v, \lambda, i} \circ \mathrm{Art}_{F_v}(u) = \prod_{\tau \in \mathrm{Hom}_{\mathbb{Q}_l}(F_v, E)} [K(b,c)_v\mathrm{diag}(1, \cdots, \underset{i}{u}, 1, \cdots, 1)K(b,c)_v]\tau(u)^{-\lambda_{\tau,n-i+1} - i + 1 }$$ for any $u \in \mathcal{O}_{F_v}^{\times}$ and $\chi_{v, \lambda, i} \circ \mathrm{Art}_{F_v}(\varpi_v) = \varepsilon_l^{1-i}(\mathrm{Art}_{F_v}(\varpi_v)) \frac{U_{v,i}}{U_{v,i-1}}$. (Note that $\chi_{v,\lambda, i}$ is independent of the choice of $\varpi_v$.)

Then we have the following result.

\begin{thm}\label{ordinarirty of automorphic Galois}

    Let $\mathfrak{m}$ be a non-Eisenstein ideal of $\mathbb{T}^{S, \mathrm{ord}}(R\Gamma_{K(0,c)/K(b,c)}(X_{K(b,c)}, \mathcal{V}_{\lambda})^{\mathrm{ord}})$. We suppose the following conditions.

    1 \ $\overline{\rho_{\mathfrak{m}}}$ is decomposed generic.
    
    2 \ $l$ splits in an imaginary quadratic field contained in $F$.

    Then there exist a positive integer $\delta$ depending only on $n$ and $[F:\mathbb{Q}]$, an ideal $I$ of $\mathbb{T}^{S, \mathrm{ord}}(R\Gamma_{K(0,c)/K(b,c)}(X_K, \mathcal{V}_{\lambda})_{\mathfrak{m}}^{\mathrm{ord}})$ such that $I^{\delta}=0$ and a continuous representation $\rho_{\mathfrak{m}}: G_{F,S} \rightarrow \mathrm{GL}_n(\mathbb{T}^{S, \mathrm{ord}}(R\Gamma_{K(0,c)/K(b,c)}(X_{K(b,c)},\mathcal{V}_{\lambda})_{\mathfrak{m}}^{\mathrm{ord}})/I)$ satisfying the following conditions.

    (1) \ For any $v \notin S$, we have $\mathrm{det}(TI_n - \rho_{\mathfrak{m}}(\mathrm{Frob}_v)) = P_v(T)$.
    
    (2) \ For any $v|l$, we have $\mathrm{det}(T-\rho_{\mathfrak{m}}(g)) = \prod_{i=1}^n(T-\chi_{v, \lambda, i}(g))$ for any $g \in G_{F_v}$.
    
(3) \ For any $v|l$, we have $(\rho_{\mathfrak{m}}(g_1) -\chi_{v, \lambda, 1}(g_1))(\rho_{\mathfrak{m}}(g_2) -\chi_{v, \lambda, 2}(g_2)) \cdots (\rho_{\mathfrak{m}}(g_n) -\chi_{v, \lambda, n}(g_n)) = 0$ for any $g_1, \cdots, g_n \in G_{F_v}$.

\end{thm}

\begin{proof} See \cite[Theorem 5.5.1]{10}. By \cite[Theorem 1.3]{kos}, we can remove the assumption $[F^+:\mathbb{Q}] > 1$ in \cite[Theorem 5.5.1]{10}. \end{proof}

\begin{rem} In the above situation, we can prove more detailed properties by using the method in \cite[{\S} 3.2]{MEI} and we can use the ordinary deformation ring $R^{\triangle}_v$ in Theorem \ref{ordinary automorphy lifting}. (See \cite[{\S} 3.2]{red} for the definition of $R^{\triangle}_v$.)  \end{rem}

\begin{dfn} \label{iota ordinary definition}

Let $\iota : \overline{\mathbb{Q}}_l \stackrel{\sim}{\rightarrow} \mathbb{C}$ be an isomorphism of fields and $\pi$ be a cohomological cuspidal automorphic representation of weight $\iota\lambda$.

1 \  For $c \ge b \ge 0$ with $c \ge 1$ and $v|l$, let $\iota^{-1} (\pi_v^{\mathrm{Iw}_v(b,c)})^{\mathrm{ord}}$ be the sum of all generalized eigenspaces of $$\prod_{i=1}^{n}(\prod_{\tau \in \mathrm{Hom}_{\mathbb{Q}_l}(F_v, E)} \prod_{j=1}^i \tau(\varpi_v)^{-\lambda_{\tau, n-j+1} }) [\mathrm{Iw}_v(b,c)\mathrm{diag}(\underbrace{\varpi_v, \cdots, \varpi_v}_{i}, 1, \cdots, 1) \mathrm{Iw}_v(b,c)]$$ in $\iota^{-1} (\pi_v^{\mathrm{Iw}_v(b,c)})$ corresponding to the eigenvalues which are $l$-adic units. Note that $\Delta_v^{T_n}$ acts on $\iota^{-1} (\pi_v^{\mathrm{Iw}_v(b,c)})^{\mathrm{ord}}$ by $\Delta_v^{T_n} \rightarrow \mathcal{H}(\Delta_v^{T_n}, T_n(\mathcal{O}_{F_v})(b))_{\mathcal{O}} \rightarrow \mathcal{H}(\Delta_v, \mathrm{Iw}_v(b,c))_{\mathcal{O}}$.

2 \  We also put $\iota^{-1}(\pi^{\infty})^{K(b,c), \mathrm{ord}} := \iota^{-1}(\pi^{\infty, l})^{K(b,c)^l} \otimes (\otimes_{v \mid l} \iota^{-1} (\pi_v^{\mathrm{Iw}_v(b,c)})^{\mathrm{ord}})$. 

3 \ We say that $\pi$ is $\iota$-ordinary if there exist integers $c \ge b \ge 0$ with $c \ge 1$ such that $\iota^{-1} (\pi_v^{\mathrm{Iw}_v(b,c)})^{\mathrm{ord}} \neq 0$ for any $v|l$.

\end{dfn}

\begin{rem}For any $c' \ge b' \ge 0$, $c \ge b \ge 0$ satisfying $c' \ge c \ge 1$ and $b' \ge b$, we have $\iota^{-1} (\pi_v^{\mathrm{Iw}_v(b,c)})^{\mathrm{ord}} \subset \iota^{-1} (\pi_v^{\mathrm{Iw}_v(b',c')})^{\mathrm{ord}}$. \end{rem}

\begin{lem} \label{iota-ordinary}
    
Let $\iota : \overline{\mathbb{Q}}_l \stackrel{\sim}{\rightarrow} \mathbb{C}$ be an isomorphism of fields, $\pi$ be a cohomological cuspidal automorphic representation of $\mathrm{GL}_n(\mathbb{A}_F)$ of weight $\iota \lambda$ and $0 \le b \le c$ be integers such that $1 \le c$.

Then we have the following results.

(1) \ The dimension of $\iota^{-1} (\pi_v^{\mathrm{Iw}_v(b,c)})^{\mathrm{ord}}$ is at most 1 for any $v|l$.

(2) \ Let $c, b$ be integers satisfying $c \ge b \ge 0$ and $c \ge 1$. Then the following conditions are equivalent.  

1 \ $\iota^{-1} (\pi_v^{\mathrm{Iw}_v(b,c)})^{\mathrm{ord}} \neq 0$.

2 \ There exist characters $\chi_{v,1}, \cdots, \chi_{v,n} : F_v^{\times}/(1 + \varpi_v^b\mathcal{O}_{F_v}) \rightarrow \mathbb{C}^{\times}$ satisfying the following conditions.

$\cdot$ $\pi_v$ is a subquotient of $\mathrm{n}\textrm{-}\mathrm{Ind}^{\mathrm{GL}_n(F_v)}_{B_n(F_v)} \chi_{v,1} \otimes \cdots \chi_{v,n}$.

$\cdot$ \ $\mathrm{val}_l(\iota^{-1}\chi_{v, i}(\varpi_v)) = \displaystyle \frac{1}{e_v} \sum_{\tau \in \mathrm{Hom}_{\mathbb{Q}_l}(F_v, \overline{\mathbb{Q}}_l)} ( \lambda_{\tau, n + 1 - i} + i - 1 + \frac{1-n}{2})$ for all $i = 1, \cdots, n$.

(Here, $\mathrm{val}_l$ denotes the $l$-adic valuation normalized to $\mathrm{val}_l(l) = 1$ and $e_v$ denotes the ramification index of $F_v/\mathbb{Q}_l$.)

Moreover, if these equivalent conditions hold, we have $$\iota^{-1} (\pi_v^{\mathrm{Iw}_v(b,c)})^{\mathrm{ord}} \cong \displaystyle \iota^{-1}(\chi_{v,1}| \ |_v^{\frac{1-n}{2}} \otimes \chi_{v,2}| \ |_v^{\frac{3-n}{2}} \otimes \cdots \otimes \chi_{v,n}| \ |_v^{\frac{n-1}{2}})$$ as $\Delta_v^{T_n}$-modules.

(3) \ The following conditions are equivalent.

1 \ $\pi$ is $\iota$-ordinary.
    
2 \ For all $v|l$, the eigenvalues $\alpha_1, \cdots, \alpha_n$ of $\iota^{-1}\mathrm{rec}(\pi_v|\mathrm{det}|_v^{\frac{1-n}{2}})(\mathrm{Frob}_v)$ satisfies $\mathrm{val}_l(\alpha_i) = \displaystyle \frac{1}{e_v} \sum_{\tau \in \mathrm{Hom}_{\mathbb{Q}_l}(F_v, \overline{\mathbb{Q}}_l)} ( \lambda_{\tau, n + 1 - i} + i - 1)$ for all $i = 1, \cdots, n$ after arranging $\alpha_i$'s so that $\mathrm{val}_l(\alpha_1) \le \mathrm{val}_l(\alpha_2) \le \cdots \le \mathrm{val}_l(\alpha_n)$.

\end{lem}

\begin{proof} See \cite[proof of Lemma 5.4 and 5.7]{OG} and \cite[Lemma 2.3]{red} for (1) and (2). (3) follows from (2). \end{proof}

\begin{cor} \label{ordinary cohomology of locally symmetric space}

    Let $\iota: \overline{\mathbb{Q}}_l \stackrel{\sim}{\rightarrow} \mathbb{C}$ be an isomorphism of fields, $\lambda \in (\mathbb{Z}^{n}_{+})^{\mathrm{Hom}(F,\overline{\mathbb{Q}}_l)}$ and $c, b$ be integers satisfying $c \ge b \ge 0, c > 0$.
        
    1 \ Let $\pi$ be am $\iota$-ordinary cohomological cuspidal automorphic representation $\pi$ of weight $\iota \lambda$. Then we have the following results.
    
    $\cdot$ \ $\iota^{-1}(\pi^{\infty})^{K(b,c), \mathrm{ord}}$ is a $\mathbb{T}^{S, \mathrm{ord}}$-equivariant direct summand of $H^*(X_{K(b,c)}, \mathcal{V}_{\lambda})^{\mathrm{ord}} \otimes_{\mathcal{O}} \overline{\mathbb{Q}}_l$. \footnote{$\mathbb{T}^{S, \mathrm{ord}}$ acts on $\iota^{-1}(\pi^{\infty})^{K(b,c), \mathrm{ord}}$ by sending $X_{v,i}$ to $ \displaystyle (\prod_{\tau \in \mathrm{Hom}_{\mathbb{Q}_k}(F_v, E)} \prod_{j=1}^i \tau(\varpi_v)^{-\lambda_{\tau, n-j+1} })[\mathrm{Iw}_v(b,c)\mathrm{diag}(\underbrace{\varpi_v, \cdots, \varpi_v}_{i}, 1, \cdots, 1) \mathrm{Iw}_v(b,c)].$ Note that $X_{v, i}$'s acts on $\iota^{-1}(\pi^{\infty})^{K(b,c), \mathrm{ord}}$ as a scalar by (1) of Lemma \ref{iota-ordinary}.}
    
    $\cdot$ The Hecke eigensystem $\varphi_{\pi, \iota}^{\mathrm{ord}}: \mathbb{T}^{S, \mathrm{ord}} \rightarrow \overline{\mathbb{Q}}_l$ corresponding to $\iota^{-1}(\pi^{\infty})^{K(b,c), \mathrm{ord}}$ factors through $\mathbb{T}^{S,\mathrm{ord}} \twoheadrightarrow \mathbb{T}^{S,\mathrm{ord}}(R\Gamma(X_{K(b,c)}, \mathcal{V}_{\lambda})^{\mathrm{ord}})$.
    
    $\cdot$ \ The maximal ideal $\mathfrak{m}:=\mathrm{Ker}(\mathbb{T}^{S,\mathrm{ord}}(R\Gamma(X_{K(b,c)}, \mathcal{V}_{\lambda})^{\mathrm{ord}}) \xrightarrow{\varphi_{\iota, \pi}} \mathcal{O}_{\overline{\mathbb{Q}}_l} \twoheadrightarrow \overline{\mathbb{F}_l})$ induces $\overline{\rho_{\mathfrak{m}}} \cong \overline{r_{\iota}(\pi)}$.
    
    $\cdot$ \ If $\overline{r_{\iota}(\pi)}$ is absolutely irreducible, then $\mathfrak{m}$ is a non-Eisenstein ideal of $\mathbb{T}^{S, \mathrm{ord}}(R\Gamma(X_{K(b,c)}, \mathcal{V}_{\lambda})^{\mathrm{ord}})$.
        
    2 \ Let $\mathfrak{m}$ be a non-Eisenstein ideal of $\mathbb{T}^{S, \mathrm{ord}}(R\Gamma(X_{K(b,c)}, \mathcal{V}_{\lambda})^{\mathrm{ord}})$ and $\mathcal{T}_{\mathfrak{m}}^{\mathrm{ord}}$ be the set of all $\iota$-ordinary cohomological cuspidal automorphic representations of $\mathrm{GL}_n(\mathbb{A}_F)$ of weight $\iota\lambda$ such that $\iota^{-1}(\pi^{\infty})^{K(b,c), \mathrm{ord}} \neq 0$ and $\overline{r_{\iota}(\pi)} \cong \sigma \circ \overline{\rho_{\mathfrak{m}}}$ for some isomorphism $\sigma : \overline{\mathbb{T}^{S, \mathrm{ord}}(R\Gamma(X_{K(b,c)}, \mathcal{V}_{\lambda})^{\mathrm{ord}})/\mathfrak{m}} \stackrel{\sim}{\rightarrow} \overline{\mathbb{F}_l}$ over $\mathbb{F}$. 
        
    (1) \ There exists an isomorphism of $\mathbb{T}^{S, \mathrm{ord}}$-modules \begin{gather*}H^*(X_{K(b,c)}, \mathcal{V}_{\lambda})^{\mathrm{ord}}_{\mathfrak{m}} \otimes_{\mathcal{O}} \overline{\mathbb{Q}}_l \\
    \cong \oplus_{\pi \in \mathcal{T}^{\mathrm{ord}}_{\mathfrak{m}}} (\iota^{-1} (\pi^{\infty})^{K(b,c), \mathrm{ord}} \otimes_{\overline{\mathbb{Q}_{l}}} (H^*(\mathfrak{g}, \prod_{v|\infty}U(n);\pi_{\infty} \otimes_{\mathbb{C}} V_{\iota \lambda}) \otimes_{\mathbb{C}, \iota^{-1}} \overline{\mathbb{Q}}_l)).\end{gather*}
        
    (2) \ Let $l_0:=[F^+:\mathbb{Q}]n-1$ and $q_0:=[F^+:\mathbb{Q}]\frac{n(n-1)}{2}$.
        
    If $\mathcal{T}^{\mathrm{ord}}_{\mathfrak{m}}$ is not empty, then $H^i(X_{K(b,c)}, \mathcal{V}_{\lambda})^{\mathrm{ord}}_{\mathfrak{m}}[\frac{1}{l}] \neq 0$ for any $i \in [q_0, q_0 + l_0]$ and $H^i(X_{K(b,c)}, \mathcal{V}_{\lambda})^{\mathrm{ord}}_{\mathfrak{m}}[\frac{1}{l}] = 0$ for any $i \notin [q_0, q_0 + l_0]$.
            
    (3) \ The $E$-algebra $\mathbb{T}^{S, \mathrm{ord}}(H^*(X_{K(b,c)}, \mathcal{V}_{\lambda})^{\mathrm{ord}}_{\mathfrak{m}}[\frac{1}{l}]) = \mathbb{T}^{S, \mathrm{ord}}(R\Gamma(X_{K(b,c)}, \mathcal{V}_{\lambda})^{\mathrm{ord}}_{\mathfrak{m}})[\frac{1}{l}]$ is finite $\acute{e}$tale over $E$. Moreover, by sending $\pi$ to $\varphi^{\mathrm{ord}}_{\pi, \iota}$, we obtain a bijection between $\mathcal{T}_{\mathfrak{m}}^{\mathrm{ord}}$ and $\{ f: \mathbb{T}^{S, \mathrm{ord}}(R\Gamma(X_{K(b,c)}, \mathcal{V}_{\lambda})_{\mathfrak{m}}^{\mathrm{ord}}) \rightarrow \overline{\mathbb{Q}}_l \mid \mathcal{O} \textrm{-} \mathrm{morphism} \}$. (Note that $\varphi_{\pi, \iota}^{\mathrm{ord}} \circ \rho_{\mathfrak{m}} \cong r_{\iota}(\pi)$, where $\rho_{\mathfrak{m}}$ denote the Galois representation of Theorem \ref{ordinarirty of automorphic Galois}.)
        
    \end{cor}
        
    \begin{proof} This follows from Theorem \ref{cohomology of locally symmetric space}. Note that the action of $\Delta$ on $\mathcal{V}_{\lambda}$ is twisted by $\alpha_{\lambda}^{-1}$. \end{proof}

\begin{thm} \label{ordinary semisimple}

    Let $l$ be a prime, $\iota: \overline{\mathbb{Q}}_l \rightarrow \mathbb{C}$ and $\pi$ be an $\iota$-ordinary cohomological cuspidal automorphic representation of $\mathrm{GL}_n(\mathbb{A}_F)$ of weight $\iota\lambda$.
    
    We assume that $\overline{r_{\iota}(\pi)}$ is absolutely irreducible and decomposed generic.
    
    1 \ For all $v|l$, $r_{\iota}(\pi)|_{G_{F_v}}$ is ordinary of weight $\lambda_v$ and $\iota\mathrm{WD}(r_{\iota}(\pi)|_{G_{F_v}})^{ss} \cong \mathrm{rec}_{F_v}(\pi_v| \mathrm{det} |_v^{\frac{1-n}{2}})^{ss}$.
    
    2 \ Let $v$ be an $l$-adic place of $F$. If $\pi_v$ is unramified, then $r_{\iota}(\pi)|_{G_{F_v}}$ is crystalline and $\iota\mathrm{WD}(r_{\iota}(\pi)|_{G_{F_v}}) \cong \mathrm{rec}_{F_v}(\pi_v| \mathrm{det} |_v^{\frac{1-n}{2}})$.
         
    \end{thm}

    \begin{rem} Note that we need not assume that $F$ contains an imaginary quadratic field. \end{rem}

    \begin{proof} By the same proof as Theorem \ref{Caraiani-Newton 2}, we may assume that $F$ satisfies the conditions of Theorem \ref{ordinarirty of automorphic Galois} and $\iota^{-1}(\pi^{\infty})^{K(b,c), \mathrm{ord}} \neq 0$ for some $0 \le b \le c$ such that $1 \le c$. 

    Let $\mathfrak{m}$ be the non-Eisenstein ideal of $\mathbb{T}^{S, \mathrm{ord}}(R\Gamma(X_{K(b,c)}, \mathcal{V}_{\lambda})^{\mathrm{ord}})$ corresponding to $\overline{r_{\iota}(\pi)}$ and $\varphi_{\pi, \iota}^{\mathrm{ord}} : \mathbb{T}^{S, \mathrm{ord}}(R\Gamma(X_{K(b,c)}, \mathcal{V}_{\lambda})_{\mathfrak{m}}^{\mathrm{ord}}) \rightarrow \overline{\mathbb{Q}}_l$ be the Hecke eigensystem corresponding to $\iota^{-1}(\pi^{\infty})^{K(b,c), \mathrm{ord}}$. Then the Galois representation $\rho_{\mathfrak{m}}$ in Theorem \ref{ordinarirty of automorphic Galois} satisfies $\varphi^{\mathrm{ord}}_{\iota, \pi} \circ \rho_{\mathfrak{m}} \cong r_{\iota}(\pi)$. Therefore, we obtain the following properties. (We fix an $l$-adic place $v$ of $F$ and we put $\psi_{v,i} := \varphi_{\pi, \iota}^{\mathrm{ord}}(\chi_{v, \lambda, i})$.)

(1) \ $\mathrm{det}(T-r_{\iota}(\pi)(g)) = \prod_{i=1}^n(T-\psi_{v, i}(g))$ for any $g \in G_{F_v}$.
    
(2) \ $(r_{\iota}(\pi)(g_1) - \psi_{v, 1}(g_1))(r_{\iota}(\pi)(g_2) - \psi_{v, 2}(g_2)) \cdots (r_{\iota}(\pi)(g_n) -\psi_{v, n}(g_n)) = 0$ for any $g_1, \cdots, g_n \in G_{F_v}$.

We define $G_{F_v}$-stable subspaces $0 = V_0 \subset V_1 \subset \cdots \subset V_n \subset \overline{\mathbb{Q}}_l^{\oplus n}$ of $\overline{\mathbb{Q}}_l^{\oplus n}$ inductively by the condition that $V_i$ is the maximal $G_{F_v}$-stable subspace of $\overline{\mathbb{Q}}_l^{\oplus n}$ satisfying that $G_{F_v}$ acts on $V_i/V_{i-1}$ by the character $\psi_{v, i}$. We fix $v \in \overline{\mathbb{Q}}_l^{\oplus n}$. By the property (2), for any $g_2, \cdots, g_n \in G_{F_v}$, we have $(r_{\iota}(\pi)(g_2) - \psi_{v, 2}(g_2)) \cdots (r_{\iota}(\pi)(g_n) - \psi_{v, n}(g_n))v \in V_1$. Thus, we obtain $(r_{\iota}(\pi)(g_3) - \psi_{v, 3}(g_3)) \cdots (r_{\iota}(\pi)(g_n) -\psi_{v, n}(g_n)) v \in V_2$ for any $g_3, \cdots, g_n \in G_{F_v}$. By using the same argument repeatedly, we obtain $v \in V_n$. This implies $\overline{\mathbb{Q}}_l^{\oplus n} = V_n$. By the property (1) and the fact that $\psi_{v,i}$'s are distinct, we obtain $\mathrm{dim}_{\overline{\mathbb{Q}}_l} V_i/V_{i-1} = 1$ and consequently $V_i/V_{i-1} \cong \psi_{v, i}$.  Thus $r_{\iota}(\pi)|_{G_{F_v}}$ is ordinary of weight $\lambda_v$.

By Lemma \ref{iota-ordinary}, there exist characters $\chi_{v,1}, \cdots, \chi_{v,n} : F_v^{\times}/(1 + \varpi_v^b\mathcal{O}_{F_v}) \rightarrow \mathbb{C}^{\times}$ satisfying the following conditions.

$\cdot$ $\pi_v$ is a subquotient of $\mathrm{n}\textrm{-}\mathrm{Ind}^{\mathrm{GL}_n(F_v)}_{B_n(F_v)} \chi_{v,1} \otimes \cdots \chi_{v,n}$.

$\cdot$ $\iota^{-1} (\pi_v^{\mathrm{Iw}_v(b,c)})^{\mathrm{ord}} \cong \displaystyle \iota^{-1}(\chi_{v,1}| \ |_v^{\frac{1-n}{2}} \otimes \chi_{v,2}| \ |_v^{\frac{3-n}{2}} \otimes \cdots \otimes \chi_{v,n}| \ |_v^{\frac{n-1}{2}})$ as $\Delta_v^{T_n}$-modules.

Then $\mathrm{WD}(\psi_{v,i})(\mathrm{Art}_{F_v}(u)) = \psi_{v,i}(\mathrm{Art}_{F_v}(u)) \prod_{\tau \in \mathrm{Hom}_{\mathbb{Q}_l}(F_v, E)}\tau(u)^{\lambda_{\tau,n-i+1} + i - 1 } = \iota^{-1}\chi_{v,i}(u) $ for any $u \in \mathcal{O}_{F_v}^{\times}$ and \begin{align*}\mathrm{WD}(\psi_{v,i})(\mathrm{Art}_{F_v}(\varpi_v)) &= \psi_{v,i}(\mathrm{Art}_{F_v}(\varpi_v)) \prod_{\tau \in \mathrm{Hom}_{\mathbb{Q}_l}(F_v, E)}\tau(\varpi_v)^{\lambda_{\tau,n-i+1} + i - 1 }\\
     &= \varepsilon_l^{1-i}(\mathrm{Art}_{F_v}(\varpi_v))\iota^{-1}(\chi_{v,i}(\varpi_v)q_v^{\frac{1+n-2i}{2}})\prod_{\tau \in \mathrm{Hom}_{\mathbb{Q}_k}(F_v, E)} \tau(\varpi_v)^{ i-1 } \\
     &= \iota^{-1}(\chi_{v,i}(\varpi_v)q_{v}^{\frac{n-1}{2}}).\end{align*}
        
(Note that since $\varepsilon_l(\mathrm{Art}_{\mathbb{Q}_l}(x)) x^{-1} = |x|_l$ for any $x \in \mathbb{Q}_l^{\times}$, we have $$\varepsilon_l(\mathrm{Art}_{F_v}(x)) N_{F_v/\mathbb{Q}_l}(x)^{-1} = |x|_v$$ for any $x \in F_v^{\times}$.)
    
Therefore, we obtain $\iota \mathrm{WD}(r_{\iota}(\pi)|_{G_{F_v}})^{ss} \cong \mathrm{rec}_{F_v}(\pi_v|\mathrm{det}|_v^{\frac{1-n}{2}})^{ss}$.

    We assume that $\pi_v$ is unramified. By Theorem \ref{generic unitary} and Proposition \ref{purity}, all eigenvalues $\alpha, \beta$ of $\mathrm{rec}_{F_v}(\pi_v|\mathrm{det}|_v^{\frac{1-n}{2}})(\mathrm{Frob}_v)$ satisfy $|\frac{\alpha}{\beta}| < q_v$. Therefore, the monodromy operator of $\iota \mathrm{WD}(r_{\iota}(\pi)|_{G_{F_v}})$ is zero by $\iota\mathrm{WD}(r_{\iota}(\pi)|_{G_{F_v}})^{ss} \cong \mathrm{rec}_{F_v}(\pi_v|\mathrm{det}|_v^{\frac{1-n}{2}})^{ss}$. \end{proof}

Finally, we recall the results of \cite{ade}. We assume $l > n$.

Let $\mathfrak{m}$ be a non-Eisenstein ideal of $\mathbb{T}^S(R\Gamma(X_K, \mathcal{V}_{\lambda}))$ such that $\mathbb{T}^S(R\Gamma(X_K, \mathcal{V}_{\lambda}))/\mathfrak{m} = \mathbb{F}$ and $(Q, \{ \alpha_v \}_{v \in Q})$ be a Taylor-Wiles datum for $(\overline{\rho_{\mathfrak{m}}}, S)$. 

For $v \in Q$, let $d_v$ denote the multiplicity of the root $\alpha_v$ in $\mathrm{det}(TI_n - \overline{\rho_{\mathfrak{m}}}(\mathrm{Frob}_v))$, $\mathfrak{p}_{v,0}$ denote the parahoric subgroup of $\mathrm{GL}_n(F_v)$ corresponding to the decomposition $(d_v, n-d_v)$ and $\mathfrak{p}_{v,1}:=\mathrm{Ker}(\mathfrak{p}_{v,0} \twoheadrightarrow \mathrm{GL}_{d_v}(\mathcal{O}_{F_v}) \times \mathrm{GL}_{n-d_v}(\mathcal{O}_{F_v}) \rightarrow \mathrm{GL}_{d_v}(\mathcal{O}_{F_v}) \xrightarrow{\mathrm{det}} \mathcal{O}_{F_v}^{\times} \rightarrow \mathbb{F}_v^{\times} \twoheadrightarrow \mathbb{F}^{\times}_v(l))$.

Let $K_0(Q)$ and $K_1(Q)$ be good subgroups defined by the following conditions.

For $v \notin Q$, $K_0(Q)_v = K_1(Q)_v = K_v$.

For $v \in Q$, $K_0(Q)_v:=\mathfrak{p}_{v,0}$ and $K_1(Q)_v:=\mathfrak{p}_{v,1}$.

We put $\Delta_v:=K_0(Q)_v/K_1(Q)_v \cong \mathbb{F}_v^{\times}(l)$, $\Xi_{v,1}:= \Delta_v \times \mathbb{Z}^n$ for any $v \in Q$ and $\mathbb{T}^{S \cup Q \cup Q^c}_{Q} := \mathbb{T}^{S \cup Q \cup Q^c} \otimes_{\mathcal{O}} ( \otimes_{v \in Q} \mathcal{O}[\Xi_{v,1}]^{\mathfrak{S}_{d_v} \times \mathfrak{S}_{n-d_v}})$. Then we have a $\mathcal{O}$-morphism $\mathcal{O}[\Xi_{v,1}]^{\mathfrak{S}_{d_v} \times \mathfrak{S}_{n-d_v}} \twoheadrightarrow \mathcal{O}[\mathbb{Z}^n]^{\mathfrak{S}_{d_v} \times \mathfrak{S}_{n-d_v}} \twoheadrightarrow \mathbb{F}$ corresponding to $(T-\alpha_v)^{d_v}$ and $\mathrm{det}(T-\overline{\rho_{\mathfrak{m}}}(\mathrm{Frob}_v))/(T-\alpha_v)^{d_v}$. Let $\mathfrak{m}_v$ denote the kernel of this map. We also have certain actions of $\mathbb{T}^{S \cup Q \cup Q^c}_{Q}$ on $R\Gamma_{K_0(Q)/K_1(Q)}(X_{K_1(Q)}, \mathcal{V}_{\lambda})$ and $R\Gamma(X_{K_0(Q)}, \mathcal{V}_{\lambda})$. (See \cite[Proposition 3.10]{ade} for the definition of the action of $\mathcal{O}[\Xi_{v,1}]^{\mathfrak{S}_{d_v} \times \mathfrak{S}_{n-d_v}}$.)

Let $\mathfrak{m}_{Q}$ denotes the maximal ideal of $\mathbb{T}_Q^{S \cup Q \cup Q^c}$ generated by the pullback of $\mathfrak{m}$ and $(\mathfrak{m}_v)_{v \in Q}$. We put $\Delta_{Q}:=K_0(Q)/K_1(Q) \cong \prod_{v \in Q}\mathbb{F}_v^{\times}(l)$. We put $\mathbb{T}^{S \cup Q \cup Q^c}_{\Delta_Q} := \mathbb{T}^{S \cup Q \cup Q^c} \otimes_{\mathcal{O}} \mathcal{O}[\Delta_Q] \subset \mathbb{T}^{S \cup Q \cup Q^c}_{Q}$. Then we have a natural morphism $\mathbb{T}^{S \cup Q \cup Q^c}_{\Delta_Q} \rightarrow \mathbb{T}^{S}$.

\begin{thm} \label{Patching argument} In the above situation, we have the following results. (We put $\mathcal{S}:=(\overline{\rho_{\mathfrak{m}}}, S, \{ \mathcal{O} \}_{v \in S}, \{ \mathcal{D}_v^{\Box} \}_{v \in S})$.)

1 \ We have canonical $\mathbb{T}^{S \cup Q \cup Q^c}_{\Delta_Q}$-equivalent isomorphisms $$R\Gamma(\Delta_{Q}, R\Gamma_{K_0(Q)/K_1(Q)}(X_{K_1(Q)}, \mathcal{V}_{\lambda})_{\mathfrak{m}_{Q}}) \cong R\Gamma(X_{K_0(Q)}, \mathcal{V}_{\lambda})_{\mathfrak{m}_{Q}} \cong R\Gamma(X_K, \mathcal{V}_{\lambda})_{\mathfrak{m}}.$$

2 \ There exist a positive integer $\delta$, an ideal $I$ of $$\mathbb{T}^{S \cup Q \cup Q^c}_{\Delta_{Q}}(R\Gamma_{K_0(Q)/K_1(Q)}(X_{K_1(Q)}, \mathcal{V}_{\lambda})_{\mathfrak{m}_Q})$$ and a continuous representation $$\rho_{\mathfrak{m}_{Q}} : G_{F \cup Q} \rightarrow \mathrm{GL}_n(\mathbb{T}^{S \cup Q \cup Q^c}_{\Delta_Q}(R\Gamma_{K_0(Q)/K_1(Q)}(X_{K_1(Q)}, \mathcal{V}_{\lambda})_{\mathfrak{m}_Q})/I)$$ satisfying the following conditions.

$\cdot$ \ $\delta$ depends only on $n$ and $[F:\mathbb{Q}]$.

$\cdot$ \ $I^{\delta} = 0$.

$\cdot$ For all $v \notin S \cup Q \cup Q^c$, $\mathrm{det}(TI_n - \rho_{\mathfrak{m}_Q}(\mathrm{Frob}_v)) = P_v(T)$.

$\cdot$ $\rho_{\mathfrak{m}_{Q}}$ is type $\mathcal{S}(Q_N)$ and the $\mathcal{O}$-morphism $$R_{\rho_{\mathfrak{m}_Q}, \mathcal{S}(Q)} \rightarrow \mathbb{T}^{S \cup Q \cup Q^c}_{\Delta_Q}(R\Gamma_{K_0(Q)/K_1(Q)}(X_{K_1(Q)}, \mathcal{V}_{\lambda})_{\mathfrak{m}_Q})/I$$ corresponding to $\rho_{\mathfrak{m}_{Q}}$ is an $\mathcal{O}[\Delta_Q]$-morphism.

\end{thm}

\begin{proof} See \cite[Theorem 7.6 and Theorem 7.7]{ade}. \end{proof}

\subsubsection{Some results on extensions of number fields}

\begin{prop}\label{strong multiplicity one}

Let $F$ be a CM field, $n$ be a positive integer, $\pi, \pi'$ be cohomological cuspidal automorphic representations of $\mathrm{GL}_n(\mathbb{A}_F)$, $l$ be a prime and $\iota : \overline{\mathbb{Q}}_l \stackrel{\sim}{\rightarrow} \mathbb{C}$ be an isomorphism of fields.

If $r_{\iota}(\pi) \cong r_{\iota}(\pi')$, then $\pi \cong \pi'$.

\end{prop}

\begin{proof} If $r_{\iota}(\pi) \cong r_{\iota}(\pi')$, then $\mathrm{rec}_{F_v}(\pi_v|\mathrm{det}|_v^{\frac{1-n}{2}}) \cong \mathrm{rec}_{F_v}(\pi'_v|\mathrm{det}|_v^{\frac{1-n}{2}})$ and consequently $\pi_v \cong \pi_v'$ for almost all $v$. By strongly multiplicity one theorem, we have $\pi \cong \pi'$.  \end{proof}

\begin{prop}\label{base change}
    
Let $E/F$ be a solvable extension of CM fields and $n$ be a positive integer.

1 \ Let $\pi$ be a cohomological cuspidal automorphic representation of $\mathrm{GL}_n(\mathbb{A}_F)$ such that $r_{\iota}(\pi)|_{G_{E}}$ is irreducible.

Then $\mathrm{BC}_{E/F}(\pi)$ is a cohomological cuspidal automorphic representation of $\mathrm{GL}_n(\mathbb{A}_E)$ such that $r_{\iota}(\mathrm{BC}_{E/F}(\pi)) \cong r_{\iota}(\pi)|_{G_{E}}$.

2 \ Let $\rho : G_{F} \rightarrow \mathrm{GL}_n(\overline{\mathbb{Q}}_l)$ be a continuous representation and $\Pi$ be a cohomological cuspidal automorphic representation of $\mathrm{GL}_n(\mathbb{A}_E)$ such that $r_{\iota}(\Pi) \cong \rho|_{G_E}$. We assume that $\rho|_{G_E}$ is irreducible. Then there exists a cohomological cuspidal automorphic representation $\pi$ such that $\mathrm{BC}_{E/F}(\pi) = \Pi$ and $\rho \cong r_{\iota}(\pi)$.

\end{prop}

\begin{proof} See \cite[Propostion 6.5.13]{10}. \end{proof}

\begin{prop}\label{ordinary base change}

Let $E/F$ be a solvable extension of CM fields, $\iota : \overline{\mathbb{Q}}_l \stackrel{\sim}{\rightarrow} \mathbb{C}$ be an isomorphism of fields and $\pi$ be a cohomological cuspidal automorphic representation of $\mathrm{GL}_n(\mathbb{A}_F)$ such that $\mathrm{BC}_{E/F}(\pi)$ is cuspidal.

Then $\pi$ is $\iota$-ordinary if and only if $\mathrm{BC}_{E/F}(\pi)$ is $\iota$-ordinary.

\end{prop}

\begin{proof} This follows by (3) of Lemma \ref{iota-ordinary}. \end{proof}

\begin{prop}\label{extension of field}

Let $K$ be a number field, $K^{(a)}$ be a finite extension of $K$, $S$ be a finite set of places of $K$ ( note that $S$ may contain infinite places) and for any $v \in S$, $L_v/K_v$ be a finite Galois extension.

Then there exists a finite solvable extension $M/K$ linearly disjoint from $K^{(a)}$ over $K$ such that $M_u \cong L_v$ over $K_v$ for any $u|v \in S$. 

\end{prop}

\begin{proof} See \cite[Lemma A 2.1]{CW}. \end{proof}

\begin{cor}\label{extension of Q}

    Let $K$ be a number field, $K^{(a)}$ be a finite extension of $\mathbb{Q}$, $S, T$ be finite sets of primes such that $S \cap T$ is empty and for any $p \in S$ and any $v|p$, $L_v/K_v$ be a finite extension.
    
    Then there exists a finite totally real solvable extension $M/\mathbb{Q}$ linearly disjoint from $K^{(a)}$ over $\mathbb{Q}$ such that there exists an injection $L_v \hookrightarrow (KM)_u$ over $K_v$ for any $u|v|p \in S$ and any $p \in T$ splits completely in $M$.
    
    \end{cor}
    
    \begin{proof} We may assume that $K^{(a)}$ is Galois over $\mathbb{Q}$.

        After extending $T$, we may assume that for any Galois extension $K'$ of $\mathbb{Q}$ contained in $K^{(a)}$, there exists a prime $p \in T$ such that $p$ doesn't split in $K'$.
        
        For any $p \in S$, we fix an embedding $L_v \hookrightarrow \overline{\mathbb{Q}_p}$ over $\mathbb{Q}_p$ for any $v|p$ and $L_p$ denotes the Galois closure of $L_v$'s over $\mathbb{Q}_p$. By Proposition \ref{extension of field}, there exists a finite totally real solvable extension $M/\mathbb{Q}$ linearly disjoint from $K^{(a)}$ over $\mathbb{Q}$ such that for any $p \in S$ and $w|p$, we have $M_w \cong L_p$ over $\mathbb{Q}_p$ and for any $p \in T$ and $w|p$, we have $M_w \cong \mathbb{Q}_p$ over $\mathbb{Q}_p$. Since any prime $p \in T$ splits completely in $M$, $M$ is linearly disjoint from $K^{(a)}$ over $\mathbb{Q}$. Fix a place $u$ of $KM$ and let $v$, $w$ and $p$ be the place of $K$, $M$ and $\mathbb{Q}$ respectively. Then we obtain a $\mathbb{Q}_p$-embedding $f : L_v \hookrightarrow M_w$. Since $M_w$ is Galois over $\mathbb{Q}_p$, the image of the natural embedding $K_v \hookrightarrow (MK)_u$ is contained in $M_w$. Therefore, there exists $\sigma \in \mathrm{Gal}(M_w/\mathbb{Q}_p)$ such that $\sigma \circ f : L_v \hookrightarrow M_w \rightarrow (MK)_u$ is a $K_v$-morphism. This implies the result.  \end{proof}

\begin{lem} \label{imaginary quadratic}

Let $K$ be a number field, $K^{(a)}$ be a finite extension of $\mathbb{Q}$ and $S$ be a finite set of odd primes.

Then there exist odd primes $p_1, p_2, p_3$ such that $2$ splits in $\mathbb{Q}(\sqrt{-p_1})$, any prime in $S$ splits in $\mathbb{Q}(\sqrt{-p_i})$ for any $i$, $\mathbb{Q}(\sqrt{-p_1}, \sqrt{-p_2}, \sqrt{-p_3})$ is linearly disjoint from $K^{(a)}$ over $\mathbb{Q}$ and any prime ramified in $K(\sqrt{-p_1}, \sqrt{-p_2}, \sqrt{-p_3})$ splits in one of the fields $\mathbb{Q}(\sqrt{-p_1})$, $\mathbb{Q}(\sqrt{-p_2})$ and $\mathbb{Q}(\sqrt{-p_3})$.

\end{lem}

\begin{proof}

We may assume that $K^{(a)}$ is Galois over $\mathbb{Q}$. Let $T_1$ be a finite set of odd primes containing $S$ and all primes ramified in $K$. After extending $T_1$, we may assume that for any Galois extension $K'$ of $\mathbb{Q}$ contained in $K^{(a)}$, there exists an odd prime $p \in T_1$ such that $p$ doesn't split in $K'$. We can take an odd prime $p_1 \notin T_1$ such that $p_1 \equiv -1 \mod 8$ and $\chi_q(-p_1) = 1$ for any odd primes $q \in T_1$. (Here, $\chi_q : (\mathbb{Z}/q\mathbb{Z})^{\times} \rightarrow \mathbb{C}^{\times}$ denotes the unique quadratic character.) Then any prime in $T_1 \cup \{ 2 \}$ splits in $\mathbb{Q}(\sqrt{-p_1})$. We take an odd prime $p_2 \notin T_1 \cup \{ 2, p_1 \}$ such that $p_2 \equiv 1 \mod 4$ and $\chi_q(-p_2) = 1$ for any odd primes $q \in T_1 \cup \{ p_1 \}$. Then any prime in $T_1 \cup \{ p_1 \}$ splits in $\mathbb{Q}(\sqrt{-p_2})$. We also take an odd prime $p_3 \notin T_1 \cup \{ 2, p_1, p_2 \}$ such that $p_3 \equiv 1 \mod 4$ and $\chi_q(-p_3) = 1$ for any odd primes $q \in T_1 \cup \{ p_1, p_2 \}$. Then any prime in $T_1 \cup \{p_1, p_2\}$ splits in $\mathbb{Q}(\sqrt{-p_3})$. Moreover, $p_3$ splits in $\mathbb{Q}(\sqrt{-p_2})$. Moreover, $\mathbb{Q}(\sqrt{-p_1}, \sqrt{-p_2}, \sqrt{-p_3})$ is linearly disjoint from $K^{(a)}$ over $\mathbb{Q}$ since any prime in $T_1$ splits completely in $\mathbb{Q}(\sqrt{-p_1}, \sqrt{-p_2}, \sqrt{-p_3})$. \end{proof}

\begin{cor} \label{imaginary quadratic 2}

Let $K^{(a)}$ be a finite extension of $\mathbb{Q}$ and $S$ be a finite set of primes. Then there exist odd primes $p$ such that all primes in $S$ splits in $\mathbb{Q}(\sqrt{-p})$ and $\mathbb{Q}(\sqrt{-q})$ and moreover $K^{(a)}$ is linearly disjoint over $\mathbb{Q}$.
    
\end{cor}

\begin{proof} This is contained in Lemma \ref{imaginary quadratic}. \end{proof}

\begin{lem}\label{quadratic extension}

Let $K$ be a number field, $S_1$ be a finite set of places of $K$ and $S_2$ be a finite set of finite places of $K$ such that $S_1$ doesn't contain any $2$-adic places and $S_1 \cap S_2$ is empty.

Then there exists a quadratic extension $M/K$ such that all places in $S_1$ splits in $M$ and all places in $S_2$ ramifies in $M$. 

\end{lem}

\begin{proof} We may assume that $S_2$ is nonempty. By approximation theorem, there exists $x \in \mathcal{O}_K$ such that $|x - 1|_v < 1$ for any $v \in S_1$ and $|x - \varpi_v|_{v} < q_v^{-1}$ for any $v \in S_2$. Then $M:=K(\sqrt{x})$ satisfies the conditions of the lemma. \end{proof}

\begin{prop}\label{cyclic extension}

Let $K$ be a number field, $K^{(a)}$ be a finite extension of $K$, $S$ be a finite set of places of $K$ (note that $S$ may contain infinite places) and $N$ be a positive integer.

Then there exists a cyclic extension $M/K$ of degree $N$ linearly disjoint from $K^{(a)}$ over $K$ such that all $v \in S$ split completely in $M$.

\end{prop}

\begin{proof} See \cite[Lemma A.2.2]{CW}. \end{proof}

\subsection{Automorphy lifting and local-global compatibility}

We fix a CM field $F$ and a positive integer $n$ in this subsection.

Let us recall some useful lemmas.

\begin{prop}(Auslander-Buchsbaum formula) \label{Auslander-Buchsbaum formula}

    Let $S$ be a regular local ring and $M$ be a finite $S$-module. Then we have $\mathrm{projdim}_{S}M + \mathrm{depth}_{S}M = \mathrm{dim} \, S.$
        
    \end{prop}
    
    \begin{proof} See \cite[Theorem 19.1 and 19.2]{Mat}. \end{proof}
    
    \begin{prop} \label{Calagari-OG}
        
        Let $S$ be a regular local ring, $l$ be a nonnegative integer, $q$ be an integer and $C$ be a perfect complex of $S$-modules such that $H^i(C \otimes_{S}^{\mathbb{L}} S/\mathfrak{m}_S) = 0$ for any $i \notin [q, q+l]$ and $H^i(C \otimes_S^{\mathbb{L}} S/\mathfrak{m}_S) \neq 0$ for some $i \in [q, q+l]$.
    
        Then $\mathrm{dim} \, \mathrm{Supp}_{S}H^*(C) \ge \mathrm{dim} \, S - l$. Moreover, if this is equality, then we have the following properties.
    
        1 \ $H^{i}(C) \neq 0$ if and only if $i = l + q$.
        
        2 \ $\mathrm{depth}_{S}H^{q+l}(C) = \mathrm{dim} \, S - l$ and $\mathrm{projdim}_SH^{q+l}(C) = l$. 
    
    \end{prop}
    
    \begin{proof}
    
    See \cite[Lemma 2.9]{Leo}. \end{proof}
    
    \begin{prop}\label{depth}
    
    Let $S, R$ be Noetherian local rings, $f : S \rightarrow R$ be a local finite morphism and $M$ be a finite $R$-module.
    
    Then we have $\mathrm{depth}_{R}M = \mathrm{depth}_{S}M$.
    
    \end{prop}
    
    \begin{proof} See \cite[Lemma 2.5.7]{RG}. \end{proof}
    
    \begin{lem} \label{Tor spectral sequence}
    
    Let $S$ be a Noetherian ring, $R$ be a Noetherian $S$-algebra and $C$ be a perfect complex of $S$-modules with an $S$-morphism $R \rightarrow \mathrm{End}_{D(S)}(C)$. Let $\mathfrak{p} \in \mathrm{Supp}_{R}H^*(C)$, $\mathfrak{q}$ be the pull-back of $\mathfrak{p}$ in $S$ and $r:=\mathrm{max}\{ r' \mid H^{r'}(C)_{\mathfrak{p}} \neq 0 \}$. Then we have $H^r(C)_{\mathfrak{p}}/\mathfrak{q} \cong H^r(C \otimes_{S}^{\mathbb{L}} S/\mathfrak{q})_{\mathfrak{p}}$. In particular, we have $\mathfrak{p} \in \mathrm{Supp}_{R}H^*(C \otimes^{\mathbb{L}}_{S} S/\mathfrak{q})$.
        
    \end{lem}
    
    \begin{proof}
    
    We have an $R$-equivariant spectral sequence $E_2^{p,q}:=\mathrm{Tor}_{-p}^{S}(H^q(C)_{\mathfrak{p}}, S/\mathfrak{q}) \Rightarrow H^{p+q}(C \otimes_{S}^{\mathbb{L}} S/\mathfrak{q})_{\mathfrak{p}}$. Note that $E_2^{p, q} = 0$ for $p > 0$ or $q > r$. Thus we obtain $E_2^{0, r} = H^r(C)_{\mathfrak{p}}/\mathfrak{q} \cong H^r(C \otimes_{S}^{\mathbb{L}} S/\mathfrak{q})_{\mathfrak{p}}$. \end{proof}

\subsubsection{Automorphy lifting and local-global compatibility in crystalline cases}

\begin{thm}\label{automorphy lifting theorem in crystalline cases}

Let $l$ be an odd prime such that $l > n$, $\iota: \overline{\mathbb{Q}}_l \stackrel{\sim}{\rightarrow} \mathbb{C}$ be an isomorphism of fields, $\lambda \in (\mathbb{Z}_{+}^n)^{\mathrm{Hom}(F, \overline{\mathbb{Q}}_l)}$ and $r: G_F \rightarrow \mathrm{GL}_n(\overline{\mathbb{Q}}_l)$ be an algebraic $l$-adic representation.

We suppose the following conditions.

1 \ $\overline{r}$ is absolutely irreducible and decomposed generic. 

2 \ $\overline{r}(G_{F(\zeta_l)})$ is adequate. 

3 \ $\zeta_l \notin F$.

4 \ For all $v|l$, $r|_{G_{F_v}}$ is crystalline of $l$-adic Hodge type $\mathbf{v}_{\lambda_v}$.

5 \ There exists a cohomological cuspidal automorphic representation $\pi$ of $\mathrm{GL}_n(\mathbb{A}_F)$ of weight $\iota\lambda$ satisfying the following conditions.

(1) \ $\overline{r} \cong \overline{r_{\iota}(\pi)}$.

(2) \ For all $v|l$, $\pi_v$ is unramified and $r_{\iota}(\pi)|_{G_{F_v}} \sim r|_{G_{F_v}}$. (Note that we use the convention \ref{convension}.)

(3) \ For all $v \nmid l$, $\mathrm{WD}(r_{\iota}(\pi)|_{G_{F_v}})$ and $\mathrm{WD}(r|_{G_{F_v}})$ have the same monodromy type. 

(4) \ For all $v \nmid l$, we have $\iota \mathrm{WD}(r_{\iota}(\pi)|_{G_{F_{v}}})^{F-ss} \cong \mathrm{rec}_{F_{v}}(\pi_{v}|\mathrm{det}|_{v}^{\frac{1-n}{2}})$.

Then there exists a cohomological cuspidal automorphic representation $\Pi$ of $\mathrm{GL}_n(\mathbb{A}_F)$ of weight $\iota \lambda$ such that $r \cong r_{\iota}(\Pi)$, $\iota\mathrm{WD}(r_{\iota}(\Pi)|_{G_{F_{v}}})^{F-ss} \cong \mathrm{rec}_{F_{v}}(\Pi_{v}|\mathrm{det}|_{v}^{\frac{1-n}{2}})$ for all $v \nmid l$ and $\Pi_v$ is unramified for all $v|l$. 

\end{thm}

\begin{proof} We put $S_l:= \{ v|l \}$.

We can take a finite extension $E$ of $\mathbb{Q}_l$ contained in $\overline{\mathbb{Q}}_l$ such that $\tau (F) \subset E$ for all $\tau \in \mathrm{Hom}(F, \overline{\mathbb{Q}}_l)$ and the residue field $\mathbb{F}$ of $E$ contains all eigenvalues of $\overline{r}(g)$ for all $g \in G_F$. Moreover, we may assume that $\mathrm{Im}(r)$ and $\mathrm{Im}(r_{\iota}(\pi))$ are contained in $\mathrm{GL}_n(\mathcal{O})$ and $\overline{r} = \overline{r_{\iota}(\pi)}$. (Here, $\mathcal{O}$ denotes the ring of integers of $E$.) 

Fix a finite place $u \nmid l$ of $F$. By Proposition \ref{strong multiplicity one}, it suffices to show the existence of $\Pi$ and the local-global compatibility for $\Pi$ at $u$.

By Proposition \ref{base change}, Corollary \ref{extension of Q}, Lemma \ref{imaginary quadratic}, Proposition \ref{potential automorphy and local-global compatibility} and Proposition \ref{monodromy type irreducible component}, we may assume the following conditions. (Note that there exist positive Dirichlet density primes which are decomposed generic for $\overline{r}$ by Lemma \ref{strongly decomposed generic} later.)

$(a)$ \ We have a finite set $S$ of finite places of $F$ containing $S_l$ and $u$ such that for all $v \notin S$, $r$ and $\pi$ are unramified at $v$ and $S = S^c$.

$(b)$ \ There exist $s, t \in S \setminus (S_l\cup\{ u \})$ such that $\mathrm{char} \, \mathbb{F}_s \neq \mathrm{char} \, \mathbb{F}_t$.

$(c)$ \ For all $v \in S$, $r_{\iota}(\pi)|_{G_{F_v}} \sim r|_{G_{F_v}}$.

$(d)$ \ For all $v \in S \setminus S_{l}$, $\pi_v^{\mathrm{Iw}_v} \neq 0$ and $r|_{G_{F_v}}$ is unipotently ramified.

$(e)$ \ Any prime $p$ lying below $S$ or ramified in $F$, splits in an imaginary quadratic field contained in $F$. (In particular, $F$ is an imaginary CM field.)

$(f)$ \ For any $l$-adic place $\overline{v}$ of $F^+$, there exists an $l$-adic place $\overline{v}' \neq \overline{v}$ of $F^+$ such that $\sum_{\overline{v}, \overline{v}' \neq \overline{v}''|l} [F^+_{\overline{v}''}:\mathbb{Q}_l] \ge \frac{1}{2}[F^+:\mathbb{Q}]$. 

We put $\mathcal{S}:= ( \overline{r}, S, \{ \mathcal{O} \}_{v \in S}, \{ \mathcal{D}_v \}_{v \in S}) :=( \overline{r}, S, \{ \mathcal{O} \}_{v \in S}, \{ \mathcal{D}_v^{\Box} \}_{v \in S \setminus S_l} \cup \{ \mathcal{D}_v^{\mathrm{cris}, \lambda_v} \}_{v \in S_l} )$. Then $r$ and $r_{\iota}(\pi)$ are liftings of type $\mathcal{S}$ to $\mathcal{O}$. After extending $E$, we may assume for all $v \in S$, all irreducible components of $\mathrm{Spec} \, R_{\overline{r}|_{G_{F_v}}, \mathcal{D}_v}[\frac{1}{l}]$ are geometrically irreducible.

We define an open compact subgroup $K = \prod_{v} K_v$ of $\mathrm{GL}_n(\mathbb{A}_F^{\infty})$ as follows.

$\cdot$ If $v \notin S$ or $v|l$, $K_v:=\mathrm{GL}_n(\mathcal{O}_{F_v})$.

$\cdot$ If $v \in S \setminus (S_l \cup \{u\})$, $K_v:=\mathrm{Iw}_{v,1}$.

$\cdot$ $K_{u}$ is the parahoric subgroup of $\mathrm{GL}_n(F_{u})$ corresponding to the dual partition of the monodromy type of $\mathrm{WD}(r|_{G_{F_u}})$.

Then $K$ is a good subgroup of $\mathrm{GL}_n(\mathbb{A}_F^{\infty})$ by the assumption $(b)$. By Proposition \ref{automorphy lifting and local-global compatibility}, the assumptions (2), (3) and (4) of $5$, $(a)$ and $(d)$, we have $(\pi^{\infty})^{K} \neq 0$.

 Let $\mathfrak{m}$ be the non-Eisenstein ideal of $\mathbb{T}^S(R\Gamma(X_K, \mathcal{V}_{\lambda}))$ corresponding to $\overline{r_{\iota}(\pi)}$. (See Theorem \ref{cohomology of locally symmetric space}.) We may assume $\overline{\rho_{\mathfrak{m}}} = \overline{r_{\iota}(\pi)}$. 

By Theorem \ref{Caraiani-Newton}, there exist a positive integer $\delta$ depending only on $n$ and $[F:\mathbb{Q}]$, an ideal $I_0$ of $\mathbb{T}^S(R\Gamma(X_K, \mathcal{V}_{\lambda})_{\mathfrak{m}})$ such that $I_0^{\delta}=0$ and a lifting $\rho_{\mathfrak{m}}: G_{F, S} \rightarrow \mathrm{GL}_n(\mathbb{T}^S(R\Gamma(X_K, \mathcal{V}_{\lambda})_{\mathfrak{m}})/I_0)$ of $\overline{\rho_{\mathfrak{m}}}$ of type $\mathcal{S}$ such that for all $v \notin S$, we have det$(T-\rho_{\mathfrak{m}}(\mathrm{Frob}_v)) = P_v(T)$. Thus, we obtain the surjective morphism of $\mathcal{O}$-algebras $R_{\mathcal{D}} \twoheadrightarrow \mathbb{T}^S(R\Gamma(X_K, \mathcal{V}_{\lambda})_{\mathfrak{m}})/I_0$ corresponding to $\rho_{\mathrm{\mathfrak{m}}}$.

Let $q := $$\mathrm{dim}_{\mathbb{F}}H^1(G_{F,S}, \mathrm{ad} \, \overline{\rho_{\mathfrak{m}}}(1))$. By Proposition \ref{chebotarev}, for any positive integer $N$, there exists a Taylor-Wiles datum ($Q_N, \{ \alpha_v \}_{v\in Q_N})$ satisfying the following conditions.

$(a')$ \ $|Q_N| = q$.

$(b')$ \ For any $v \in Q_N$, we have $q_v \equiv 1$ \ mod $l^N$.

$(c')$ \ $\mathrm{dim}_{\mathbb{F}} \mathfrak{m}_{R_{\mathcal{S}(Q_N)}^S}/(\mathfrak{m}_{R_{\mathcal{S}(Q_N)}^S}^2, \mathfrak{m}_{R_{\mathcal{D}_v}})_{v \in S} = q-n^2[F^+:\mathbb{Q}]=:g$.

Then by Theorem \ref{Patching argument} and Theorem \ref{Caraiani-Newton} again, after changing $\delta$ (but independent of $N$), there exist an ideal $I_N$ of $$\mathbb{T}_{\Delta_{Q_N}}^{S \cup Q_N \cup Q_N^c}(R\Gamma_{K_0(Q_N)/K_1(Q_N)}(X_{K_1(Q_N)}, \mathcal{V}_{\lambda})_{\mathfrak{m}_{Q_N}})$$ such that $I_N^{\delta}=0$ and a lifting $$\rho_{\mathfrak{m}, N}: G_{F, S \cup Q_N} \rightarrow \mathrm{GL}_n(\mathbb{T}_{\Delta_{Q_N}}^{S \cup Q_N \cup Q_N^c}(R\Gamma_{K_0(Q_N)/K_1(Q_N)}(X_{K_1(Q_N)}, \mathcal{V}_{\lambda})_{\mathfrak{m}_{Q_N}})/I_N)$$ of $\overline{\rho_{\mathfrak{m}}}$ of type $\mathcal{S}(Q_N)$ such that for all $v \notin S \cup Q_N \cup Q_N^c$, det$(T-\rho_{\mathfrak{m}, N}(\mathrm{Frob}_v)) = P_v(T)$ and the corresponding morphism of $\mathcal{O}$-algebras $$R_{\mathcal{S}(Q_N)} \twoheadrightarrow \mathbb{T}_{Q_N}^{S \cup Q_N \cup Q_N^c}(R\Gamma_{K_0(Q_N)/K_1(Q_N)}(X_{K_1(Q_N)}, \mathcal{V}_{\lambda})_{\mathfrak{m}_{Q_N}})/I_N$$ is a surjective morphism of $\mathcal{O}[\Delta_{Q_N}]$-algebras. Note that $$R\Gamma_{K_0(Q_N)/K_1(Q_N)}(X_{K_1(Q_N)}, \mathcal{V}_{\lambda})_{\mathfrak{m}_{Q_N}}$$ is a perfect complex of $\mathcal{O}[\Delta_{Q_N}]$-modules by Proposition \ref{perfect complex}.

Let $d:=[F^+:\mathbb{Q}]n^2-1$, $\Delta_0$ be the trivial group, $R_0:=R_{\mathcal{S}}$, $T_0:= \mathbb{T}^S(R\Gamma(X_K, \mathcal{V}_{\lambda})_{\mathfrak{m}})$ and $C_0 := R\mathrm{Hom}_{\mathcal{O}}(R\Gamma(X_K, \mathcal{V}_{\lambda})_{\mathfrak{m}}, \mathcal{O})[-d]$. For any $N$, let $\Delta_N:=\Delta_{Q_N}$, $R_N:=R_{\mathcal{S}(Q_N)}$, $T_N:=\mathbb{T}_{\Delta_N}^{S \cup Q_N \cup Q_N^c}(R\Gamma_{K_0(Q_N)/K_1(Q_N)}(X_{K_1(Q_N)}, \mathcal{V}_{\lambda})_{\mathfrak{m}_{Q_N}})$ and $$C_N := R\mathrm{Hom}_{\mathcal{O}[\Delta_N]}(R\Gamma_{K_0(Q_N)/K_1(Q_N)} (X_{K_1(Q_N)}, \mathcal{V}_{\lambda})_{\mathfrak{m}_{Q_N}}, \mathcal{O}[\Delta_N])[-d].$$ Let $\mathcal{T}:=\mathcal{O}[[X_{v,i,j} \mid v \in S, i, j = 1, \cdots, n]]/(X_{v_0,1,1})$, $\Delta_{\infty}:=\mathbb{Z}_l^{\oplus q}$, $S_{\infty} := \mathcal{T}[[\Delta_{\infty}]]$ and $R^{\mathrm{loc}}:=\widehat{\otimes}_{v \in S} R_{\overline{r}|_{G_{F_v}}, \mathcal{D}_v}$. ($v_0$ is some fixed place in $S$.) We fix a surjection $\Delta_{\infty} \twoheadrightarrow \Delta_{N}$ for each $N$. Then the kernel of this map is contained in $(l^N\mathbb{Z}_l)^{\oplus q}$ by $(b')$. We regard $S_{\infty}$ as an augmented algebra over $\mathcal{O}$ and $\mathfrak{a}$ denotes the augmentation ideal of $S_{\infty}$

Moreover, we have a canonical isomorphism $C_N \otimes^{\mathbb{L}}_{\mathcal{O}[\Delta_N]} \mathcal{O} \cong C_0$ in $D(\mathcal{O})$ and this induces a $\mathcal{O}$-morphism $T_N \otimes_{\mathcal{O}[\Delta_N]} \mathcal{O} \rightarrow T_0$ by Theorem \ref{Patching argument}. We also have a canonical isomorphism of $\mathcal{O}$-algebras $R_N \otimes_{\mathcal{O}[\Delta_N]} \mathcal{O} \cong R_0$ fitting in the following commutative diagram by Lemma \ref{trace density}

\[\xymatrix{
 R_{N} \ar@{->>}[d] \ar@{->>}[r] & T_{N}/I_{N} \ar@{->>}[d]  \\
 R_0 \ar@{->>}[r] & T_0/(I_0+I_{N}). \\
}\] 

We put $\mathcal{S}(Q_0):=\mathcal{S}$.

We take a representative $\rho_{\mathcal{S}}: G_F \rightarrow \mathrm{GL}_n(R_{0})$ of the universal deformation of $\overline{\rho_{\mathfrak{m}}}$ of type $\mathcal{S}=\mathcal{S}(Q_0)$. Moreover, for each $N$, we take a representative $\rho_{\mathcal{S}(Q_N)}: G_F \rightarrow \mathrm{GL}_n(R_{N})$ of the universal deformation of $\overline{\rho_{\mathfrak{m}}}$ of type $\mathcal{S}(Q_N)$ whose push-forward to $R_0$ is equal to $\rho_{\mathcal{S}}$. This induces an isomorphism $R^S_{\mathcal{S}(Q_N)} \cong R_N\widehat{\otimes}_{\mathcal{O}}\mathcal{T}$ by Lemma \ref{framed deformation} and $R^{\mathrm{loc}}$-algebra structures on $R_N \widehat{\otimes}_{\mathcal{O}} \mathcal{T}$ and $R_N$ such that the morphism $R_N \widehat{\otimes} \mathcal{T} \twoheadrightarrow R_0 \widehat{\otimes} \mathcal{T}$ becomes an $R^{\mathrm{loc}}$-morphism. Moreover, for each $N$, we can take a surjection $R_{\infty}:=R^{\mathrm{loc}}[[Y_1, \cdots, Y_g]] \twoheadrightarrow R_N\widehat{\otimes}_{\mathcal{O}}\mathcal{T}$ of $R^{\mathrm{loc}}$-algebras such that the morphism $R_N \widehat{\otimes} \mathcal{T} \twoheadrightarrow R_0 \widehat{\otimes} \mathcal{T}$ becomes an $R_{\infty}$-morphism by the condition $(c')$.

Note that the following argument is essentially contained in the proof of \cite[Lemma 2.6.4]{RG}. By \cite[Proposition 6.4.12 and 6.4.16]{10} (or more easily \cite[{\S} 2]{RG}), we obtain the following.

(i) \ A bounded complex $C_{\infty}$ of finite free $S_{\infty}$-modules, a commutative $S_{\infty}$-subalgebra $T_{\infty}$ of $\mathrm{End}_{D(S_{\infty})}(C_{\infty})$ and an ideal $I_{\infty}$ of $T_{\infty}$ satisfying $I_{\infty}^{\delta} = 0$. 

(ii) \ An $S_{\infty}$-algebra structure on $R_{\infty}$ such that $R_{\infty} \twoheadrightarrow R_0$ factors through $R_{\infty} \twoheadrightarrow R_{\infty}/\mathfrak{a}$ and an $S_{\infty}$-algebra surjection $R_{\infty} \twoheadrightarrow T_{\infty}/I_{\infty}$.

(iii) \ An isomorphism $C_{\infty}\otimes^{\mathbb{L}}_{S_{\infty}} \mathcal{O} \cong C_0$ in $D(\mathcal{O})$ inducing an $\mathcal{O}$-morphism $T_{\infty}/\mathfrak{a} \rightarrow T_0$ fitting in the following commutative diagram

\[\xymatrix{
 R_{\infty} \ar@{->>}[d] \ar@{->>}[r] & T_{\infty}/I_{\infty} \ar@{->>}[d]  \\
 R_0 \ar@{->>}[r] & T_0/(I_0+I_{\infty}). \\
}\]

We put $l_0:=[F^+:\mathbb{Q}]n-1$ and $q_0:=[F^+:\mathbb{Q}]\frac{n(n-1)}{2}$.

Then we have $\mathrm{dim} \, R_{\infty}[\frac{1}{l}] = g + |S|n^2 + [F:\mathbb{Q}]\frac{n(n-1)}{2} = q - n^2[F^+:\mathbb{Q}] + |S|n^2 + [F:\mathbb{Q}]\frac{n(n-1)}{2} = |S|n^2 + q - [F^+:\mathbb{Q}]n = $$\mathrm{dim} \, S_{\infty, \mathfrak{a}} - l_0$ by Proposition \ref{dimension of unristricted}, Proposition \ref{psd} and Lemma \ref{completed tensor product irreducibility}. Note that since $T_{\infty, \mathfrak{a}}$ is finite over $S_{\infty, \mathfrak{a}}$ and Ann$_{T_{\infty, \mathfrak{a}}} (H^*(C_{\infty, \mathfrak{a}}))$ is a nilpotent ideal of $T_{\infty, \mathfrak{a}}$ by Lemma \ref{nilpotent}, we have $\mathrm{dim} \, \mathrm{Supp}_{S_{\infty, \mathfrak{a}}} H^*(C_{\infty, \mathfrak{a}}) = \mathrm{dim} \, T_{\infty, \mathfrak{a}}$. Moreover, since $I_{\infty}$ is a nilpotent ideal of $T_{\infty, \mathfrak{a}}$ and $\mathrm{Spec} \, T_{\infty, \mathfrak{a}}/I_{\infty}$ is a closed subscheme of $\mathrm{Spec} \, R_{\infty, \mathfrak{a}}$, we obtain dimSupp$_{S_{\infty, \mathfrak{a}}}H^*(C_{\infty, \mathfrak{a}}) \le \mathrm{dim} \, R_{\infty, \mathfrak{a}} \le \mathrm{dim} \, R_{\infty}[\frac{1}{l}] = $$\mathrm{dim} \, S_{\infty, \mathfrak{a}} - l_0$. By (2) of 2 of Theorem \ref{cohomology of locally symmetric space} and by the fact that $\iota^{-1}(\pi^{\infty})^K$ contributes to $H^{*}(X_K, \mathcal{V}_{\lambda})_{\mathfrak{m}} \otimes_{\mathcal{O}} \overline{\mathbb{Q}}_l$, we obtain that $H^i(C_{\infty, \mathfrak{a}} \otimes_{S_{\infty}, \mathfrak{a}}^{\mathbb{L}} S_{\infty, \mathfrak{a}}/\mathfrak{a}) \cong H^i(C_0 \otimes^{\mathbb{L}}_{\mathcal{O}} E) = $Hom$_E(H^{d-i}(X_K, \mathcal{V}_{\lambda})_{\mathfrak{m}}[\frac{1}{l}], E)$ is zero for $i \notin [q_0, q_0 + l_0]$ and nonzero for $i \in [q_0, q_0 + l_0]$. This implies $\mathrm{dim} \, _{S_{\infty, \mathfrak{a}}}$Supp$_{S_{\infty, \mathfrak{a}}}H^*(C_{\infty, \mathfrak{a}}) = \mathrm{dim} \, S_{\infty, \mathfrak{a}} - l_0$, $H^{i}(C_{\infty, \mathfrak{a}})$ is zero for any $i \neq q_0+l_0$, $M := H^{q_0+l_0}(C_{\infty, \mathfrak{a}}) \neq 0$ and depth$_{S_{\infty, \mathfrak{a}}}M = $$\mathrm{dim} \, S_{\infty, \mathfrak{a}} - l_0$ by Proposition \ref{Calagari-OG}.

We write $\mathfrak{p} \in $$\mathrm{Spec} \, T_0$ for the point corresponding to $\pi$. We regard $\mathfrak{p}$ as a point of $\mathrm{Spec} \, R_{\infty}$ by $\mathrm{Spec} \, T_0 \hookrightarrow \mathrm{Spec} \, R_0 \hookrightarrow \mathrm{Spec} \, R_{\infty}$. Then $\mathfrak{p}$ is a regular point of $\mathrm{Spec} \, R_{\infty}$ by the assumption (4), Proposition \ref{purity}, Proposition \ref{regular point}, Proposition \ref{psd} and 3 of Lemma \ref{completed tensor product irreducibility}. (Note that $R_{\infty} \twoheadrightarrow R_0$ is an $R^{\mathrm{loc}}$-morphism.) By 2 of Lemma \ref{completed tensor product regularity}, $\widehat{R_{\infty, \mathfrak{p}}}$ is isomorphic to a formal power series ring over $E$. In particular, we have a lift $\widehat{R_{\infty, \mathfrak{p}}}\rightarrow \widehat{T_{\infty, \mathfrak{p}}}$ of $\widehat{R_{\infty, \mathfrak{p}}} \twoheadrightarrow \widehat{T_{\infty, \mathfrak{p}}}/I_{\infty}$ over $E$. This is a finite map since $I_{\infty}$ is nilpotent. We put $\widehat{M} := M \otimes_{T_{\infty, \mathfrak{a}}} \widehat{T_{\infty, \mathfrak{p}}} (\neq 0)$. Then we obtain $\mathrm{depth}_{\widehat{R_{\infty, \mathfrak{p}}}}\widehat{M} = \mathrm{depth}_{\widehat{T_{\infty, \mathfrak{p}}}}\widehat{M} = \mathrm{depth}_{T_{\infty, \mathfrak{p}}} M_{\mathfrak{p}} \ge \mathrm{depth}_{S_{\infty, \mathfrak{a}}} M = \mathrm{dim} \, S_{\infty, \mathfrak{a}} - l_0 = \mathrm{dim} \, R_{\infty}[\frac{1}{l}] \ge \mathrm{dim} \, \widehat{R_{\infty, \mathfrak{p}}}$. In fact, the first equality follows from Proposition \ref{depth}. The third inequalities follow by the fact that an $M$-regular sequence in $\mathfrak{m}_{S_{\infty, \mathfrak{a}}}$ is also an $M_{\mathfrak{p}}$-regular sequence in $\mathfrak{p}T_{\infty, \mathfrak{p}}$. The second equality is well-known (or $\ge$ is proved by the same way as the third inequality). 

By Proposition \ref{Auslander-Buchsbaum formula}, we obtain that $\widehat{M}$ is a non-zero finite free $\widehat{R_{\infty, \mathfrak{p}}}$-module and consequently $\widehat{R_{\infty, \mathfrak{p}}} \hookrightarrow \widehat{T_{\infty, \mathfrak{p}}}$ is injective. Since $\widehat{R_{\infty, \mathfrak{p}}}$ is regular, $I_{\infty} \cap \widehat{R_{\infty, \mathfrak{p}}}=0$. Therefore, we obtain isomorphisms $\widehat{R_{\infty, \mathfrak{p}}} \stackrel{\sim}{\rightarrow} \widehat{T_{\infty, \mathfrak{p}}}/I_{\infty}$ and $R_{\infty, \mathfrak{p}} \stackrel{\sim}{\rightarrow} T_{\infty, \mathfrak{p}}/I_{\infty}$. In particular, the irreducible component $\mathcal{C}$ of $\mathrm{Spec} \, R_{\infty, \mathfrak{a}}$ containing $\mathfrak{p}$ is contained in $\mathrm{Im}(\mathrm{Spec} \, T_{\infty, \mathfrak{a}}/I_{\infty} \rightarrow \mathrm{Spec} \, R_{\infty, \mathfrak{a}}) = \mathrm{Supp}_{R_{\infty, \mathfrak{a}}}M/I_{\infty}$. By the assumption (c) and Lemma \ref{completed tensor product irreducibility}, the point $\mathfrak{q}$ of $\mathrm{Spec} \, R_{\infty, \mathfrak{a}}$ corresponding to $r$ is contained in $\mathcal{C}$ and consequently is contained in $\mathrm{Im}(\mathrm{Spec} \, T_{\infty, \mathfrak{a}}/I_{\infty} \rightarrow \mathrm{Spec} \, R_{\infty, \mathfrak{a}}) = \mathrm{Supp}_{R_{\infty, \mathfrak{a}}}M/I_{\infty}$. By Lemma \ref{Tor spectral sequence} and $C_{\infty, \mathfrak{a}} \otimes_{S_{\infty, \mathfrak{a}}}^{\mathbb{L}} S_{\infty, \mathfrak{a}}/\mathfrak{a} \cong C_{0}[\frac{1}{l}]$, we obtain $\mathfrak{q} \in \mathrm{Supp}_{R_0[\frac{1}{l}]}H^*(C_{0})[\frac{1}{l}]/I_0 =  \mathrm{Im}(\mathrm{Spec} \, T_0[\frac{1}{l}]/I_0 \rightarrow \mathrm{Spec} \, R_0[\frac{1}{l}])$.

By the bijection in $(3)$ of 2 of Theorem \ref{cohomology of locally symmetric space}, there exists a cohomological cuspidal automorphic representation $\Pi$ of $\mathrm{GL}_n(\mathbb{A}_F)$ of weight $\iota \lambda$ such that $r_{\iota}(\Pi) \cong r$ and $(\Pi^{\infty})^K \neq 0$. In particular, $\Pi_v$ is unramified for all $v|l$ and $\Pi_{u}^{K_{u}} \neq 0$. This implies $\iota \mathrm{WD}(r_{\iota}(\Pi)|_{G_{F_{u}}})^{F-ss} \cong $rec$_{F_{u}}(\Pi_{u}|\mathrm{det}|^{\frac{1-n}{2}}_{u})$ by Proposition \ref{automorphy lifting and local-global compatibility}.\end{proof}

By the above proof, we obtain the following theorem.

\begin{thm}\label{selmer group}

Let $l$ be an odd prime such that $l > n$, $\iota: \overline{\mathbb{Q}}_l \stackrel{\sim}{\rightarrow} \mathbb{C}$ be an isomorphism of fields and $\pi$ be a cohomological cuspidal automorphic representation of $\mathrm{GL}_n(\mathbb{A}_F)$.

We suppose the following conditions.
    
1 \ $\overline{r_{\iota}(\pi)}$ is absolutely irreducible and decomposed generic. 
    
2 \ $\overline{r_{\iota}(\pi)}(G_{F(\zeta_l)})$ is adequate. 

3 \ $\zeta_l \notin F$.

4 \ For any $v \mid l$, there exists a finite extension $F_v'/F_v$ such that $\mathrm{BC}_{F_v'/F_v}(\pi_v)$ is unramified. 

5 \ For all $v \nmid l$, we have $\iota \mathrm{WD}(r_{\iota}(\pi)|_{G_{F_{v}}})^{F-ss} \cong \mathrm{rec}_{F_{v}}(\pi_{v}|\mathrm{det}|_{v}^{\frac{1-n}{2}})$.

We take a finite extension $E$ of $\mathbb{Q}_l$ contained in $\overline{\mathbb{Q}}_l$ such that $\mathrm{Im}(r_{\iota}(\pi)) \subset \mathrm{GL}_n(E)$ and we regard $r_{\iota}(\pi)$ as a representation $G_F \rightarrow \mathrm{GL}_n(E)$.

Then we obtain $H^1_f(F, \mathrm{ad} \, r_{\iota}(\pi)) = 0$.
    
\end{thm}

\begin{proof} 

By Lemma \ref{vanishing of selmer group}, we may assume that the conditions $(a) \sim (f)$ hold in the proof of Theorem \ref{automorphy lifting theorem in crystalline cases}. (We put $r:=r_{\iota}(\pi)$.) We will use notations in the proof of Theorem \ref{automorphy lifting theorem in crystalline cases}. Note that $T_{0, \mathfrak{p}} \cong E$ by 2 (3) of Theorem \ref{cohomology of locally symmetric space}. By Lemma \ref{Tor spectral sequence}, we have $H^{l_0 + q_0}(C_0 \otimes_{\mathcal{O}}E)_{\mathfrak{p}} = M_{\mathfrak{p}}/\mathfrak{a} = \widehat{M}/\mathfrak{a}$. This is a nonzero finite free $\widehat{R_{\infty, \mathfrak{p}}}/\mathfrak{a}$-module and the action of $\widehat{R_{\infty, \mathfrak{p}}}/\mathfrak{a}$ on $H^{l_0 + q_0}(C_0 \otimes_{\mathcal{O}}E)_{\mathfrak{p}}$ factors through $(T_0/(I_0 + I_{\infty}))_{\mathfrak{p}}$ as well as $\widehat{R_{0, \mathfrak{p}}}$ since $I_0 + I_{\infty}$ is zero in $T_{0, \mathfrak{p}} \cong E$. Thus, we obtain isomorphisms $\widehat{R_{\infty, \mathfrak{p}}}/\mathfrak{a} \stackrel{\sim}{\rightarrow} \widehat{R_{0, \mathfrak{p}}} = R_{0, \mathfrak{p}} \stackrel{\sim}{\rightarrow} T_{0, \mathfrak{p}} \cong E$. This implies $H^1_f(F, \mathrm{ad} \, r_{\iota}(\pi)) = 0$ by Lemma \ref{definition of selmer group}. \end{proof}

\subsubsection{Automorphy lifting and local-global compatibility in ordinary cases}

In this subsection, we prove the automorphy lifting of some ordinary Galois representations and the local-global compatibility for the cuspidal automorphic representations corresponding to them as in the crystalline cases. In ordinary cases, we already have the automorphy lifting theorem \cite[Theorem 6.1.2]{10}. However, we can remove an assumption for the neatness of the level structure by the methods of \cite[section 8]{small} and \cite[{\S}5.4]{MEI}. First, we recall some foundational results.

\begin{dfn}

Let $R$ be a Noetherian local ring and $C$ be a complex of $R$-modules.
    
We say that $C$ is minimal if $C$ is a bounded complex of finite free $R$-modules and all differentials on $C \otimes_{R} R/\mathfrak{m}_R$ are zero.
    
\end{dfn}
    
\begin{lem} \label{minimal complex}
    
Let $R$ be a Noetherian local ring.
    
1 \ For any perfect complex $C$ of $R$-modules, there exist a minimal complex $F$ of $R$-modules and a quasi-isomorphism $f : F \rightarrow C$.

2 \ Let $C$ be a minimal complex of $R$-modules and $R \rightarrow S$ be a local morphism of Noetherian local algebras. Then $C \otimes_R S$ is minimal.
    
3 \ For minimal complexes $C$ and $D$ of $R$-modules, if $f : C \rightarrow D$ is a quasi-isomorphism, then $f$ is an isomorphism.

4 \ Let $\{ I_n \}_n$ be a decreasing sequence of open ideals of $R$ such that $R \cong \varprojlim_n R/I_n$ and for each positive integer $n$, $C_n$ and $D_n$ be minimal complexes of $R/I_n$-modules such that $C_{n+1}/I_{n} \cong C_n$ and $D_{n+1}/I_n \cong D_n$. Then $C_{\infty} := \varprojlim_{n} C_n$ and $D_{\infty} := \varprojlim_{n} D_n$ are minimal complexes of $R$-modules and the canonical morphism $\mathrm{Hom}_{D(R)}(C_{\infty}, D_{\infty}) \stackrel{\sim}{\rightarrow} \varprojlim_{n} \mathrm{Hom}_{D(R/I_n)}(C_n, D_n)$ is an isomorphism.

\end{lem}
    
\begin{proof} 1 is contained in \cite[Lemma 2.3]{Leo}.

2 follows from the definition. 

3 \ $f$ induces a quasi-isomorphism $C/\mathfrak{m}_R \rightarrow D/\mathfrak{m}_R$. Since $H^i(C/\mathfrak{m}_R) = C^i/\mathfrak{m}_R$ and $D^i/\mathfrak{m}_R = H^i(D/\mathfrak{m}_R)$ for any $i$, $f$ induces an isomorphism $C^i/\mathfrak{m}_R \rightarrow D^i/\mathfrak{m}_R$ for any $i$. By Nakayama's Lemma, $\mathrm{Coker}(f) = 0$. Therefore, $\mathrm{Ker}(f)$ is a direct summand of $C$. By Nakayama's Lemma again, we obtain $\mathrm{Ker}(f)=0$.

4 \ See \cite[Lemma 2.13]{Leo}. \end{proof}

\begin{thm}\label{ordinary automorphy lifting}

    Let $l$ be an odd prime such that $l > n$, $\iota: \overline{\mathbb{Q}}_l \stackrel{\sim}{\rightarrow} \mathbb{C}$ be an isomorphism of fields, $r: G_F \rightarrow \mathrm{GL}_n(\overline{\mathbb{Q}}_l)$ be an algebraic $l$-adic representation and $U$ be a set of finite places of $F$ not containing $l$-adic places.

    We suppose the following conditions.
    
    1 \ $\overline{r}$ is absolutely irreducible and decomposed generic. 
    
    2 \ $\overline{r}(G_{F(\zeta_l)})$ is adequate. 
    
    3 \ $\zeta_l \notin F$.
    
    4 \ For all $v|l$, $r|_{G_{F_v}}$ is ordinary.
    
    5 \ There exists an $\iota$-ordinary cohomological cuspidal automorphic representation $\pi$ of $\mathrm{GL}_n(\mathbb{A}_F)$ satisfying the following conditions.
    
    (1) \ $\overline{r} \cong \overline{r_{\iota}(\pi)}$.
    
    (2) \ For all $v \in U$, $\mathrm{WD}(r_{\iota}(\pi)|_{G_{F_v}})$ and $\mathrm{WD}(r|_{G_{F_v}})$ have the same monodromy type. 
    
    (3) \ For all $v \in U$, we have $\iota \mathrm{WD}(r_{\iota}(\pi)|_{G_{F_{v}}})^{F-ss} \cong \mathrm{rec}_{F_{v}}(\pi_{v}|\mathrm{det}|_{v}^{\frac{1-n}{2}})$.
    
    Then there exists an $\iota$-ordinary cohomological cuspidal automorphic representation $\Pi$ of $\mathrm{GL}_n(\mathbb{A}_F)$ such that $r \cong r_{\iota}(\Pi)$ and $\iota\mathrm{WD}(r_{\iota}(\Pi)|_{G_{F_{v}}})^{F-ss} \cong \mathrm{rec}_{F_{v}}(\Pi_{v}|\mathrm{det}|_{v}^{\frac{1-n}{2}})$ for all $v \in U$. 

\end{thm}

\begin{proof} We put $S_l:= \{ v|l \}$. By Proposition \ref{strong multiplicity one}, we may assume that $U$ is finite. (In fact, we may assume that $U$ is empty or consists of one element.)

Let $\lambda$ (resp. $\mu$) denote the element of $(\mathbb{Z}_+^n)^{\mathrm{Hom}(F, \overline{\mathbb{Q}}_l)}$ (resp. $(\mathbb{Z}_+^n)^{\mathrm{Hom}(F, \overline{\mathbb{Q}}_l)}$) such that $\lambda_v$ (resp. $\mu_v$) is the weight of $r|_{G_{F_v}}$ (resp, $r_{\iota}(\pi)|_{G_{F_v}}$) for all $v \mid l$. For any $v|l$, we take a $G_{F_v}$-stable increasing filtration $\{ \mathrm{Fil}_i \}_{i=1}^{n}$ on $r|_{G_{F_v}}$ (resp. $\{ \mathrm{Fil}'_i \}_{i=1}^{n}$ on $r_{\iota}(\pi)|_{G_{F_v}}$) satisfying the condition (1) of Lemma \ref{ordinary Galois representation}. We put $\phi_{\lambda, v, i} := \mathrm{Fil}_{i}/\mathrm{Fil}_{i-1}$ (resp. $\psi_{\mu, v, i} := \mathrm{Fil}'_i/\mathrm{Fil}'_{i-1}$).

We can take a finite extension $E$ of $\mathbb{Q}_l$ contained in $\overline{\mathbb{Q}}_l$ such that $\tau (F) \subset E$ for all $\tau \in \mathrm{Hom}(F, \overline{\mathbb{Q}}_l)$, $\zeta_l \in E$ and the residue field $\mathbb{F}$ of $E$ contains all eigenvalues of $\overline{r}(g)$ for all $g \in G_F$. Moreover, we may assume that $\mathrm{Im}(r)$ and $\mathrm{Im}(r_{\iota}(\pi))$ are contained in $\mathrm{GL}_n(\mathcal{O})$ and $\overline{r} = \overline{r_{\iota}(\pi)}$. (Here, $\mathcal{O}$ denotes the ring of integers of $E$.) Let $\varpi$ be a prime of $\mathcal{O}$.

By Propositions \ref{base change}, \ref{ordinary base change},  Corollary \ref{extension of Q}, Lemma \ref{imaginary quadratic}, Proposition \ref{potential automorphy and local-global compatibility} and Lemma \ref{monodromy type irreducible component}, we may assume the following conditions.

$(a)$ \ We have a finite set $T$ of finite places of $F$ containing $S_l \cup U$ such that for all $v \notin T$, $r$ and $\pi$ are unramified at $v$.

$(b)$ \ For all $v \in U$, $r_{\iota}(\pi)|_{G_{F_v}} \sim r|_{G_{F_v}}$.

$(c)$ \ For all $v \in S_l$, we have that $\pi_v^{\mathrm{Iw}_v} \neq 0$, $\overline{r}|_{G_{F_v}}$ is trivial, $r|_{G_{F_v}}$ is semistable, $\psi_{\mu, v, i} \circ \mathrm{Art}_{F_v}(u) = \phi_{\lambda, v, i} \circ \mathrm{Art}_{F_v}(u) =1$ for all $l$-torsion elements $u \in \mathcal{O}_{F_v}^{\times}$ and $[F_v:\mathbb{Q}_l] > \frac{n(n+1)}{2} + 1$. 

$(d)$ \ For any $v \in T \setminus S_{l}$, we have that $\pi_v^{\mathrm{Iw}_v} \neq 0$, $r|_{G_{F_v}}$ is unipotently ramified, $\overline{r}|_{G_{F_v}}$ is trivial and $q_v \equiv 1 \mod l$.

$(e)$ \ Any prime $p$ lying below $T$ or ramified in $F$, splits in an imaginary quadratic field contained in $F$. (In particular, $F$ is an imaginary CM field.)

Take a finite place $u$ of $F$ such that $u \notin T$ and $q_u \mod l \neq 1$. This is possible by $\zeta_l \notin F$. Note that we have the following properties by the assumption (a), 1 of Lemma \ref{coincide monodromy type} and Theorem \ref{Ila Varma}.

$(f) \ r|_{G_{F_u}} \sim r_{\iota}(\pi)|_{G_{F_u}}$.

$(g) \ \iota\mathrm{WD}(r_{\iota}(\pi)|_{G_{F_u}})^{F-ss} \cong \mathrm{rec}_{F_u}(\pi_u|\mathrm{det}|_u^{\frac{1-n}{2}})$.

By using Lemma \ref{imaginary quadratic} again, we may assume that the prime lying below $u$ splits in an imaginary quadratic field contained in $F$. We put $S := T \cup \{ u \}$. 

For $v|l$, we put $\Lambda_{v, 1} := \mathcal{O}[[\mathcal{O}_{F_v}^{\times}(l)^{\oplus n}]]$ and $p_{v, \lambda}$, $p_{v, \mu}: \Lambda_{v, 1} \rightarrow \mathcal{O}$ denote the $\mathcal{O}$-morphisms induced by $(\phi_{\lambda, v, i}|_{I_{F_v}}), (\psi_{\mu, v, i}|_{I_{F_v}})$ respectively. We put $\Lambda_1:= \widehat{\otimes}_{v|l} \Lambda_{v,1}$. By the condition $(c)$, for any $v|l$, there exists a minimal prime $Q_v$ of $\Lambda_{v, 1}$ contained in $\mathrm{Ker}(p_{v, \lambda})$ and $\mathrm{Ker}(p_{v, \mu})$. Note that $\Lambda_{v} := \Lambda_{v, 1}/Q_v \cong \mathcal{O}[[X_1, \cdots, X_{n[F_v:\mathbb{Q}_l]}]]$. We put $\Lambda := \widehat{\otimes}_{v | l} \Lambda_v$.

For each $v \in T \setminus (S_l \cup U)$, let $\chi_{v, 1}, \cdots, \chi_{v, n}: \mathbb{F}_v^{\times} \rightarrow 1 + \varpi\mathcal{O}$ be characters. (Later, we only consider the case that $\chi_{v, i}$ are trivial and the case that $\chi_{v, i}$ are pairwise different.) We put $\mathcal{S}^{\chi} := ( \overline{r}, S, \{ \Lambda_{v} \}_{v \in S}, \{ \mathcal{D}_v(\chi) \}_{v \in S}) := ( \overline{r}, S, \{ \Lambda_{v} \}_{v | l} \cup \{ \mathcal{O} \}_{v \in S \setminus S_l}, \{ \mathcal{D}_v^{\chi} \}_{v \in T \setminus (S_l \cup U)} \cup \{ \mathcal{D}_v^{\mathrm{det, ord}} \}_{v|l} \cup \{ \mathcal{D}_v^1 \}_{v \in U \cup \{ u \}} )$. Then $r$ (resp. $r_{\iota}(\pi)$) is a lifting of type $\mathcal{S}$ to $\mathcal{O}$ when we regard $\mathcal{O}$ as a $\Lambda$-algebra by $(p_{v, \lambda})$ (resp. $(p_{v, \mu})$). After extending $E$, we may assume for any $v \in S$, all irreducible components of $\mathrm{Spec} \, R_{\overline{r}|_{G_{F_{v}}}, \mathcal{D}_v(\chi)}[\frac{1}{l}]$ and $\mathrm{Spec} \, R_{\overline{r}|_{G_{F_{v}}}, \mathcal{D}_v(\chi)}/\varpi$ are geometrically irreducible.

We define an open compact subgroup $K = \prod_{v} K_v$ of $\mathrm{GL}_n(\mathbb{A}_F^{\infty})$ as follows.

$\cdot$ If $v \notin S$, then $K_v:=\mathrm{GL}_n(\mathcal{O}_{F_v})$.

$\cdot$ If $v \in S_l$, then $K_v:=\mathrm{Iw}_{v,1} = \mathrm{Iw}_{v}(1,1)$.

$\cdot$ If $v \in S \setminus (S_l \cup U)$, then $K_v := \mathrm{Iw}_v$.

$\cdot$ If $v \in U$, then $K_v$ is the parahoric subgroup of $\mathrm{GL}_n(F_{v})$ corresponding to the dual partition of the monodromy type of $\mathrm{WD}(r|_{G_{F_v}})$.

By Proposition \ref{automorphy lifting and local-global compatibility}, the assumption (2) and (3) of $5$, $(a)$, $(c)$ and $(d)$, we have $\iota^{-1}(\pi^{\infty})^{K, \mathrm{ord}} \neq 0$. In order to prove that $K$ is a good subgroup of $\mathrm{GL}_n(\mathbb{A}_F^{\infty})$, it suffices to show that for any $k \in K_u = \mathrm{Iw}_u$, the group $G$ ($\subset \mathcal{O}_{\overline{F_u}}^{\times}$ ) generated by all eigenvalues of $k$ has no $l$-torsion elements. This follows by $G/(G \cap (1 + \mathfrak{m}_{\mathcal{O}_{\overline{F_u}}^{\times}})) \hookrightarrow \mathbb{F}_u^{\times}$ and $q_u \mod l \neq 1$.

Let $\mathfrak{m}$ be the non-Eisenstein ideal of $\mathbb{T}^{S, \mathrm{ord}}(R\Gamma(X_K, \mathcal{V}_{\lambda})^{\mathrm{ord}})$ corresponding to $\overline{r_{\iota}(\pi)}$. (See Corollary \ref{ordinary cohomology of locally symmetric space}.) We may assume $\overline{\rho}_{\mathfrak{m}} = \overline{r_{\iota}(\pi)}$. We define an $\mathcal{O}[K_S]$-module $\mathcal{V}_{\mu}(\chi^{-1}) := \mathcal{V}_{\mu} \otimes_{\mathcal{O}} \mathcal{O}(\chi^{-1})$, where $K_S$ acts on $\mathcal{V}_{\mu}$ by projection to $K_l$ and on $\mathcal{O}(\chi^{-1})$ by projection to $K_{T \setminus (S_l \cup U)}$. For $c \ge 1$, we put $\Lambda_{1, c}:=\mathcal{O}[\prod_{v|l} \mathrm{Ker}(T_n(\mathcal{O}_{F_v}/\varpi_v^c) \rightarrow T_n(\mathbb{F}_v))] \cong \mathcal{O}[K(1,c)/K(c,c)]$. This is naturally regarded as a quotient of $\Lambda_1$ and we have $\Lambda_{1} \cong \varprojlim_c \Lambda_{1,c} \cong \varprojlim_c \Lambda_{1, c}/\varpi^c$.

Then $R\Gamma_{K(1,c)/K(c,c)}(X_{K(c,c)}, \mathcal{V}_{\mu}(\chi^{-1}))_{\mathfrak{m}}^{\mathrm{ord}}$ is a perfect complex of $\Lambda_{1,c}$-modules by Proposition \ref{perfect complex}. By Lemma \ref{minimal complex}, we assume that this is a minimal complex of $\Lambda_{1,c}$-modules. 

Let $d:=[F^+:\mathbb{Q}]n^2-1$. We put $$A_1(\mu, \chi, c) := R\mathrm{Hom}_{\Lambda_{1,c}}(R\Gamma_{K(1,c)/K(c,c)}(X_{K(c,c)}, \mathcal{V}_{\mu}(\chi^{-1}))_{\mathfrak{m}}^{\mathrm{ord}}, \Lambda_{1,c})[-d]. $$ Note that for any $c$, we have canonical $\mathbb{T}^{S,\mathrm{ord}}$-equivariant isomorphisms \begin{equation}A_1(\mu, \chi, c) \otimes_{\Lambda_{1,c}}^{\mathbb{L}}\Lambda_{1,c}/\varpi \cong A_1(\mu, 1, c) \otimes_{\Lambda_{1,c}}^{\mathbb{L}}\Lambda_{1,c}/\varpi\end{equation} in $D(\Lambda_{1,c}/\varpi)$ and \begin{equation} A_1(\mu, \chi, c) \otimes^{\mathbb{L}}_{\Lambda_{c}} \Lambda_{c-1} \cong A_1(\mu, \chi, c-1) \end{equation} in $D(\Lambda_{1, c-1})$ by Proposition \ref{level independence} which are compatible. 

For any $c$, we can take an isomorphism $A_1(\mu, \chi, c) \otimes^{\mathbb{L}}_{\Lambda_{c}} \Lambda_{c-1} \cong A_1(\mu, \chi, c-1)$ of minimal complexes of $\Lambda_{1,c-1}$-modules inducing $(5)$ by 3 of Lemma \ref{minimal complex}. We put $A_1(\mu, \chi) := \varprojlim_c A_1(\mu, \chi, c) \cong \varprojlim_c A_1(\mu, \chi, c)/\varpi^c$. This is a minimal complex of $\Lambda_1$-modules with a $\mathbb{T}^{S, \mathrm{ord}}$-action and there exists a $\mathbb{T}^{S, \mathrm{ord}}$-equivariant isomorphism \begin{equation}A_1(\mu, \chi) \otimes^{\mathbb{L}}_{\Lambda_{1}} \Lambda_{1}/\varpi \cong A_1(\mu, 1) \otimes^{\mathbb{L}}_{\Lambda_{1}} \Lambda_{1}/\varpi\end{equation} in $D(\Lambda_1/\varpi)$ by 4 of Lemma \ref{minimal complex}.

We put $B_1(\mu, \chi) := A_1(\mu, \chi) \otimes_{\mathcal{O}}\mathcal{O}(-\nu - w_0\mu)$. Here, we put $\nu:=(0, 1, \cdots, n-1) \in (\mathbb{Z}_+^n)^{\mathrm{Hom}(F,E)}$ and $w_0\mu:=(\mu_{\tau, n}, \mu_{\tau, n-1}, \cdots, \mu_{\tau, 1})_{\tau \in \mathrm{Hom}(F, \overline{\mathbb{Q}}_l)}$. By Proposition \ref{weight independence}, for any $\mu' \in (\mathbb{Z}^n_{+})^{\mathrm{Hom}(F,E)}$, we have a $\mathbb{T}^{S,\mathrm{ord}}$-equivariant isomorphism $B_1(\mu, \chi) \cong B_1(\mu', \chi)$ in $D(\Lambda_1)$. We put $B(\mu, \chi) := B_1(\mu, \chi) \otimes^{\mathbb{L}}_{\Lambda_1} \Lambda$, $\mathbb{T}^{S, \Lambda_1}:=\mathbb{T}^{S} \otimes_{\mathcal{O}} \Lambda_1 \subset \mathbb{T}^{S,\mathrm{ord}}$. Note that $\mathbb{T}^{S, \Lambda_1}(B(\mu, \chi))$ is a finite local $\Lambda$-algebra.

\begin{prop}\label{ordinary R=T}

There exist an positive integer $\delta$ depending only on $n$, $[F:\mathbb{Q}]$ and $|S|$, an ideal $I_0(\chi)$ of $\mathbb{T}^{S, \Lambda_1}(B(\mu, \chi))$ such that $I_0(\chi)^{\delta} = 0$ and a local surjective $\Lambda$-morphism $f_{\mathcal{S}^{\chi}} : R_{\mathcal{S}^{\chi}} \twoheadrightarrow \mathbb{T}^{S, \Lambda_1}(B(\mu, \chi))/I_0(\chi)$ such that for all $v \notin S$, $\mathrm{det}(T - f_{\mathcal{S}_{\chi}} \circ \rho_{\mathcal{S}^{\chi}}(\mathrm{Frob}_v)) = P_v(T)$. ($\rho_{\mathcal{S}^{\chi}}$ denotes a representative of the universal deformation of $\overline{\rho_{\mathfrak{m}}}$ of type $\mathcal{S}^{\chi}$.)

\end{prop}

\begin{proof} Note that we have $\mathbb{T}^{S, \Lambda_1}(B(\mu, \chi))$ is naturally isomorphic to $$\varprojlim_{c} \mathbb{T}^{S, \Lambda_1}(R\Gamma_{K(1,c)/K(c,c)}(X_{K(c,c)}, \mathcal{V}_{\mu}(\chi^{-1})/\varpi^c)_{\mathfrak{m}}^{\mathrm{ord}} \otimes_{\Lambda_1}^{\mathbb{L}} \Lambda \otimes_{\mathcal{O}} \mathcal{O}(-\nu - w_0\mu))$$ by 4 of Lemma \ref{minimal complex}. By Lemma \ref{trace density}, it suffices to show that there exist an positive integer $\delta$ depending only on $n$, $[F:\mathbb{Q}]$ and $|S|$ such that for any positive integer $c$, there exist an ideal $I_0(c, \chi)$ of $$A_{c} := \mathbb{T}^{S, \Lambda_1}(R\Gamma_{K(1,c)/K(c,c)}(X_{K(c,c)}, \mathcal{V}_{\mu}(\chi^{-1}))_{\mathfrak{m}}^{\mathrm{ord}} \otimes_{\Lambda_1}^{\mathbb{L}} \Lambda \otimes_{\mathcal{O}} \mathcal{O}(-\nu - w_0\mu))$$ such that $I_0(c, \chi)^{\delta} = 0$ and a lifting of type $\mathcal{S}^{\chi}$ $$\rho_{\mathfrak{m}, c} : G_{F, S} \rightarrow \mathrm{GL}_n(A_c/I_0(c, \chi))$$ such that for all $v \notin S$, $\mathrm{det}(T - \rho_{\mathfrak{m}, c}(\mathrm{Frob}_v)) = P_v(T)$. 
    
Let $K'(c,c)$ (resp. $K'(1,c)$) is the good subgroup of $\mathrm{GL}_n(\mathbb{A}_{F}^{\infty})$ defined by $K'(c,c)_v=K(c,c)_v$ (resp. $K'(1,c)_v=K(1,c)_v$) if $v \notin U$ and $K'(c,c)_v= K'(1,c)_v = \mathrm{Iw}_v$ if $v \in U$. By Theorem \ref{unipotently ramified} and Theorem \ref{ordinarirty of automorphic Galois}, there exist an positive integer $\delta$ depending only on $n$, $[F:\mathbb{Q}]$ and $|S|$ such that for any positive integer $c$, there exist an ideal $I_0(c, \chi)$ of $$B_c := \mathbb{T}^{S, \Lambda_1}(R\Gamma_{K'(1,c)/K'(c,c)}(X_{K'(c,c)}, \mathcal{V}_{\lambda}(\chi^{-1}))_{\mathfrak{m}}^{\mathrm{ord}} \otimes_{\Lambda_1}^{\mathbb{L}} \Lambda \otimes_{\mathcal{O}} \mathcal{O}(-\nu - w_0\mu))$$ such that $I_0(c, \chi)^{\delta} = 0$ and a lifting of type $\mathcal{S}^{\chi}$ $$\rho_{\mathfrak{m}, c}' : G_{F, S} \rightarrow \mathrm{GL}_n(B_c/I_0(c, \chi))$$ such that for all $v \notin S$, $\mathrm{det}(T - \rho'_{\mathfrak{m}, c}(\mathrm{Frob}_v)) = P_v(T)$.

Note that for all $c \ge 1$, the canonical morphism $$R\Gamma_{K(1,c)/K(c,c)}(X_{K(c,c)}, \mathcal{V}_{\mu}(\chi^{-1}))^{\mathrm{ord}}_{\mathfrak{m}} \rightarrow R\Gamma_{K'(1,c)/K'(c,c)}(X_{K'(c,c)}, \mathcal{V}_{\mu}(\chi^{-1}))^{\mathrm{ord}}_{\mathfrak{m}}$$ is a $\mathbb{T}^{S,\mathrm{ord}}$-equivariant section of the map $\frac{\mathrm{tr}_{K'(c,c)/K(c,c)}}{(K(c,c):K'(c,c))}$ in $D(\Lambda_{1,c})$. Note that $(K(c,c):K'(c,c)) = \prod_{v \in U}\prod_{i=1}^{k_v}(\mathrm{GL}_{n_{i,v}}(\mathbb{F}_{v}) : B_{n_{i,v}}(\mathbb{F}_{v})) \equiv \prod_{v \in U}((n_{1,v}!) \cdots (n_{k_v,v}!)) \mod l$ is a unit in $\mathcal{O}$. (Here, $[n_{1,v}, \cdots ,n_{k_v,v}]$ is the partition corresponding to $K_v$. Note $q_v \equiv 1 \mod l$.) This implies the result. \end{proof}

Let $q := $$\mathrm{dim}_{\mathbb{F}}H^1(G_{F,S}, \mathrm{ad} \, \overline{\rho_{\mathfrak{m}}}(1))$. By Proposition \ref{adequate}, for any positive integer $N$, there exists a Taylor-Wiles datum ($Q_N, \{ \alpha_v \}_{v\in Q_N})$ satisfying the following conditions. 

$(a')$ \ $|Q_N| = q$.

$(b')$ \ For all $v \in Q_N$, $q_v \equiv 1$ \ mod $l^N$.

$(c')$ \ $\mathrm{dim}_{\mathbb{F}} \mathfrak{m}_{R_{\mathcal{S}^{\chi}(Q_N)}^S}/(\mathfrak{m}_{R_{\mathcal{S}^{\chi}(Q_N)}^S}^2, \mathfrak{m}_{R_{\mathcal{D}^{\chi}_v}})_{v \in S} = q-n^2[F^+:\mathbb{Q}]=:g$. 

(Note that these conditions are independent of $\chi$.)

We put $\Delta_N:=K_0(Q_N)/K_1(Q_N)$ and $A_1(\mu, \chi, c, Q_N)$:= $$R\mathrm{Hom}_{\Lambda_{1,c}[\Delta_{N}]}(R\Gamma_{K(1,c)_0(Q_N)/K(c,c)_1(Q_N)}(X_{K(c,c)_1(Q_N)}, \mathcal{V}_{\mu}(\chi^{-1}))_{\mathfrak{n}_{Q_N}}^{\mathrm{ord}}, \Lambda_{1,c}[\Delta_{N}])[-d].$$ 

(Here, $\mathfrak{n}_{Q_N}$ denotes the maximal ideal of $\mathbb{T}_{Q_N}^{S \cup Q_N \cup Q_N^c, \mathrm{ord}}:=\mathbb{T}^{S \cup Q_N \cup Q_N^c, \mathrm{ord}} \otimes _{\mathbb{T}^{S \cup Q_N \cup Q_N^c}} \mathbb{T}_{Q_N}^{S \cup Q_N \cup Q_N^c}$ generated by the pullback of $\mathfrak{m}$ and $\mathfrak{m}_{Q_N}$. )

This is a perfect complex of $\Lambda_{1,c}[\Delta_{N}]$-modules and we have canonical $\mathbb{T}_{Q_N}^{S \cup Q_N \cup Q_N^c, \mathrm{ord}}$-equivariant isomorphisms \begin{equation}A_1(\mu, \chi, c, Q_N) \otimes_{\Lambda_{1,c}[\Delta_N]}^{\mathbb{L}}\Lambda_{1,c}/\varpi[\Delta_N] \cong A_1(\mu, 1, c, Q_N) \otimes_{\Lambda_{1,c}[\Delta_N]}^{\mathbb{L}} \Lambda_{1,c}/\varpi[\Delta_N]\end{equation} in $D(\Lambda_{1,c}[\Delta_N]/\varpi)$ and \begin{equation}A_1(\mu, \chi, c-1, Q_N) \otimes^{\mathbb{L}}_{\Lambda_{c}[\Delta_N]} \Lambda_{c-1}[\Delta_N] \cong A_1(\mu, \chi, c-1, Q_N)\end{equation} in $D(\Lambda_{1, c-1})$ by Proposition \ref{level independence}, which are compatible.

Moreover, we have a canonical isomorphism \begin{equation}A_1(\mu, \chi, c, Q_N) \otimes_{\Lambda_{1,c}[\Delta_N]}^{\mathbb{L}} \Lambda_{1,c} \cong A_1(\mu, \chi, c)\end{equation} in $D(\Lambda_{1,c})$ which induces a $\Lambda_{1,c}$-morphism $\mathbb{T}^{S \cup Q_N \cup Q_N^c, \mathrm{ord}}_{\Delta_N}(A_1(\mu, \chi, c, Q_N)) \otimes_{\Lambda_{1,c}[\Delta_N]} \Lambda_{1,c} \rightarrow \mathbb{T}^{S, \mathrm{ord}}(A_1(\mu, \chi, c))$ by Theorem \ref{Patching argument}, which is compatible with the isomorphisms (4), (5), (7) and (8). (Here, we put $$\mathbb{T}_{\Delta_N}^{S \cup Q_N \cup Q_N^c, \mathrm{ord}}:=\mathbb{T}^{S \cup Q_N \cup Q_N^c, \mathrm{ord}} \otimes _{\mathbb{T}^{S \cup Q_N \cup Q_N^c}} \mathbb{T}_{\Delta_N}^{S \cup Q_N \cup Q_N^c}.)$$

We may assume that $A_1(\mu, \chi, c, Q_N)$ is a minimal complex and we can take an isomorphism $A_1(\mu, \chi, c, Q_N) \otimes_{\Lambda_{1,c}[\Delta_N]}^{\mathbb{L}} \Lambda_{1,c-1}[\Delta_N] \cong A_1(\mu, \chi, c-1, Q_N)$ of complexes of $\Lambda_{1,c-1}[\Delta_N]$-modules inducing $(8)$ by Lemma \ref{minimal complex}. We put $A_1(\mu, \chi, Q_N):=\varprojlim_c A_1(\mu, \chi, c, Q_N) \cong \varprojlim_c A_1(\mu, \chi, c, Q_N)/\varpi^c$. This is a minimal complex of $\Lambda_1[\Delta_{N}]$-modules with a $\mathbb{T}_{\Delta_N}^{S \cup Q_N \cup Q_N^c, \mathrm{ord}}$-action and there are a $\mathbb{T}_{\Delta_N}^{S \cup Q_N \cup Q_N^c, \mathrm{ord}}$-equivalent isomorphism \begin{equation}A_1(\mu, \chi, Q_N) \otimes^{\mathbb{L}}_{\Lambda_{1}[\Delta_{N}]} \Lambda_{1}/\varpi[\Delta_{N}] \cong A_1(\mu, 1, Q_N) \otimes^{\mathbb{L}}_{\Lambda_{1}[\Delta_{N}]} \Lambda_{1}/\varpi[\Delta_{N}]\end{equation} in $D(\Lambda_1/\varpi[\Delta_{N}])$ and an isomorphism \begin{equation}A_1(\mu, \chi, Q_N) \otimes_{\Lambda_{1}[\Delta_N]}^{\mathbb{L}} \Lambda_1 \cong A_1(\mu, \chi)\end{equation} in $D(\Lambda_1)$ which induces a $\Lambda_1$-morphism $\mathbb{T}^{S \cup Q_N \cup Q_N^c, \mathrm{ord}}_{\Delta_N}(A_1(\mu, \chi, Q_N)) \otimes_{\Lambda_1[\Delta_N]}\Lambda_1 \rightarrow \mathbb{T}^{S, \mathrm{ord}}(A_1(\mu, \chi))$ by 4 of Lemma \ref{minimal complex}. Moreover, the isomorphisms (6), (10) and (11) are compatible.

We put $B_1(\mu, \chi, Q_N) := A_1(\mu, \chi, Q_N) \otimes_{\mathcal{O}}\mathcal{O}(-\nu - w_0\lambda)$, $B(\mu, \chi, Q_N) := B_1(\mu, \chi, Q_N) \otimes^{\mathbb{L}}_{\Lambda_1[\Delta_{N}]} \Lambda[\Delta_{N}]$, $\mathbb{T}_{\Delta_N}^{S \cup Q_N \cup Q_N^c, \Lambda_1}:=\mathbb{T}_{\Delta_N}^{S \cup Q_N \cup Q_N^c} \otimes_{\mathcal{O}} \Lambda_1 \subset \mathbb{T}_{\Delta_N}^{S \cup Q_N \cup Q_N^c,\mathrm{ord}}$. Note that $\mathbb{T}_{\Delta_N}^{S \cup Q_N \cup Q_N^c, \Lambda_1}(B(\mu, \chi, Q_N))$ is a finite local $\Lambda[\Delta_{N}]$-algebra.

By the same proof as in Proposition \ref{ordinary R=T} and by using Proposition \ref{Patching argument}, we obtain the following proposition.

\begin{prop}

After changing $\delta$ (but independent of $N$), there exist an ideal $I_N(\chi)$ of $\mathbb{T}_{\Delta_N}^{S \cup Q_N \cup Q_N^c, \Lambda_1}(B(\mu, \chi, Q_N))$ such that $I_N(\chi)^{\delta}=0$ and a surjection of local $\Lambda[\Delta_{N}]$-algebras $f_{\mathcal{S}^{\chi}(Q_N)}: R_{\mathcal{S}^{\chi}(Q_N)} \twoheadrightarrow \mathbb{T}_{\Delta_N}^{S \cup Q_N \cup Q_N^c, \Lambda_1}(B(\mu, \chi, Q_N))/I_N(\chi)$ such that for all $v \notin S \cup Q_N \cup Q_N^c$, $\mathrm{det}(T-f_{\mathcal{S}^{\chi}(Q_N)} \circ \rho_{\mathcal{S}^{\chi}(Q_N)}(\mathrm{Frob}_v)) = P_v(T)$. ($\rho_{\mathcal{S}^{\chi}(Q_N)}$ denotes a representative of the universal deformation of $\overline{\rho}_{\mathfrak{m}}$ of type $\mathcal{S}^{\chi}(Q_N)$.)

\end{prop}

In the following, we assume that for each $v \in T \setminus (S_l \cup U)$, $\chi_{v, 1}, \cdots, \chi_{v, n}:\mathbb{F}_v^{\times} \rightarrow 1 + \varpi\mathcal{O}$ are pairwise distinct characters. This is possible by $l \ge n$, $\zeta_l \in E$ and $q_v \equiv 1 \mod l$ for all $v \in T \setminus (S_l \cup U)$.

We put $R_0:=R_{\mathcal{S}^{1}}$, $R_0':=R_{\mathcal{S}^{\chi}}$, $T_0:= \mathbb{T}^{S, \Lambda_1}(B(\mu, 1))$, $T_0':=\mathbb{T}^{S, \Lambda_1}(B(\mu, \chi))$, $C_0 := B(\mu, 1)$, $C_0':=B(\mu,\chi)$, $I_0:=I_0(1)$ and $I_0':=I_0(\chi)$. For all $N$, we put $R_N:=R_{\mathcal{S}^{1}(Q_N)}$, $R_N':=R_{\mathcal{S}^{\chi}(Q_N)}$, $T_N:=\mathbb{T}_{\Delta_N}^{S \cup Q_N \cup Q_N^c, \Lambda_1}(B(\mu, 1, Q_N))$, $T_N':= \mathbb{T}_{\Delta_N}^{S \cup Q_N \cup Q_N^c, \Lambda_1}(B(\mu, \chi, Q_N))$, $C_N := B(\mu, 1, Q_N)$, $C_N' := B(\mu, \chi, Q_N)$, $I_N:=I_N(1)$ and $I_N':=I_N(\chi)$. We also put $\mathcal{T}:=\Lambda[[ X_{v,i,j} \mid v \in S, i, j = 1, \cdots, n ]]/(X_{v_{0,1,1}})$ ($v_0$ is some fixed place of $F$ in $S$.), $\Delta_{\infty}:=\mathbb{Z}_l^{\oplus q}$, $S_{\infty}:=\Lambda[[\Delta_{\infty}]] \widehat{\otimes}_{\Lambda} \mathcal{T}$, $R^{\mathrm{loc}} := \widehat{\otimes}_{v \in S}R_{\overline{r}|_{G_{F_v}}, \mathcal{D}_v(1)}$ and $R^{'\mathrm{loc}}:= \widehat{\otimes}_{v \in S}R_{\overline{r}|_{G_{F_v}}, \mathcal{D}_v(\chi)}$. We regard $S_{\infty}$ as an augmented $\Lambda$-algebra and $\mathfrak{a}$ denotes the augmentation ideal of $S_{\infty}$. We fix a surjection $\Delta_{\infty} \twoheadrightarrow \Delta_{N}$ for each $N$. Then the kernel of this map is contained in $(l^N\mathbb{Z}_l)^{\oplus q}$ by $(b')$.

We also have isomorphisms of $\Lambda$-algebras $R_N \otimes_{\Lambda[\Delta_N]} \Lambda \cong R_0$ and $R_N' \otimes_{\Lambda[\Delta_N]} \Lambda \cong R_0'$ fitting in the following commutative diagrams by Lemma \ref{trace density}

\[\xymatrix{
 R_{N} \ar@{->>}[d] \ar@{->>}[r] & T_{N}/I_{N} \ar@{->>}[d] &  R_{N}' \ar@{->>}[d] \ar@{->>}[r] & T_{N}'/I_{N}' \ar@{->>}[d] \\
 R_0 \ar@{->>}[r] & T_0/(I_0+I_{N}) & R_0' \ar@{->>}[r] & T_0'/(I_0'+I_{N}') \\
}\] and the induced morphisms $R_N/\varpi \rightarrow \overline{T_N}/(I_N+I_N')$ and $R_N'/\varpi \rightarrow \overline{T_N'}/(I_{N} + I_N')$ are equal if we identify $R_N/\varpi = R_N'/\varpi$ and $\overline{T}_N = \overline{T_N'}$. (Here, we put $\overline{T_N} = \overline{T_N'}$ is the image of $T_N$, $T_N'$ in $\mathrm{End}_{D(\Lambda[\Delta_N])}(C_N/\varpi) = \mathrm{End}_{D(\Lambda[\Delta_N])}(C_N'/\varpi)$. )

We put $\mathcal{S}^1(Q_0):=\mathcal{S}^1$ and $\mathcal{S}^{\chi}(Q_0):=\mathcal{S}^{\chi}$. We take representatives $\rho_{\mathcal{S}^{1}}: G_F \rightarrow \mathrm{GL}_n(R_{0}), \ \rho_{\mathcal{S}^{\chi}}: G_F \rightarrow \mathrm{GL}_n(R_{0}')$ of the universal deformations of $\overline{\rho_{\mathfrak{m}}}$ of type $\mathcal{S}^1(Q_0), \ \mathcal{S}^{\chi}(Q_0)$ respectively which are equal after modulo $\varpi$.  Moreover, for each $N$, we take representatives $\rho_{\mathcal{S}^1(Q_N)}: G_F \rightarrow \mathrm{GL}_n(R_{N})$, $\rho_{\mathcal{S}^{\chi}(Q_N)}: G_F \rightarrow \mathrm{GL}_n(R_{N}')$ of the universal deformation of $\overline{\rho_{\mathfrak{m}}}$ of type $\mathcal{S}^1(Q_N)$, $\mathcal{S}^{\chi}(Q_N)$ respectively which are equal after modulo $\varpi$ and whose push-forwards to $R_0$, $R_0'$ are equal to $\rho_{\mathcal{S}^1}$, $\rho_{\mathcal{S}^{\chi}}$ respectively. These induce isomorphisms $R^S_{\mathcal{S}^{1}(Q_N)} \cong R_N\widehat{\otimes}_{\Lambda}\mathcal{T}$ and $R^S_{\mathcal{S}^{\chi}(Q_N)} \cong R_N' \widehat{\otimes}_{\Lambda} \mathcal{T}$ by Lemma \ref{framed deformation}, $R^{\mathrm{loc}}$-algebra structures on $R_N \widehat{\otimes}_{\Lambda} \mathcal{T}$ and $R_N$, and $R^{' \mathrm{loc}}$-algebra structures on $R_N' \widehat{\otimes}_{\Lambda} \mathcal{T}$ and $R_N'$ which are equal after modulo $\varpi$. Moreover, the morphism $R_N \widehat{\otimes}_{\Lambda} \mathcal{T} \twoheadrightarrow R_0 \widehat{\otimes}_{\Lambda} \mathcal{T}$ (resp. $R_N' \widehat{\otimes}_{\Lambda} \mathcal{T} \twoheadrightarrow R_0' \widehat{\otimes}_{\Lambda} \mathcal{T}$) becomes an $R^{\mathrm{loc}}$-morphism (resp. $R^{'\mathrm{loc}}$-morphism).  By the above condition $(c')$, we have a surjection $R_{\infty}:=R^{\mathrm{loc}}[[Y_1, \cdots, Y_g]] \twoheadrightarrow R_N\widehat{\otimes}_{\Lambda}\mathcal{T}$ of $R^{\mathrm{loc}}$-algebras and a surjection $R'_{\infty}:=R^{' \mathrm{loc}}[[Y_1, \cdots, Y_g]] \twoheadrightarrow R_N'\widehat{\otimes}_{\Lambda}\mathcal{T}$ of $R^{'\mathrm{loc}}$-algebras which are equal after modulo $\varpi$ such that the morphism $R_N \widehat{\otimes}_{\Lambda} \mathcal{T} \twoheadrightarrow R_0 \widehat{\otimes}_{\Lambda} \mathcal{T}$ (resp. $R_N' \widehat{\otimes}_{\Lambda} \mathcal{T} \twoheadrightarrow R_0' \widehat{\otimes}_{\Lambda} \mathcal{T}$) is a morphism of $R_{\infty}$-algebras (resp. $R_{\infty}'$-algebras). 

Note that the following argument is essentially contained in \cite[\S 5.4]{MEI} and \cite[\S 6.4]{10}. By \cite[Proposition 6.4.12 and 6.4.16]{10}, we obtain the following.

(1) \ Bounded complexes $C_{\infty}$ and $C_{\infty}'$ of finite free $S_{\infty}$-modules, commutative $S_{\infty}$-subalgebras $T_{\infty} \subset \mathrm{End}_{S_{\infty}}(C_{\infty})$ and $T_{\infty}' \subset \mathrm{End}_{S_{\infty}}(C_{\infty}')$, ideals $I_{\infty}$ and $I_{\infty}'$ of $T_{\infty}$, $T_{\infty}'$ respectively satisfying $I_{\infty}^{\delta} = 0$ and $I_{\infty}^{'\delta}=0$. 

(2) \ $S_{\infty}$-algebra structures on $R_{\infty}$ and $R_{\infty}'$ such that the morphisms $R_{\infty} \twoheadrightarrow R_0$ and $R_{\infty}' \twoheadrightarrow R_0'$ factor through $R_{\infty} \twoheadrightarrow R_{\infty}/\mathfrak{a}$ and $R_{\infty}' \twoheadrightarrow R_{\infty}'/\mathfrak{a}$ respectively and surjections of $S_{\infty}$-algebras $R_{\infty} \twoheadrightarrow T_{\infty}/I_{\infty}$ and $R_{\infty}' \twoheadrightarrow T'_{\infty}/I'_{\infty}$.

(3) \ Isomorphisms $C_{\infty}\otimes^{\mathbb{L}}_{S_{\infty}} \Lambda \cong C_0$, $C'_{\infty}\otimes^{\mathbb{L}}_{S_{\infty}} \Lambda \cong C'_0$ in $D(\Lambda)$ inducing $\Lambda$-algebra morphisms $T_{\infty}\otimes_{S_{\infty}} \Lambda \rightarrow T_0$ and $T'_{\infty}\otimes_{S_{\infty}} \Lambda \rightarrow T'_0$ fitting in the following commutative diagrams.

\[\xymatrix{
     R_{\infty} \ar@{->>}[d] \ar@{->>}[r] & T_{\infty}/I_{\infty} \ar@{->>}[d] & R_{\infty}' \ar@{->>}[d] \ar@{->>}[r] & T_{\infty}'/I_{\infty}' \ar@{->>}[d] \\
     R_0 \ar@{->>}[r] & T_0/(I_0+I_{\infty}) & R_0' \ar@{->>}[r] & T_0'/(I_0'+I_{\infty}') \\
    }\] 

(4) \ An isomorphism $C_{\infty} \otimes^{\mathbb{L}}_{S_{\infty}} S_{\infty}/\varpi \cong C_{\infty}' \otimes^{\mathbb{L}}_{S_{\infty}} S_{\infty}/\varpi$ in $D(S_{\infty}/\varpi)$. Under this identification, $T_{\infty}$ and $T_{\infty}'$ have the same image $\overline{T_{\infty}}$ in $\mathrm{End}_{D(S_{\infty}/\varpi)}(C_{\infty} \otimes^{\mathbb{L}}_{S_{\infty}} S_{\infty}/\varpi) = \mathrm{End}_{D(S_{\infty}/\varpi)}(C_{\infty}' \otimes^{\mathbb{L}}_{S_{\infty}} S_{\infty}/\varpi)$ and the actions of $R_{\infty}/\varpi$ and $R_{\infty}'/\varpi$ on $H^*(C_{\infty} \otimes^{\mathbb{L}}_{S_{\infty}} S_{\infty}/\varpi)/(I_{\infty} + I_{\infty}') = H^*(C_{\infty}' \otimes^{\mathbb{L}}_{S_{\infty}} S_{\infty}/\varpi)/(I_{\infty}+I_{\infty}')$ are equal if we identify $R_{\infty}/\varpi = R_{\infty}'/\varpi$.

\vspace{0.5 \baselineskip}

We put $l_0:=[F^+:\mathbb{Q}]n-1$ and $q_0:=[F^+:\mathbb{Q}]\frac{n(n-1)}{2} $. Let $\mathfrak{p}$ denote the kernel of the $\mathcal{O}$-morphism $S_{\infty} \twoheadrightarrow \Lambda \xrightarrow{p_{\mu}} \mathcal{O}$. Then we have $\mathrm{dim} \, R_{\infty}[\frac{1}{l}] = g + |S|n^2 + [F:\mathbb{Q}]\frac{n(n+1)}{2} = q - n^2[F^+:\mathbb{Q}] + |S|n^2 + [F:\mathbb{Q}]\frac{n(n+1)}{2} = |S|n^2 + q + [F^+:\mathbb{Q}]n = $$\mathrm{dim} \, S_{\infty, \mathfrak{p}} - l_0$ by Lemma \ref{completed tensor product irreducibility}, Proposition \ref{Level raising} and Proposition \ref{ordinary deformation ring}. Note that since $T_{\infty, \mathfrak{p}}$ is finite over $S_{\infty, \mathfrak{p}}$ and Ann$_{T_{\infty, \mathfrak{p}}} (H^*(C_{\infty, \mathfrak{p}}))$ is a nilpotent ideal of $T_{\infty, \mathfrak{p}}$ by Lemma \ref{nilpotent}, we have $\mathrm{dim} \, \mathrm{Supp}_{S_{\infty, \mathfrak{p}}} H^*(C_{\infty, \mathfrak{p}}) = \mathrm{dim} \, T_{\infty, \mathfrak{p}}$. Moreover, since $I_{\infty}$ is a nilpotent ideal of $T_{\infty, \mathfrak{p}}$ and $\mathrm{Spec} \, T_{\infty, \mathfrak{p}}/I_{\infty}$ is a closed subscheme of $\mathrm{Spec} \, R_{\infty, \mathfrak{p}}$, we have dimSupp$_{S_{\infty, \mathfrak{p}}}H^*(C_{\infty, \mathfrak{p}}) \le\mathrm{dim} \, R_{\infty, \mathfrak{p}} \le \mathrm{dim} \, R_{\infty}[\frac{1}{l}] = $$\mathrm{dim} \, S_{\infty, \mathfrak{p}} - l_0$. By (2) of 2 of Corollary \ref{ordinary cohomology of locally symmetric space} and the fact that $\iota^{-1}(\pi^{\infty})^{K, \mathrm{ord}}$ contributes to $H^{*}(X_K, \mathcal{V}_{\mu}(1))^{\mathrm{ord}}_{\mathfrak{m}} \otimes_{\mathcal{O}} \overline{\mathbb{Q}}_l$, we obtain that $$H^i(C_{\infty, \mathfrak{p}} \otimes_{S_{\infty, \mathfrak{p}}}^{\mathbb{L}} S_{\infty, \mathfrak{p}}/\mathfrak{p}) = \mathrm{Hom}_E(H^{d-i}(X_K, \mathcal{V}_{\mu}(1))^{\mathrm{ord}}_{\mathfrak{m}}[\frac{1}{l}] \otimes_{E}E(-\nu-w_0\mu), E)$$ is zero for $i \notin [q_0, q_0 + l_0]$ and nonzero for $i \in [q_0, q_0 + l_0]$.

This implies dimSupp$_{S_{\infty, \mathfrak{p}}}H^*(C_{\infty, \mathfrak{p}}) = \mathrm{dim} \, S_{\infty, \mathfrak{p}} - l_0$, $H^{i}(C_{\infty, \mathfrak{p}})$ is zero for $i \neq l_0 + q_0$, $M := H^{l_0+q_0}(C_{\infty, \mathfrak{p}}) \neq 0$ and depth$_{S_{\infty, \mathfrak{p}}}M = $$\mathrm{dim} \, S_{\infty, \mathfrak{p}} - l_0$ by Proposition \ref{Calagari-OG}.

Let $\mathfrak{q}$ be the point of $\mathrm{Spec} \, T_{0}$ corresponding to $\pi$ and we regard $\mathfrak{q}$ as points of $\mathrm{Spec} \, T_{\infty}$ and $\mathrm{Spec} \, R_{\infty}$. Then $\mathfrak{q}$ lies above $\mathfrak{p}$, $M_{\mathfrak{q}} \neq 0$ and we have $\mathrm{depth}_{T_{\infty, \mathfrak{q}}}M_{\mathfrak{q}} \ge \mathrm{depth}_{S_{\infty, \mathfrak{p}}}M = $$\mathrm{dim} \, S_{\infty, \mathfrak{p}} - l_0 = \mathrm{dim} \, R_{\infty}[\frac{1}{l}] \ge $$\mathrm{dim} \, R_{\infty, \mathfrak{q}} \ge $$\mathrm{dim} \, T_{\infty, \mathfrak{q}}$ since an $M$-regular sequence in $\mathfrak{m}_{S_{\infty, \mathfrak{p}}}$ is an $M_{\mathfrak{q}}$-regular sequence in $\mathfrak{q}T_{\infty, \mathfrak{q}}$. This implies that all irreducible components of $\mathrm{Spec} \, T_{\infty}[\frac{1}{l}]$ containing $\mathfrak{q}$ have dimension $\mathrm{dim} \, S_{\infty, \mathfrak{p}} - l_0 = \mathrm{dim} \, R_{\infty}[\frac{1}{l}]$. Therefore, there exists an irreducible component $\mathcal{C}_0$ of $\mathrm{Spec} \, R_{\infty}$ of dimension $\mathrm{dim} \, R_{\infty}$ containing $\mathfrak{q}$ and contained in Im($\mathrm{Spec} \, T_{\infty}/I_{\infty} \rightarrow \mathrm{Spec} \, R_{\infty})= $Supp$_{R_{\infty}}H^*(C_{\infty})/I_{\infty}$.

For any $v \in S_l \cup U \cup \{ u \}$, $r_{\iota}(\pi)|_{G_{F_v}}$ is contained in a unique irreducible components $\mathcal{C}_{0, v}$ of $\mathrm{Spec} \, R_{\overline{r}|_{G_{F_v}}, \mathcal{D}_v(\chi)}=\mathrm{Spec} \, R_{\overline{r}|_{G_{F_v}}, \mathcal{D}_v(1)}$ by the assumption (3) and $(g)$, Proposition \ref{purity}, Proposition \ref{regular point} and Proposition \ref{ordinary deformation ring}. Moreover, if $v \in S_l$, then $\mathcal{C}_{0,v}$ is the irreducible component of $\mathrm{Spec} \, R_{\overline{r}|_{G_{F_v}}, \mathcal{D}_v(1)}$ having the maximum dimension and the generic point of $\mathcal{C}_{0,v}$ is the unique generic point of $\mathrm{Spec} \, R_{\overline{r}|_{G_{F_v}}, \mathcal{D}_v(1)}$ specializing to some generic points of $\mathcal{C}_{0,v}/\varpi$ by Proposition \ref{ordinary deformation ring}.

By Lemma \ref{completed tensor product irreducibility}, the assumptions $(b)$, $(f)$ and Proposition \ref{ordinary deformation ring}, for each $v \in T \setminus (S_l \cup U)$, there exist irreducible components $\mathcal{C}_{0, v}$ and $\mathcal{C}_{1,v}$ of $\mathrm{Spec} \, R_{\overline{r}|_{G_{F_v}}, \mathcal{D}_v(1)}$ such that $$\mathcal{C}_0 = \mathrm{Spec} \, (\widehat{\otimes}_{v \in S}\Gamma(\mathcal{C}_{0,v}, \mathcal{O}_{\mathcal{C}_{0, v}}))[[Y_1, \cdots, Y_g]]$$ and $r$ is contained in the irreducible component $$\mathcal{C}_1 := \mathrm{Spec} \, (\widehat{\otimes}_{v \in S_l \cup U \cup \{u\}}\Gamma(\mathcal{C}_{0, v}, \mathcal{O}_{\mathcal{C}_{0,v}})) \widehat{\otimes} (\widehat{\otimes}_{v \in T \setminus (S_l \cup U)}\Gamma(\mathcal{C}_{1, v}, \mathcal{O}_{\mathcal{C}_{1,v}}))[[Y_1, \cdots, Y_g]]$$ of $\mathrm{Spec} \, R_{\infty}$. Let $\eta_0$ and $\eta_1$ be the generic points of $\mathcal{C}_{0}$ and $\mathcal{C}_1$ respectively and we take generic points $\overline{\eta_0}$ and $\overline{\eta_1}$ of $\mathcal{C}_0/\varpi$ and $\mathcal{C}_1/\varpi$ respectively. We regard $\overline{\eta_0}$ and $\overline{\eta_1}$ as points of $\mathrm{Spec} \, R_{\infty}'$ by the canonical isomorphism $R_{\infty}/\varpi \cong R_{\infty}'/\varpi$. 

By Proposition \ref{Level raising}, Lemma \ref{completed tensor product irreducibility} and Proposition \ref{ordinary deformation ring}, $\overline{\eta_0}$ and $\overline{\eta_1}$ generalize to the generic point $\eta'$ of an irreducible component $$\mathcal{C}' := \mathrm{Spec} \, (\widehat{\otimes}_{v \in S_l \cup U \cup \{u\}}\Gamma(\mathcal{C}_{0, v}, \mathcal{O}_{\mathcal{C}_{0,v}})) \widehat{\otimes} (\widehat{\otimes}_{v \in T \setminus (S_l \cup U)}R_{\overline{r}|_{G_{F_v}}, \mathcal{D}_v(\chi)})[[Y_1, \cdots, Y_g]]$$ of $\mathrm{Spec} \, R_{\infty}'$ and $\eta_1$ (resp. $\eta_2$, $\eta'$) is the unique generic point of $\mathrm{Spec} \, R_{\infty}$ (resp. $\mathrm{Spec} \, R_{\infty}$, $\mathrm{Spec} \, R_{\infty}'$) specializing to $\overline{\eta_0}$ (resp. $\overline{\eta_1}$, $\overline{\eta_0}$ or $\overline{\eta_1}$). In particular, $\mathrm{Spec} \, (R_{\infty})_{\overline{\eta_i}}$ and $\mathrm{Spec} \, (R_{\infty}')_{\overline{\eta_i}}$ are irreducible one-dimensional schemes whose generic points have characteristic zero for each $i=0, 1$. Note that for each $i=0,1$, $(T_{\infty})_{\overline{\eta_i}}$ is empty or irreducible such that $\mathrm{dim} \, (T_{\infty})_{\overline{\eta_i}}/\varpi = 0$ since $\mathrm{Spec} \, (T_{\infty})_{\overline{\eta_i}}$ is a closed subscheme of $\mathrm{Spec} \, (R_{\infty})_{\overline{\eta_i}}$. Similarly, for $i=0,1$, $(T_{\infty}')_{\overline{\eta_i}}$ is empty or irreducible such that $\mathrm{dim} \, (T_{\infty}')_{\overline{\eta_i}}/\varpi = 0$.

We will use the following lemma.

\begin{lem} \label{Taylor's trick}

    Let $S$ be an excellent ring, $f$ be a regular element of $S$, $C$ be a bounded complex of $S$-modules such that $H^i(C)$ is a finite $S$-module for any $i$, $T$ be a finite $S$-algebra with an $S$-algebra morphism $T \rightarrow \mathrm{End}_{D(S)}(C)$ and $\overline{\xi}$ be a prime ideal of $T$ containing $f$ such that $\mathrm{Spec} \, T_{\overline{\xi}}$ is irreducible and $\mathrm{dim} \, T_{\overline{\xi}}/f = 0$. (This implies $\mathrm{dim} \, T_{\overline{\xi}} \le 1$.)
    
    1 \ If $\mathrm{dim} \, T_{\overline{\xi}} = 0$, then $\mathrm{lg}_{T_{\overline{\xi}}/f}((C \otimes_{S}^{\mathbb{L}} S/f)_{\overline{\xi}}) := \sum_i (-1)^i\mathrm{lg}_{T_{\overline{\xi}}/f}H^i(C \otimes_{S}^{\mathbb{L}} S/f)_{\overline{\xi}} = 0$.
    
    2 \ If $\mathrm{dim} \, T_{\overline{\xi}} = 1$, then there exists a positive integer $a_{\overline{\xi}}$ such that $a_{\overline{\xi}} \mathrm{lg}_{T_{\xi}}(C_{\xi}) = \mathrm{lg}_{T_{\overline{\xi}}/f} ((C \otimes_{S}^{\mathbb{L}} S/f)_{\overline{\xi}})$, where $\xi$ denotes the generic point of $\mathrm{Spec} \, T_{\overline{\xi}}$ and $\mathrm{lg}_{T_{\xi}}(C_{\xi}) := \sum_i (-1)^i\mathrm{lg}_{T_{\xi}}H^i(C)_{\xi}$.
    
    \end{lem}
    
    \begin{proof} See \cite[Lemma 6.3.4]{10}. \end{proof}

We already have $\eta_0, \overline{\eta_0} \in \mathrm{Spec} \, T_{\infty}$ (in particular, $\mathrm{dim} \, T_{\infty, \overline{\eta_0}} = 1$) and $\mathrm{lg}_{T_{\infty, \eta_0}}(C_{\infty, \eta_0}) = (-1)^{l_0+q_0}\mathrm{lg}_{T_{\infty,\eta_0}}M_{\eta_0} \neq 0$ since $\eta_0$ lies below $\mathfrak{p}$. By 2 of Lemma \ref{Taylor's trick}, $\mathrm{lg}_{T_{\infty, \overline{\eta_0}}}((C_{\infty} \otimes_{S_{\infty}}^{\mathbb{L}} S_{\infty}/\varpi)_{\overline{\eta_0}}) = \mathrm{lg}_{T'_{\infty, \overline{\eta_0}}}((C'_{\infty} \otimes_{S_{\infty}}^{\mathbb{L}} S_{\infty}/\varpi)_{\overline{\eta_0}}) \neq 0$. Therefore, we have $\mathrm{dim} \, T_{\infty, \overline{\eta_0}}' = 1$ by 1 of Lemma \ref{Taylor's trick} and consequently $\eta' \in \mathrm{Spec} \, T_{\infty}'$. By 2 of Lemma \ref{Taylor's trick}, we obtain $\mathrm{lg}_{T_{\infty, \eta'}'}(C'_{\infty, \eta'}) \neq 0$. Note that we have $\mathrm{dim} \, T_{\infty, \overline{\eta_1}}' = 1$. Again by using 2 of Lemma \ref{Taylor's trick}, we obtain $\mathrm{lg}_{T_{\infty, \overline{\eta_1}}}((C_{\infty} \otimes_{S_{\infty}}^{\mathbb{L}} S_{\infty}/\varpi)_{\overline{\eta_1}}) = \mathrm{lg}_{T'_{\infty, \overline{\eta}_1}}((C'_{\infty} \otimes_{S_{\infty}}^{\mathbb{L}} S_{\infty}/\varpi)_{\overline{\eta_1}}) \neq 0$ and again by using 1 of Lemma \ref{Taylor's trick}, we obtain that $\eta_1$ is contained in $\mathrm{Spec} \, T_{\infty}$. Thus $\mathcal{C}_1$ and $r$ are contained in $\mathrm{Im}( \mathrm{Spec} \, T_{\infty}/I_{\infty} \rightarrow \mathrm{Spec} \, R_{\infty}) = \mathrm{Supp}_{R_{\infty}}H^*(C_{\infty})/I_{\infty}$. By Lemma \ref{Tor spectral sequence}, $r$ is contained in $\mathrm{Supp}_{R_{0}[\frac{1}{l}]}H^*(C_{0} \otimes^{\mathbb{L}}_{\Lambda, p_{\lambda}} E)/I_0.$ Note that we have $\mathrm{Supp}_{R_{0}[\frac{1}{l}]}H^*(C_{0} \otimes^{\mathbb{L}}_{\Lambda, p_{\lambda}} E) = \mathrm{Supp}_{R_{0}[\frac{1}{l}]}\mathrm{Hom}_E(H^{*}(X_K, \mathcal{V}_{\lambda}(1))^{\mathrm{ord}}_{\mathfrak{m}}[\frac{1}{l}] \otimes_{E}E(-\nu-w_0\lambda))/I_0 = \mathrm{Im}(\mathrm{Spec} \, \mathbb{T}^{S, \Lambda_1}(H^{*}(X_K, \mathcal{V}_{\lambda}(1))^{\mathrm{ord}}_{\mathfrak{m}}[\frac{1}{l}] \otimes_{E}E(-\nu-w_0\lambda))[\frac{1}{l}]/I_0 \hookrightarrow \mathrm{Spec} \, T_0[\frac{1}{l}]/I_0 \hookrightarrow \mathrm{Spec} \, R_0[\frac{1}{l}]).$ 

By the bijection in (3) of 2 of Corollary \ref{ordinary cohomology of locally symmetric space}, we obtain an $\iota$-ordinary cohomological cuspidal automorphic representation $\Pi$ of $\mathrm{GL}_n(\mathbb{A}_F)$ of weight $\iota\lambda$ such that $r_{\iota}(\Pi) \cong r$ and $(\Pi^{\infty})^K \neq 0$. In particular, we have $\Pi_{v}^{K_{v}} \neq 0$ for all $v \in U$. This implies $\iota \mathrm{WD}(r_{\iota}(\Pi)|_{G_{F_{v}}})^{F-ss} \cong $rec$_{F_{v}}(\Pi_{v}|\mathrm{det}|^{\frac{1-n}{2}}_{v})$ by Proposition \ref{automorphy lifting and local-global compatibility}. \end{proof}

\section{Potential automorphy and local-global compatibility}

In this section, we prove the potential automorphy of Galois representations $r$ and the local-global compatibility for the cuspidal automorphic representations corresponding to them in the following cases under some technical assumptions on the residual representations. 

1 \ $r$ is essentially self-dual and potentially diagonalizable at all $l$-adic places. (See Theorem \ref{potential diagonalizable automorphy}.)

2 \ $r$ is ordinary at all $l$-adic places (See Theorem \ref{potential ordinary automorphy}.) 

3 \ $r$ has a sufficiently regular weight and is potentially diagonalizable at all $l$-adic places (See Theorem \ref{potential tensor automorphy}.)

We fix a CM field $F$ and a positive integer $n$. 

\subsection{Some results on inductions and characters.}

\begin{lem}\label{local global induction}

Let $v$ be a finite place of $F$, $s_1$, $s_2$, $r_1$ and $r_2$ be $n$-dimensional Frobenius semisimple Weil-Deligne representations of $W_{F_v}$ over $\mathbb{C}$. We assume that $s_1$ and $s_2$ are weakly tempered, $r_1 \oplus r_2 \cong s_1 \oplus s_2$, $r_1 \prec s_1$ and $r_2 \prec s_2$. Then $r_1 \cong s_1$ and $r_2 \cong s_2$.
        
\end{lem}
    
\begin{proof}  For $i = 1,2$, let $\sigma_i$, $\tau_i$ be the monodromy types of $r_i$ and $s_i$ respectively. Then we have $\sigma_i \prec \tau_i$ for any $i = 1, 2$ and $\sigma_1 \sqcup \sigma_2 = \tau_1 \sqcup \tau_2$. By Lemma \ref{partition sum}, we obtain $\sigma_i = \tau_i$ for $i = 1, 2$. This implies the lemma by $(2) \Rightarrow (1)$ of Lemma \ref{monodromy operator}. \end{proof}

\begin{lem}\label{local global induction 2}

Let $l$ be a prime, $\iota : \overline{\mathbb{Q}}_l \stackrel{\sim}{\rightarrow} \mathbb{C}$ be an isomorphism of fields, $M/F$ be a solvable extension of CM fields of degree $d$, $r: G_M \rightarrow \mathrm{GL}_n(\overline{\mathbb{Q}}_l)$ be an algebraic $l$-adic representation and $v$ be a non-$l$-adic places of $F$.

We assume that $\mathrm{Ind}^{G_F}_{G_M}r$ is irreducible and there exists a cohomological cuspidal automorphic representation $\Pi$ of $\mathrm{GL}_{dn}(\mathbb{A}_F)$ such that $\mathrm{Ind}_{G_M}^{G_F}r \cong r_{\iota}(\Pi)$ and $\iota \mathrm{WD}(r_{\iota}(\Pi)|_{G_{F_v}})^{F-ss} \cong \mathrm{rec}_{F_v}(\Pi_v|\mathrm{det}|_v^{\frac{1-dn}{2}})$.

Then there exists a cohomological cuspidal automorphic representation $\pi$ of $\mathrm{GL}_n(\mathbb{A}_M)$ such that $\mathrm{AI}_{M/F}(\pi)|| \mathrm{det} ||_F^{\frac{(d-1)n}{2}} \cong \Pi$, $r \cong r_{\iota}(\pi)$ and $\iota \mathrm{WD}(r_{\iota}(\pi)|_{G_{M_u}})^{F-ss} \cong \mathrm{rec}_{M_u}(\pi_u|\mathrm{det}|_u^{\frac{1-n}{2}})$ for all $u | v$.
        
\end{lem}

\begin{proof}

We may assume that $d$ is a prime.

Let $\sigma$ (resp. $\chi$) be a generator of $\mathrm{Gal}(M/F)$ (resp. $\mathrm{Hom}(\mathrm{Gal}(M/F), \mathbb{C}^{\times})$). Then $\mathrm{Ind}^{G_F}_{G_M}r \otimes \iota^{-1}\chi \cong \mathrm{Ind}^{G_F}_{G_M}r$. This implies $\Pi \otimes \chi \circ \mathrm{Art_F} \circ \mathrm{det} \cong \Pi$ by Lemma \ref{strong multiplicity one}. By \cite[Theorem 3.4.2]{AC}, there exists a cuspidal automorphic representation $\pi'$ of $\mathrm{GL}_n(\mathbb{A}_{M})$ such that $\mathrm{BC}_{M/F}(\Pi) \cong \pi' \boxplus \pi^{'\sigma} \boxplus \cdots \boxplus \pi^{'\sigma^{d-1}}$. Therefore, $\pi:=\pi'||\mathrm{det}||_{M}^{\frac{(1-d)n}{2}}$ is cohomological and we have $r \oplus r^{\sigma} \cdots \oplus r^{\sigma^{d-1}} \cong r_{\iota}(\pi) \oplus r_{\iota}(\pi^{\sigma}) \oplus \cdots \oplus r_{\iota}(\pi^{\sigma^{d-1}})$. By the irreducibility of $r$, we may assume $r \cong r_{\iota}(\pi)$. For any $u | v$, we have $\iota\mathrm{WD}(r_{\iota}(\pi)|_{G_{M_u}})^{F-ss} \oplus \iota \mathrm{WD}(r_{\iota}(\pi^{\sigma})|_{G_{M_u}})^{F-ss} \oplus \cdots \oplus \iota \mathrm{WD}(r_{\iota}(\pi^{\sigma^{d-1}})|_{G_{M_u}})^{F-ss} \cong \mathrm{rec}_{M_u}(\pi_u|\mathrm{det}|_u^{\frac{1-n}{2}}) \oplus \mathrm{rec}_{M_u}(\pi^{\sigma}_u|\mathrm{det}|_u^{\frac{1-n}{2}}) \oplus \cdots \oplus \mathrm{rec}_{M_u}(\pi^{\sigma^{d-1}}_u|\mathrm{det}|_u^{\frac{1-n}{2}})$. By Theorem \ref{Ila Varma}, Proposition \ref{purity}, 1 of Remark \ref{weakly tempered} and Lemma \ref{local global induction}, we obtain $\iota \mathrm{WD}(r_{\iota}(\pi)|_{G_{M_u}})^{F-ss} \cong \mathrm{rec}_{M_u}(\pi_u|\mathrm{det}|_u^{\frac{1-n}{2}})$. \end{proof}

\begin{lem} \label{purity tensor}

Let $l$ be a prime, $\iota : \overline{\mathbb{Q}}_l \stackrel{\sim}{\rightarrow} \mathbb{C}$ be an isomorphism of fields, $\pi$ be a cohomological cuspidal automorphic representation of $\mathrm{GL}_n(\mathbb{A}_F)$ and $v$ be a non-$l$-adic place of $F$.
    
We assume that there exist a finite extension $F'/F$ of CM fields, a positive integer $d$, an algebraic $l$-adic representation $r : G_{F'} \rightarrow \mathrm{GL}_d(\overline{\mathbb{Q}}_l)$, a cohomological cuspidal automorphic representation $\Pi$ of $\mathrm{GL}_{dn}(\mathbb{A}_{F'})$ and a finite place $u$ of $F'$ lying above $v$ such that $r_{\iota}(\Pi) \cong r_{\iota}(\pi)|_{G_{F'}} \otimes r$, $\iota\mathrm{WD}(r_{\iota}(\Pi)|_{G_{F'_u}})^{F-ss} \cong \mathrm{rec}_{F'_u}(\Pi_u|\mathrm{det}|_u^{\frac{1-dn}{2}})$ and $r|_{G_{F'_u}}$ is unramified pure. 
        
Then we have $\iota \mathrm{WD}(r_{\iota}(\pi)|_{G_{F_v}})^{F-ss} \cong \mathrm{rec}_{F_v}(\pi_v|\mathrm{det}|_v^{\frac{1-n}{2}})$.
    
\end{lem}
    
\begin{proof}
    
There exist unramified characters $\chi_1, \cdots, \chi_n : {F'}_u^{\times} \rightarrow \mathbb{C}^{\times}$ such that $\iota\mathrm{WD}(r|_{G_{F'_u}})^{F-ss} \cong (\chi_1 \oplus \cdots \oplus \chi_d) \circ \mathrm{Art}^{-1}_{F'_u}$ and $|\chi_i(\varpi_u)| = |\chi_j(\varpi_u)|$ for all $i, j$. 

Then we have $(\oplus_{i=1}^d \mathrm{rec}_{F'_u}(\mathrm{BC}_{F'_u/F_v}(\pi_v)|\mathrm{det}|_u^{\frac{1-n}{2}}\chi_i))^{ss} \cong \iota\mathrm{WD}((r_{\iota}(\pi)|_{G_{F'}} \otimes r)|_{G_{F'_u}})^{ss} \cong \iota\mathrm{WD}(r_{\iota}(\Pi)|_{G_{F'_u}})^{ss} \cong \mathrm{rec}_{F'_u}(\Pi_u |\mathrm{det}|_u^{\frac{1-dn}{2}})^{ss}$ by Theorem \ref{Ila Varma}. This implies $\boxplus_{i=1}^d\mathrm{BC}_{F'_u/F_v}(\pi_v)|\mathrm{det}|_u^{\frac{1-n}{2}}\chi_i \cong \Pi_u |\mathrm{det}|_u^{\frac{1-dn}{2}}$ by 1 of Remark \ref{weakly tempered}, Proposition \ref{purity} and Lemma \ref{semisimple}. Thus, we obtain $$\oplus_{i=1}^d \mathrm{rec}_{F'_u}(\mathrm{BC}_{F'_u/F_v}(\pi_v)|\mathrm{det}|_u^{\frac{1-n}{2}}\chi_i) \cong \oplus_{i=1}^d \iota\mathrm{WD}(r_{\iota}(\pi)|_{G_{F'_u}})^{F-ss} \chi_i \circ \mathrm{Art}_{F'_u}^{-1}.$$
    
Therefore, we obtain $\mathrm{rec}_{F'_u}(\mathrm{BC}_{F'_u/F_v}(\pi_v)|\mathrm{det}|_u^{\frac{1-n}{2}}\chi_1) \cong \iota\mathrm{WD}(r_{\iota}(\pi)|_{G_{F'_u}})^{F-ss} \chi_1 \circ \mathrm{Art}_{F'_u}^{-1}|_{G_{F'_u}}$ by Lemma \ref{local global induction}. This implies $\iota \mathrm{WD}(r_{\iota}(\pi)|_{G_{F_v}})^{F-ss} \cong \mathrm{rec}_{F_v}(\pi_v|\mathrm{det}|_v^{\frac{1-n}{2}})$ by Proposition \ref{automorphy lifting and local-global compatibility}. \end{proof}

We can prove the following lemma by the same way as Lemma \ref{coefficient}.

\begin{lem}\label{coefficient 2} Let $l$ be a prime, $r : G_F \rightarrow \mathrm{GL}_n(\overline{\mathbb{Q}}_l)$ be a continuous representation, $\tau \in \mathrm{Aut}(F)$.

Then, we obtain the following results.
        
1 \ If $r$ is crystalline (resp. semistable, de Rham, Hodge-Tate) at $v \mid l$, then $r^{\tau}$ is crystalline (resp. semistable, de Rham, Hodge-Tate) at $v^{\tau}$.
        
2 \ If $r$ is Hodge-Tate at all $v \mid l$, then $\mathrm{HT}_{\delta \tau}(r^{\tau}) = \mathrm{HT}_{\delta}(r)$ for any $\delta \in \mathrm{Hom}(F, \overline{\mathbb{Q}}_l)$. 
        
\end{lem}

\begin{prop}\label{automorphic induction}

Let $M/F$ be a solvable extension of imaginary CM fields of degree $d$, $l$ be a prime, $\lambda \in (\mathbb{Z}_{+}^n)^{\mathrm{Hom}(M, \overline{\mathbb{Q}}_l)}$, $\iota : \overline{\mathbb{Q}}_l \stackrel{\sim}{\rightarrow} \mathbb{C}$ be an isomorphism of fields and $\pi$ be a cohomological cuspidal automorphic representation of $\mathrm{GL}_n(\mathbb{A}_M)$ of weight $\iota \lambda$ such that $r_{\iota}(\pi)|_{G_{F_v}}$ is de Rham of $l$-adic Hodge type $\mathbf{v}_{\lambda_v}$ for any $v|l$, $(\mathrm{Ind}^{G_F}_{G_M} r_{\iota}(\pi))|_{G_{F_w}}$ is Hodge-Tate regular for any $w|l$ and $\mathrm{Ind}^{G_F}_{G_M} r_{\iota}(\pi)$ is irreducible.
        
Then $\mathrm{AI}_{M/F}(\pi)||\mathrm{det}||_F^{\frac{(d-1)n}{2}}$ is a cohomological cuspidal automorphic representation of $\mathrm{GL}_n(\mathbb{A}_F)$ of weight $\iota \mu$ such that $r_{\iota}(\mathrm{AI}_{M/F}(\pi)||\mathrm{det}||_F^{\frac{(d-1)n}{2}}) \cong \mathrm{Ind}^{G_F}_{G_M} r_{\iota}(\pi)$. (Here, $\mu$ denotes the element of $(\mathbb{Z}^{dn}_+)^{\mathrm{Hom}(F,\overline{\mathbb{Q}}_l)}$ such that $\mathbf{v}_{\mu_v}$ is the $l$-adic Hodge type of $(\mathrm{Ind}^{G_{F}}_{G_M}r_{\iota}(\pi))|_{G_{F_v}}$ for all $v \mid l$.)
        
\end{prop}
        
\begin{proof} We may assume that $d$ is a prime. Since $\mathrm{Ind}^{G_F}_{G_M} r_{\iota}(\pi)$ is irreducible, we have $r_{\iota}(\pi) \ncong r_{\iota}(\pi)^{\sigma} = r_{\iota}(\pi^{\sigma})$, where $\sigma$ denotes the generator of $\mathrm{Gal}(M/F)$. This is equivalent to $\pi \ncong \pi^{\sigma}$ by Lemma \ref{strong multiplicity one}. This implies the cuspidality of $\mathrm{AI}_{M/F}(\pi)||\mathrm{det}||_F^{\frac{(d-1)n}{2}}$ by \cite[Theorem 3.4.2]{AC}. 
    
For $v \mid l$, $\tau \in \mathrm{Hom}_{\mathbb{Q}_l}(F_v, \overline{\mathbb{Q}}_l)$, $w \mid v$ and $\widetilde{\tau} \in \mathrm{Hom}_{\mathbb{Q}_l}(M_w, \overline{\mathbb{Q}}_l)$, we have $\mathrm{HT}_{\tau}((\mathrm{Ind}^{G_{F}}_{G_M}r_{\iota}(\pi))|_{G_{F_v}}) = \mathrm{HT}_{\widetilde{\tau}}((\mathrm{Ind}^{G_{F}}_{G_M}r_{\iota}(\pi))|_{G_{M_w}}) = \sqcup_{i=0}^{d-1} \mathrm{HT}_{\widetilde{\tau}}(r_{\iota}(\pi)^{\sigma^{-i}}|_{G_{M_w}}) = \sqcup_{i=0}^{d-1} \mathrm{HT}_{\widetilde{\tau} \sigma^{i}}(r_{\iota}(\pi)|_{G_{M_{w^{\sigma^i}}}}) = \sqcup_{i=0}^{d-1} \{ \lambda_{\widetilde{\tau} \sigma^i, j} + n - j \mid j = 1, \cdots, n \} \in ((\mathbb{Z} + \frac{dn-1}{2})^{dn}/\mathfrak{S}_{dn})$ by Lemma \ref{coefficient 2}. This consists of $dn$-distinct integers by the Hodge-Tate regularity of $(\mathrm{Ind}^{G_F}_{G_M} r_{\iota}(\pi))|_{G_{F_v}}$.
    
On the other hand, the infinitesimal character $z \in ((\mathbb{Z} + \frac{dn-1}{2})^{dn}/\mathfrak{S}_{dn})^{\mathrm{Hom}(F, \mathbb{C})}$ of $(\mathrm{AI}_{M/F}(\pi)||\mathrm{det}||_F^{\frac{(d-1)n}{2}})_{\infty}$ is given by $z_{\tau} = \{ - \lambda_{\iota^{-1} \widetilde{\tau}, j} - n + j + \frac{dn-1}{2} \mid \widetilde{\tau} \in \mathrm{Hom}(M, \mathbb{C}) \ \mathrm{satisfying} \ \widetilde{\tau}|_F=\tau , j=1, \cdots, n \} \in ((\mathbb{Z} + \frac{dn-1}{2})^{dn}/\mathfrak{S}_{dn})$ for any $\tau \in \mathrm{Hom}(F, \mathbb{C})$. This implies that $\mathrm{AI}_{M/F}(\pi)||\mathrm{det}||_F^{\frac{(d-1)n}{2}}$ is cohomological. \end{proof}

\begin{prop}\label{character}

We assume that $F$ is an imaginary CM field.

Let $l$ be a prime, $S$ be a finite set of finite places of $F$ containing all $l$-adic places and satisfying $S = S^c$.

Let $\chi : G_{F^+} \rightarrow \overline{\mathbb{Q}}_l^{\times}$ be an algebraic $l$-adic character such that $\chi(c_v) = \chi(c_w)$ for all $v, w \mid \infty$ and for $v \in S$, $\psi_v : G_{F_v} \rightarrow \overline{\mathbb{Q}}_l^{\times}$ be a continuous character such that $(\psi_v\psi_{v^c}^c)|_{I_{F_v}} = \chi|_{I_{F_v}}$ and if $v \mid l$, then $\psi_v$ is de Rham.

(1) \ If all places in $S$ are unramified over $F^+$, then there exists an algebraic $l$-adic character such that $\theta \theta^c = \chi|_{G_{F}}$ and $\theta|_{I_{F_v}} = \psi_v|_{I_{F_v}}$ for all $v \in S$.

(2) \ We assume that $l > 2$. Let $\overline{\theta} : G_{F} \rightarrow \overline{\mathbb{F}_l}^{\times}$ be a continuous character satisfying $\overline{\chi}|_{G_{F}} = \overline{\theta} \overline{\theta}^c$ and $\overline{\theta}|_{G_{F_v}} = \overline{\psi_v}$.

Then there exists an algebraic $l$-adic character $\theta : G_F \rightarrow \overline{\mathbb{Q}}_l^{\times}$ which is a lifting of $\overline{\theta}$ such that $\theta \theta^c = \chi|_{G_F}$ and $\theta|_{I_{F_v}} = \psi_v|_{I_{F_v}}$ for all $v \in S$.

\end{prop}

\begin{proof} See \cite[Lemma A.2.5]{CW}.  \end{proof}

\begin{rem}

See \cite[A.2]{CW} for more detailed properties of characters.

\end{rem}

\subsection{Potential automorphy and local-global compatibility in self-dual potentially diagonalizable cases}

We will use the following proposition in polarizable cases.

\begin{prop}\label{potential conjugate}
 
(1) Let $F$ be an imaginary CM field.

(2) Let $d$ be a positive integer. 
    
(3) \ Let $l \ge 2(d + 1)$ be a prime such that $F \nsubseteq F^{+} (\zeta_l)$.

(4) Let $\mu: G_{F^+} \longrightarrow \mathcal{O}^{\times}_{\overline{\mathbb{Q}}_l}$ be an algebraic $l$-adic character such that $\mu(c_v) = \mu(c_w)$ for all $v, w|\infty$.
    
    (5) Let $\overline{r}: G_F \longrightarrow \mathrm{GL}_n(\overline{\mathbb{F}_l})$ be a continuous representation satisfying the following conditions.
    
    $\cdot$ $\overline{r}|_{G_{F(\zeta_l)}}$ is absolutely irreducible.
    
    $\cdot$ $d$ is bigger than or equal to the maximal dimension of irreducible constituents of the restriction of $\overline{r}$ to the closed subgroup of $G_{F(\zeta_l)}$ topologically generated by all Sylow pro-$l$-subgroups.
    
    $\cdot$ There exists a perfect symmetric $G_F$-equivalent pairing $\overline{r} \times \overline{r}^c \rightarrow \overline{\mu}|_{G_{F}}$.

    (6) Let $S$ be a finite set of finite places of $F$ satisfying the following conditions.

    $\cdot$ $S = S^c$. 
    
    $\cdot$ All $v \in S$ split over $F^+$.
    
    $\cdot$ $S$ contains all $l$-adic places.
    
    $\cdot$ $\overline{r}$ and $\mu$ are unramified outside $S$.

(7) \ For any $v \in S$, let $\rho_v : G_{F_v} \rightarrow \mathrm{GL}_n(\mathcal{O}_{\overline{\mathbb{Q}}_l})$ be a lifting of $\overline{r}|_{G_{F_v}}$ satisfying the following conditions.

$\cdot$ If $v \nmid l$, $\rho_v$ is unipotently ramified and $\mathrm{WD}(\rho_{v} \otimes \overline{\mathbb{Q}}_l)$ and $\mathrm{WD}(\rho_{v^c} \otimes \overline{\mathbb{Q}}_l)$ have the same monodromy type.

$\cdot$ If $v | l$, $\rho_v$ is Hodge-Tate regular potentially diagonalizable crystalline such that $\rho_{v^c}^c \sim \rho_v^{\vee} \mu|_{G_{F_v}}$.

    Then there exists a lifting $r: G_F \rightarrow \mathrm{GL}_n(\mathcal{O}_{\overline{\mathbb{Q}}_l})$ of $\overline{r}$ satisfying the following conditions.
    
    (a) There exists a perfect symmetric $G_{F}$-equivariant pairing $r \times r^c \rightarrow \mu|_{G_{F}}$.
    
    (b) For any $v \notin S$, $r|_{G_{F_v}}$ is unramified.
    
    (c) For any $v|l$, $r|_{G_{F_v}} \sim \rho_v$.
    
    (d) For any $v \in S \setminus \{ u|l \}$, $\mathrm{WD}(r|_{G_{F_v}} \otimes \overline{\mathbb{Q}}_l)$ and $\mathrm{WD}(\rho_v \otimes \overline{\mathbb{Q}}_l)$ have the same monodromy type.

Moreover, for any $\iota : \overline{\mathbb{Q}}_l \stackrel{\sim}{\rightarrow} \mathbb{C}$, there exist a finite CM Galois extension $L/\mathbb{Q}$ linearly disjoint from $F^{(a)}$ over $\mathbb{Q}$ and a polarizable cohomological cuspidal automorphic representation $\pi$ of $\mathrm{GL}_n(\mathbb{A}_{F'})$ such that $r|_{G_{F'}} \cong r_{\iota}(\pi)$. (We put $F':=LF$.)

\end{prop}

\begin{proof}

We may assume that $\rho_v$ is robustly smooth for all $v \in S \setminus \{ u|l \}$ by Lemma \ref{robustly smooth monodromy type}. Moreover, we may assume that $\rho_{v^c}^c \cong \rho_v^{\vee} \mu|_{G_{F_v}}$ for any $v \in S$ by 1 of Lemma \ref{p=l} and by the fact that for any $v \nmid l$, $\mathrm{WD}(\rho_{v^c} \otimes \overline{\mathbb{Q}}_l)$ and $\mathrm{WD}(\rho_{v}^{\vee}\mu|_{G_{F_v}} \otimes \overline{\mathbb{Q}}_l)$ have the same monodromy type. By \cite[Proposition 3.2.1]{CW}, there exists a lifting $r : G_{F} \rightarrow \mathrm{GL}_n(\mathcal{O}_{\overline{\mathbb{Q}}_l})$ of $\overline{r}$ satisfying the following conditions. 

$(a')$ There exists a $G_{F}$-equivariant symmetric perfect pairing $( \ , \ ): \ r \times r^c \rightarrow \mu|_{G_{F}}$. (We fix a complex conjugation $c \in G_{F^+}$.)
        
$(b')$ For any finite place $v \notin S$, $r|_{G_{F_v}}$ is unramified.
        
$(c')$ For any finite place $v \in S$, $r|_{G_{F_v}} \sim \rho_v$.

There exists a finite place $\overline{v} \nmid 2l$ of $F^+$ such that $\overline{v}$ splits in $F$ and $\overline{v}$ doesn't lie below any place in $S$. Let $v$ be a finite place of $F$ lying above $\overline{v}$. By Proposition \ref{character}, we can take an algebraic $l$-adic character $\theta: G_{F} \rightarrow \mathcal{O}_{\overline{\mathbb{Q}}_l}^{\times}$ such that $\theta \theta^c = (\mu^{-1}\varepsilon_{l}^{1-2n})|_{G_{F}}$, $\sigma := \mathrm{Ind}^{G_{F^+}}_{G_{F}}r\theta$ is Hodge-Tate regular and $\overline{\theta}^2|_{I_{F_v}}$ is not trivial. Then we have that $\overline{\sigma}|_{G_{F^+(\zeta_l)}} = \mathrm{Ind}^{G_{F^+(\zeta_l)}}_{G_{F(\zeta_l)}}\overline{r\theta|_{G_{F(\zeta_l)}}}$ is absolutely irreducible since $F \nsubseteq F^+(\zeta_l)$ and $\overline{r\theta}|_{G_{F(\zeta_l)}} \ncong (\overline{r\theta})^c|_{G_{F(\zeta_l)}} (\cong \overline{\varepsilon_l^{1-2n}r^{\vee}\theta^{-1}}|_{G_{F(\zeta_l)}})$ by the non-triviality of $\overline{\theta}^2|_{I_{F_v}}$ and the unramifiedness of $\overline{r}|_{G_{F_v}}$. Then $\langle \ , \ \rangle : \ \sigma \times \sigma \rightarrow \varepsilon_{l}^{1-2n}$, $\langle f,g \rangle :=(f(e),g(c))-(f(c),g(e))$ \ is a $G_{F^+}$-equivariant alternating perfect pairing. (We use the complex conjugation $c \in G_{F^+}$ in $(a')$. $e$ denotes the unit element of $G_{F^+}$.) 

By the same proof as in \cite[Theorem 3.1.2]{CW}, we obtain a finite totally real Galois extension $L/\mathbb{Q}$ linearly disjoint from $F^{(a)}F^{\mathrm{Ker}(\overline{\sigma}|_{F})}(\zeta_l)$ over $\mathbb{Q}$ and a polarizable $\iota$-ordinary cohomological cuspidal automorphic representation $\pi$ of $\mathrm{GL}_{2n}(\mathbb{A}_{F^{+'}})$ such that $\pi \cong \pi^{\vee}$ and $\overline{r_{\iota}(\pi)} \cong \overline{\sigma}|_{G_{F^{+'}}}$. (We put $F^{+'}:=LF^+$. Note that we need to consider $\mathrm{Res}_{F(\zeta_N)^+/\mathbb{Q}}\tilde{T}$ instead of $\tilde{T}$ in the proof of \cite[Theorem 3.1.2]{CW}.) Then by \cite[Theorem 4.2.1 and Lemma 2.2.4]{CW}, there exists a polarizable cohomological cuspidal automorphic representation $\pi$ of $\mathrm{GL}_n(\mathbb{A}_{F'})$ such that $r|_{G_{F'}}$ such that $r_{\iota}(\pi) \cong r|_{G_{F'}}$. (We put $F':=FF^{+'}$.) Since $\iota \mathrm{WD}(r_{\iota}(\pi)|_{G_{F'_v}})^{F-ss} \cong \mathrm{rec}_{F'_v}(\pi_v |\mathrm{det}|_v^{\frac{1-n}{2}})$ for all $v \nmid l$ by Theorem \ref{polarizable local-global compatibility}, $r|_{G_{F_v}}$ is robustly smooth for all $v \nmid l$. By 5 of Lemma \ref{coincide monodromy type} and Lemma \ref{regular point}, $\mathrm{WD}(r|_{G_{F_v}} \otimes \overline{\mathbb{Q}}_l)$ and $\mathrm{WD}(\rho_{v} \otimes \overline{\mathbb{Q}}_l)$ have the same monodromy type for any $v \in S \setminus \{ u|l \}$.  \end{proof}

For our potential automorphy theorems, we need the following stronger notions than the decomposed genericity.

\begin{dfn}\label{strongly decomposed generic}

Let $l$ be a prime, $\overline{r}: G_F \rightarrow \mathrm{GL}_n(\overline{\mathbb{F}_l})$ be a continuous representation and $p \neq l$ be a prime.
            
We say that $p$ is strongly decomposed generic for $\overline{r}$ if $p$ splits completely in $F$, $\overline{r}|_{G_{F_v}}$ is unramified for all $v|p$ and the eigenvalues $\alpha_{v,1}, \cdots, \alpha_{v,n}$ of $\overline{r}(\mathrm{Frob}_v)$ satisfies $\frac{\alpha_{v,i}}{\alpha_{v,j}} \neq p$ for all $i \neq j$ and $\frac{\alpha_{v,i}}{\alpha_{v^c,j}} \neq p$ for all $i, j$.

We say that $p$ is fully decomposed generic for $\overline{r}$ if
     $p$ splits completely in $F$, for all $v|p$, $\overline{r}|_{G_{F_v}}$ is unramified and the eigenvalues $\alpha_{v,1}, \cdots, \alpha_{v,n}$ of $\overline{r}(\mathrm{Frob}_v)$ satisfies $\frac{\alpha_{v,i}}{\alpha_{v,j}} \neq p$ for all $i, j$. 

We say that $p$ is fully strongly decomposed generic for $\overline{r}$ if $p$ is strongly decomposed generic for $\overline{r}$ and fully decomposed generic for $\overline{r}$.

We say that $\overline{r}$ is strongly (resp. fully, fully strongly) decomposed generic if there exists a prime $p$ which is strongly (resp. fully, fully strongly) decomposed generic for $\overline{r}$.

    \end{dfn}

\begin{lem} \label{decomposed genericity density}
    
    Let $l$ be a prime, $\overline{r}: G_F \rightarrow \mathrm{GL}_n(\overline{\mathbb{F}_l})$ be a continuous decomposed generic (resp. strongly decomposed generic, fully decomposed generic, fully strongly decomposed generic ) representation.

1 \ There exist positive Dirichlet density primes which are decomposed generic (resp. strongly decomposed generic, fully decomposed generic, fully strongly decomposed generic) for $\overline{r}$.

2 \ Let $E$ be a Galois CM extension of $\mathbb{Q}$ linearly disjoint from the Galois closure of $\overline{F}^{\mathrm{Ker}\overline{r}}(\zeta_l)$ over $\mathbb{Q}$. Then $\overline{r}|_{G_{EF}}$ is decomposed generic (resp. strongly decomposed generic, fully decomposed generic, fully strongly decomposed generic).

\end{lem}

\begin{proof} We assume that there exists a prime $p \neq l$ such that $p$ is fully strongly decomposed generic for $\overline{r}$. (The proofs of other cases are similar.) Let $M$ be the Galois closure of $\overline{F}^{\mathrm{Ker}\overline{r}}(\zeta_l)$ over $\mathbb{Q}$.

1 \ By Chebotarev's density theorem, there exist positive Dirichlet density primes $q \neq p$ which is unramified in $M$ and satisfies $\mathrm{Frob}_q = \mathrm{Frob}_p$ in $\mathrm{Gal}(M/\mathbb{Q})$ (up to conjugations by $\mathrm{Gal}(M/\mathbb{Q})$). Then $q$ splits completely in $F$ since $\mathrm{Frob}_{q}$ is trivial on the Galois closure of $F$ over $\mathbb{Q}$, for any finite place $v \mid q$ of $F$, there exists $w \mid p$ such that $\mathrm{Frob}_{v} = \mathrm{Frob}_{w}$ in $\mathrm{Gal}(M / F)$ and $p \equiv q \mod l$. Therefore, $q$ is fully strongly decomposed generic for $\overline{r}$.

2 \ By the assumption, we obtain the isomorphism $\mathrm{Gal}(EM/\mathbb{Q}) \stackrel{\sim}{\rightarrow} \mathrm{Gal}(M/\mathbb{Q}) \times \mathrm{Gal}(E/\mathbb{Q}), \sigma \mapsto (\sigma|_{M}, \sigma|_{E})$. By Chebotarev's density theorem, there exists a prime $q \neq l$ such that $q$ is unramified in $EM$, $\mathrm{Frob}_q = \mathrm{Frob}_p$ in $\mathrm{Gal}(M/\mathbb{Q})$ and $\mathrm{Frob}_q$ is trivial in $\mathrm{Gal}(E/\mathbb{Q})$. By the same argument as above, $q$ is fully strongly decomposed generic for $\overline{r}|_{G_{EF}}$. \end{proof}

\begin{lem} \label{induction conjugate} We assume that $F$ is an imaginary CM field.

    (1) \ Let $l$ be a prime, $r : G_{F} \rightarrow \mathrm{GL}_n(\mathcal{O}_{\overline{\mathbb{Q}}_l})$ be a continuous representation and $w$ be an integer such that for any $v \mid l$, there exists a finite extension $F_v'/F_v$ such that $r^c|_{G_{F_v'}} \sim r^{\vee}\varepsilon_l^{-w}|_{G_{F_v'}}$.
    
    (2) \ Let $\theta : G_{F} \rightarrow \mathcal{O}_{\overline{\mathbb{Q}}_l}^{\times}$ be an algebraic $l$-adic character such that $\theta \theta^c$ is trivial and $\sigma := \mathrm{Ind}^{G_{F^+}}_{G_{F}}r\theta$ is Hodge-Tate regular at all $l$-adic places. 
    
    (3) \ Let $F^{(a)}$ be a finite extension of $\mathbb{Q}$.
    
    (4) \ Let $S$ be a finite set of primes not containing $l$.
    
    Then there exists a finite totally real solvable extension $L/\mathbb{Q}$ satisfying the following property. (We put $E:=LF^+$.)
    
    (a) \ $L$ is linearly disjoint from $F^{(a)}$ over $\mathbb{Q}$.
    
    (b) \ All primes in $S$ split completely in $L$.
    
    (c) \ For all $v \mid l$, $(\sigma|_{G_{E_v}} =) \sigma^c|_{G_{E_v}} \sim \sigma^{\vee} \varepsilon_l^{-w}|_{G_{E_v}}$.
    
    \end{lem}
    
    \begin{proof}
    
    By Corollary \ref{extension of Q}, there exists a finite totally real solvable extension $L/\mathbb{Q}$ satisfying the following property. (We put $E:=LF^+$.)
    
    (a) \ $L$ is linearly disjoint from $F^{(a)}$ over $\mathbb{Q}$.
    
    (b) \ All primes in $S$ split completely in $L$.
    
    (c) \ $r^c|_{G_{(FE)_v}} \sim r^{\vee} \varepsilon_l^{-w}|_{G_{(FE)_v}}$ and $v$ splits over $E$ for all $v \mid l$. (Thus, we obtain $r|_{G_{(FE)_{v}}} \sim (r^{\vee})^c \varepsilon_l^{-w}|_{G_{(FE)_v}}$.)
    
    Then for any $l$-adic place $v$ of $E$, if we identify $G_{(FE)_w} = G_{E_{v}}$ for a place $w | v$ of $FE$, we have $\sigma|_{G_{E_{v}}} = (r \theta)|_{G_{(FE)_w}} \oplus (r \theta)^c|_{G_{(FE)_w}} = r|_{G_{(FE)_w}} \theta|_{G_{(FE)_w}} \oplus r^c|_{G_{(FE)_w}} \theta^c|_{G_{(FE)_w}}$ and $\sigma^{\vee} \varepsilon_l^{-w}|_{G_{E_{v}}} = (r^{\vee} \theta^{-1} \varepsilon_l^{-w})|_{G_{(FE)_w}} \oplus (r^{\vee} \theta^{-1} \varepsilon_l^{-w})^c|_{G_{(FE)_w}} = (r^{\vee})|_{G_{(FE)_w}} \theta^{c} \varepsilon_l^{-w}|_{G_{(FE)_w}} \oplus (r^{\vee})^c|_{G_{(FE)_w}} \theta \varepsilon_l^{-w}|_{G_{(FE)_w}}$. Thus, we have $\sigma^c|_{G_{(FE)_v}} = \sigma|_{G_{(FE)_{v}}} \sim (\sigma^{\vee}\varepsilon_l^{-w})|_{G_{(FE)_{v}}}$ by 6 of Lemma \ref{p=l}. \end{proof}

\begin{lem}\label{disjoint CM}

Let $l$ be a prime such that $F \nsubseteq F^+(\zeta_l)$ and $E/F$ be a CM extension which is linearly disjoint from $F(\zeta_l)$ over $F$. Then $E \nsubseteq E^+(\zeta_l)$.

\end{lem}

\begin{proof} It suffices to show that $E$ is linearly disjoint from $E^+(\zeta_l)$ over $E^+$ and this follows from that $E$ is linearly disjoint from $F^+(\zeta_l)$ over $F^+$. This is equivalent to that $F$ is linearly disjoint from $F^+(\zeta_l)$ over $F^+$ and $E$ is linearly disjoint from $F(\zeta_l)$ over $F$. \end{proof}

The main theorem in this subsection is the following.

\begin{thm}\label{potential diagonalizable automorphy}

Let $l \geq 2(n+1)$ be a prime, $r: G_F \rightarrow \mathrm{GL}_n(\overline{\mathbb{Q}}_l)$ be an algebraic $l$-adic representation and $F^{(a)}$ be a finite extension of $\mathbb{Q}$.

We assume the following conditions.

1 \ $F \nsubseteq F^+(\zeta_l)$ or $F$ is a totally real field.

2 \ $\overline{r}$ is strongly decomposed generic.

3 \ $\overline{r}|_{G_{F(\zeta_l)}}$ is absolutely irreducible.

4 \ For all $v|l$, $r|_{G_{F_v}}$ is potentially diagonalizable and Hodge-Tate regular.

5 \ There exists an integer $w$ such that for any $v|l$, there exists a finite extension $F_v'/F_v$ such that $r^c|_{G_{F'_v}} \sim (r^{\vee}\varepsilon_l^{-w}) |_{G_{F_v'}}$. 

6 \ There exists a continuous character $\overline{\chi}: G_{F^+} \rightarrow \overline{\mathbb{F}_l}^{\times}$ such that we have $\overline{\chi}(c_v)=\overline{\chi}(c_w)$ for all $v, w|\infty$ and $\overline{r} \cong \overline{r}^{\vee} \overline{\chi}|_{G_{F}}$. 

Then for any $\iota: \overline{\mathbb{Q}}_l \stackrel{\sim}{\longrightarrow} \mathbb{C}$, there exist a finite Galois CM extension $L/\mathbb{Q}$ linearly disjoint from $F^{(a)}$ over $\mathbb{Q}$ and a cohomological cuspidal automorphic representation $\pi$ of $\mathrm{GL}_n(\mathbb{A}_{F'})$ such that $r_{\iota}(\pi) \cong r|_{G_{F'}}$, $\iota\mathrm{WD}(r_{\iota}(\pi)|_{G_{F_{v}'}})^{F-ss} \cong \mathrm{rec}_{F_{v}'}(\pi_v|\mathrm{det}|_v^{\frac{1-n}{2}})$ for $v \nmid l$ and $\pi_v$ is unramified for all $v|l$. (Here, we put $F':=LF$.)

\end{thm}

\begin{rem}

If we assume that $F^{(a)}$ contains the Galois closure of $\overline{F}^{\mathrm{Ker}(\overline{r})}(\zeta_l)$ over $\mathbb{Q}$, $\overline{r_{\iota}(\pi)}$ is absolutely irreducible and decomposed generic by 2 of Lemma \ref{decomposed genericity density}. Therefore, the weight of $\pi$ is determined by the Hodge-Tate weight of $r$ by Corollary \ref{Caraiani-Newton 2}.

\end{rem}

\begin{proof} We may assume that $F$ is an imaginary CM field, all $l$-adic places of $F^+$ split in $F$ by Corollary \ref{imaginary quadratic 2} and Lemma \ref{disjoint CM} and $F^{(a)}$ contains the Galois closure of $\overline{F}^{\mathrm{Ker}(\overline{r})}(\zeta_l)$ over $\mathbb{Q}$.

We can take an algebraic $l$-adic character $\chi : G_{F^+} \rightarrow \mathcal{O}_{\overline{\mathbb{Q}}_l}^{\times}$ such that $\chi \mod \mathfrak{m}_{\mathcal{O}_{\overline{\mathbb{Q}}_l}} = \overline{\chi}$ and $\mathrm{HT}_{\tau}(\chi) = \{ w \}$ for all $\tau \in \mathrm{Hom}(F^+, \overline{\mathbb{Q}}_l)$. By Lemma \ref{decomposed genericity density}, there exists an odd prime $p \neq l$ such that $p$ is strongly decomposed generic for $\overline{r}$ and $r$ and $\chi$ are unramified above $p$. 

There exists a finite place $u$ of $F$ such that $u$ splits over $F^+$, $r$ is unramifed at $u$, $u^c$ and $u$ doesn't divide $2pl$. By Lemma \ref{quadratic extension}, we can construct a quadratic extension $F'$ of $F$ such that $F'/F$ splits at all $v|p$, $F'/F$ ramifies at $u$ and $F'/F$ is unramified at $u^c$. After twisting $r$ by the corresponding quadratic character, we may assume $\overline{r}|_{G_{F(\zeta_l)}} \ncong \overline{r}^c|_{G_{F(\zeta_l)}}$. Note that we still have $\overline{r} \cong \overline{r}^{\vee} \overline{\chi}|_{G_F}$ and $p$ is strongly decomposed generic for $\overline{r}$. Therefore, (Ind$_{G_{F}}^{G_{F^+}}\overline{r})|_{G_{F^+(\zeta_l)}} = \mathrm{Ind}^{G_{F^+(\zeta_l)}}_{G_{F(\zeta_l)}}(\overline{r}|_{G_{F(\zeta_l)}})$ is absolutely irreducible. (We use $F \nsubseteq F^+(\zeta_l)$.) Since $p$ is strongly decomposed generic for $\overline{r}$, $p$ is decomposed generic for Ind$_{G_{F}}^{G_{F^+}}\overline{r}$.

By Proposition \ref{character}, we can take an algebraic $l$-adic character $\theta: G_{F} \rightarrow \mathcal{O}_{\overline{\mathbb{Q}}_l}^{\times}$ such that $\theta \theta^c$ and $\overline{\theta}$ are trivial, $\sigma:= \mathrm{Ind}^{G_{F^+}}_{G_{F}} (r\theta)$ is Hodge-Tate regular and $\theta$ is unramified above $p$. Note that $p$ is decomposed generic for $\overline{\sigma}$ and $\overline{\sigma}|_{G_{F^+(\zeta_l)}}$ is absolutely irreducible.

By the assumption 6, there exists a $G_F$-equivalent perfect pairing $( \ , \ ): \overline{r} \times \overline{r} \rightarrow \overline{\chi}|_{G_F}$. By the irreducibility of $\overline{r}$, $( \ , \ )$ is symmetric or alternating. After replacing $\chi$ by $\delta_{F^+/F}\chi$ if necessary, we may assume that $\chi$ is totally even (resp. totally odd) if $( \ , \ )$ is symmetric (resp. alternating). We fix a complex conjugation $c \in G_{F^+}$ and define the perfect $G_{F^+}$-equivalent symmetric pairing $\langle \ , \ \rangle : \overline{\sigma} \times \overline{\sigma}^c \rightarrow \chi$ as follows. ($e$ denotes the unit element of $G_{F^+}$.)

$\langle f, g \rangle := (f(e), g(c)) + (f(c), g(e))$ if ( \ , \ ) is symmetric.

$\langle f, g \rangle := (f(e), g(c)) - (f(c), g(e))$ if ( \ , \ ) is alternating.

We take a finite set $S$ of primes such that $S$ contains $p, l$ and all primes lying below ramified places of $\sigma$ and $\chi$, and satisfies $S = S^c$. By Corollary \ref{imaginary quadratic 2}, we can take an imaginary quadratic field $E_1$ linearly disjoint from $F^{(a)}\overline{F^+}^{\mathrm{Ker}(\overline{\sigma})}$ over $\mathbb{Q}$ such that all $q \in S$ split in $E_1$.

By 2 of Lemma \ref{p=l}, Lemma \ref{induction conjugate} and Corollary \ref{extension of Q}, there exists a finite solvable totally real extension $M/\mathbb{Q}$ satisfying the following conditions. (We put $F_1:=ME_1F^+$ and we write $(S \setminus \{l, p\})_{F_1}$ for the set of all finite places of $F_1$ lying above primes contained in $S \setminus \{ p, l \}$.)

(1) \ $M$ is linearly disjoint from $E_1F^{(a)}\overline{F^+}^{\mathrm{Ker}(\overline{\sigma})}$ over $\mathbb{Q}$.

(2) \ $p$ splits completely in $M$.

(3) \ For any $v|l$, $(\sigma|_{G_{F_{1, v}}} = ) \sigma^c|_{G_{F_{1,v}}} \sim (\chi \sigma^{\vee})|_{G_{F_{1, v}}}$.

(4) \ For any $v \in (S \setminus \{l, p\})_{F_1}$, $\sigma|_{G_{F_{1,v}}}$ is unipotently ramified. 

This implies $\overline{\sigma}|_{G_{F_1(\zeta_l)}}$ is absolutely irreducible and $p$ is decomposed generic for $\overline{\sigma}|_{G_{F_1}}$. Note that for any $v \nmid l$, $\mathrm{WD}(\sigma|_{G_{F_{1,v}}})$ and $\mathrm{WD}(\sigma|_{G_{F_{1, v^c}}})$ have the same monodromy type since $\sigma|_{G_{F_1}} = \sigma^c|_{G_{F_1}}$. Note also that all Sylow pro-$l$-subgroups of $G_{F_1(\zeta_l)}$ are contained in $G_{F_1F(\zeta_l)}$ and $l \ge 2(n+1)$ (this implies that $\overline{\sigma}(G_{F_1(\zeta_l)})$ is adequate by Proposition \ref{adequate}) and $F_1^+(\zeta_l) \subsetneq F_1(\zeta_l)$ by Lemma \ref{disjoint CM}. 

By Proposition \ref{potential conjugate}, there exist a finite CM Galois extension $L$ of $\mathbb{Q}$ and a polarizable cohomological cuspidal automorphic representation $\pi_1$ of $\mathrm{GL}_{2n}(\mathbb{A}_{F_2})$ satisfying the following conditions. (We put $F_2:=LF_1$.)

    $(a)$ $L$ is linearly disjoint from the Galois closure of $F^
    {(a)}\overline{F_1}^{\mathrm{Ker}(\overline{\sigma}|_{G_{F_1}})}$ over $\mathbb{Q}$.
    
    $(b)$ $\overline{r_{\iota}(\pi_1)} \cong \overline{\sigma}|_{G_{F_2}}$.
    
    $(c)$ $r_{\iota}(\pi_1)^c \cong \chi|_{G_{{F_2}}}r_{\iota}(\pi_1)^{\vee}$.

    $(d)$ For all $v \notin S_{F_2}$, $\pi_{1, v}$ is unramified.
    
    $(e)$ For all $v | l$, $\pi_{1, v}$ is unramified. 
    
$(f)$ For all $v|l$, $r_{\iota}(\pi_1)|_{G_{F_{2, v}}} \sim \sigma|_{G_{F_{2, v}}}$.
    
$(g)$ For $v \in S_{F_2} \setminus \{ w|l \}$, $\mathrm{WD}(r_{\iota}(\pi_1)|_{G_{F_{2, v}}})$ and $\mathrm{WD}(\sigma|_{G_{F_{2,v}}})$ have the same monodromy type.

Note that $\overline{\sigma}|_{G_{F_2}}$ is decomposed generic by Lemma \ref{decomposed genericity density}, $\overline{\sigma}|_{G_{F_2(\zeta_l)}}$ is absolutely irreducible and $\overline{\sigma}(G_{F_2(\zeta_l)})$ is adequate. Note also that $\mathrm{WD}(r_{\iota}(\pi_1)|_{G_{F_{2, v}}})^{F-ss} \cong \mathrm{rec}_{F_{2,v}}(\pi_1 |\mathrm{det}|_v^{\frac{1-2n}{2}})$ for any finite place $v$ of $F_2$ by Theorem \ref{polarizable local-global compatibility}.

By Theorem \ref{automorphy lifting theorem in crystalline cases}, there exists a cohomological cuspidal automorphic representation $\Pi$ of $\mathrm{GL}_{2n}(\mathbb{A}_{F_2})$ such that $r_{\iota}(\Pi) \cong \sigma|_{G_{F_2}}$, $\iota \mathrm{WD}(r_{\iota}(\Pi)|_{G_{F_{2, v}}})^{F-ss} \cong \mathrm{rec}_{F_{2,v}}(\Pi_{v}|\mathrm{det}|_v^{\frac{1-2n}{2}})$ for all $v \nmid l$ and $\Pi_v$ is unramified for all $v|l$. We put $F':=FF_2$. By Lemma \ref{local global induction 2}, there exists a cuspidal automorphic representation $\pi$ of $\mathrm{GL}_n(\mathbb{A}_{F'})$ such that $r|_{G_{F'}} \cong r_{\iota}(\pi)$ and $\iota \mathrm{WD}(r_{\iota}(\pi)|_{G_{F'_{v}}})^{F-ss} \cong \mathrm{rec}_{F'_{v}}(\pi_{v}|\mathrm{det}|_v^{\frac{1-n}{2}})$ for all $v \nmid l$ and $\pi_{v}$ is unramified for all $v|l$.  \end{proof}

\subsection{Potential automorphy and local-global compatibility in ordinary cases}

\begin{prop}\label{ordinary potential autormorphy}

Let $l$ be a prime such that $l \ge 2(n+1)$, $\iota: \overline{\mathbb{Q}}_l \stackrel{\sim}{\longrightarrow} \mathbb{C}$ be an isomorphism of fields, $F^{(a)}$ be a finite extension of $\mathbb{Q}$ and $r: G_F \rightarrow \mathrm{GL}_n(\overline{\mathbb{Q}}_l)$ be an algebraic $l$-adic representation.

We suppose the following conditions.

1 \ $\overline{r}|_{G_{F(\zeta_l)}}$ is absolutely irreducible. 
    
2 \ $\overline{r}$ is decomposed generic. 

3 \ $\zeta_l \notin F$.
    
4 \ For all $v|l$, $r|_{G_{F_v}}$ is ordinary.

Then there exist a finite CM Galois extension $L$ of $\mathbb{Q}$ linearly disjoint from $F^{(a)}$ over $\mathbb{Q}$ and an $\iota$-ordinary cohomological cuspidal automorphic representation $\pi$ of $\mathrm{GL}_n(\mathbb{A}_{F'})$ such that $r|_{G_{F'}} \cong r_{\iota}(\pi)$. (We put $F':=FL$.)

\end{prop}

\begin{proof}

We may assume $F^{(a)}$ contains the Galois closure of $\overline{F}^{\mathrm{Ker}(\overline{r})}(\zeta_l)$ over $\mathbb{Q}$. By \cite[proof of Theorem 1.1]{opta}, there exist a finite CM Galois extensions $L$ of $\mathbb{Q}$ linearly disjoint from $F^{(a)}$ over $\mathbb{Q}$ and an $\iota$-ordinary cohomological cuspidal automorphic representation $\Pi$ of $\mathrm{GL}_n(\mathbb{A}_{F'})$ such that $\overline{r_{\iota}(\Pi)} \cong \overline{r}|_{G_{F'}}$. (We put $F':=FL$. ) Since $F'(\zeta_l)$ is linearly disjoint from $\overline{F}^{\mathrm{Ker}(\overline{r})}(\zeta_l)$ over $F(\zeta_l)$, $\overline{r}|_{G_{F'(\zeta_l)}}$ is absolutely irreducible by the condition 1. Thus $\overline{r}(G_{F'(\zeta_l)})$ is adequate by Proposition \ref{adequate}. Since $F'$ is linearly disjoint from $F(\zeta_l)$ over $F$, we have $\zeta_l \notin F'$ by the condition 3. $\overline{r}|_{G_{F'}}$ is decomposed generic by Lemma \ref{decomposed genericity density} and the condition 2. Therefore, by Theorem \ref{ordinary automorphy lifting} in the case that $U$ is empty, there exists a cohomological cuspidal automorphic representation $\pi$ of $\mathrm{GL}_n(\mathbb{A}_{F'})$ such that $r_{\iota}(\pi) \cong r|_{G_{F'}}$. \end{proof}

For our applications, we need to modify the result \cite[Theorem 1.1]{opta} as the following.

\begin{thm}\label{ordinary residual potential automorphy}

Let $l$ be an odd prime such that $l \ge 2(n+1)$, $\overline{r}: G_F \rightarrow \mathrm{GL}_n(\overline{\mathbb{F}_l})$ be a continuous representation, $F^{(a)}$ be a finite extension of $\mathbb{Q}$, $\iota: \overline{\mathbb{Q}}_l \stackrel{\sim}{\rightarrow} \mathbb{C}$ be an isomorphism of fields and $T$ be a finite set of primes containing $l$.

We assume the following conditions.

1 \ $\overline{r}|_{G_{F(\zeta_l)}}$ is absolutely irreducible. 

2 \ $\overline{r}$ is decomposed generic. 

3 \ $\zeta_l \notin F$.

Then there exist a finite CM Galois extension $L$ of $\mathbb{Q}$ linearly disjoint from $F^{(a)}$ over $\mathbb{Q}$ and an $\iota$-ordinary cohomological cuspidal automorphic representation $\pi$ of $\mathrm{GL}_n(\mathbb{A}_{F'})$ satisfying the following conditions. (We put $F':=FL$.)
        
(1) \ $\overline{r_{\iota}(\pi)} \cong \overline{r}|_{G_{F'}}$.

(2) \ $\pi_v$ is unramified for all finite places $v$ of $F'$ lying above $T$.
        
(3) \ For all $v|l$, $r_{\iota}(\pi)|_{G_{F'_v}}$ is crystalline ordinary and $\mathrm{HT}_{\tau}(r_{\iota}(\pi)|_{G_{F_v'}}) = \{ n - 1, n - 2, \cdots, 0 \}$ for all $\tau \in \mathrm{Hom}_{\mathbb{Q}_l}(F'_v, \overline{\mathbb{Q}}_l)$. 

(4) \ For all finite places $v$ of $F'$, $\mathrm{WD}(r_{\iota}(\pi)|_{G_{F'_v}})$ is pure.
        
(5) \ For all finite places $v$ of $F'$, $\iota\mathrm{WD}(r_{\iota}(\pi)|_{G_{F'_v}})^{F-ss} \cong \mathrm{rec}_{F_v'}(\pi_v|\mathrm{det}|_v^{\frac{1-n}{2}})$.

\end{thm}

\begin{proof} We may assume that $F^{(a)}$ contains the Galois closure of $\overline{F}^{\mathrm{Ker}(\overline{r})}(\zeta_l)$ over $\mathbb{Q}$. We only explain the case $n > 2$. In the case $n = 2$, we can prove the theorem similarly. (Note that \cite[Proposition 4.2.6]{pw} also holds for $U_{\lambda}$ by the proof. See \cite[p53]{Calabi} for the definition of $U_{\lambda}$.)

By \cite[proof of Theorem 1.1]{opta} and \cite[Proposition 4.2.6]{pw}, there exist finite CM Galois extensions $M$ of $\mathbb{Q}$ linearly disjoint from $F^{(a)}$ over $\mathbb{Q}$, an algebraic $l$-adic representation $s : G_{F_1} \rightarrow \mathrm{GL}_n(\overline{\mathbb{Q}}_l)$ and a finite order character $\chi : G_{F_1} \rightarrow \overline{\mathbb{Q}}_l^{\times}$ satisfying the following conditions. (We put $F_1:=FM$. In \cite[proof of Theorem 1.1]{opta}, we don't need to consider the prime $l'$ and the elliptic curve over $\mathbb{Q}$ and we use $\mathcal{W}_{\lambda}$ instead of $V_{\lambda}$. See \cite[below of Lemma 4.2.4]{pw} for the definition of $\mathcal{W}_{\lambda}$ and \cite[below of Lemma 3.1]{opta} for the definition of $V_{\lambda}$. In \cite[Proposition 4.2]{opta}, we add all primes in $T$ to $S_3$, define $\Omega_{p, 0}$ similarly as $\Omega_{l', 0}$ for any $p \in S_3 \setminus \{l\}$ and replace $\Omega_{l,0}$ by the inverse image of $\prod_{\tau \in \mathrm{Hom}(E, \overline{\mathbb{Q}}_l)} U(v_{\tau}, \lambda)(\overline{\mathbb{F}_l})$ via the reduction map $\prod_{\tau} (T_{0}(\mathcal{O}_{\overline{\mathbb{Q}}_l}) \rightarrow T_{0}(\overline{\mathbb{F}_l}))$, where $v_{\tau}$ denotes the valuation on $\mathbb{Q}(\zeta_N)$ induced by $\tau|_{\mathbb{Q}(\zeta_N)}$ and $U(v_{\tau}, \lambda)$ is a set satisfying the property of \cite[Proposition 4.2.6]{pw}.)

(a) \  $s\chi$ is a direct summand of $H^{N-2}_{\acute{e}t}(Y_{\overline{F_1}}, \overline{\mathbb{Q}}_l)$, where $Y$ is a smooth hypersurface in $\mathbb{P}^{N-1}_{F_1}$ for a sufficiently large $N$ such that $Y$ has good reduction at all finite places of $F_1$ lying above $T$.

(b) \ $\overline{s} \cong \overline{r}|_{G_{F_1}}$. 

(c) \ For any $v|l$, $(s \chi)|_{G_{F_{1,v}}}$ is crystalline ordinary.

(d) \ $\mathrm{HT}_{\tau}(s) = \{ n - 1, n - 2, \cdots, 0 \}$ for all $\tau \in \mathrm{Hom}(F_{1}, \overline{\mathbb{Q}}_l)$. (See \cite[Theorem 3.11]{pw}.)

Note that $\zeta_l \notin E$, $\overline{r}|_{G_{F_1}}$ is absolutely irreducible and decomposed generic by Lemma \ref{decomposed genericity density}. By Corollary \ref{extension of Q}, we may assume that $\chi|_{G_{F'}}$ is unramified at all finite places.

By Proposition \ref{ordinary potential autormorphy}, there exist a finite CM Galois extension $L$ of $\mathbb{Q}$ linearly disjoint from $F^{(a)}E$ over $\mathbb{Q}$ and an $\iota$-ordinary cohomological cuspidal automorphic representation $\pi$ of $\mathrm{GL}_n(\mathbb{A}_{F'})$ such that $r_{\iota}(\pi) \cong s|_{G_{F'}}$. (We put $F':=F_1L$.)

By Proposition \ref{purity local-global} and Remark \ref{Scholze}, we obtain the properties (4) and (5) at all non-$l$-adic places of $F'$. This implies the property (2) for $v \mid p \in T \setminus \{ l \}$ since $r_{\iota}(\pi)|_{G_{F_v}}$ is unramified. Moreover, for all $v|l$ of $F'$, $\mathrm{WD}(r_{\iota}(\pi)|_{G_{F'_v}})$ is unramified pure by \cite[Theorem 1]{cp} and $\iota\mathrm{WD}(r_{\iota}(\pi)|_{G_{F'_v}})^{ss} \cong \mathrm{rec}_{F'_v}(\pi_v|\mathrm{det}|_v^{\frac{1-n}{2}})^{ss}$ by Theorem \ref{ordinary semisimple}. Therefore, we obtain $\iota\mathrm{WD}(r_{\iota}(\pi)|_{G_{F'_v}})^{F-ss} \cong \mathrm{rec}_{F'_v}(\pi_v|\mathrm{det}|_v^{\frac{1-n}{2}})$ by the same argument as Proposition \ref{purity local-global} and consequently $\pi_v$ is unramified for all $v|l$. \end{proof}

We will use the following result in polarizable cases again.

\begin{prop}\label{conjugate self-dual ordinary}

    (1) Let $F$ be an imaginary CM field.

    (2) Let $d$ be a positive integer. 
    
(3) \ Let $l \ge 2(d + 1)$ be a prime such that $\zeta_l \notin F$.
    
    (4) Let $\mu: G_{F^+} \longrightarrow \mathcal{O}^{\times}_{\overline{\mathbb{Q}}_l}$ be an algebraic $l$-adic character such that $\mu(c_v) = \mu(c_w)$ for all $v, w|\infty$.
    
    (5) Let $\iota : \overline{\mathbb{Q}}_l \stackrel{\sim}{\rightarrow} \mathbb{C}$ be an isomorphism of fields and $\pi$ be a polarizable $\iota$-ordinary cohomological cuspidal automorphic representation of $\mathrm{GL}_n(\mathbb{A}_F)$ satisfying the following conditions.
    
    $\cdot$ $\overline{r_{\iota}(\pi)}|_{G_{F(\zeta_l)}}$ is absolutely irreducible.
    
    $\cdot$ $d$ is bigger than or equal to the maximal dimension of irreducible constituents of the restriction of $\overline{r_{\iota}(\pi)}$ to the closed subgroup of $G_{F(\zeta_l)}$ topologically generated by all Sylow pro-$l$-subgroups.
    
    $\cdot$ There exists a perfect symmetric $G_F$-equivalent pairing $\overline{r_{\iota}(\pi)} \times \overline{r_{\iota}(\pi)}^c \rightarrow \overline{\mu}|_{G_{F}}$. 

    (6) Let $S$ be a finite set of finite places of $F$ satisfying the following conditions.

    $\cdot$ $S = S^c$.
    
    $\cdot$ All $v \in S$ split over $F^+$.
    
    $\cdot$ $S$ contains all $l$-adic places.
    
    $\cdot$ $\pi$ and $\mu$ are unramified outside $S$.

(7) \ For any $v \in S$, let $\rho_v : G_{F_v} \rightarrow \mathrm{GL}_n(\mathcal{O}_{\overline{\mathbb{Q}}_l})$ be a lifting of $\overline{r_{\iota}(\pi)}|_{G_{F_v}}$ satisfying the following conditions.

$\cdot$ If $v \nmid l$, $\rho_v$ is unipotently ramified and $\mathrm{WD}(\rho_{v} \otimes \overline{\mathbb{Q}}_l)$ and $\mathrm{WD}(\rho_{v^c} \otimes \overline{\mathbb{Q}}_l)$ have the same monodromy type.

$\cdot$ If $v | l$, $\rho_v$ is ordinary crystalline such that $\rho_{v^c}^c \sim \rho_v^{\vee} \mu|_{G_{F_v}}$.

Then there exists a polarizable $\iota$-ordinary cohomological cuspidal automorphic representation $\Pi$ of $\mathrm{GL}_n(\mathbb{A}_{F})$ satisfying the following conditions.

(a) \ $\overline{r_{\iota}(\Pi)} = \overline{r_{\iota}(\pi)}$.

(b) There exists a perfect symmetric $G_F$-equivariant paring $r_{\iota}(\Pi)^c \times r_{\iota}(\Pi) \rightarrow \mu|_{G_{F}}$.
    
(c) For any $v \notin S$ or $v \mid l$, $\Pi_v$ is unramified.
    
(d) For any $v|l$, $r_{\iota}(\Pi)|_{G_{F_v}} \sim \rho_v$. 
    
(e) For $v \in S \setminus \{ u|l \}$, $\mathrm{WD}(r_{\iota}(\Pi)|_{G_{F_v}})$ and $\mathrm{WD}(\rho_v \otimes \overline{\mathbb{Q}}_l)$ have the same monodromy type.

\end{prop}

\begin{proof} We may assume that $\rho_v$ is robustly smooth for all $v \in S \setminus \{ u|l \}$ by Lemma \ref{robustly smooth monodromy type}. Moreover, we may assume that $\rho_{v^c}^c \cong \rho_v^{\vee} \mu|_{G_{F_v}}$ for any $v \in S$ by 1 of Lemma \ref{p=l} and the fact that for $v \nmid l$, $\mathrm{WD}(\rho_{v} \otimes \overline{\mathbb{Q}}_l)$ and $\mathrm{WD}(\rho_{v}^{\vee}\mu|_{G_{F_v}} \otimes \overline{\mathbb{Q}}_l)$ have the same monodromy type.

By \cite[Proposition 1.5.1 and Theorem 2.4.2]{CW}, there exists a lifting $r$ of $\overline{r_{\iota}(\pi)}$ satisfying the following conditions.

(a) There exists a perfect symmetric $G_F$-equivalent pairing $r \times r^c \rightarrow \mu|_{G_{F}}$.

(b) For any $v \notin S$, $r|_{G_{F_v}}$ is unramified.
    
(c) For any $v|l$, $r|_{G_{F_v}}$ is crystalline ordinary and $r|_{G_{F_v}} \sim \rho_v$.

(d) For $v \in S \setminus \{ u|l \}$, $r|_{G_{F_v}} \sim \rho_{v}$.

By \cite[Theorem 2.4.1]{CW}, there exists a polarizable $\iota$-ordinary cohomological cuspidal automorphic representation $\Pi$ of $\mathrm{GL}_n(\mathbb{A}_F)$ such that $r_{\iota}(\Pi) \cong r$. 

Since $\iota \mathrm{WD}(r_{\iota}(\Pi)|_{G_{F_v}})^{F-ss} \cong \mathrm{rec}_{F_v}(\Pi_v|\mathrm{det}|_v^{\frac{1-n}{2}})$ for all $v$ by Theorem \ref{polarizable local-global compatibility}, $r|_{G_{F_v}}$ is robustly smooth for $v \in S \setminus \{ u \mid l \}$ and $\Pi_v$ is unramified for all $u \mid l$. By Lemma \ref{regular point} and 5 of Lemma \ref{coincide monodromy type}, $\mathrm{WD}(r|_{G_{F_v}} \otimes \overline{\mathbb{Q}}_l)$ and $\mathrm{WD}(\rho_{v} \otimes \overline{\mathbb{Q}}_l)$ have the same monodromy type for $v \in S \setminus \{ u \mid l \}$. \end{proof}

\begin{lem} \label{Thorne}

Let $l$ be a prime, $\iota: \overline{\mathbb{Q}}_l \stackrel{\sim}{\rightarrow} \mathbb{C}$ be an isomorphism of fields and $\pi$ be a cohomological cuspidal automorphic representation of $\mathrm{GL}_n(\mathbb{A}_F)$ of weight $\iota\lambda$.

We assume that for any $v|l$, $r_{\iota}(\pi)|_{G_{F_v}}$ is de Rham of $l$-adic Hodge type $\mathbf{v}_{\lambda_v}$ and $\iota \mathrm{WD}(r_{\iota}(\pi)|_{G_{F_v}})^{ss} \cong \mathrm{rec}_{F_v}(\pi_v|\mathrm{det}|_v^{\frac{1-n}{2}})^{ss}$.

Then $r_{\iota}(\pi)|_{G_{F_v}}$ is ordinary for all $v | l$ if and only if $\pi$ is $\iota$-ordinary.

\end{lem}

\begin{proof}

This follows from Lemma \ref{ordinary Galois representation} and (3) of Proposition \ref{iota-ordinary}. \end{proof}

\begin{cor} \label{ordinary automorphic induction}

Let $l$ be a prime, $\iota: \overline{\mathbb{Q}}_l \stackrel{\sim}{\rightarrow} \mathbb{C}$ be an isomorphism of fields, $M/F$ be a solvable extension of imaginary CM field of degree $d$ and $\pi$ be an $\iota$-ordinary cohomological cuspidal automorphic representation of $\mathrm{GL}_n(\mathbb{A}_M)$ of weight $\iota \lambda$ such that $\overline{r_{\iota}(\pi)}$ is absolutely irreducible and decomposed generic.

If $\mathrm{Ind}^{G_F}_{G_M}\overline{r_{\iota}(\pi)}$ is irreducible and $\mathrm{Ind}^{G_F}_{G_M}r_{\iota}(\pi)|_{G_{F_v}}$ is ordinary for all $v|l$, then $\mathrm{AI}_{M/F}(\pi)||\mathrm{det}||_F^{\frac{n(d-1)}{2}}$ is an $\iota$-ordinary cohomological cuspidal automorphic representation of $\mathrm{GL}_{nd}(\mathbb{A}_F)$.

\end{cor}

\begin{proof} 
    
By Theorem \ref{ordinary semisimple}, $r_{\iota}(\pi)|_{G_{F_v}}$ is ordinary of weight $\lambda_v$ and $\iota \mathrm{WD}(r_{\iota}(\pi)|_{G_{F_v}})^{ss} \cong \mathrm{rec}_{F_v}(\pi_v |\mathrm{det}|_v^{\frac{1-n}{2}})^{ss}$. Let $\mu$ be the element of $(\mathbb{Z}^{dn}_{+})^{\mathrm{Hom}(F, \overline{\mathbb{Q}}_l)}$ such that $(\mathrm{Ind}^{G_{F}}_{G_{M}}r_{\iota}(\pi))|_{G_{F_v}}$ is $l$-adic Hodge type $\mu_{v}$. By Proposition \ref{automorphic induction}, $\mathrm{AI}_{M/F}(\pi)||\mathrm{det}||_F^{\frac{n(d-1)}{2}}$ is a cohomological cuspidal automorphic representation of $\mathrm{GL}_{nd}(\mathbb{A}_F)$ of weight $\iota \mu$. 

We will prove the $\iota$-ordinarity of $\mathrm{AI}_{M/F}(\pi)||\mathrm{det}||_F^{\frac{n(d-1)}{2}}$. By Propositions \ref{ordinary base change}, \ref{base change} and Corollary \ref{extension of Q}, we may assume that all $l$-adic places of $F$ split completely in $M$. 

Then we have \begin{align*} \iota \mathrm{WD}(\mathrm{Ind}^{G_F}_{G_M}r_{\iota}(\pi)|_{G_{F_v}})^{ss} \\
= \oplus_{i=0}^{d-1} \iota \mathrm{WD}(r_{\iota}(\pi)^{\sigma^{-i}}|_{G_{M_{w}}})^{ss} \\
= (\oplus_{i=0}^{d-1} \iota \mathrm{WD}(r_{\iota}(\pi)|_{G_{M_{w^{\sigma^i}}}})^{\sigma^{-i}})^{ss} \\
 = (\oplus_{i=0}^{d-1} \mathrm{rec}_{M_{w^{\sigma^{i}}}}(\pi_{w^{\sigma^{i}}}| \mathrm{det} |_{w^{\sigma^{i}}}^{\frac{1-n}{2}})^{\sigma^{-i}})^{ss} \\
 = \mathrm{rec}_{F_v}((\mathrm{AI}_{M/F}(\pi)_v|\mathrm{det}|_{v}^{\frac{n(d-1)}{2}})|\mathrm{det}|_v^{\frac{1-dn}{2}})^{ss},\end{align*} where $w$ is a place of $M$ lying above $v$ and $\sigma$ is a generator of $\mathrm{Gal}(M/F)$. By Lemma \ref{Thorne}, $\mathrm{AI}_{M/F}(\pi)||\mathrm{det}||_F^{\frac{n(d-1)}{2}}$ is $\iota$-ordinary. \end{proof}

\begin{lem} \label{ordinary character induction}
    
We assume that $F$ is an imaginary CM field.

(a) \ Let $l$ be an odd prime such that all $l$-adic places of $F^+$ splits in $F$.

(b) \ Let $m$ and $k$ be positive integers,.

(c) \ Let $S$ be a finite set of finite places of $F$ not containing $2$-adic places.

(d) \ Let $\overline{r} : G_F \rightarrow \mathrm{GL}_n(\overline{\mathbb{F}_l})$ be a continuous representation such that $\overline{r}|_{G_{F(\zeta_l)}}$ is absolutely irreducible.

(e) \ Let $F^{(a)}$ be a finite extension of $F$.

Then there exist a cyclic extension $M/F$ of CM fields of degree $m$ linearly disjoint from $F^{(a)}$ over $F$ and an algebraic $l$-adic character $\theta : G_{M} \rightarrow \mathcal{O}_{\overline{\mathbb{Q}}_l}^{\times}$ satisfying the following conditions.

1 \ All places in $S$ split completely in $M$.

2 \ $\theta \theta^c$ is trivial

3 \ $(\overline{r} \otimes \overline{\mathrm{Ind}^{G_F}_{G_M}}\theta)|_{G_{F(\zeta_l)}}$ is absolutely irreducible.

4 \ For any $v \mid l$, $(\mathrm{Ind}^{G_{F}}_{G_M}\theta)|_{G_{F_v}}$ is crystalline ordinary and the weight $\lambda$ of $(\mathrm{Ind}^{G_{F}}_{G_M}\theta)|_{G_{F_v}}$ satisfy $\lambda_{\tau, i} - \lambda_{\tau, i+1} \ge k$ for any $\tau \in \mathrm{Hom}_{\mathbb{Q}_l}(F_v, \overline{\mathbb{Q}}_l)$ and $i=1, \cdots, n-1$.

5 \ $\overline{\theta}|_{G_{M_v}}$ is trivial for any finite place $v$ of $M$ lying above a place in $S$. 

6 \ $\theta|_{G_{M_v}}$ is unramified for any finite place $v$ of $M$ lying above a place in $S \setminus \{ u \mid l \}$.

\end{lem}

\begin{rem} Note that in the above situation, we have a perfect symmetric $G_F$-equivariant pairing $\langle \ \ , \ \ \rangle : \mathrm{Ind}^{G_{F}}_{G_M}\theta \times (\mathrm{Ind}^{G_{F}}_{G_M}\theta)^c \rightarrow \mathcal{O}_{\overline{\mathbb{Q}}_l}, \ (f, g) \mapsto \sum_{\sigma \in \mathrm{Gal}(M/F)} f(\sigma)g(c \sigma c)$. \end{rem}

\begin{proof} Let $U$ be the set of primes consisting of all primes lying below places in $S$ and $l$. We fix a finite place $\overline{v_0}$ of $F^+$ not lying above any prime in $U$, $\overline{v_0} \nmid 2$ and $\overline{r}$ is unramified at all $v|\overline{v_0}$. By Lemma \ref{cyclic extension}, we can take a cyclic totally real extension $M^+/F^+$ of degree $m$ such that $\overline{v_0}$ and all finite places of $F^+$ lying above $U$ split completely in $M^+$ and $M^+$ is linearly disjoint from $F^{(a)}\overline{F}^{\mathrm{Ker}(\overline{r})}(\zeta_l)$ over $F^+$. We write $\overline{w_1}, \cdots, \overline{w_{m}}$ for the places of $M^+$ lying above $\overline{v_0}$. By Lemma \ref{quadratic extension}, we can also take a quadratic totally real extension $M^{+'}$ of $M^{+}$ such that all finite places of $M^+$ lying above $U$ split completely in $M^{+'}$ and $\overline{w_1}$ ramifies in $M^{+'}$ and $\overline{w_2}, \cdots, \overline{w_{m}}$ are unramified in $M^{+'}$. 

We put $M:=FM^+$, $M':=FM^{+'}$ and $\overline{\theta}: G_M \twoheadrightarrow \mathrm{Gal}(M'/M) \stackrel{\sim}{\rightarrow} \{ \pm 1 \} \subset \overline{\mathbb{F}_l}^{\times}$. Then we have $\overline{\theta} = \overline{\theta}^c$ and $\overline{\theta}\overline{\theta}^c$ is trivial since $M'$ is a CM field. Moreover, we have that $\overline{\theta}|_{G_{M_v}}$ is trivial for any finite place $v$ of $M$ lying above primes in $U$. Let $\sigma$ be a generator of $\mathrm{Gal}(M/F) = \mathrm{Gal}(M(\zeta_l)/F(\zeta_l))$. Then $(\overline{r}|_{G_{F}} \otimes \mathrm{Ind}^{G_{F}}_{G_M}\overline{\theta})|_{G_{F(\zeta_l)}} = \mathrm{Ind}_{G_{M(\zeta_l)}}^{G_{F(\zeta_l)}}(\overline{r}|_{G_{M(\zeta_l)}} \otimes \overline{\theta}|_{G_{M(\zeta_l)}})$ is absolutely irreducible since $\overline{r}|_{G_{M(\zeta_l)}} \otimes \overline{\theta}|_{G_{M(\zeta_l)}} \ncong (\overline{r}|_{G_{M(\zeta_l)}} \otimes \overline{\theta}|_{G_{M(\zeta_l)}})^{\sigma^{i}}$ for any $i = 1, \cdots, m-1$.

Let $T$ be a set of $l$-adic places of $F$ such that $T \sqcup T^c$ is the set of all $l$-adic places of $F$. For any $v \in T$, we fix an $l$-adic place $w_v$ of $M$ lying above $v$. Then $w_v, w_v^{\sigma}, \cdots, w_v^{\sigma^{m-1}}$ are all $l$-adic places of $M$ lying above $v$.

By Proposition \ref{character}, there exists an algebraic $l$-adic character $\theta : G_{M} \rightarrow \mathcal{O}_{\overline{\mathbb{Q}}_l}^{\times}$ satisfying the following condition.

1 \ $\theta$ is a lifting of $\overline{\theta}$.

2 \ $\theta \theta^c$ is trivial.

3 \ $\theta$ is unramified at all finite places lying above primes in $U \setminus \{ l \}$.

4 \ $\theta$ is crystalline at all $v \mid l$.

5 \ $\mathrm{HT}_{\tau}(\theta|_{G_{M_{w_v^{\sigma^i}}}}) = \{ k'i \}$ for any $v \in T$, $i = 0, \cdots, m-1$ and $\tau \in \mathrm{Hom}_{\mathbb{Q}_l}(M_{w_v^{\sigma^i}}, \overline{\mathbb{Q}}_l)$, where $k' \ge k$ is a positive integer such that $\theta$ is residually trivial on the inertia groups at all $l$-adic places. (Therefore, we have $\mathrm{HT}_{\tau}(\theta|_{G_{M_{w_v^{c\sigma^i}}}}) = \{ -k'i \}$ for any $v \in T$, $i = 0, \cdots, m-1$ and $\tau \in \mathrm{Hom}_{\mathbb{Q}_l}(M_{w_v^{c\sigma^i}}, \overline{\mathbb{Q}}_l)$.)

Then for any $v \in T$, if we identify $G_{M_{w_v}}=G_{F_v}$, then we have $(\mathrm{Ind}^{G_F}_{G_M}\theta)|_{G_{F_v}} = \theta|_{G_{M_{w_v}}} \oplus \theta^{\sigma^{-1}}|_{G_{M_{w_v}}} \oplus \cdots \oplus \theta^{\sigma^{1-m}}|_{G_{M_{w_v}}} = \theta|_{G_{M_{w_v}}} \oplus (\theta|_{G_{M_{w_v^{\sigma}}}})^{\sigma^{-1}} \oplus \cdots \oplus (\theta|_{G_{M_{w_v^{\sigma^{m-1}}}}})^{\sigma^{1-m}}$ and $(\mathrm{Ind}^{G_F}_{G_M}\theta)|_{G_{F_{v^c}}} = ((\mathrm{Ind}^{G_F}_{G_M}\theta)|_{G_{F_v}})^{\vee, c}$. Therefore, $\mathrm{Ind}^{G_{F}}_{G_M}\theta$ satisfy the condition 4. \end{proof}

The main theorem in this subsection is the following.

\begin{thm}\label{potential ordinary automorphy}

    Let $l > n^2$ be a prime, $r: G_F \rightarrow \mathrm{GL}_n(\overline{\mathbb{Q}}_l)$ be an algebraic $l$-adic representation and $F^{(a)}$ be a finite extension of $\mathbb{Q}$.

    We assume the following conditions.
        
    1 \ $\zeta_l \notin F$.
        
    2 \ $\overline{r}$ is fully decomposed generic.
        
    3 \ $\overline{r}|_{G_{F(\zeta_l)}}$ is absolutely irreducible.
        
    4 \ For all $v|l$, $r|_{G_{F_v}}$ is ordinary.

    Then for any $\iota: \overline{\mathbb{Q}}_l \stackrel{\sim}{\longrightarrow} \mathbb{C}$, there exist a finite Galois CM extension $L/\mathbb{Q}$ linearly disjoint from $F^{(a)}$ over $\mathbb{Q}$ and an $\iota$-ordinary cohomological cuspidal automorphic representation $\pi$ of $\mathrm{GL}_n(\mathbb{A}_{F'})$ such that $r_{\iota}(\pi) \cong r|_{G_{F'}}$ and $\iota\mathrm{WD}(r_{\iota}(\pi)|_{G_{F_{u}'}})^{F-ss} \cong \mathrm{rec}_{F_{u}'}(\pi_u|\mathrm{det}|_u^{\frac{1-n}{2}})$ for any $u \nmid l$. (Here, we put $F':=LF$.)
        
 \end{thm}

\begin{proof}

    We may assume that $n \ge 2$, $F$ is an imaginary CM field, all $l$-adic places of $F^+$ split in $F$ by Corollary \ref{imaginary quadratic 2} and $F^{(a)}$ contains the Galois closure of $\overline{F}^{\mathrm{Ker}(\overline{r})}(\zeta_l)$ over $\mathbb{Q}$. First, note that it suffices to prove that for any finite place $u' \nmid l$ of $F$, there exist $L$ and $\pi$ as in the statement but satisfying that $\iota\mathrm{WD}(r_{\iota}(\pi)|_{G_{F_{u}'}})^{F-ss} \cong \mathrm{rec}_{F_{u}'}(\pi_u|\mathrm{det}|_u^{\frac{1-n}{2}})$ for some $u \nmid u'$ instead of for any $u \nmid l$. In fact, if we can prove this result, then by applying this result for $r_{\iota}(\pi)$ again at other non-$l$-adic places and using Proposition \ref{potential automorphy and local-global compatibility}, we obtain the result. In the following, we fix a prime $q \neq l$ and a place $u' \mid q$ of $F$. We may assume that $q$ splits in an imaginary quadratic field in $F$.

    By Theorem \ref{ordinary residual potential automorphy}, there exist a finite CM Galois extension of $L/\mathbb{Q}$ linearly disjoint from $F^{(a)}$ over $\mathbb{Q}$ and an $\iota$-ordinary cohomological cuspidal automorphic representation $\Pi$ of $\mathrm{GL}_n(\mathbb{A}_{F_1})$ such that $\overline{r_{\iota}(\Pi)} \cong \overline{r}|_{G_{F_1}}$, $\Pi_{u}$ is unramified and $\mathrm{WD}(r_{\iota}(\Pi)|_{G_{F_{1,u}}})$ is pure for any $u \mid u'$. (We put $F_1:=LF$. We only need the purity at $u \mid u'$ in this proof.) Therefore, $\overline{r}|_{G_{F_1}}$ is fully decomposed generic by Lemma \ref{decomposed genericity density}, $\overline{r}|_{G_{F_1(\zeta_l)}}$ is absolutely irreducible and $\zeta_l \notin F_1$. We can take an odd prime $p \neq l, q$ such that $p$ is fully decomposed generic for $\overline{r}|_{G_{F_1}}$ by Lemma \ref{decomposed genericity density}.
    
    By Lemma \ref{ordinary character induction}, there exist a cyclic CM extension $M/F_1$ of degree $n$ and an algebraic $l$-adic character $\phi : G_{M} \rightarrow \mathcal{O}_{\overline{\mathbb{Q}}_l}^{\times}$ satisfying the following conditions.

    $(a)$ \ $p$ splits completely in $M$.

    $(b)$ \ $(\overline{r|_{G_{F_1}} \otimes \mathrm{Ind}^{G_{F_1}}_{G_M}\phi})|_{G_{F_1(\zeta_l)}}$ is absolutely irreducible.

    $(c)$ \ $r|_{G_{F_1}} \otimes \mathrm{Ind}^{G_{F_1}}_{G_M}\phi$, $r_{\iota}(\Pi) \otimes \mathrm{Ind}^{G_{F_1}}_{G_M}\phi$ and $\mathrm{Ind}^{G_{F_1}}_{G_M}\phi$ are ordinary at all $l$-adic places.

    $(d)$ \ $\phi \phi^c$ is trivial.

    $(e)$ \ $\overline{\phi}|_{G_{M_v}}$ is trivial for any $v|p$.

$(f)$ \ $\phi|_{G_{M_v}}$ is unramified for any $v|p$.

By the conditions $(a)$ and $(e)$, we obtain that $p$ is decomposed generic for $\overline{r|_{G_{F_1}} \otimes \mathrm{Ind}^{G_{F_1}}_{G_M} \phi}$.

    By Corollary \ref{imaginary quadratic 2} and Corollary \ref{extension of Q}, there exists a finite solvable CM extension $L'/\mathbb{Q}$ satisfying the following conditions. (We put $F_2:=L'F_1$.)
    
    $\cdot$ \ $L'$ is linearly disjoint from $MF^{(a)}\overline{F_1}^{\mathrm{Ker}(\overline{\mathrm{Ind}^{G_{F_1}}_{G_M}\phi})}$ over $\mathbb{Q}$.
    
    $\cdot$ \ $p$ splits completely in $L'$.
    
    $\cdot$ \ For all $w|l$, $(\mathrm{Ind}^{G_{F_1}}_{G_{M}}\phi)|_{G_{F_{2,w}}}$ is crystalline.

$\cdot$ \ For all $w \nmid l$, $(\mathrm{Ind}^{G_{F_1}}_{G_{M}}\phi)|_{G_{F_{2,w}}}$ is unramified.

$\cdot$ \ For all $u \mid q$, $(\overline{\mathrm{Ind}^{G_{F_1}}_{G_M}\phi})|_{G_{F_{2,u}}}$ and $\overline{r}|_{G_{F_{2,u}}}$ are trivial and $r|_{G_{F_{2,u}}}$ is unipotently ramified. 
    
We put $N:=F_2M$, $\chi:=r_{\iota}^{-1}(\phi|_{G_N})$ and we fix a finite place $u$ of $F_2$ lying above $u'$. By Corollary \ref{ordinary automorphic induction}, $\mathrm{AI}_{N/F_2}(\chi)||\mathrm{det}||_{F_2}^{\frac{n-1}{2}}$ is a polarizable $\iota$-ordinary cohomological cuspidal automorphic representation of $\mathrm{GL}_{n}(\mathbb{A}_{F_2})$ and $r_{\iota}(\mathrm{AI}_{N/F_2}(\chi)||\mathrm{det}||^{\frac{n-1}{2}}_{F_2}) \cong \mathrm{Ind}^{G_{F_2}}_{G_{N}}\phi|_{G_N}$. By Proposition \ref{conjugate self-dual ordinary}, there exists an $\iota$-ordinary polarized cohomological cuspidal automorphic representations $\pi'$ of $\mathrm{GL}_{n}(\mathbb{A}_{F_2})$ satisfying the following conditions.
    
    $\cdot$ \ $r_{\iota}(\pi')^c \cong r_{\iota}(\pi')^{\vee}$.
    
    $\cdot$ \ $\overline{r_{\iota}(\pi')} \cong \mathrm{Ind}^{G_{F_2}}_{G_{N}}\overline{\phi}|_{G_{N}}$.
    
    $\cdot$ \ For any $u \mid u'$, $\mathrm{WD}(r|_{G_{F_{2,u}}})$ and $\mathrm{WD}(r_{\iota}(\pi')|_{G_{F_{2,u}}})$ have the same monodromy type.
    
    We put $\tau_1 := r_{\iota}(\Pi)|_{G_{F_2}} \otimes \mathrm{Ind}^{G_{F_2}}_{G_{N}}\phi|_{G_{N}}$, $\tau_2:= r_{\iota}(\Pi)|_{G_{F_2}} \otimes r_{\iota}(\pi')$ and $\tau_3 := r|_{G_{F_2}} \otimes \mathrm{Ind}^{G_{F_2}}_{G_N}\phi|_{G_{N}}$. By the above constructions, we have the following properties.
    
    $\cdot$ \ $\zeta_l \notin F_2$. 
    
    $\cdot$ \ $\overline{\tau_1} \cong \overline{\tau_2} \cong \overline{\tau_3}$.
    
    $\cdot$ \ $\overline{\tau_1}|_{G_{F_2(\zeta_l)}}$ is absolutely irreducible.
    
    $\cdot$ \ $\overline{\tau_1}$ is decomposed generic.
    
    $\cdot$ \  $\overline{\tau_1}(G_{F_2(\zeta_l)})$ is adequate by Proposition \ref{adequate}. (We use $l \nmid n$ and $l \ge 2(n+1)$.)
    
    $\cdot$ \ $\tau_1|_{G_{F_{2,w}}}$, $\tau_2|_{G_{F_{2,w}}}$ and $\tau_3|_{G_{F_{2,w}}}$ are ordinary for all $w|l$.
    
    $\cdot$ \ $\tau_2|_{G_{F_{2,u}}}$ and $\tau_3|_{G_{F_{2,u}}}$ have the same monodromy type for all $u \mid u'$.
    
    $\cdot$ \ $\mathrm{WD}(\tau_2|_{G_{F_{2,u}}})$ is pure for all $u \mid u'$ by Theorem \ref{polarizable local-global compatibility} and 5 of Lemma \ref{purity lemma}.

   We put $\Pi_1 := \mathrm{AI}_{N/F_2}(\mathrm{BC}_{N/F_2}(\mathrm{BC}_{F_2/F_1}(\Pi))\chi)||\mathrm{det}||_{F_2}^{\frac{n(n-1)}{2}}$. This is an $\iota$-ordinary cohomological cuspidal automorphic representation of $\mathrm{GL}_{n^2}(\mathbb{A}_{F_2})$ and we have $r_{\iota}(\Pi_1) \cong \tau_1$ by Proposition \ref{ordinary base change} and Corollary \ref{ordinary automorphic induction}.
    
    By using Theorem \ref{ordinary automorphy lifting} (when $U = \emptyset$), there exists an $\iota$-ordinary cohomological cuspidal automorphic representation $\Pi_2$ of $\mathrm{GL}_{n^2}(\mathbb{A}_{F_2})$ such that $\tau_2 \cong r_{\iota}(\Pi_2)$. Then we have $\iota \mathrm{WD}(r_{\iota}(\Pi_2)|_{G_{F_{2,u}}})^{F-ss} \cong \mathrm{rec}_{F_{2,u}}(\Pi_{2,u}|\mathrm{det}|_u^{\frac{1-n^2}{2}})$ for any $u \mid u'$ by Proposition \ref{purity local-global} and the purity of $\mathrm{WD}(\tau_2|_{G_{F_{2,u}}})$. By Theorem \ref{ordinary automorphy lifting} again (when $U = \{ u \mid u' \}$), there exists an $\iota$-ordinary cohomological cuspidal automorphic representation $\Pi_3$ of $\mathrm{GL}_{n^2}(\mathbb{A}_{F_2})$ such that $r_{\iota}(\Pi_3) \cong \tau_3$ and $\iota \mathrm{WD}(r_{\iota}(\Pi_3)|_{G_{F_{2,u}}})^{F-ss} \cong \mathrm{rec}_{F_{2,u}}(\Pi_{3,u}|\mathrm{det}|_u^{\frac{1-n^2}{2}})$ for any $u \mid u'$. By Lemma \ref{local global induction 2}, we obtain a cohomological cuspidal automorphic representation $\pi''$ of $\mathrm{GL}_{n}(\mathbb{A}_{N})$ such that $r_{\iota}(\pi'') \cong r|_{G_{N}}$ and $\iota \mathrm{WD}(r_{\iota}(\pi'')|_{G_{N_{u}}})^{F-ss} \cong \mathrm{rec}_{N_{u}}(\pi_{u}''|\mathrm{det}|_u^{\frac{1-n}{2}})$ for all $u \mid u'$. Thus we obtain a cohomological cuspidal automorphic representation $\pi$ of $\mathrm{GL}_{n}(\mathbb{A}_{F_2})$ such that $r_{\iota}(\pi) \cong r|_{G_{F_2}}$ and $\iota \mathrm{WD}(r_{\iota}(\pi)|_{G_{F_{2, u}}})^{F-ss} \cong \mathrm{rec}_{F_{2 ,u}}(\pi_{u}|\mathrm{det}|_u^{\frac{1-n}{2}})$ for any $u \mid u'$ by Proposition \ref{base change} and Proposition \ref{potential automorphy and local-global compatibility}. 

Finally, we will prove the $\iota$-ordinarity of $\pi$. Note that $\mathrm{BC}_{F_2/F_1}(\Pi)$ is an $\iota$-ordinary cohomological cuspidal automorphic representation of $\mathrm{GL}_n(\mathbb{A}_{F_2})$ and $\overline{r}|_{G_{F_2}} \cong \overline{r_{\iota}(\mathrm{BC}_{F_2/F_1}(\Pi))}$ by Propositions \ref{ordinary base change} and \ref{base change}. Therefore, we obtain an $\iota$-ordinary cohomological automorphic representation $\Pi'$ of $\mathrm{GL}_n(\mathbb{A}_{F_2})$ such that $r|_{G_{F_2}} \cong r_{\iota}(\Pi')$ by Theorem \ref{ordinary automorphy lifting} (when $U = \emptyset$). Then Propositions \ref{strong multiplicity one} induces $\Pi' \cong \pi$ and consequently we obtain the $\iota$-ordinarity of $\pi$. \end{proof}

\subsection{Potential automorphy and local-global compatibility in sufficiently regular weight potentially diagonalizable cases} \label{potentially diagonalizable}

\begin{lem}\label{diagonalizable induction character}

We assume that $F$ is an imaginary CM field.

(a) \ Let $l$ be an odd prime such that all $l$-adic places of $F^+$ split in $F$.
        
(b) \ Let $m$ be a positive integer. 

(c) \ Let $S$ be a finite set of finite places of $F$ not containing $2$-adic places.

(d) \ Let $\overline{r} : G_F \rightarrow \mathrm{GL}_n(\overline{\mathbb{F}_l})$ be a continuous representation such that $\overline{r}|_{G_{F(\zeta_l)}}$ is absolutely irreducible.

(e) \ Let $w_1$ and $w_2$ be integers.

(f) \ For any $v \mid l$, let $\rho_{1, v}, \rho_{2, v} : G_{F_v} \rightarrow \mathrm{GL}_m(\mathcal{O}_{\overline{\mathbb{Q}}_l})$ be residually trivial diagonalizable Hodge-Tate regular $l$-adic representations such that $\rho_{1, v^c}^c \sim \rho_{1, v}^{\vee} \varepsilon_l^{-w_1}$ and $\rho_{2, v^c}^c \sim \rho_{2, v}^{\vee} \varepsilon_l^{-w_2}$.

(g) \ Let $F^{(a)}$ be a finite extension of $F$.

Then there exist a cyclic extension $M/F$ of CM fields of degree $m$ linearly disjoint from $F^{(a)}$ over $F$ and algebraic $l$-adic characters $\phi_1, \phi_2 : G_{M} \rightarrow \mathcal{O}_{\overline{\mathbb{Q}}_l}^{\times}$ satisfying the following conditions.

1 \ All places in $S$ split completely in $M$.
        
2 \ $\overline{\phi_1} = \overline{\phi_2}$.

3 \ $(\overline{r} \otimes \overline{\mathrm{Ind}^{G_F}_{G_M}}\phi_1)|_{G_{F(\zeta_l)}}$ is absolutely irreducible.

4 \ For any $v|l$, $(\mathrm{Ind}^{G_{F}}_{G_{M}}\phi_1)|_{G_{F_v}} \sim \rho_{1, v}|_{G_{F_v}}$ and $(\mathrm{Ind}^{G_{F}}_{G_{M}}\phi_2)|_{G_{F'_v}} \sim \rho_{2, v}|_{G_{F_{v}}}$.

5 \ For any $i = 1, 2$, there exists an algebraic $l$-adic character $\chi_i : G_{F^+} \rightarrow \mathcal{O}_{\overline{\mathbb{Q}}_l}^{\times}$ such that $\chi_i(c_v) = \chi_i(c_w)$ for any $v, w \mid \infty$ and $\phi_i\phi_i^c = \chi_i|_{G_{M}}$.

6 \ $\overline{\phi_1}|_{G_{M_v}}$ is trivial for any finite place $v$ of $M$ lying above a place in $S$. 

7 \ $\phi_1|_{G_{M_v}}$ and $\phi_2|_{G_{M_v}}$ are unramified for any finite place $v$ of $M$ lying above a place in $S \setminus \{ u \mid l \}$.

\end{lem}

\begin{rem} Note that in the above situation, we have a perfect symmetric $G_F$-equivariant pairing $\langle \ \ , \ \ \rangle : \mathrm{Ind}^{G_{F}}_{G_M}\phi_i \times (\mathrm{Ind}^{G_{F}}_{G_M}\phi_i)^c \rightarrow \chi_i, \ (f, g) \mapsto \sum_{\sigma \in \mathrm{Gal}(M/F)} \chi_i^{-1}(\sigma)f(\sigma)g(c \sigma c)$.  \end{rem}

\begin{proof} By Lemma \ref{ordinary character induction}, there exist a cyclic extension $M/F$ of CM fields of degree $m$ and an algebraic $l$-adic character $\psi : G_{M} \rightarrow \mathcal{O}_{\overline{\mathbb{Q}}_l}^{\times}$ satisfying the following conditions.

$\cdot$ \ All places in $S \cup \{ v|l \}$ split completely in $M$.

$\cdot$ \ $\psi \psi^c$ is trivial.

$\cdot$ \ $(\overline{r} \otimes \overline{\mathrm{Ind}^{G_F}_{G_M}}\psi)|_{G_{F(\zeta_l)}}$ is absolutely irreducible.

$\cdot$ \ $\overline{\psi}|_{G_{M_v}}$ is trivial for any finite place $v$ of $M$ lying above a place in $S$. 

$\cdot$ \ $M$ is linearly disjoint from $F^{(a)}$ over $F$.

$\cdot$ \ $\psi|_{G_{M_v}}$ is unramified for any finite place $v$ of $M$ lying above a place in $S \setminus \{ u \mid l \}$.

Let $T$ be a set of $l$-adic places of $F$ such that $T \sqcup T^c$ is the set of all $l$-adic places of $F$. For any $v \in T$, we fix an $l$-adic place $w_v$ of $M$ lying above $v$. Then $w_v, w_v^{\sigma}, \cdots, w_v^{\sigma^{m-1}}$ are all $l$-adic places of $M$ lying above $v$, where $\sigma$ denotes the generator of $\mathrm{Gal}(M/F)$.

For $v \in T$ and any $i = 1, 2$, there exist crystalline characters $\chi_{i, v,1}, \cdots, \chi_{i, v,m} : G_{F_v} \rightarrow \mathcal{O}_{\overline{\mathbb{Q}}_l}$ such that $\rho_{i, v} \sim \chi_{i, v,1} \oplus \chi_{i,v,2} \oplus \cdots \oplus \chi_{i,v,m}$ since $\rho_{i, v}$ is diagonalizable. Then we have $\rho_{i, v^c} \sim (\chi_{i,v, 1}^{-c} \oplus \chi_{i,v, 2}^{-c} \oplus \cdots \oplus \chi_{i,v, m}^{-c} )\varepsilon_l^{-w_i}$ by the assumption $(f)$. By Lemma \ref{character}, there exists an algebraic $l$-adic character $\psi_i : G_{M} \rightarrow \mathcal{O}_{\overline{\mathbb{Q}}_l}^{\times}$ such that $\psi_i|_{G_{M_v}}$ is unramified for any finite place $v$ of $M$ lying above a place in $S \setminus \{ u \mid l \}$, $\psi_i \psi^c_i = \varepsilon_l^{-w_i}$ and $\psi_i|_{I_{M_{w_v^{\sigma^{j-1}}}}} = \chi_{i,v,j}|_{I_{F_v}}$ for any $v \in T$ if we identify $G_{F_v} = G_{M_{w_v^{\sigma^{j-1}}}}$.

Let $\phi_i := \psi_i \tilde{\psi_i}^{-1} \tilde{\psi}$, where $\tilde{\psi}_i$ (resp. $\tilde{\psi}$) denotes the Teichmuller lift of $\overline{\psi}_i$ (resp. $\overline{\psi}$). Then we obtain that $\overline{\phi}_1 = \overline{\phi}_2 = \overline{\psi}$ and since the considered residual representations are trivial at all $v \in T$, we have $(\mathrm{Ind}^{G_{F}}_{G_{M}}\phi_i)|_{G_{F_v}} = \phi_i|_{G_{M_{w_v}}} \oplus (\phi_i|_{G_{M_{w_v^{\sigma}}}})^{\sigma^{-1}} \oplus \cdots \oplus (\phi_i|_{G_{M_{w_v^{\sigma^{m-1}}}}})^{\sigma^{1-m}} \sim \chi_{i, v,1} \oplus \chi_{i, v,2} \oplus \cdots \oplus \chi_{i, v,m} \sim \rho_{i, v}$ and $(\mathrm{Ind}^{G_{F}}_{G_{M}}\phi_i)|_{G_{F_{v^c}}} = ((\phi_i|_{G_{M_{w_v}}})^{-c} \oplus (\phi_i|_{G_{M_{w_v^{\sigma}}}})^{-\sigma^{-1}c} \oplus \cdots \oplus (\phi_i|_{G_{M_{w_v^{\sigma^{m-1}}}}})^{-\sigma^{1-m}c}) \varepsilon_l^{-w_i} \sim (\chi_{i, v, 1}^{-c} \oplus \chi_{i, v, 2}^{-c} \oplus \cdots \oplus \chi_{i, v, m}^{-c} )\varepsilon_l^{-w_i} \sim \rho_{i, v^c}$ by 2 and 6 of Lemma \ref{p=l}. Therefore, we obtain $(\mathrm{Ind}^{G_{F}}_{G_{M}}\phi_i)|_{G_{F_v}} \sim \rho_{i, v}$ and $(\mathrm{Ind}^{G_{F}}_{G_{M}}\phi_i)|_{G_{F_{v^c}}} \sim \rho_{i, v^c}$ by 1 of Lemma \ref{p=l}. \end{proof}

\begin{thm}\label{potential tensor automorphy} Let $l > 8n^3$ be a prime, $r: G_F \rightarrow \mathrm{GL}_n(\overline{\mathbb{Q}}_l)$ be an algebraic $l$-adic representation, $F^{(a)}$ be a finite extension of $\mathbb{Q}$ and $U$ be a finite set of primes not containing $l$. We assume the following conditions.
        
1 \ $F \nsubseteq F^+(\zeta_l)$ or $F$ is a totally real field.
        
2 \ $\overline{r}$ is fully strongly decomposed generic.
        
3 \ $\overline{r}|_{G_{F(\zeta_l)}}$ is absolutely irreducible.
        
4 \ For all $v|l$, $r|_{G_{F_v}}$ is potentially diagonalizable and Hodge-Tate regular.
        
5 \ There exists an integer $w$ such that for all $v|l$, there exists a finite extension $F_v'/F_v$ such that $r^c|_{G_{F'_v}} \sim (r^{\vee}\varepsilon_l^{-w}) |_{G_{F_v'}}$. 
    
6 \ For all $\tau \in \mathrm{Hom}(F, \overline{\mathbb{Q}}_l)$, $\lambda \neq \lambda' \in \mathrm{HT}_{\tau}(r)$, we have $|\lambda - \lambda'| \ge n$.
        
Then for any $\iota: \overline{\mathbb{Q}}_l \stackrel{\sim}{\longrightarrow} \mathbb{C}$, there exist a finite Galois CM extension $L/\mathbb{Q}$ linearly disjoint from $F^{(a)}$ over $\mathbb{Q}$, a cohomological cuspidal automorphic representation $\Pi$ of $\mathrm{GL}_{2n^2}(\mathbb{A}_{F'})$ and a polarizable $\iota$-ordinary cohomological cuspidal automorphic representation $\pi$ of $\mathrm{GL}_{2n}(\mathbb{A}_{F'})$ satisfying the following conditions. (We put $F':=LF$.)
        
(a) \ $\overline{r_{\iota}(\pi) \otimes r|_{G_{F'}}}$ is absolutely irreducible and decomposed generic.
    
(b) \ $\pi$ is unramified at all places lying above primes in $U$.

(c) \ $r_{\iota}(\pi) \otimes r|_{G_{F'}} \cong r_{\iota}(\Pi)$.
    
(d) \ $\iota\mathrm{WD}(r_{\iota}(\Pi)|_{G_{F'_{v}}})^{F-ss} \cong \mathrm{rec}_{F'_{v}}(\Pi_v|\mathrm{det}|_v^{\frac{1-2n^2}{2}})$ for all $v \nmid l$.

\end{thm}

Before starting the proof of Theorem \ref{potential tensor automorphy}, we give a rough strategy and a key point of the proof. The idea of the proof is similar to Theorem \ref{potential ordinary automorphy}. By Theorem \ref{ordinary residual potential automorphy}, there exist a finite CM extension $F'/F$ and a cohomological cuspidal automorphic representation $\Pi$ of $\mathrm{GL}_n(\mathbb{A}_{F'})$ such that $\overline{r}|_{G_{F'}} \cong \overline{r_{\iota}(\Pi)}$. By Proposition \ref{character} and Lemma \ref{diagonalizable induction character}, we can take an appropriate quadratic CM extension $E$ of $F^{'+}$, a certain cyclic extension $M/E$ of degree $2n$ and certain characters $\theta_1, \theta_2 : G_{F'} \rightarrow \overline{\mathbb{Q}}_l^{\times}$ and $\phi_1, \phi_2 : G_{M} \rightarrow \overline{\mathbb{Q}}_l^{\times}$ satisfying the following conditions. \footnote{For simplicity, we ignore some technical properties such as the residual irreducibility and the decomposed genericity.}

$\cdot$ \ $\overline{\theta}_1 = \overline{\theta}_2$ and $\overline{\phi}_1 = \overline{\phi}_2$.

$\cdot$ \ $\mathrm{Ind}^{G_{E}}_{G_M}\phi_i$ is polarizable for any $i = 1, 2$.

$\cdot$ \ For any $w \mid l$, we have $(\mathrm{Ind}^{G_{E}}_{G_{M}}\phi_1)|_{G_{E_w}} \sim (\mathrm{Ind}^{G_{F'^{+}}}_{G_{F'}}r|_{G_{F'}} \theta_1)|_{G_{E_w}}$ and $(\mathrm{Ind}^{G_{E}}_{G_{M}}\phi_2)|_{G_{E_w}} \sim (\mathrm{Ind}^{G_{F'^{+}}}_{G_{F'}}r_{\iota}(\Pi)\theta_2)|_{G_{E_w}}$

$\cdot$ \ $(\mathrm{Ind}^{G_{F'^{+}}}_{G_{F'}}r|_{G_{F'}}\theta_2)|_{G_{E}} \otimes (\mathrm{Ind}^{G_{E}}_{G_{M}}\phi_2)$ is Hodge-Tate regular. Note that we need the condition 6 for this property because the Hodge-Tate weights of $r_{\iota}(\Pi)$ are consecutive.

We basically compare the two representations $\gamma_1 := (\mathrm{Ind}^{G_{F'^{+}}}_{G_{F'}}r|_{G_{F'}}\theta_1)|_{G_{E}} \otimes (\mathrm{Ind}^{G_{E}}_{G_{M}}\phi_2)$ and $\gamma_2 := (\mathrm{Ind}^{G_{F'^{+}}}_{G_{F'}}r_{\iota}(\Pi)\theta_2)|_{G_{E}} \otimes (\mathrm{Ind}^{G_{E}}_{G_{M}}\phi_1)$. Note that we have $\overline{\gamma}_1 \cong \overline{\gamma}_2$ and $\gamma_1|_{G_{E_v}} \sim \gamma_2|_{G_{E_v}}$ for all $v \mid l$, but $\mathrm{WD}(\gamma_1|_{G_{E_v}})$ and $\mathrm{WD}(\gamma_2|_{G_{E_v}})$ don't have the same monodromy type for all $v \nmid l$ in general.

In order to apply Theorem \ref{automorphy lifting theorem in crystalline cases} to these representations, one idea is that as in the proof of Theorem \ref{potential ordinary automorphy}, by using Propositions \ref{potential conjugate} and \ref{conjugate self-dual ordinary} and replacing $E$ by a suitable solvable CM extension of $E$, we deform $\mathrm{Ind}^{G_{E}}_{G_{M}}\phi_1$ and $\mathrm{Ind}^{G_{E}}_{G_{M}}\phi_2$ to conjugate self-dual automorphic representations $r_{\iota}(\pi_1)$ and $r_{\iota}(\pi_2)$ satisfying the following conditions.

$\cdot$ \ $r_{\iota}(\pi_i)|_{G_{E_v}} \sim (\mathrm{Ind}^{G_{E}}_{G_M}\phi_i)|_{G_{E_v}}$ for all $v \mid l$ and any $i = 1, 2$.

$\cdot$ \ $\mathrm{WD}(r_{\iota}(\pi_1)|_{G_{E_v}})$ and $\mathrm{WD}((\mathrm{Ind}^{G_{F'^{+}}}_{G_{F'}}r|_{G_{F'}}\theta_1)|_{G_{E_{v}}})$ (resp. $\mathrm{WD}(r_{\iota}(\pi_2)|_{G_{E_v}})$ and $\mathrm{WD}((\mathrm{Ind}^{G_{F'^{+}}}_{G_{F'}}r_{\iota}(\Pi)\theta_2)|_{G_{E_{v}}})$) have the same monodromy type for all $v \nmid l$. 

However, we have the following problem to apply Theorem \ref{automorphy lifting theorem in crystalline cases} to the representations $(\mathrm{Ind}^{G_{F'^{+}}}_{G_{F'}}r|_{G_{F'}}\theta_1)|_{G_{E}} \otimes r_{\iota}(\pi_2)$ and $(\mathrm{Ind}^{G_{F'^{+}}}_{G_{F'}}r_{\iota}(\Pi)\theta_2)|_{G_{E}} \otimes r_{\iota}(\pi_1)$.  

\vspace{0.5 \baselineskip}

\textbf{Problem} The automorphy lifting theorems in this paper don't suffice to prove the automorphy of $(\mathrm{Ind}^{G_{F'^{+}}}_{G_{F'}}r_{\iota}(\Pi)\theta_2)|_{G_{E}} \otimes r_{\iota}(\pi_1)$ from the automorphy of $(\mathrm{Ind}^{G_{F'^{+}}}_{G_{F'}}r_{\iota}(\Pi)\theta_2)|_{G_{E}} \otimes \mathrm{Ind}^{G_{E}}_{G_{M}} \phi_1$. In fact, Theorem \ref{automorphy lifting theorem in crystalline cases} doesn't work because these representations may have different monodromies at some non-$l$-adic places and Theorem \ref{potential ordinary automorphy} doesn't work because these representations may not be ordinary. 

\vspace{0.5 \baselineskip}

The above problem can be solved by taking one more tensor product as in the following way. We consider a polarizable automorphic Galois representation $\mathrm{Ind}^{G_{E}}_{G_M}\psi$ which is ordinary and has a sufficiently regular weight. As above $\mathrm{Ind}^{G_{E}}_{G_M}\phi_i$ and $r_{\iota}(\pi_i)$, we can take conjugate self-dual automorphic deformations $r_{\iota}(\pi_i')$ ($i = 1, 2$) of $\mathrm{Ind}^{G_{E}}_{G_M}\psi$ such that $r_{\iota}(\pi_i')|_{G_{E_v}} \sim (\mathrm{Ind}^{G_{E}}_{G_M}\psi)|_{G_{E_v}}$ for all $v \mid l$ and any $i = 1, 2$ and $\mathrm{WD}(r_{\iota}(\pi_1')|_{G_{E_v}})$ and $\mathrm{WD}((\mathrm{Ind}^{G_{F'^{+}}}_{G_{F'}}r|_{G_{F'}}\theta_1)|_{G_{E_{v}}})$ (resp. $\mathrm{WD}(r_{\iota}(\pi_2')|_{G_{E_v}})$ and $\mathrm{WD}((\mathrm{Ind}^{G_{F'^{+}}}_{G_{F'}}r_{\iota}(\Pi)\theta_2)|_{G_{E_{v}}})$) have the same monodromy type for all $v \nmid l$.

We claim that $\tau_1 := \gamma_2 \otimes r_{\iota}(\pi_1') = (\mathrm{Ind}^{G_{F'^{+}}}_{G_{F'}}r_{\iota}(\Pi)\theta_2)|_{G_{E}} \otimes (\mathrm{Ind}^{G_{E}}_{G_{M}}\phi_1) \otimes r_{\iota}(\pi_1')$ is automorphic and $\tau_1$ and $\tau_2 := \gamma_1 \otimes r_{\iota}(\pi_2') = (\mathrm{Ind}^{G_{F'^{+}}}_{G_{F'}}r|_{G_{F'}}\theta_1)|_{G_{E}} \otimes (\mathrm{Ind}^{G_{E}}_{G_{M}}\phi_2) \otimes r_{\iota}(\pi_2')$ satisfy all conditions of Theorem \ref{automorphy lifting theorem in crystalline cases}. First note that the representation $(\mathrm{Ind}^{G_{F'^{+}}}_{G_{F'}}r_{\iota}(\Pi)\theta_2)|_{G_{E}} \otimes r_{\iota}(\pi_1')$ is automorphic by the automorphy of $(\mathrm{Ind}^{G_{F'^{+}}}_{G_{F'}}r_{\iota}(\Pi)\theta_2)|_{G_{E}} \otimes \mathrm{Ind}^{G_{E}}_{G_M}\psi$ and the ordinary automorphy lifting theorem (see Theorem \ref{potential ordinary automorphy}). Thus we obtain the automorphy of $\tau_1$ by Proposition \ref{automorphic induction}. Note that $\mathrm{WD}(\tau_1|_{G_{F_v}})$ is pure for any finite place $v$ and thus we obtain the local-global compatibility assumption 5 (4) of Theorem \ref{automorphy lifting theorem in crystalline cases}. We can easily check other conditions of Theorem \ref{automorphy lifting theorem in crystalline cases}. Thus we obtain the automorphy of $\tau_2$ and the local-global compatibility of the corresponding automorphic representation at non-$l$-adic places by Theorem \ref{automorphy lifting theorem in crystalline cases} and from this, we can easily obtain the result.

\vspace{0.5 \baselineskip}

\textbf{Proof of Theorem \ref{potential tensor automorphy}} We may assume that $F$ is an imaginary CM field such that all $l$-adic places of $F^+$ split in $F$ by Corollary \ref{imaginary quadratic 2} and Lemma \ref{disjoint CM} and $F^{(a)}$ contains the Galois closure of $\overline{F}^{\mathrm{Ker}(\overline{r})}(\zeta_l)$ over $\mathbb{Q}$.

By Theorem \ref{ordinary residual potential automorphy}, there exist a finite CM Galois extension of $L/\mathbb{Q}$ linearly disjoint from the Galois closure of $F^{(a)}$ over $\mathbb{Q}$ and an $\iota$-ordinary cohomological cuspidal automorphic representation $\Pi$ of $\mathrm{GL}_n(\mathbb{A}_{F_1})$ satisfying the following conditions. (We put $F_1:=LF$.)
    
$(1')$ \ $\overline{r_{\iota}(\Pi)} \cong \overline{r}|_{G_{F_1}}$.
    
$(2')$ \ $\mathrm{WD}(r_{\iota}(\Pi)|_{G_{F_{1,v}}})$ is pure for all finite places $v$ of $F_1$.
    
$(3')$ \ For any $v|l$, $\Pi_v$ is unramified and $r_{\iota}(\Pi)|_{G_{F_{1,v}}}$ is crystalline ordinary.
    
$(4')$ If $r$ is ramified above a prime $q \neq l$, then $\Pi_u$ is unramified for all $u|q$. 
    
$(5')$ $\Pi_u$ is unramified at all places lying above primes in $U$.

$(6')$ $\mathrm{HT}_{\tau}(r_{\iota}(\Pi)) = \{ n-1, n -2, \cdots, 0 \}$ for all $\tau \in \mathrm{Hom}(F_1, \overline{\mathbb{Q}}_l)$.

Therefore, $\overline{r}|_{G_{F_1}}$ is fully strongly decomposed generic by Lemma \ref{decomposed genericity density}, $\overline{r}|_{G_{F_1(\zeta_l)}}$ is absolutely irreducible and $F_1 \nsubseteq F_1^+(\zeta_l)$ by Lemma \ref{disjoint CM}. By Proposition \ref{semistable ordinary deformation ring} and by the properties $(3')$ and $(6')$, for any $v|l$, there exists a finite extension $F_{v}'/F_v$ such that $r_{\iota}(\Pi)^c|_{G_{F_v'}} \sim (r_{\iota}(\Pi)^{\vee} \varepsilon_l^{-(n-1)})|_{G_{F_v'}}$. By Lemma \ref{decomposed genericity density}, we can take an odd prime $p \neq l$ which is fully strongly decomposed generic for $\overline{r}|_{G_{F_1}}$ and $r|_{G_{F_1}}$, $r_{\iota}(\Pi)$ are unramified above $p$. There exists a finite place $w$ of $F_1$ such that $w$ splits over $F_1^+$, $\overline{r}|_{G_{F_1}}$ is unramifed at $w$, $w^c$ and $w$ doesn't divide $2pl$. By Lemma \ref{quadratic extension}, we can construct a quadratic extension $F_1'$ of $F_1$ such that $F_1'/F_1$ splits at all $v|p$, $F_1'/F_1$ ramifies at $w$ and $F_1'/F_1$ is unramified at $w^c$. After twisting $r|_{G_{F_1}}$ and $r_{\iota}(\Pi)$ by the corresponding quadratic character, we may assume $\overline{r}|_{G_{F_1(\zeta_l)}} \ncong \overline{r}^c|_{G_{F_1(\zeta_l)}}$. (Note that we still have that $p$ is fully strongly decomposed generic for $\overline{r}|_{G_{F_1}}$.) Therefore, (Ind$_{G_{F_1}}^{G_{F_1^+}}\overline{r}|_{G_{F_1}})|_{G_{F^+(\zeta_l)}} = \mathrm{Ind}^{G_{F_1^+(\zeta_l)}}_{G_{F_1(\zeta_l)}}(\overline{r}|_{G_{F_1(\zeta_l)}})$ is absolutely irreducible. (We use $F_1 \nsubseteq F_1^+(\zeta_l)$.) Since $p$ is fully strongly decomposed generic for $\overline{r}$, $p$ is fully decomposed generic for Ind$_{G_{F_1}}^{G_{F_1^+}}\overline{r}$.

By Proposition \ref{character}, we can take algebraic $l$-adic characters $\theta, \theta': G_{F_1} \rightarrow \mathcal{O}_{\overline{\mathbb{Q}}_l}^{\times}$ satisfying the following properties. (We put $\sigma_1:= $Ind$^{G_{F_1^+}}_{G_{F_1}} (r|_{G_{F_1}}\theta)$ and $\sigma_2:= $Ind$^{G_{F_1^+}}_{G_{F_1}} (r_{\iota}(\Pi)\theta')$.)
    
$\cdot$ \ $\theta \theta^c$ and $\theta' \theta^{'c}$ are trivial.
    
$\cdot$ \ $\overline{\theta}$ and $\overline{\theta'}$ are trivial.
    
    $\cdot$ \ For all $v \nmid l$, $\theta|_{G_{F_{1,v}}}$ and $\theta'|_{G_{F_{1,v}}}$ are unramified and for all $v \mid l$, $\theta|_{G_{F_{1,v}}}$ and $\theta'|_{G_{F_{1,v}}}$ are crystalline. \footnote{Note that for any algebraic $l$-adic character $\chi$, there exists $m >> 0$ such that $\chi^m$ satisfies this property.}

    $\cdot$ \ $\sigma_1 \otimes \sigma_2$ is Hodge-Tate regular and $\sigma_2$ is ordinary. (We use the assumption 6 and the property $(6')$.)
    
Note that $p$ is fully decomposed generic for $\overline{\sigma_1}$ and $\overline{\sigma_1}|_{G_{F^+_1(\zeta_l)}}$ is absolutely irreducible. Let $S$ be the set of all primes lying below not potentially unramified places of $\sigma_1$. Let $T$ be the set of all primes lying below not potentially unramified places of $\sigma_2$. Then we have $S \cap T = \{ l \}$ and $p \notin S \cup T$. By Corollary \ref{extension of Q}, Lemma \ref{induction conjugate} and Corollary \ref{imaginary quadratic 2}, there exists a finite solvable CM extension $L_1/\mathbb{Q}$ satisfying the following conditions. (We put $F_2:=L_1F_1$.)
    
$\cdot$ \ $L_1$ is linearly disjoint from $F^{(a)}\overline{F_1}^{\mathrm{Ker}(\overline{\sigma}_1)}$ over $\mathbb{Q}$.

$\cdot$  \ There exists an imaginary quadratic field $E_1$ contained in $L_1$ such that all $q \in S \cup T$ split in $E_1$.

$\cdot$ \ $p$ splits completely in $L_1$.

$\cdot$ \ For all $v|l$, $\sigma_1|_{G_{2,v}}$ is diagonalizable, the representations $\overline{\sigma_1}|_{G_{F_{2,v}}}$ and $\overline{\sigma_2}|_{G_{F_{2,v}}}$ are trivial, $\sigma_1^c|_{G_{F_{2,v}}} = \sigma_1|_{G_{F_{2,v}}} \sim (\sigma_1^{\vee} \varepsilon_l^{-w})|_{G_{F_{2,v}}}$ and $\sigma_2^c|_{G_{F_{2,v}}} = \sigma_2|_{G_{F_{2,v}}} \sim (\sigma_2^{\vee} \varepsilon_l^{-(n-1)})|_{G_{F_{2,v}}}$.

$\cdot$ \ For all $v$ not lying above $S$, $\sigma_1|_{G_{F_{2,v}}}$ is unramified.

$\cdot$ \ For all $v$ not lying above $T$, $\sigma_2|_{G_{F_{2,v}}}$ is unramified.

$\cdot$ \ For all $v$ lying above $S \cup T \setminus \{ l \}$, the representations $\sigma_1|_{G_{F_{2,v}}}$ and $\sigma_2|_{G_{F_{2,v}}}$ are unipotently ramified and $\overline{\sigma_1}|_{G_{F_{2,v}}}$ and $\overline{\sigma_2}|_{G_{F_{2,v}}}$ are trivial. 

By Lemma \ref{diagonalizable induction character}, we can construct a cyclic CM extension $N/F_2$ of degree $2n$ and algebraic $l$-adic characters $\phi_1, \phi_2 : G_{N} \rightarrow \mathcal{O}_{\overline{\mathbb{Q}}_l}^{\times}$ satisfying the following conditions.

$\cdot$ \ $N$ is linearly disjoint from $MF^{(a)}\overline{F_2}^{(\mathrm{Ker}(\overline{\sigma_1}|_{G_{F_2}}))}$ over $F_2$.
    
$\cdot$ \ $\overline{\phi}_1 = \overline{\phi}_2$.

$\cdot$ \ There exists an algebraic $l$-adic character $\chi_i : G_{F^+_2} \rightarrow \mathcal{O}_{\overline{\mathbb{Q}}_l}^{\times}$ such that $\chi_i(c_v) = \chi_i(c_w)$ for any $v, w \mid \infty$ and $\phi_i\phi_i^c = \chi_i|_{G_{N}}$.

$\cdot$ \ $\overline{\sigma_1|_{G_{F_2}} \otimes \mathrm{Ind}_{G_{N}}^{G_{F_2}}\phi_1}|_{G_{F_2(\zeta_l)}}$ is absolutely irreducible.

$\cdot$ \ For any $v|l$, $\sigma_i|_{G_{F_{3,v}}} \sim (\mathrm{Ind}^{G_{F_2}}_{G_N}\phi_i)|_{G_{F_{2,v}}}$.

$\cdot$ \ $p$ and $l$ split completely in $N$.

$\cdot$ \ $\overline{\phi_1}|_{G_{N_v}}$ is trivial for any $v|p$.

$\cdot$ \ $\phi_1|_{G_{N_v}}$ and $\phi_2|_{G_{N_v}}$ are unramified for any $v|p$.

Note that then $p$ is fully decomposed generic for $\overline{\sigma_1 \otimes \mathrm{Ind}^{G_{F_2}}_{G_N}\phi_1}$.

By Lemma \ref{ordinary character induction}, there exist a cyclic CM extension $M/F_2$ of degree $2n$ and an algebraic $l$-adic character $\psi : G_{M} \rightarrow \mathcal{O}_{\overline{\mathbb{Q}}_l}^{\times}$ satisfying the following conditions.

$(a)$ \ $p$ and $l$ split completely in $M$.

$(b)$ \ $\overline{\sigma_1 \otimes \mathrm{Ind}^{G_{F_2}}_{G_N} \phi_1 \otimes \mathrm{Ind}^{G_{F_2}}_{G_M}\psi}|_{G_{F_2(\zeta_l)}}$ is absolutely irreducible.

$(c)$ \ $\mathrm{Ind}^{G_{F_2}}_{G_M}\psi$ and $\sigma_2|_{G_{F_2}} \otimes \mathrm{Ind}^{G_{F_2}}_{G_M}\psi$ are crystalline ordinary and $\sigma_1|_{G_{F_2}} \otimes \sigma_2|_{G_{F_2}} \otimes \mathrm{Ind}^{G_{F_2}}_{G_M}\psi$ is Hodge-Tate regular at all $l$-adic places.

$(d)$ \ $\psi \psi^c$ is trivial.

$(e)$ \ $\overline{\psi}|_{G_{M_v}}$ is trivial for any $v|p$.

$(f)$ \ $\psi|_{G_{M_v}}$ is unramified for any $v|p$.

By the conditions $(a)$ and $(e)$, the prime $p$ is decomposed generic for $\overline{\sigma_1 \otimes \mathrm{Ind}^{G_{F_2}}_{G_N} \phi_1 \otimes \mathrm{Ind}^{G_{F_2}}_{G_M}\psi}$.

By Corollary \ref{extension of Q} and Lemma \ref{induction conjugate}, there exists a finite solvable totally real extension $L_2/\mathbb{Q}$ satisfying the following conditions. (We put $F_3:=L_2F_2$.)
    
$\cdot$ \ $L_2$ is linearly disjoint from $F^{(a)}\overline{F_2}^{\mathrm{Ker}(\overline{\sigma_1}|_{G_{F_2}}) \cap \mathrm{Ker}(\mathrm{Ind}^{G_{F_2}}_{G_N} \overline{\phi}_1) \cap \mathrm{Ker}(\mathrm{Ind}^{G_{F_2}}_{G_M}\overline{\psi})}$ over $\mathbb{Q}$.

$\cdot$ \ $p$ splits completely in $L_2$.

$\cdot$ \ For all $v \nmid l$, $(\mathrm{Ind}^{G_{F_2}}_{G_{M}}\psi)|_{G_{F_{3,v}}}$, $(\mathrm{Ind}^{G_{F_2}}_{G_{N}}\phi_1)|_{G_{F_{3,v}}}$ and $(\mathrm{Ind}^{G_{F_2}}_{G_{N}}\phi_2)|_{G_{F_{3,v}}}$ are unramified.

$\cdot$ \ For all finite place $v$ of $F_3$ lying above a prime in $S \cup T$, $\overline{\mathrm{Ind}^{G_{F_2}}_{G_{M}}\psi}|_{G_{F_{3,v}}}$, $\overline{\mathrm{Ind}^{G_{F_2}}_{G_{N}}\phi_1}|_{G_{F_{3,v}}}$ and $\overline{\mathrm{Ind}^{G_{F_2}}_{G_{N}}\phi_2}|_{G_{F_{3,v}}}$ are trivial.

We put $M_1:=MF_3$ and $N_1:=NF_3$. By Proposition \ref{conjugate self-dual ordinary}, $l \nmid 2n$ and $\zeta_l \notin F_3$, there exist polarizable $\iota$-ordinary cohomological cuspidal automorphic representations $\pi_1, \pi_2$ of $\mathrm{GL}_{2n}(\mathbb{A}_{F_3})$ satisfying the following conditions. 

$\cdot$ \ $\overline{r_{\iota}(\pi_1)} \cong \overline{r_{\iota}(\pi_2)} \cong \overline{\mathrm{Ind}^{G_{F_3}}_{G_{M_1}}\psi|_{G_{M_1}}}$.

$\cdot$ \ $r_{\iota}(\pi_1)^c = r_{\iota}(\pi_1)^{\vee}$ and $r_{\iota}(\pi_2)^c = r_{\iota}(\pi_2)^{\vee}$.

$\cdot$ \ For all $v \nmid l$, $r_{\iota}(\pi_1)|_{G_{F_{3,v}}}$ and $\sigma_1|_{G_{F_{3,v}}}$ (resp. $r_{\iota}(\pi_2)|_{G_{F_{3,v}}}$ and $\sigma_2|_{G_{F_{3,v}}}$) have the same monodromy type.

$\cdot$ \ For all $v|l$, $r_{\iota}(\pi_1)|_{G_{F_{3,v}}} \sim r_{\iota}(\pi_2)|_{G_{F_{3,v}}} \sim (\mathrm{Ind}^{G_{F_3}}_{G_{M_1}}\psi|_{G_{M_1}})|_{G_{F_{3,v}}}$.


We put $\tau_1 := \sigma_2|_{G_{F_3}} \otimes \mathrm{Ind}^{G_{F_3}}_{G_{N_1}}\phi_1|_{G_{N_1}} \otimes r_{\iota}(\pi_1)$, $\tau_2:=\sigma_1|_{G_{F_3}} \otimes \mathrm{Ind}^{G_{F_3}}_{G_{N_1}}\phi_2|_{G_{N_1}} \otimes r_{\iota}(\pi_2)$, $\rho_1 := \sigma_2|_{G_{F_3}} \otimes \mathrm{Ind}^{G_{F_3}}_{G_{M_1}}\psi|_{G_{M_1}}$ and $\rho_2 := \sigma_2|_{G_{F_3}} \otimes r_{\iota}(\pi_1)$. By the above constructions, we have the following properties.
        
$\cdot$ \ $\zeta_l \notin F_3$. 
        
$\cdot$ \ $\overline{\tau_1} \cong \overline{\tau_2}$ and $\overline{\rho_1} \cong \overline{\rho_2}$.
        
$\cdot$ \ $\overline{\tau_1}|_{G_{F_3(\zeta_l)}}$ and $\overline{\rho_1}|_{G_{F_3(\zeta_l)}}$ are absolutely irreducible.
        
$\cdot$ \ $\overline{\tau_1}$ and $\overline{\rho_1}$ are decomposed generic.
        
$\cdot$ \  $\overline{\tau_1}(G_{F_3(\zeta_l)})$ and $\overline{\rho_1}(G_{F_3(\zeta_l)})$ are adequate by Proposition \ref{adequate}. (We use $l \nmid 2n$ and $l \ge 2(n+1)$.)
    
$\cdot$ \ $\tau_1|_{G_{F_{3,v}}} \sim \tau_2|_{G_{F_{3,v}}}$ for all $v|l$. (We use 6 of Lemma \ref{p=l}.)
        
$\cdot$ \ $\rho_1|_{G_{F_{3,v}}}$ and $\rho_2|_{G_{F_{3,v}}}$ are crystalline ordinary for all $v|l$.

$\cdot$ \ $\tau_1|_{G_{F_{3,v}}}$ and $\tau_2|_{G_{F_{3,v}}}$ have the same monodromy type for all $v \nmid l$.
    
$\cdot$ \ $\mathrm{WD}(\tau_1|_{G_{F_{3,v}}})$ and $\mathrm{WD}(\rho_2|_{G_{F_{3,v}}})$ are pure for all finite places $v$ of $F_4$. (We use 5 of Lemma \ref{purity lemma}.)

By Proposition \ref{ordinary automorphic induction} and Proposition \ref{ordinary base change}, we have an $\iota$-ordinary cohomological cuspidal automorphic representation $\pi_1'$ of $\mathrm{GL}_{4n^2}(\mathbb{A}_{F_3})$ such that $r_{\iota}(\pi_1') \cong \rho_1$. By Theorem \ref{ordinary automorphy lifting} (in the case $U := \emptyset$), there exists an $\iota$-ordinary cohomological cuspidal automorphic representation $\pi_2'$ of $\mathrm{GL}_{4n^2}(\mathbb{A}_{F_3})$ such that $\rho_2 \cong r_{\iota}(\pi_2')$. Since $\mathrm{WD}(r_{\iota}(\pi_2')|_{G_{F_{3,v}}})$ is pure and $\iota \mathrm{WD}(r_{\iota}(\pi_2')|_{G_{F_{3,v}}})^{ss} \cong \mathrm{rec}_{F_{3,v}}(\pi_{2,v}'|\mathrm{det}|_{v}^{\frac{1-4n^2}{2}})^{ss}$ for all $v \mid l$ by Theorem \ref{ordinary semisimple}, we obtain $$\iota \mathrm{WD}(r_{\iota}(\pi_2')|_{G_{F_{3,v}}})^{F-ss} \cong \mathrm{rec}_{F_{3,v}}(\pi_{2,v}'|\mathrm{det}|_{v}^{\frac{1-4n^2}{2}})$$ for all $v \mid l$ by the same argument as Proposition \ref{purity local-global} and consequently $\pi_{2,v}'$ is unramified for all $v \mid l$. By Proposition \ref{automorphic induction}, there exists a cohomological cuspidal automorphic representation $\Pi_1$ of $\mathrm{GL}_{8n^3}(\mathbb{A}_{F_3})$ such that $r_{\iota}(\Pi_1) \cong \tau_1 ( = \rho_2 \otimes \mathrm{Ind}^{G_{F_3}}_{G_{N_1}}\phi_1|_{G_{N_1}})$ and $\Pi_{1,v}$ is unramified for all $v \mid l$. For all $w \nmid l$, we have $\iota \mathrm{WD}(r_{\iota}(\Pi_1)|_{G_{F_{3,w}}})^{F-ss} \cong \mathrm{rec}_{F_{3,w}}(\Pi_{1,w}|\mathrm{det}|_w^{\frac{1-8n^3}{2}})$ by Proposition \ref{purity local-global} and the fact that $\mathrm{WD}(\tau_1|_{G_{F_{3,w}}})$ is pure. By Theorem \ref{automorphy lifting theorem in crystalline cases}, there exists a cohomological cuspidal automorphic representation $\Pi_2$ of $\mathrm{GL}_{8n^3}(\mathbb{A}_{F_3})$ such that $\tau_2 \cong r_{\iota}(\Pi_2)$ and $\iota \mathrm{WD}(r_{\iota}(\Pi_2)|_{G_{F_{3,w}}})^{F-ss} \cong \mathrm{rec}_{F_{3,w}}(\Pi_{2,w}|\mathrm{det}|_w^{\frac{1-8n^3}{2}})$ for all $w \nmid l$.

By Lemma \ref{local global induction 2}, we obtain a cohomological cuspidal automorphic representation $\Pi_3$ of $\mathrm{GL}_{2n^2}(\mathbb{A}_{FN_1})$ such that $r_{\iota}(\mathrm{BC}_{FN_1/F_3}(\pi_2)) \otimes r|_{G_{FN_1}} \cong r_{\iota}(\Pi_3)$ and $\iota \mathrm{WD}(r_{\iota}(\Pi_3)|_{G_{(FN_1)_{w}}})^{F-ss} \cong \mathrm{rec}_{(FN_1)_{w}}(\Pi_{3,w}|\mathrm{det}|_w^{\frac{1-2n^2}{2}})$ for all $w \nmid l$. By Proposition \ref{base change} and Proposition \ref{potential automorphy and local-global compatibility}, we obtain a cohomological cuspidal automorphic representation $\Pi_4$ of $\mathrm{GL}_{2n^2}(\mathbb{A}_{FF_3})$ such that $r_{\iota}(\mathrm{BC}_{FF_3/F_3}(\pi_2)) \otimes r|_{G_{FF_3}} \cong r_{\iota}(\Pi_4)$ and $$\iota \mathrm{WD}(r_{\iota}(\Pi_4)|_{G_{(FF_3)_{w}}})^{F-ss} \cong \mathrm{rec}_{(FF_3)_{w}}(\Pi_{4,w}|\mathrm{det}|_w^{\frac{1-2n^2}{2}})$$ for all $w \nmid l$.  \qed

\vspace{0.5 \baselineskip}

If the following conjecture holds, we can remove the assumption 6 in Theorem \ref{potential tensor automorphy}. (Note that we already have the local-global compatibility at non-$l$-adic places and the local-global compatibility up to semisimplification at $l$-adic places in ordinary cases. See Theorem \ref{potential ordinary automorphy} and Theorem \ref{ordinary semisimple}.)

\begin{conj}

Let $\Pi$ be an $\iota$-ordinary cohomological cuspidal automorphic representation of $\mathrm{GL}_n(\mathbb{A}_{F})$ such that $\overline{r_{\iota}(\Pi)}$ is absolutely irreducible and decomposed generic, $S$ be a finite set of non-$l$-adic finite places of $F$ such that $\Pi_v$ is unramified for all $v \in S$, $k$ be a positive integer and $F^{(a)}$ be a finite extension of $\mathbb{Q}$.

Then there exist a solvable CM extension $L$ of $\mathbb{Q}$ and an $\iota$-ordinary cohomological cuspidal automorphic representation $\pi$ of $\mathrm{GL}_n(\mathbb{A}_{F'})$ of weight $\lambda$ satisfying the following conditions. (We put $F':=FL$.)

1 \ $L$ is linearly disjoint from $F^{(a)}$ over $\mathbb{Q}$.

2 \ $\overline{r_{\iota}(\Pi)}|_{G_{F'}} \cong \overline{r_{\iota}(\pi)}$.

3 \ $\pi$ is unramified for all finite places lying above $S$.

4 \ $\lambda_{\tau, i} - \lambda_{\tau, i+1} \ge k$ for all $\tau \in \mathrm{Hom}(F', \mathbb{C})$ and $i = 1, \cdots, n-1$.

\end{conj}

\section{Applications}

We fix a CM field $F$ and a positive integer $n$ again.

\subsection{Applications to local-global compatibility}

For a cohomological cuspidal automorphic representation $\pi$ of $\mathrm{GL}_n(\mathbb{A}_F)$ and $\sigma \in \mathrm{Aut}(\mathbb{C})$, we put \ $^{\sigma}\pi^{\infty}:=\mathbb{C} \otimes_{\sigma, \mathbb{C}} \pi^{\infty}$. Let $M_{\pi}$ be the fixed field of $\{ \sigma \in \mathrm{Aut}(\mathbb{C}) \mid \ ^\sigma\pi^{\infty} \cong \pi^{\infty} \}$.

Note that the following lemmas.

\begin{lem} \label{example 1}

For a cohomological cuspidal automorphic representation $\pi$ of $\mathrm{GL}_n(\mathbb{A}_F)$, we have the following properties.

1 \ For any $\sigma \in \mathrm{Aut}(\mathbb{C})$, there exists a cohomological cuspidal automorphic representation \ $^{\sigma}\pi$ of $\mathrm{GL}_n(\mathbb{A}_F)$ such that \ $(^{\sigma}\pi)^{\infty} \cong \ ^{\sigma}\pi^{\infty}$

2 \ There exists an integer $w$ such that \ $^{c}\pi \cong \pi^{\vee}||\mathrm{det}||_F^{-w}$, where $c \in \mathrm{Aut}(\mathbb{C})$ is the complex conjugation.

3 \ $M_{\pi}$ is a CM field.

\end{lem}

\begin{proof} 1 is \cite[Theorem 3.13]{purity}.
    
2 and 3 \ By \cite[Theorem 3.13]{purity}, $M_{\pi}$ is a finite extension of $\mathbb{Q}$.
        
By Proposition \ref{purity}, there exists an integer $w$ such that $\pi ||\mathrm{det}||_F^{\frac{w}{2}}$ is unitary. We fix $\sigma \in \mathrm{Aut}(\mathbb{C})$. Then we have $|\omega_{^{\sigma}\pi}| = | \omega_{\pi} |$ by the purity of $\omega_{\pi}$. Therefore, $^{\sigma}\pi|| \mathrm{det} ||_F^{\frac{w}{2}}$ is also unitary. This implies \ $^{c}(^{\sigma}\pi ||\mathrm{det}||_F^{\frac{w}{2}}) = ^{\sigma}\pi^{\vee} ||\mathrm{det}||_F^{-\frac{w}{2}}$. This is equivalent to \ $^{c\sigma}\pi = ^{\sigma}\pi^{\vee} ||\mathrm{det}||_F^{-w}$ and $^{\sigma^{-1}c\sigma}\pi = \pi^{\vee} ||\mathrm{det}||_F^{-w}$. Thus, $^{\sigma^{-1} c \sigma} \pi \cong \ ^{c}\pi$ and $^{c \sigma^{-1} c \sigma} \pi \cong \pi$. This implies that $M_{\pi}$ is a CM field. \end{proof}

\begin{rem}

1 \ For a cohomological cuspidal automorphic representation $\pi$ of $\mathrm{GL}_n(\mathbb{A}_F)$, we can prove that $\pi \cong \pi^{\vee}||\mathrm{det}||_F^{-w}$ for some integer $w$ if and only if $M_{\pi}$ is a totally real field. Moreover, a similar result holds for pure compatible systems of $l$-adic Galois representations.

\vspace{0.5 \baselineskip}

2 \ By Proposition \ref{strong multiplicity one} and Lemma \ref{example 1}, we have $$M_{\pi} = \mathrm{Im}(\varphi_{\pi} : \mathcal{H}(\mathrm{GL}_n(\mathbb{A}_{F}^{\infty, S_{\pi}}), \mathrm{GL}_n(\widehat{\mathcal{O}_F}^{S_{\pi}}))_{\mathbb{Q}} \rightarrow \mathbb{C}),$$ where $S_{\pi}$ is the set of all finite places $v$ of $F$ such that $\pi_v$ is ramified and $\varphi_{\pi}$ denotes the Hecke eigensystem corresponding to $\pi^{\infty, S_{\pi}, \mathrm{GL}_n(\widehat{\mathcal{O}_F}^{S_{\pi}})}$.

\end{rem}

\begin{lem} \label{decomposed genericity 1}
    
Let $\pi$ be a cohomological cuspidal automorphic representation of $\mathrm{GL}_n(\mathbb{A}_F)$. 
    
1 \ For almost all $l$ and $\iota : \overline{\mathbb{Q}}_l \stackrel{\sim}{\rightarrow} \mathbb{C}$, the representation $\overline{r_{\iota}(\pi)}$ is fully strongly decomposed generic.

2 \ Moreover, if $n=2$, then for any algebraic Hecke character $\chi : \mathbb{A}_{F}^{\times}/F^{\times} \rightarrow \mathbb{C}$ and any positive integer $k$, the representation $\overline{\mathrm{Symm}^kr_{\iota}(\pi) \otimes r_{\iota}(\chi)}$ is fully strongly decomposed generic for almost all $l$ and $\iota : \overline{\mathbb{Q}}_l \stackrel{\sim}{\rightarrow} \mathbb{C}$.
    
\end{lem}
    
\begin{proof}1 \ We fix a prime $p$ such that $p$ splits completely in $F$ and $\pi_v$ are unramified for all $v|p$. Then by Theorem \ref{generic unitary} and Proposition \ref{purity}, the eigenvalues $\alpha_{v,1}, \cdots, \alpha_{v,n}$ of $\mathrm{rec}_{F_v}(\pi_v|\mathrm{det}|_{v}^{\frac{1-n}{2}})(\mathrm{Frob}_v)$ satisfies $p^{-1} < |\frac{\alpha_{v,i}}{\alpha_{v,j}}|, |\frac{\alpha_{v,i}}{\alpha_{v^c,j}}| < p$ for all $i, j$. Note that $\alpha_{v,j}$'s are algebraic numbers by Lemma \ref{example 1}. Therefore, for almost all $l \neq p$ and $\iota : \overline{\mathbb{Q}}_l \stackrel{\sim}{\rightarrow} \mathbb{C}$, $\iota^{-1}(\alpha_{v,i}/\alpha_{v,j} - p)$, $\iota^{-1}(\alpha_{v,i}/\alpha_{v^c,j} - p)$ are $l$-adic units for all $v|p$, $i, j$. This implies the result.
    
2 \ By Proposition \ref{purity}, there exists an integer $w$ such that $\pi||\mathrm{det}||_F^{\frac{w}{2}}$ is unitary.
    
    By \cite[Proposition 3.7]{sym}, for Dirichlet density one finite places $v$, $\pi_v$ is unramified and the absolute values of the eigenvalues of $\mathrm{rec}_{F_v}(\pi_v |\mathrm{det}|_v^{-\frac{1}{2}})(\mathrm{Frob}_v)$ are bigger than $q_v^{\frac{1+w}{2} - \frac{1}{2k}}$ and smaller than $q_v^{\frac{1+w}{2} + \frac{1}{2k}}$. This implies that there exists a prime $p$ which splits completely in $F$ such that for all $v|p$, $\pi_v$ and $\chi_v$ are unramified and the eigenvalues $\alpha_{v, 1}, \alpha_{v, 2}$ of $\mathrm{rec}_{F_v}(\pi_v|\mathrm{det}|_v^{-\frac{1}{2}})(\mathrm{Frob}_v)$ satisfies $|\alpha_{v,2}/\alpha_{v,1}|^i, |\alpha_{v,1}/\alpha_{v,2}|^i$, $|(\alpha_{v,1}^i\alpha_{v,2}^{k-i}\chi_{v}(\varpi_v)/\alpha_{v^c, 1}^j\alpha_{v^c, 2}^{k-j}\chi_{v^c}(\varpi_{v^c}))| < p$ for all $i,j = 0, \cdots, k$. This implies that for almost all $l$, $\iota: \overline{\mathbb{Q}}_l \stackrel{\sim}{\longrightarrow} \mathbb{C}$, $\iota^{-1}((\alpha_{v,1}/\alpha_{v,2})^i - p),\ \iota^{-1}((\alpha_{v,2}/\alpha_{v,1})^i - p)$ and $\iota^{-1}((\alpha_{v,1}^i\alpha_{v,2}^{k-i}\chi_v(\varpi_v)/\alpha_{v^c, 1}^j\alpha_{v^c, 2}^{k-j}\chi_{v^c}(\varpi_{v^c})) - p)$ are $l$-adic units for all $v|p$ and $i, j = 0, \cdots, k$. This implies the result. \end{proof}
    
\begin{lem}\label{Hecke field 2}
    
Let $\pi$ be a cohomological cuspidal automorphic representation of $\mathrm{GL}_n(\mathbb{A}_F)$ of weight $\lambda$, $\widetilde{FM_{\pi}}$ be the Galois closure of $FM_{\pi}$ over $\mathbb{Q}$ (note that this is CM field by 3 of Lemma \ref{example 1}) and $\mathcal{L}$ be a set of primes.

We assume that any $l \in \mathcal{L}$ satisfies the following conditions.

1 \ $l$ splits completely in $\widetilde{FM_{\pi}}^+$ and all $l$-adic places of $\widetilde{FM_{\pi}}^+$ are inert in $\widetilde{FM_{\pi}}$.

2 \ $\overline{r_{\iota}(\pi)}$ is absolutely irreducible.

Then there exist an integer $w$ and a subset $\mathcal{L}_0$ of $\mathcal{L}$ such that $\mathcal{L} \setminus \mathcal{L}_0$ is finite and for any $l \in \mathcal{L}_0$, $\iota : \overline{\mathbb{Q}}_l \stackrel{\sim}{\rightarrow} \mathbb{C}$ and $v|l$, the representation $r_{\iota}(\pi)|_{G_{F_v}}$ is potentially diagonalizable Hodge-Tate regular and there exists a finite extension $F'_v/F_v$ such that $r_{\iota}(\pi)^c|_{G_{F'_v}} \sim r_{\iota}(\pi)^{\vee}\varepsilon_l^{-w} |_{G_{F'_v}}$.
    
    \end{lem}

    \begin{proof} Note that for any $l \in \mathcal{L}$, all $l$-adic places of $F^+$ (resp. $M_{\pi}^+$) are inert in $F$ (resp. $M_{\pi}$) by the assumption 1.
    
    Let $\mathcal{L}_0$ be the subset of $\mathcal{L}$ consisting of all $l \in \mathcal{L}$ satisfying the following conditions.
    
    $\cdot$ \ $l \ge $max$\{ \lambda_{\tau, 1} - \lambda_{\tau, n} + n + 1 \mid \tau \in \mathrm{Hom}(F, \mathbb{C}) \}$.
    
    $\cdot$ \  $l$ is unramified in $F$.
        
    $\cdot$ \ $\pi_v$ is unramified for any $v|l$.
    
$\cdot$ \ For all $\iota : \overline{\mathbb{Q}}_l \stackrel{\sim}{\rightarrow} \mathbb{C}$, $\overline{r_{\iota}(\pi)}$ is decomposed generic.
    
    Then $\mathcal{L} \setminus \mathcal{L}_0$ is finite by Lemma \ref{decomposed genericity 1}. We fix $l \in \mathcal{L}_0$, $\iota : \overline{\mathbb{Q}}_l \stackrel{\sim}{\rightarrow} \mathbb{C}$ and $v|l$. Then $r_{\iota}(\pi)|_{G_{F_v}}$ is potentially diagonalizable Hodge-Tate regular by Theorem \ref{Caraiani-Newton 2} and Proposition \ref{trace irreducible component}. By the assumption 1, there exists an isomorphism $\tilde{c} \in \mathrm{Gal}(\overline{\mathbb{Q}}_l/\mathbb{Q}_l)$ such that $\iota \tilde{c} \iota^{-1} = c$ on $\widetilde{M_{\pi}F}$. By Lemma \ref{example 1}, there exists an integer $w$ such that \ $^c\pi \cong \pi^{\vee}|| \mathrm{det} ||_{F}^{-w + n - 1}$. This implies \ $^{\tilde{c}}r_{\iota}(\pi) \cong r_{\iota}(\pi)^{\vee}\varepsilon_l^{-w}$. There exists $\tau \in \mathrm{Gal}(\widetilde{M_{\pi}F}/\mathbb{Q})$ such that $\iota^{-1}\tau$ extends to the morphism $\widetilde{\tau} : F_v \rightarrow \overline{\mathbb{Q}}_l$ over $\mathbb{Q}_l$. Since $\widetilde{M_{\pi}F}$ is a CM field, we have $\tau^{-1} c \tau = c$. Consequently, we obtain $c|_{F} = \tau^{-1} c \tau|_{F} = (\iota \widetilde{\tau})^{-1} c \iota \widetilde{\tau}|_{F} = \widetilde{\tau}^{-1} \widetilde{c} \widetilde{\tau}|_{F}$. By Proposition \ref{trace irreducible component} and Theorem \ref{Caraiani-Newton 2} again, there exists a finite extension $F'_v/F_v$ such that $r_{\iota}(\pi)^c|_{G_{F'_v}} = (r_{\iota}(\pi)|_{G_{F_v}})^{\tilde{\tau}^{-1} \tilde{c} \tilde{\tau}}|_{G_{F'_v}} \sim \ ^{\tilde{c}}r_{\iota}(\pi)|_{G_{F_v'}} = r_{\iota}(\pi)^{\vee}\varepsilon_l^{-w}|_{G_{F'_v}}$. \end{proof}

\begin{lem}\label{supercuspidal residually irreducible}

Let $\pi$ be a cohomological cuspidal automorphic representation of $\mathrm{GL}_n(\mathbb{A}_F)$ such that $\pi_v$ is supercuspidal for some $v$. Then there exists a set of primes $\mathcal{L}$ having Dirichlet density one such that $\overline{r_{\iota}(\pi)}|_{G_{F(\zeta_l)}}$ is absolutely irreducible for any $l \in \mathcal{L}$ and $\iota : \overline{\mathbb{Q}_{l}} \stackrel{\sim}{\rightarrow} \mathbb{C}$.

\end{lem}

\begin{proof} Fix a finite place $v$ of $F$ such that $\pi_v$ is supercuspidal. First, we will prove that $\mathrm{rec}_{F_v}(\pi_v|\mathrm{det}|_v^{\frac{1-n}{2}})$ is defined over a finite extension of $M_{\pi}$.
    
Since $\mathrm{rec}_{F_v}(\pi_v|\mathrm{det}|_v^{\frac{1-n}{2}})$ is irreducible, the monodromy operator is zero and $\mathrm{rec}_{F_v}(\pi_v|\mathrm{det}|_v^{\frac{1-n}{2}})(\phi_v)^n$ is a scalar for some $n > 0$ by Schur's Lemma. ( $\phi_v$ is a Frobenius lift.) By Lemma \ref{example 1}, we have $\mathrm{Im}(\mathrm{tr}(\mathrm{rec}_{F_v}(\pi_v|\mathrm{det}|_v^{\frac{1-n}{2}}))) \subset M_{\pi}$. Therefore, there exists a smooth character $\chi : W_{F_v} \rightarrow M_{\pi}^{\times}$ such that $\mathrm{rec}_{F_v}(\pi_v|\mathrm{det}|_v^{\frac{1-n}{2}}) \otimes \chi$ has a finite image. Since any irreducible representation of a finite group over algebraically closed fields of characteristic zero is defined over a finite extension of $\mathbb{Q}$, there exists a finite extension $M$ of $M_{\pi}$ and an absolutely irreducible smooth representation $r : W_{F_v} \rightarrow \mathrm{GL}_n(M)$ such that $r \otimes_{M} \mathbb{C} \cong \mathrm{rec}_{F_v}(\pi_v|\mathrm{det}|_v^{\frac{1-n}{2}})$. 

Since $\mathrm{Im}(r)$ is finitely generated, we have a positive integer $n$ such that $\mathrm{Im}(r) \subset \mathrm{GL}_n(\mathcal{O}_{M}[\frac{1}{n}])$. Since $r$ is absolutely irreducible, $\mathrm{Im}(r)$ generates $M_n(M)$ over $M$. Therefore, after replacing $n$, we may assume that $\mathrm{Im}(r)$ generates $M_n(\mathcal{O}_M[\frac{1}{n}])$ over $\mathcal{O}_M[\frac{1}{n}]$. Thus, for any prime ideal $\lambda$ of $\mathcal{O}_M$ not lying above any prime factor of $n$, the residual representation $\overline{r}_{\lambda} : W_{F_v} \rightarrow \mathrm{GL}_n(\mathcal{O}_M[\frac{1}{n}]) \rightarrow \mathrm{GL}_n(\mathbb{F}_{\lambda})$ is absolutely irreducible.

Let $l \nmid n$ a prime such that $v$ doesn't lie above $l$. By Theorem \ref{Ila Varma}, we have $\overline{r_{\iota}(\pi)}|_{W_{F_v}} \cong \overline{r}_{\lambda}$ for any $\iota : \overline{\mathbb{Q}}_l \stackrel{\sim}{\rightarrow} \mathbb{C}$. (Here, $\lambda$ is the prime ideal of $\mathcal{O}_M$ induced by $\iota$.) In particular, $\overline{r_{\iota}(\pi)}$ is absolutely irreducible. By using Theorem \ref{Caraiani-Newton}, Lemma \ref{decomposed genericity 1} and Proposition \ref{irreducibility} later, we obtain the result. \end{proof}

Note that Theorems \ref{self-dual local-global}, \ref{two-dimensional local-global} and \ref{Ramanujan} have already been proved if $F$ is a totally real field.

    \begin{thm} \label{self-dual local-global}
    
        Let $\pi$ be a cohomological cuspidal automorphic representation of $\mathrm{GL}_{n}(\mathbb{A}_F)$ and $\mathcal{L}$ be a set of primes having Dirichlet density one. 
        
        We suppose the following conditions.
                
        (1) \ For any $l \in \mathcal{L}$ and $\iota: \overline{\mathbb{Q}}_l \stackrel{\sim}{\longrightarrow} \mathbb{C}$, the representation $\overline{r_{\iota}(\pi)}|_{G_{F(\zeta_l)}}$ is absolutely irreducible.
        
    (2) \ There exists an algebraic Hecke character $\chi : \mathbb{A}_{F^+}^{\times}/(F^+)^{\times} \rightarrow \mathbb{C}^{\times}$ such that $\pi \cong \pi^{\vee} \otimes \chi \circ N_{F/F^+} \circ \mathrm{det}$ and $\chi_v(-1) = \chi_w(-1)$ for all $v, w \mid \infty$.
    
        Then there exists a subset $\mathcal{L}_0$ of $\mathcal{L}$ having positive Dirichlet density such that for all $l \in \mathcal{L}_0$, $\iota: \overline{\mathbb{Q}}_l \stackrel{\sim}{\longrightarrow} \mathbb{C}$ and finite places $v \nmid l$ of $F$, we have $\iota\mathrm{WD}(r_{\iota}(\pi)|_{G_{F_v}})^{F-ss} \cong \mathrm{rec}_{F_v}(\pi_v|\mathrm{det}|_{v}^{\frac{1-n}{2}})$.
    
    \end{thm}
    
    \begin{rem}
    
    Note that if $\mathcal{R}_{\pi}$ is irreducible very weakly compatible system or $\pi_v$ is supercuspidal for some $v$, then the assumption $(1)$ holds by Proposition \ref{irreducibility} later or Lemma \ref{supercuspidal residually irreducible}. (See Definition \ref{automorphic compatible system} for the definition of $\mathcal{R}_{\pi}$.)
    
    \end{rem}
    
    \begin{proof}
    
    By Lemma \ref{decomposed genericity 1} and \ref{Hecke field 2}, there exists a subset $\mathcal{L}_0$ of $\mathcal{L}$ having positive Dirichlet density such that for any $l \in \mathcal{L}_0$ and $\iota : \overline{\mathbb{Q}}_l \stackrel{\sim}{\rightarrow} \mathbb{C}$, the representation $r_{\iota}(\pi)$ satisfies the conditions of Theorem \ref{potential diagonalizable automorphy}.
    
    Therefore, there exist a finite CM Galois extension $F'/F$ and a cohomological cuspidal automorphic representation $\Pi$ of $\mathrm{GL}_n(\mathbb{A}_{F'})$ such that $r_{\iota}(\Pi) \cong r_{\iota}(\pi)|_{G_{F'}}$ and $\iota\mathrm{WD}(r_{\iota}(\Pi)|_{G_{F'_v}})^{F-ss} \cong \mathrm{rec}_{F'_v}(\Pi_v|\mathrm{det}|_v^{\frac{1-n}{2}})$ for all $v \nmid l$. This implies the theorem by Proposition \ref{potential automorphy and local-global compatibility}. \end{proof}
    
\begin{lem} \label{GL_2}

Let $\pi$ be a cohomological cuspidal automorphic representation of $\mathrm{GL}_2(\mathbb{A}_F)$. We assume that $\pi$ is not an automorphic induction of a character. Then there exists a set of primes $\mathcal{L}$ having Dirichlet density one such that for all $l \in \mathcal{L}$ and $\iota: \overline{\mathbb{Q}}_l \stackrel{\sim}{\rightarrow} \mathbb{C}$, the group $\overline{r_{\iota}(\pi)}(G_{F(\zeta_l)})$ contains a conjugate of $\mathrm{SL}_2(\mathbb{F}_l)$. 

\end{lem}

\begin{proof}  See \cite[Lemmas 7.1.2, 7.1.3 and 7.1.10]{10}. \end{proof}

    \begin{prop}\label{automorphic LR}
    
    Let $\pi$ be a cohomological cuspidal automorphic representation of $\mathrm{GL}_2(\mathbb{A}_F)$. We assume that $\pi$ is not an automorphic induction of a character. Then for any positive even integer $n$, there exists a set of primes $\mathcal{L}_0$ having positive Dirichlet density such that for any $l \in \mathcal{L}_0$ and $\iota: \overline{\mathbb{Q}}_l \stackrel{\sim}{\longrightarrow} \mathbb{C}$, there exist a finite CM extension $F'/F$ and a cohomological cuspidal automorphic representation $\Pi$ of $\mathrm{GL}_{n+1}(\mathbb{A}_{F'})$ such that $\mathrm{Symm}^nr_{\iota}(\pi)|_{G_{F'}} \cong r_{\iota}(\Pi)$ and $\iota\mathrm{WD}(r_{\iota}(\Pi)|_{G_{F'_v}})^{F-ss} \cong \mathrm{rec}_{F'_v}(\Pi_{v}|\mathrm{det}|_v^{\frac{-n}{2}})$ for all $v \nmid l$.
    
    \end{prop}

    \begin{proof}
    
    Note that we have $r_{\iota}(\pi) \cong r_{\iota}(\pi)^{\vee}r_{\iota}(\omega_{\pi})\varepsilon_l^{-1}$ for all $l$, $\iota$. This implies $\mathrm{Symm}^nr_{\iota}(\pi) \cong (\mathrm{Symm}^nr_{\iota}(\pi))^{\vee} r_{\iota}(\omega_{\pi})^{n}\varepsilon_l^{-n}$ for all positive integer $n$ and if $n$ is even, then $\mathrm{Symm}^nr_{\iota}(\pi)r_{\iota}(\omega_{\pi})^{-\frac{n}{2}} \cong (\mathrm{Symm}^nr_{\iota}(\pi)r_{\iota}(\omega_{\pi})^{-\frac{n}{2}})^{\vee}\varepsilon_l^{-n}$.
    
    Fix a positive even integer $n$. By Lemmas \ref{decomposed genericity 1}, \ref{Hecke field 2} and \ref{GL_2}, there exists a subset $\mathcal{L}_0$ of $\mathcal{L}$ such that $\mathcal{L}_0$ has a positive Dirichlet density and for all $l \in \mathcal{L}_0$ and $\iota : \overline{\mathbb{Q}}_l \stackrel{\sim}{\rightarrow} \mathbb{C}$ such that $\mathrm{Symm}^nr_{\iota}(\pi)r_{\iota}(\omega_{\pi})^{-\frac{n}{2}}$ satisfies the conditions of Theorem \ref{potential diagonalizable automorphy}. Therefore, there exist a finite CM Galois extension $F'/F$ and a cohomological cuspidal automorphic representation $\Pi'$ of $\mathrm{GL}_{n+1}(\mathbb{A}_{F'})$ such that $r_{\iota}(\Pi') \cong \mathrm{Symm}^nr_{\iota}(\pi)r_{\iota}(\omega_{\pi})^{-\frac{n}{2}}|_{G_{F'}}$ and $\iota\mathrm{WD}(r_{\iota}(\Pi')|_{G_{F'_v}})^{F-ss} \cong \mathrm{rec}_{F'_v}(\Pi'_v|\mathrm{det}|_v^{-\frac{n}{2}})$ for all $v \nmid l$. We obtain the result by putting $\Pi := \Pi' \otimes \omega_{\pi}^{\frac{n}{2}} \circ N_{F'/F} \circ \mathrm{det}$. \end{proof}

       \begin{thm}\label{two-dimensional local-global}
    
        Let $\pi$ be a cohomological cuspidal automorphic representation of $\mathrm{GL}_2(\mathbb{A}_F)$.
        
        Then there exists a set of primes $\mathcal{L}_0$ having positive Dirichlet density such that for any $l \in \mathcal{L}_0$, $\iota : \overline{\mathbb{Q}}_l \stackrel{\sim}{\longrightarrow} \mathbb{C} $ and $v\nmid l $, we have $\iota \mathrm{WD}(r_{\iota}(\pi)|_{G_{F_v}})^{F-ss} \cong \mathrm{rec}_{F_v}(\pi_v|\mathrm{det}|_v^{-\frac{1}{2}})$.
        
    \end{thm}
    
    \begin{proof} 
        
    We may assume that $\pi$ is not an automorphic induction of a character. 
    
    By Proposition \ref{automorphic LR}, there exists a set of primes $\mathcal{L}_0$ having positive Dirichlet density such that for any $l \in \mathcal{L}_0$ and $\iota: \overline{\mathbb{Q}}_l \stackrel{\sim}{\rightarrow} \mathbb{C}$, we obtain a finite CM extension $F'/F$ and a cohomological cuspidal automorphic representation $\Pi$ of $\mathrm{GL}_{3}(\mathbb{A}_{F'})$ such that $\mathrm{Symm}^2r_{\iota}(\pi)|_{G_{F'}} \cong r_{\iota}(\Pi)$ and $\iota\mathrm{WD}(r_{\iota}(\Pi)|_{G_{F'_v}})^{F-ss} \cong \mathrm{rec}_{F'_v}(\Pi_{v}|\mathrm{det}|_v^{-1})$ for all $v \nmid l$.
    
    We fix $l \in \mathcal{L}_0$, $\iota : \overline{\mathbb{Q}}_l \stackrel{\sim}{\rightarrow} \mathbb{C}$ and $v \nmid l$. Note that by Theorem \ref{Ila Varma} and Lemma \ref{semisimple}, $\mathrm{Symm}^2\mathrm{rec}_{F_v'}(\mathrm{BC}_{F'_v/F_u}(\pi_u)|\mathrm{det}|_v^{-\frac{1}{2}}) \cong \mathrm{rec}_{F'_v}(\Pi_v|\mathrm{det}|_v^{-1})$, where $u$ is the finite place of $F$ lying below $v$. Therefore, we have $$\mathrm{Symm}^2\iota \mathrm{WD}(r_{\iota}(\pi)|_{G_{F_u}})^{F-ss}|_{W_{F_v'}} \cong \mathrm{Symm}^2\mathrm{rec}_{F_u}(\pi_u|\mathrm{det}|_u^{-\frac{1}{2}})|_{W_{F'_v}}$$. By Theorem \ref{Ila Varma}, we may assume that $\pi_u$ is a character twist of the Steinberg representation. Then, it suffices to show that the monodromy operator of $\mathrm{WD}(r_{\iota}(\pi)|_{G_{F_u}})$ is not zero by Theorem \ref{automorphy lifting and local-global compatibility}. This follows by $\mathrm{Symm}^2\iota \mathrm{WD}(r_{\iota}(\pi)|_{G_{F_u}})^{F-ss}|_{W_{F_v'}} \cong \mathrm{Symm}^2\mathrm{rec}_{F_u}(\pi_u|\mathrm{det}|_u^{-\frac{1}{2}})|_{W_{F'_v}}.$ \end{proof}
    
    We next prove the Ramanujan conjecture for cohomological cuspidal automorphic representations of $\mathrm{GL}_2(\mathbb{A}_F)$. 
    
    \begin{thm} (Ramanujan conjecture) \label{Ramanujan}
    
        Let $\pi$ be a cohomological cuspidal automorphic representation of $\mathrm{GL}_2(\mathbb{A}_F)$.
        
        Then $\pi_v$ is tempered for all finite places $v$ in $F$.
        
        \end{thm}
    
    \begin{proof}
    
    We may assume that $\pi$ is not an automorphic induction of a character. If $\pi_v$ is a character twist of the Steinberg representation, then $\pi_v$ is tempered. After replacing $F$ by its solvable CM extension, it suffices to show that if $\pi_v$ is unramified, then $\pi_v$ is tempered by Proposition \ref{base change}, Corollary \ref{extension of Q} and Lemma \ref{GL_2}. Fix a finite place $v$ of $F$ such that $\pi_v$ is unramified. There exists an integer $w$ such that $\pi||\mathrm{det}||_F^{\frac{w}{2}}$ is unitary. It suffices to show that the eigenvalues $\beta_{v,1}, \beta_{v,2}$ of $\mathrm{rec}_{F_v}(\pi_v)(\mathrm{Frob}_v)$ satisfy $|\beta_{v,1}| = |\beta_{v,2}| = q_v^{\frac{w}{2}}$. This is equivalent to $q_v^{\frac{wn-1}{2}} \le |\beta_{v,1}|^n, |\beta_{v,2}|^n \le q_v^{\frac{wn+1}{2}}$ for any positive even integer $n$. By Proposition \ref{automorphic LR}, for any positive even integer $n$, there exist a prime $l$, $\iota: \overline{\mathbb{Q}}_l \stackrel{\sim}{\rightarrow} \mathbb{C}$, a finite Galois CM extension $F'/F$ and a cohomological cuspidal automorphic representation $\Pi$ of $\mathrm{GL}_{n+1}(\mathbb{A}_{F'})$ such that $v \nmid l$, $r_{\iota}(\Pi) \cong \mathrm{Symm}^nr_{\iota}(\pi)|_{G_{F'}}$ and $\iota\mathrm{WD}(r_{\iota}(\Pi)|_{G_{F'_u}})^{F-ss} \cong \mathrm{rec}_{F'_u} (\Pi_u|\mathrm{det}|^{\frac{-n}{2}})$ for all $u \nmid l$. Note that $|\omega_{\Pi}| = || \ ||_{F'}^{-\frac{(n+1)nw}{2}}$ and $\Pi||\mathrm{det}||_{F'}^{\frac{nw}{2}}$ is unitary. For all $u|v$, $\Pi_u$ is unramified by $\iota\mathrm{WD}(\mathrm{Symm}^nr_{\iota}(\pi)|_{G_{F'_u}})^{F-ss} \cong \mathrm{rec}_{F'_u} (\Pi_u|\mathrm{det}|^{\frac{-n}{2}})$. This implies that the absolute values of eigenvalues of $\mathrm{rec}_{F'_u}(\Pi_u)(\mathrm{Frob}_u)$ are smaller than $q_u^{\frac{nw+1}{2}}$ and bigger than $q_u^{\frac{nw-1}{2}}$ by Theorem \ref{generic unitary} and Proposition \ref{purity}. In particular, $q_v^{\frac{wn-1}{2}} \le |\beta_{v,1}|^n, |\beta_{v,2}|^n \le q_v^{\frac{wn+1}{2}}$. Therefore, we obtain the result. \end{proof}
    
    \begin{thm} \label{sufficiently regular}
    
    Let $\pi$ be a cohomological cuspidal automorphic representation of $\mathrm{GL}_{n}(\mathbb{A}_F)$ and $\mathcal{L}$ be a set of primes having Dirichlet density one.
            
    We suppose the following conditions.

    (1) \ For all $l \in \mathcal{L}$ and $\iota: \overline{\mathbb{Q}}_l \stackrel{\sim}{\longrightarrow} \mathbb{C}$, the representation $\overline{r_{\iota}(\pi)}|_{G_{F(\zeta_l)}}$ is absolutely irreducible.
        
    (2) \ The weight $\lambda \in (\mathbb{Z}^{n}_+)^{\mathrm{Hom}(F,\mathbb{C})}$ of $\pi$ satisfies $\lambda_{\tau, i} - \lambda_{\tau, i+1} \ge n-1$ for all $\tau \in \mathrm{Hom}(F, \mathbb{C})$ and $i=1, \cdots, n-1$.
        
    Then there exists a subset $\mathcal{L}_0$ of $\mathcal{L}$ having positive Dirichlet density such that for all $l \in \mathcal{L}_0$, $\iota: \overline{\mathbb{Q}}_l \stackrel{\sim}{\longrightarrow} \mathbb{C}$ and finite places $v \nmid l$ of $F$, we have $\iota\mathrm{WD}(r_{\iota}(\pi)|_{G_{F_v}})^{F-ss} \cong \mathrm{rec}_{F_v}(\pi_v|\mathrm{det}|_{v}^{\frac{1-n}{2}})$.
        
    \end{thm}
    
    \begin{proof} By Lemmas \ref{decomposed genericity 1}, \ref{Hecke field 2} and Theorem \ref{Caraiani-Newton 2}, there exists a subset $\mathcal{L}_0$ of $\mathcal{L}$ having positive Dirichlet density such that for any $l \in \mathcal{L}_0$ and $\iota : \overline{\mathbb{Q}}_l \stackrel{\sim}{\rightarrow} \mathbb{C}$ such that $r_{\iota}(\pi)$ satisfies the conditions of Theorem \ref{potential tensor automorphy}. Therefore, there exist a finite Galois CM extension $F'/F$, a cohomological cuspidal automorphic representation $\Pi$ of $\mathrm{GL}_{2n^2}(\mathbb{A}_{F'})$ and a polarizable $\iota$-ordinary cohomological cuspidal automorphic representation $\pi_1$ of $\mathrm{GL}_{2n}(\mathbb{A}_{F'})$ satisfying the following conditions.
        
1 \ If $\pi$ is ramified at $v \nmid l$, then $\pi_1$ is unramified at all places above $v$.

2 \ $r_{\iota}(\pi_1) \otimes r_{\iota}(\pi)|_{G_{F'}} \cong r_{\iota}(\Pi)$.
    
3 \ $\iota\mathrm{WD}(r_{\iota}(\Pi)|_{G_{F'_{v}}})^{F-ss} \cong \mathrm{rec}_{F'_{v}}(\Pi_v|\mathrm{det}|_v^{\frac{1-2n^2}{2}})$ for $v \nmid l$.
        
Thus, the result follows from Lemma \ref{purity tensor} and Theorem \ref{polarizable local-global compatibility}. \end{proof}

\begin{thm} \label{ordinary local-global}
        
    Let $l$ be a prime such that $l > n^2$, $\iota: \overline{\mathbb{Q}}_l \stackrel{\sim}{\longrightarrow} \mathbb{C}$ be an isomorphism of fields and $\pi$ be an $\iota$-ordinary cohomological cuspidal automorphic representation of $\mathrm{GL}_{n}(\mathbb{A}_F)$. 
        
        We suppose the following conditions.
        
        (1) \ $\overline{r_{\iota}(\pi)}$ is fully decomposed generic.
        
        (2) \ $\overline{r_{\iota}(\pi)}|_{G_{F(\zeta_l)}}$ is absolutely irreducible. 
        
        (3) \ $\zeta_l \notin F$.
        
        Then for all finite places $v \nmid l$ of $F$, we have $\iota\mathrm{WD}(r_{\iota}(\pi)|_{G_{F_v}})^{F-ss} \cong \mathrm{rec}_{F_v}(\pi_v|\mathrm{det}|_{v}^{\frac{1-n}{2}})$.

    \end{thm}

    \begin{proof} By Theorem \ref{ordinary semisimple}, $r_{\iota}(\pi)$ satisfies the conditions of Theorem \ref{potential ordinary automorphy}. Thus, there exist a finite CM Galois extension $F'/F$ and an $\iota$-ordinary cohomological cuspidal automorphic representation $\Pi$ of $\mathrm{GL}_n(\mathbb{A}_{F'})$ such that $r_{\iota}(\Pi) \cong r_{\iota}(\pi)|_{G_{F'}}$ and $\iota\mathrm{WD}(r_{\iota}(\Pi)|_{G_{F'_u}})^{F-ss} \cong \mathrm{rec}_{F'_u}(\Pi_u|\mathrm{det}|_u^{\frac{1-n}{2}})$ for any $u \nmid l$. This implies the theorem by Proposition \ref{potential automorphy and local-global compatibility}.  \end{proof}

\subsection{Applications to the potential automorphy and the purity of compatible systems of $l$-adic Galois representations} \label{Application}

We recall definitions of some properties of compatible systems of $l$-adic Galois representations.

\begin{dfn} A rank $n$ very weakly compatible system $$\mathcal{R}=(M, S, \{ Q_v(X) \}_{v \notin S}, \{ r_{\lambda} \}_{\lambda}, \{ H_{\tau} \}_{\tau \in \mathrm{Hom}(F, \overline{M})})$$ of $l$-adic representations of $G_F$ defined over $M$ means a tuple of the following data.

1 \ $M$ is a number field.

2 \ $S$ is a finite set of finite places of $F$.

3 \ For each $v \notin S$, $Q_v(X) \in M[X]$ is a monic polynomial of degree $n$.

4 \ For each $\tau \in \mathrm{Hom}(F, \overline{M})$, $H_{\tau}$ is an element of $\mathbb{Z}^n/\mathfrak{S}_n$.

5 \ For each finite place $\lambda$ of $M$, $r_{\lambda} : G_{F} \rightarrow \mathrm{GL}_n(\overline{M_{\lambda}})$ is a semisimple continuous representation satisfying the following conditions.

$(a)$ \ If $v \notin S$ and $v \nmid \mathrm{char} \, \mathbb{F}_{\lambda}$, the representation $r_{\lambda}|_{G_{F_v}}$ is unramified and $\mathrm{det}(TI_n - r_{\lambda}(\mathrm{Frob}_v)) = Q_{v}(X)$.

$(b)$ \ There exists a set $\mathcal{L}$ of primes having Dirichlet density one such that for any $l \in \mathcal{L}$, $\lambda|l$ and $v|l$, the representation $r_{\lambda}|_{G_{F_v}}$ is crystalline and $\mathrm{HT}_{\iota \tau}(r_{\lambda}) = H_{\tau}$ for any $\tau \in \mathrm{Hom}(F, \overline{M})$ and any embedding $\iota : \overline{M} \hookrightarrow \overline{M_{\lambda}}$ over $M$.

A rank $n$ extremely weakly compatible system $\mathcal{R}=(M, S, \{ Q_v(X) \}, \{ r_{\lambda} \}, \{ H_{\tau} \})$ of $l$-adic representations of $G_F$ defined over $M$ means a tuple of the above data without the condition $(b)$ of 5.

\end{dfn}

\begin{dfn} \label{automorphic compatible system}

For a cohomological cuspidal automorphic representation $\pi$ of $\mathrm{GL}_n(\mathbb{A}_F)$ of weight $\lambda$, we have the following data.

$\cdot$ $S_{\pi}$ is the set of all finite places $v$ of $F$ such that $\pi_v$ is ramified.

$\cdot$ $M_{\pi}$. (See Lemma \ref{example 1}.)

$\cdot$ For any $v \notin S_{\pi}$, $Q_v(X):=\mathrm{det}(XI_n - \mathrm{rec}_{F_v}(\pi_v|\mathrm{det}|_v^{\frac{1-n}{2}})(\mathrm{Frob}_v))$.

$\cdot$ For a finite place $\lambda$ of $M_{\pi}$, $r_{\pi, \lambda} : G_F \rightarrow \mathrm{GL}_n(\overline{M_{\pi, \lambda}})$ is the semisimple continuous representation such that $r_{\pi, \lambda}|_{G_{F_v}}$ is unramified and $\mathrm{det}(TI_n - r_{\pi, \lambda}(\mathrm{Frob}_v)) = Q_v(T)$ for any $v \notin S_{\pi}$ and $v \nmid \mathrm{char} \, \mathbb{F}_{\lambda}$. (Note that if $\iota : \overline{\mathbb{Q}}_l \stackrel{\sim}{\rightarrow} \mathbb{C}$ induces $\lambda$ on $M_{\pi}$ and we identify $\overline{\mathbb{Q}}_l = \overline{M_{\pi, \lambda}}$, then we have $r_{\pi, \lambda} = r_{\iota}(\pi)$. )

$\cdot$ For any $\tau \in \mathrm{Hom}(F, \overline{M})$, $H_{\pi, \tau} := \{ \lambda_{\tau, i} + n - i \mid i = 1, \cdots, n \} \in (\mathbb{Z}^n/\mathfrak{S}_n)$.

Then $\mathcal{R}_{\pi}:=(M_{\pi}, S_{\pi}, \{ Q_{\pi, v}(X) \}, \{ r_{\pi, \lambda} \}, \{ H_{\pi, \tau} \})$ is a rank $n$ extremely weakly compatible system.

\end{dfn}
    
\begin{dfn}

1 \ For an extremely weakly compatible system $\mathcal{R}:=(M, S, \{ Q_{v}(X) \}, \{ r_{\lambda} \}, \{ H_{\tau} \})$, we say that $\mathcal{R}$ is irreducible if there exists a set $\mathcal{L}$ of primes having Dirichlet density one such that for any $l \in \mathcal{L}$ and $\lambda | l$, $r_{\lambda}$ is irreducible.
    
2 \  For an extremely weakly compatible system $\mathcal{R}$, we say that $\mathcal{R}$ is strongly irreducible if for any finite extension $F'/F$ such that $\mathcal{R}|_{G_{F'}}$ is irreducible.

3 \ For a rank $n$ very weakly compatible system $\mathcal{R}:=(M, S, \{ Q_{v}(X) \}, \{ r_{\lambda} \}, \{ H_{\tau} \})$, we say that $\mathcal{R}$ is regular if for any $\tau$, $H_{\tau}$ consists of $n$-distinct integers.

4 \ For an extremely weakly compatible system $\mathcal{R}:=(M, S, \{ Q_{v}(X) \}, \{ r_{\lambda} \}, \{ H_{\tau} \})$ and an integer $w$, we say that $\mathcal{R}$ is pure of weight $w$ if for any $v \notin S$, any root $\alpha \in \overline{M}$ of $Q_v(X)$ is a Weil $q_v^w$-number.
        
5 \ For an extremely weakly compatible system $\mathcal{R}:=(M, S, \{ Q_{v}(X) \}, \{ r_{\lambda} \}, \{ H_{\tau} \})$, we say that $\mathcal{R}$ is geometrically pure if $\mathcal{R}$ is pure and for any finite place $\lambda$ of $M$, finite place $v \nmid \mathrm{char} \, \mathbb{F}_{\lambda}$ of $F$, $\sigma \in W_{F_v}$ and any eigenvalue $\alpha \in \overline{M_{\lambda}}$ of $r_{\lambda}(\sigma)$, there exists an integer $w$ such that $\alpha$ is a Weil $q_v^w$-number. For example, $\mathcal{R}$ is geometrically pure if for any $\lambda$, $r_{\lambda}$ is a subquotient of $H^*_{\acute{e}t}(X_{\overline{F}}, \overline{M_{\lambda}}(i))$ for a proper smooth variety $X$ over $F$, $i \in \mathbb{Z}$. (See \cite[Lemma I.5.7]{LLC}.)

6 \ If $F$ is totally real, a rank $1$ very weakly compatible system $\mathcal{X} = (M, S, \{ X - \alpha_v \}, \{ \chi_{\lambda} \}, \{ H_{\tau} \} )$ is odd (resp. even) if $\chi_{\lambda}(c_v) = -1$ (resp. $1$) for all $v \mid \infty$ and some $\lambda$. (This implies $\chi_{\lambda}(c_v) = -1$ (resp. $1$) for all $\lambda$.) 

7 \ For a very weakly compatible system $\mathcal{R}:=(M, S, \{ Q_{v}(X) \}, \{ r_{\lambda} \}, \{ H_{\tau} \})$, we say that $\mathcal{R}$ has a parallel weight if $H_{\tau} = H_{\sigma}$ for any $\tau, \sigma \in \mathrm{Hom}(F, \overline{M})$.

\end{dfn}

\begin{dfn} 
    
1 \ For rank $n$ extremely weakly compatible systems $$\mathcal{R}_1=(M, S_1, \{ Q_{1, v}(X) \}, \{ r_{1, \lambda} \}, \{ H_{1, \tau} \}), \mathcal{R}_2 = (M, S_2, \{ Q_{2, v}(X) \}, \{ r_{2, \lambda} \}, \{ H_{2, \tau} \})$$ of $l$-adic representations of $G_F$ defined over $M$, we say $\mathcal{R}_1 \cong \mathcal{R}_2$ if $Q_{1, v}(X) = Q_{2, v}(X)$ for Dirichlet density one finite places $v \notin S_1 \cup S_2$.

Note that if $\mathcal{R}_1 \cong \mathcal{R}_2$, then $r_{1, \lambda} \cong r_{2, \lambda}$ for any $\lambda$ and $Q_{1, v}(X) = Q_{2, v}(X)$ for any $v \notin S_1 \cup S_2$. If $\mathcal{R}_1$ and $\mathcal{R}_2$ are very weakly compatible systems, then $H_{1, \tau} = H_{2, \tau}$ for any $\tau \in \mathrm{Hom}(F, \overline{M})$.

2 \ For a rank $n$ extremely weakly compatible system $$\mathcal{R}=(M, S, \{ Q_{v}(X) \}, \{ r_{\lambda} \}, \{ H_{\tau} \}),$$ we say that $\mathcal{R}$ is automorphic if there exists an embedding $\iota : M \hookrightarrow \mathbb{C}$ and a cohomological cuspidal automorphic representation $\pi$ of $\mathrm{GL}_n(\mathbb{A}_F)$ such that $\iota Q_v(X) = \mathrm{det}(TI_n - \mathrm{rec}_{F_v}(\pi_v|\mathrm{det}|_v^{\frac{1-n}{2}})(\mathrm{Frob}_v))$ for Dirichlet density one places $v \notin S \cup S_{\pi}$. By Lemma \ref{example 1}, this is equivalent to that for any $\iota : M \hookrightarrow \mathbb{C}$, there exists a cohomological cuspidal automorphic representation $\pi$ of $\mathrm{GL}_n(\mathbb{A}_F)$ such that $\iota Q_v(X) = \mathrm{det}(TI_n - \mathrm{rec}_{F_v}(\pi_v|\mathrm{det}|_v^{\frac{1-n}{2}})(\mathrm{Frob}_v))$ for Dirichlet density one places $v \notin S \cup S_{\pi}$.
\end{dfn}

Note that the following lemmas.

\begin{prop}\label{irreducibility}

Let $\mathcal{R} = (M, S, \{ Q_v(X) \}, \{r_{\lambda}\}, \{ H_{\tau} \} )$ be a rank $n$ regular very weakly compatible system of $l$-adic representations of $G_F$ defined over $M$.

Then there exists a set $\mathcal{L}$ of primes having Dirichlet density one such that for any $l \in \mathcal{L}$, $\lambda|l$ and any irreducible subrepresentation $s$ of $r_{\lambda}$, the representation $\overline{s}|_{G_{F(\zeta_l)}}$ is absolutely irreducible.

\end{prop}

\begin{proof}

See \cite[Proposition 5.3.2]{CW}. By the proof, we only need to assume that $\mathcal{R}$ is very weakly compatible system. \end{proof}

\begin{lem}\label{Hecke field}

Let $\mathcal{R} = (M, S, \{ Q_v(X) \}, \{r_{\lambda}\}, \{ H_{\tau} \} )$ be a rank $n$ irreducible regular very weakly compatible system of $l$-adic representations of $G_F$ defined over $M$ which is pure of weight $w$. 

Then there exist a set $\mathcal{L}_0$ of primes having positive Dirichlet density such that for any $l \in \mathcal{L}_0$, $\lambda|l$ and $v|l$, we have the following properties.

1 \ $\overline{r_{\lambda}}|_{G_{F(\zeta_l)}}$ is absolutely irreducible.

2 \ $r_{\lambda}|_{G_{F_v}}$ is potentially diagonalizable.

3 \ There exists a finite extension $F'_v/F_v$ such that $r_{\lambda}^c|_{G_{F'_v}} \sim r_{\lambda}^{\vee}\varepsilon_l^{-w} |_{G_{F'_v}}$.

\end{lem}

\begin{proof} By \cite[Lemma 1.2]{irr}, we may assume that $M$ is a CM field.
    
    Let $\widetilde{MF}$ be the Galois closure of $MF$ over $\mathbb{Q}$.
    
    We define the set $\mathcal{L}_0$ of primes consisting of all primes $l \in \mathcal{L}$ satisfying the following conditions.
    
    (a) \ $l \ge $max$\{ |h_{\tau} - h_{\tau}'| + 2 \mid \tau \in \mathrm{Hom}(F, \overline{M}), h_{\tau}, h_{\tau}' \in H_{\tau} \}$.
    
    (b) \  $l$ is unramified in $F$.
    
    (c) \ $l$ splits completely in $\widetilde{MF}^+$ and all $l$-adic places in $\widetilde{MF}^+$ are inert in $\widetilde{MF}$. (Note that in this case, all $l$-adic places of $F^+$ (resp. $M^+$) are inert in $F$ (resp. $M$).)

(d) \ $\overline{r_{\lambda}}|_{G_{F(\zeta_l)}}$ is absolutely irreducible.

    Then $\mathcal{L}_0$ has a Dirichlet density $\frac{1}{[\widetilde{MF}:\mathbb{Q}]}$ by Lemma \ref{irreducibility}.
    
    We fix $l \in \mathcal{L}_0$, $\lambda|l$ and $v|l$. Then $r_{\lambda}|_{G_{F_v}}$ is potentially diagonalizable by Proposition \ref{trace irreducible component}. There exists a finite place $\tilde{\lambda_1}$ (resp. $\tilde{\lambda_2}$) of $\widetilde{MF}$ such that $\tilde{\lambda_1}$ (resp. $\tilde{\lambda_2}$) lies above $v$ (resp. $\lambda$). Then there exists $\tau \in \mathrm{Gal}(\widetilde{MF}/\mathbb{Q})$ such that $\tau(\tilde{\lambda_1}) = \tilde{\lambda_2}$. Therefore, $\tau|_{F}$ induces $\tilde{\tau} : F_{v} \hookrightarrow \widetilde{MF}_{\tilde{\lambda_2}}$ over $\mathbb{Q}_l$.
    
    By the condition (c), the generator $c$ of $\mathrm{Gal}(\widetilde{MF}/\widetilde{MF}^+)$ induces the generator $\tilde{c}$ of $\mathrm{Gal}(\widetilde{MF}_{\tilde{\lambda_2}}/\mathbb{Q}_l)$. Note that we have $\tilde{\tau}^{-1} \tilde{c} \tilde{\tau}|_{G_{F}} = \tau^{-1} c \tau|_{F}$ and \ $^{\tilde{c}}r_{\lambda} \cong r_{\lambda}^{\vee} \varepsilon_l^{-w}$ since $\mathcal{R}$ is pure of weight $w$. Since $\widetilde{MF}$ is a CM field, we have $\tau^{-1} c \tau = c$. By Proposition \ref{trace irreducible component}, there exists a finite extension $F'_v/F_v$ such that $r_{\lambda}^c|_{G_{F'_v}} = r_{\lambda}^{\tau^{-1} c \tau}|_{G_{F'_v}} = (r_{\lambda}|_{G_{F_v}})^{\tilde{\tau}^{-1} \tilde{c} \tilde{\tau}}|_{G_{F'_v}} \sim \ ^{\tilde{c}}r_{\lambda}|_{G_{F_v'}} = r_{\lambda}^{\vee}\varepsilon_l^{-w}|_{G_{F'_v}}$.\end{proof}

\begin{lem} \label{decomposed genericity}

Let $\mathcal{R} = (M, S, \{ Q_v(X) \}, \{r_{\lambda}\}, \{ H_{\tau} \} )$ be a rank $n$ pure extremely weakly compatible system of $l$-adic representations of $G_F$ defined over $M$.

Then for almost all $l$ and $\lambda |l$, the representation $\overline{r_{\lambda}}$ is fully strongly decomposed generic.

\end{lem}

\begin{proof}

We take a prime $p$ which splits completely in $F$ and doesn't lie below $S$. Then for any $v|p$, the roots $\alpha_{v,1}, \cdots, \alpha_{v,n} \in \overline{M}$ of $Q_{v}(X)$ satisfies $|\iota(\frac{\alpha_{v,i}}{\alpha_{v,j}})| = |\iota(\frac{\alpha_{v,i}}{\alpha_{v^c,j}})| = 1$ for any $i, j$ and any $\iota : \overline{M} \hookrightarrow \mathbb{C}$. Thus, for almost all $l$ and $\iota : \overline{M} \hookrightarrow \overline{\mathbb{Q}}_l$, $\iota(\frac{\alpha_{v,i}}{\alpha_{v,j}} - p)$ and $\iota(\frac{\alpha_{v,i}}{\alpha_{v^c,j}} - p)$ are $l$-adic units for any $i, j$. \end{proof}

\begin{thm}\label{potential pure}

Let $\mathcal{R}=(M, S, \{ Q_v(X) \}, \{r_{\lambda}\}, \{ H_{\tau} \} )$ be a rank $n$ irreducible regular pure very weakly compatible system of $l$-adic representations of $G_F$ defined over $M$ and $F^{(a)}$ be a finite extension of $\mathbb{Q}$. We assume that there exists a rank $1$ very weakly totally even or odd compatible system $\mathcal{X}$ of $l$-adic representations of $G_{F^+}$ defined over $M$ such that $\mathcal{R} \cong \mathcal{R}^{\vee} \otimes \mathcal{X}|_{G_{F}}$.

Then there exists a finite Galois CM extension $L/\mathbb{Q}$ linearly disjoint from $F^{(a)}$ over $\mathbb{Q}$ such that $\mathcal{R}|_{G_{F'}}$ is automorphic. (We put $F':=LF$.)

Moreover, if $\mathcal{R}$ is geometrically pure, then there exists a set $\mathcal{L}_0$ of primes having positive Dirichlet density such that for any $l \in \mathcal{L}_0$, $\lambda | l$ and $v \nmid l$, the representation $\mathrm{WD}(r_{\lambda}|_{G_{F_v}})$ is pure.

\end{thm}

\begin{proof}

By Lemma \ref{Hecke field} and Lemma \ref{decomposed genericity}, there exists a set $\mathcal{L}_0$ of primes having positive Dirichlet density such that for any $l \in \mathcal{L}_0$, $\lambda|l$, $r_{\lambda}$ satisfies the condition of Theorem \ref{potential diagonalizable automorphy}. Therefore, for any $l \in \mathcal{L}_0$, $\lambda | l$ and $\iota : \overline{M_{\lambda}} \stackrel{\sim}{\rightarrow} \mathbb{C}$, we obtain a finite CM Galois extension $L/\mathbb{Q}$ linearly disjoint from $F^{(a)}$ over $\mathbb{Q}$ and a cohomological cuspidal automorphic representation $\pi$ of $\mathrm{GL}_n(\mathbb{A}_{F'})$ such that $r_{\iota}(\pi) \cong r_{\lambda}|_{G_{F'}}$ and $\iota\mathrm{WD}(r_{\iota}(\pi)|_{G_{F'_v}})^{F-ss} \cong \mathrm{rec}_{F'_v}(\pi_v|\mathrm{det}|^{\frac{1-n}{2}}_{v})$ for all $v \nmid l$. (We put $F':=LF$.) 

If $\mathcal{R}$ is geometrically pure, $\iota^{-1}\mathrm{rec}_{F'_v}(\pi_v|\mathrm{det}|_v^{\frac{1-n}{2}}) \cong \mathrm{WD}(r_{\lambda}|_{G_{F'_v}})^{F-ss}$ is pure for any $v \nmid l$ by 1 of Proposition \ref{purity local-global}. By 1 and 2 of Lemma \ref{purity lemma}, $\mathrm{WD}(r_{\lambda}|_{G_{F_v}})$ is pure for any $v \nmid l$. \end{proof}

\begin{cor}\label{potential symmetric power}
     
Let $\mathcal{R}=(M, S, \{ Q_v(X) \}, \{r_{\lambda}\}, \{ H_{\tau} \} )$ be a rank $2$ strongly irreducible regular pure very weakly compatible system of $l$-adic representations of $G_F$ defined over $M$ and $F^{(a)}$ be a finite extension of $\mathbb{Q}$. We suppose that $\mathrm{det}\mathcal{R}$ has a parallel weight.

Then for each positive integer $n$, there exists a finite Galois CM extension $L/\mathbb{Q}$ linearly disjoint from $F^{(a)}$ over $\mathbb{Q}$ such that $\mathrm{Symm}^n\mathcal{R}|_{G_{F'}}$ is automorphic. (We put $F':=LF$.)
    
\end{cor}

\begin{rem}

From this corollary, we can obtain a certain result about the distribution of the Frobenius eigenvalues as in \cite[Theorem 7.1.5 $\sim$ 7.1.7]{STH} and \cite[Theorem 7.2.3]{pw}. 

\end{rem}

\begin{proof} There exists an integer $w$ such that $\mathrm{det}\mathcal{R} \otimes \{ \varepsilon_l^{w} \}$ corresponds to a finite order character $\eta$ of $G_{F}$. By \cite[Lemma 5.3.3]{pw}, there exists a finite Galois totally real extension $L/\mathbb{Q}$ and a finite order character $\psi$ of $G_{FL}$ such that $\eta|_{G_{FL}} = \psi^2$. After replacing $F$ by $FL$ and twisting $\mathcal{R}$ by $\psi^{-1}$, we may assume $\mathrm{det}\mathcal{R} = \{ \varepsilon_l^{-w} \}$. By \cite[Lemma 7.1.3]{10}, there exists a set of primes $\mathcal{L}$ having Dirichlet density one such that for all $l \in \mathcal{L}$ and $\lambda|l$, the group $\overline{r_{\lambda}}(G_{F(\zeta_l)})$ contains a conjugate of $\mathrm{SL}_2(\mathbb{F}_l)$. This implies that for all positive integer $n$, Symm$^n\mathcal{R}$ is irreducible and Symm$^n\mathcal{R} \cong \mathrm{Symm}^n\mathcal{R}^{\vee} \otimes \{\varepsilon_l^{-wn}\}$. Therefore, the result follows by Theorem \ref{potential pure}. \end{proof}

\begin{thm} \label{regular pure}

Let $\mathcal{R}=(M, S, \{ Q_v(X) \}, \{r_{\lambda}\}, \{ H_{\tau} \} )$ be a rank $n$ pure irreducible very weakly compatible system of $l$-adic representations of $G_F$ defined over $M$. We assume $|\lambda - \lambda'| \ge n$ for all $\tau \in \mathrm{Hom}(F, \overline{M})$ and $\lambda \neq \lambda' \in H_{\tau}$. 

Then there exists a finite Galois CM extension $L/\mathbb{Q}$ linearly disjoint from $F^{(a)}$ over $\mathbb{Q}$ such that $\mathcal{R}|_{G_{F'}}$ is automorphic. (We put $F':=LF$.)

Moreover, if $\mathcal{R}$ is geometrically pure, then there exists a set $\mathcal{L}_0$ of primes having positive Dirichlet density such that for any $l \in \mathcal{L}_0$, $\lambda | l$ and $v \nmid l$, the representation $\mathrm{WD}(r_{\lambda}|_{G_{F_v}})$ is pure.

\end{thm}

\begin{proof}

    By Lemma \ref{Hecke field}, \ref{decomposed genericity} and Theorem \ref{potential tensor automorphy}, there exists a set of primes $\mathcal{L}_0$ having positive Dirichlet density such that for any $l \in \mathcal{L}_0$, $\lambda|l$ and $\iota: \overline{M_{\lambda}} \stackrel{\sim}{\rightarrow} \mathbb{C}$, there exist a finite CM extension $F'$ of $F$, a polarizable cohomological cuspidal $\pi_1$ of $\mathrm{GL}_{2n}(\mathbb{A}_{F'})$ and a cohomological cuspidal automorphic representation $\Pi$ of $\mathrm{GL}_{2n^2}(\mathbb{A}_{F'})$ such that $r_{\lambda}|_{G_{F'}} \otimes r_{\iota}(\pi_1) \cong r_{\iota}(\Pi)$, $\iota \mathrm{WD}(r_{\iota}(\Pi)|_{G_{F'_v}})^{F-ss} \cong \mathrm{rec}_{F'_v}(\Pi_v|\mathrm{det}|_v^{\frac{1-2n^2}{2}})$ for all $v \nmid l$ and $\pi_1$ is unramified at all finite places of $F'$ lying above $S$. By Theorem \ref{polarizable local-global compatibility} and Proposition \ref{purity local-global}, $\iota^{-1}\mathrm{rec}_{F'_v}(\Pi_v|\mathrm{det}|_v^{\frac{1-2n^2}{2}}) \cong \mathrm{WD}(r_{\lambda}|_{G_{F'_{v}}})^{F-ss} \otimes \mathrm{WD}(r_{\iota}(\pi_1)|_{G_{F'_v}})^{F-ss}$ is pure for all $v \nmid l$. By 1, 2, 5 of Lemma \ref{purity lemma} and Theorem \ref{polarizable local-global compatibility}, $\mathrm{WD}(r_{\lambda}|_{G_{F_v}})$ is pure for any $v \nmid l$. \end{proof}

\begin{thm}

Let $l$ be a prime such that $l > n^2$ and $r : G_{F} \rightarrow \mathrm{GL}_n(\overline{\mathbb{Q}}_l)$ be an algebraic $l$-adic representation. We assume the following conditions.

1 \ $\overline{r}|_{G_{F(\zeta_l)}}$ is absolutely irreducible.

2 \ $\overline{r}$ is fully decomposed generic.

3 \ $\zeta_l \notin F$.

4 \ For any $v \mid l$, $r|_{G_{F_v}}$ is ordinary.

5 \ For any $v \nmid l$, $g \in W_{F_v}$ and any eigenvalue $\alpha$ of $r(g)$, there exists an integer $w$ such that $\alpha$ is a Weil $q_v^{w}$-number.

Then for any $v \nmid l$, $\mathrm{WD}(r|_{G_{F_v}})$ is pure.
 
\end{thm}

\begin{proof} By Theorem \ref{potential ordinary automorphy}, there exist a finite CM extension $F'/F$ and a cohomological cuspidal automorphic representation $\pi$ of $\mathrm{GL}_n(\mathbb{A}_{F'})$ such that $r|_{G_{F'}} \cong r_{\iota}(\pi)$ such that $\iota \mathrm{WD}(r_{\iota}(\pi)|_{G_{F'_u}})^{F-ss} \cong \mathrm{rec}_{F_{u}'}(\pi_u|\mathrm{det}|_u^{\frac{1-n}{2}})$ for any $u \nmid l$. By Proposition \ref{purity local-global}, $\iota^{-1}\mathrm{rec}_{F_u'}(\pi_u|\mathrm{det}|_u^{\frac{1-n}{2}}) \cong \mathrm{WD}(r|_{G_{F_u'}})^{F-ss}$ is pure. By 1 and 2 of Lemma \ref{purity lemma}, $\mathrm{WD}(r|_{G_{F_v}})$ is pure.  \end{proof}

\printbibliography

Graduate School of Mathematical Sciences, The University of Tokyo, 3-8-1 Komaba, Meguro-ku, Tokyo 153-8914, Japan.

Email address: soccerkjr2252@g.ecc.u-tokyo.ac.jp

\end{document}